\documentclass[11pt,reqno]{amsart} 
\pdfoutput=1

\usepackage[foot]{amsaddr} 

\usepackage{tcolorbox}

\usepackage{bbm}

\usepackage{pbox}

\usepackage{soul}

\usepackage[displaymath, mathlines]{lineno}

\usepackage{microtype}

\usepackage[top=1.25in,bottom=1.25in,left=1.25in,right=1.25in]{geometry} 

\usepackage{xcolor}

\definecolor{my_red}{rgb}{0.5,0.0,0.0}
\definecolor{my_blue}{rgb}{0.0,0.0,0.6}
\definecolor{my_green}{rgb}{0.0,0.5,0.0}
\definecolor{light_gray}{gray}{0.6}

\usepackage{enumitem}

\usepackage[labelfont=bf]{caption}

\usepackage[font={small},labelfont=md]{subfig}

\usepackage[noadjust,nosort]{cite}

\usepackage{amsmath, amsthm, amssymb, tikz, mathtools, array, mathrsfs, tensor, ifthen, xparse, graphicx}

\usepackage{polynom}

\usepackage{esint}

\usepackage{multicol}

\usepackage{dsfont}
	\newcommand{\one}{\mathds{1}}

\usepackage{framed}

\usepackage{stackengine}

\usepackage{upref}

\usepackage[pagebackref,colorlinks=true,linkcolor=blue,citecolor=blue,urlcolor=light_gray]{hyperref}

\usepackage{orcidlink}

\usepackage{comment}

\allowdisplaybreaks

\numberwithin{equation}{section}


\newcommand{\eq}[1]{\begin{linenomath}\postdisplaypenalty=0\begin{align*} #1 \end{align*}\end{linenomath}}
\newcommand{\eeq}[1]{\begin{linenomath}\postdisplaypenalty=0\begin{align} \begin{split} #1 \end{split} \end{align}\end{linenomath}}
\newcommand{\eeqs}[1]{\begin{linenomath}\postdisplaypenalty=0\begin{align} #1 \end{align}\end{linenomath}}

\newcommand{\stackref}[2]{
\readlist*\mylist{#1}
\stackrel{\mbox{\footnotesize\foreachitem\x\in\mylist[]{\ifnum\xcnt=1\else,\fi\eqref{\x}}}}{#2}
}
\newcommand{\stackrefp}[2]{
\readlist*\mylist{#1}
\stackrel{\hphantom{\mbox{\footnotesize\foreachitem\x\in\mylist[]{\ifnum\xcnt=1\else,\fi\eqref{\x}}}}}{#2}
}
\newcommand{\stackrefpp}[3]{
\readlist*\mylist{#1}
\readlist*\mylistt{#2}
\stackrel{\parbox{\widthof{\footnotesize\foreachitem\x\in\mylistt[]{\ifnum\xcnt=1\else,\fi\eqref{\x}}}}{\centering\footnotesize\foreachitem\x\in\mylist[]{{\ifnum\xcnt=1\else,\fi\eqref{\x}}}}}{#3}
}

\def\eps{\varepsilon}
\def\vphi{\varphi}

\newcommand{\E}{\mathbb{E}}

\renewcommand{\P}{\mathbb{P}}

\newcommand{\R}{\mathbb{R}}

\newcommand{\Z}{\mathbb{Z}}

\renewcommand{\AA}{\mathcal{A}}
\newcommand{\BB}{\mathcal{B}}
\newcommand{\CC}{\mathcal{C}}
\newcommand{\DD}{\mathcal{D}}
\newcommand{\EE}{\mathcal{E}}
\newcommand{\FF}{\mathcal{F}}
\newcommand{\GG}{\mathcal{G}}

\newcommand{\JJ}{\mathcal{J}}

\newcommand{\MM}{\mathcal{M}}
\newcommand{\NN}{\mathcal{N}}

\newcommand{\PP}{\mathcal{P}}

\newcommand{\UU}{\mathcal{U}}

\newcommand{\XX}{\mathcal{X}}
\newcommand{\YY}{\mathcal{Y}}

\newcommand{\Asf}{\mathsf{A}}

\newcommand{\Gsf}{\mathsf{G}}
\newcommand{\Hsf}{\mathsf{H}}

\newcommand{\Tsf}{\mathsf{T}}

\newcommand{\ve}{\varepsilon}

\newcommand{\vc}[1]{{\boldsymbol #1}}

\newcommand{\wt}[1]{\widetilde{#1}}
\newcommand{\wh}[1]{\widehat{#1}}

\DeclareMathOperator{\dist}{dist}
\DeclareMathOperator{\Geo}{Geo}
\DeclareMathOperator{\e}{e} 

\newcommand{\ann}{\mathsf{ann}}
\newcommand{\crc}{\mathsf{cir}}

\renewcommand{\b}{b}
\newcommand{\chem}{\mathsf{ch}}

\newcommand{\givenk}[3][]{#1[ #2 \: #1| \: #3 #1]} 

\newcommand{\cc}{\mathsf{c}} 
\newcommand{\dd}{\mathrm{d}} 

\newcommand\mydots{\hbox to 1em{.\hss.\hss.}}

\newcommand{\f}{\frac}
\newcommand{\lng}{\mathcal N}
\newcommand{\intr}{\mathsf{int}}
\newcommand{\ext}{\mathsf{ext}}
\newcommand{\dualBD}{\operatorname{Bd}^\star}
\newcommand{\primBD}{\operatorname{Bd}}
\newcommand{\dualp}{\zeta}
\newcommand{\primalp}{\gamma}

\newcommand{\incir}{\mathcal I}
\newcommand{\outcir}{\mathcal I}

\newcommand{\dincir}{\mathcal U}
\newcommand{\doutcir}{\mathcal U}

\DeclarePairedDelimiter\ceil{\lceil}{\rceil}
\DeclarePairedDelimiter\floor{\lfloor}{\rfloor}

            \makeatletter
            \DeclareFontFamily{OMX}{MnSymbolE}{}
            \DeclareSymbolFont{MnLargeSymbols}{OMX}{MnSymbolE}{m}{n}
            \SetSymbolFont{MnLargeSymbols}{bold}{OMX}{MnSymbolE}{b}{n}
            \DeclareFontShape{OMX}{MnSymbolE}{m}{n}{
                <-6>  MnSymbolE5
               <6-8.5>  MnSymbolE6
               <7-8.5>  MnSymbolE7
               <8-8.5>  MnSymbolE8
               <9-10> MnSymbolE9
              <10-12> MnSymbolE10
              <12->   MnSymbolE12
            }{}
            \DeclareFontShape{OMX}{MnSymbolE}{b}{n}{
                <-6>  MnSymbolE-Bold5
               <6-8.5>  MnSymbolE-Bold6
               <7-8.5>  MnSymbolE-Bold7
               <8-8.5>  MnSymbolE-Bold8
               <9-10> MnSymbolE-Bold9
              <10-12> MnSymbolE-Bold10
              <12->   MnSymbolE-Bold12
            }{}
            
            \let\llangle\@undefined
            \let\rrangle\@undefined
            \DeclareMathDelimiter{\llangle}{\mathopen}%
                                 {MnLargeSymbols}{'164}{MnLargeSymbols}{'164}
            \DeclareMathDelimiter{\rrangle}{\mathclose}%
                                 {MnLargeSymbols}{'171}{MnLargeSymbols}{'171}
            \makeatother
            
    \DeclareFontFamily{U}{matha}{\hyphenchar\font45}
    \DeclareFontShape{U}{matha}{m}{n}{ <-6> matha5 <6-8.5> matha6 <7-8.5>
    matha7 <8-8.5> matha8 <9-10> matha9 <10-12> matha10 <12-> matha12 }{}
    \DeclareSymbolFont{matha}{U}{matha}{m}{n}
    \DeclareFontFamily{U}{mathx}{\hyphenchar\font45}
    \DeclareFontShape{U}{mathx}{m}{n}{ <-6> mathx5 <6-8.5> mathx6 <7-8.5>
    mathx7 <8-8.5> mathx8 <9-10> mathx9 <10-12> mathx10 <12-> mathx12 }{}
    \DeclareSymbolFont{mathx}{U}{mathx}{m}{n}
    
    \DeclareMathDelimiter{\llbrack} {4}{matha}{"76}{mathx}{"30}
    \DeclareMathDelimiter{\rrbrack} {5}{matha}{"77}{mathx}{"38}
            
\newcommand{\ball}{\mathsf{Ball}}
\newcommand{\jor}{\mathcal{J}}
\newcommand{\opc}{\mathcal{P}}
\newcommand{\cdc}{\mathcal{D}}
\newcommand{\cir}{\mathcal{C}}
\newcommand{\meet}{\wedge}
\newcommand{\join}{\vee}

\newtheorem{theorem}{Theorem}[section]
\newtheorem{proposition}[theorem]{Proposition}

\newtheorem{lemma}[theorem]{Lemma}
\newtheorem{claim}[theorem]{Claim}

\newtheorem{theirthm}{Theorem} 

\theoremstyle{definition}
\newtheorem{definition}[theorem]{Definition}
\newtheorem{example}[theorem]{Example}

\newtheorem{remark}[theorem]{Remark}

\newenvironment{proofclaim}[1][Proof]
	{\begin{proof}[#1]\renewcommand{\qedsymbol}{$\square$ (Claim)}}
	{\end{proof}\renewcommand{\qedsymbol}{$\square$}}

\usepackage{xargs}
\usepackage[colorinlistoftodos,prependcaption,textsize=footnotesize]{todonotes}
\newcommandx{\addmath}[2][1=]{\todo[linecolor=red,backgroundcolor=red!25,bordercolor=red,#1]{#2}}
\newcommandx{\fixtext}[2][1=]{\todo[linecolor=blue,backgroundcolor=blue!25,bordercolor=blue,#1]{#2}}
\newcommandx{\note}[2][1=]{\todo[linecolor=yellow,backgroundcolor=yellow!25,bordercolor=yellow,#1]{#2}}

\usepackage{xcolor}
\usepackage{graphicx}
\usepackage{lipsum}

\RequirePackage{color}
\definecolor{darkblue}{rgb}{0.0,0.0,0.7}
\definecolor{darkred}{rgb}{0.5,0.0,0.0}
\definecolor{darkgreen}{rgb}{0.0,0.5,0.0}
\definecolor{indigo}{rgb}{0.3,0,0.5}

\allowdisplaybreaks

\title[Geodesic length in critical FPP]{An upper bound on geodesic length in 2D critical first-passage percolation}



\author[E. Bates]{Erik Bates$^*$\,\orcidlink{0000-0002-3472-036X}}
\address{$^*$Department of Mathematics, North Carolina State University}
\email{ebates@ncsu.edu}

\author[D. Harper]{David Harper$^{\dagger}$}
\address{$^\dagger$Department of Mathematics, Georgia Institute of Technology}
\email{dharper40@gatech.edu}

\author[X. Shen]{Xiao Shen$^{*}$}
\email{xshen9@ncsu.edu}

\author[E. Sorensen]{Evan Sorensen$^\ddagger$}
\address{$^\ddagger$Department of Mathematics, Columbia University}
\email{es4203@columbia.edu}

\allowdisplaybreaks

\begin{document}


\begin{abstract}
We consider i.i.d.~first-passage percolation (FPP) on the two-dimensional square lattice, in the critical case where edge-weights take the value zero with probability $\tfrac{1}{2}$.
Critical FPP is unique in that the Euclidean lengths of geodesics are superlinear---rather than linear---in the distance between their endpoints.
This fact was speculated by Kesten in 1986 but not confirmed until 2019 by Damron and Tang, who showed a lower bound on geodesic length that is polynomial with degree strictly greater than $1$.
In this paper we establish the first nontrivial upper bound.
Namely, we prove that for a large class of critical edge-weight distributions, the shortest geodesic from the origin to a box of radius $R$ uses at most $R^{2+\eps}\pi_3(R)$ edges with high probability, for any $\eps > 0$.
Here $\pi_3(R)$ is the polychromatic 3-arm probability from classical Bernoulli percolation; upon inserting its conjectural asymptotic, our bound converts to $R^{4/3 + \eps}$. In any case, it is known that $\pi_3(R) \lesssim R^{-\delta}$ for some $\delta > 0$, so our bound gives an exponent strictly less than $2$.
In the special case of Bernoulli($\tfrac{1}{2}$) edge-weights, we replace the additional factor of $R^\eps$ with a constant and give an expectation bound.
\end{abstract}

\maketitle
\setcounter{tocdepth}{1} 


\tableofcontents 
\pagebreak
\section{Introduction}

\subsection{The model of critical first-passage percolation (FPP)}

Let $E(\Z^2)$ denote the edge set of the square lattice $\Z^2$.
Consider a family of i.i.d.~random variables $(t_e)_{e\in E(\Z^2)}$ defined on some probability space $(\Omega,\FF,\P)$ such that
\eeq{ \label{critical_assumption}
\P(t_e<0) = 0 \quad \text{and} \quad \P(t_e=0) = \tfrac{1}{2}.
}
We say that $t_e$ is the \textit{weight} of edge $e$. For each pair $x,y\in\Z^2$, let $\PP(x,y)$
denote the collection of all self-avoiding nearest-neighbor paths starting at $x$ and ending at $y$.
The \textit{passage time} between $x$ and $y$ is the random quantity
\eeq{ \label{fpp_def}
T(x,y) = \inf_{\primalp \in \PP(x,y)} T(\primalp), \quad \text{where} \quad T(\primalp) = \sum_{e\in\primalp} t_e.
}
When $x=y$, we allow the empty path so that $T(x,x)=0$.
The map $T(\cdot\,,\cdot)$ is thus a pseudometric on $\Z^2$, and it naturally extends to sets:
for $\mathcal A,\mathcal B\subseteq\Z^2$, we define
\eq{
T(\mathcal A,\mathcal B)=\inf_{x\in \mathcal A,\, y\in \mathcal B} T(x,y).
}
A path $\primalp$ is said to be a \textit{geodesic} between $\mathcal{A}$ and $\mathcal{B}$ if $\primalp$ starts at a vertex in $\mathcal{A}$, ends at a vertex in $\mathcal{B}$, and achieves the minimal passage time $T(\primalp) = T(\mathcal A,\mathcal B)$.
It is known that with probability one, geodesics exist between every pair of points \cite[Cor.~1.3]{wierman_reh78}. 
Therefore, geodesics exist between any two finite sets $\mathcal{A}$ and $\mathcal{B}$.

This model can be projected to classical Bernoulli percolation by declaring that all edges $e$ with $t_e = 0$ are \textit{open}, while those with $t_e>0$ are \textit{closed}. 
The assumption \eqref{critical_assumption} means that the resulting projection is \textit{critical} percolation; in particular, there is no infinite connected cluster of open edges.
Consequently, for $x$ and $y$ far apart, a geodesic between $x$ and $y$ will typically use a large number of open edges without penalty, but will also need to traverse a small number of closed edges to go between distinct open clusters.
On the criticality assumption \eqref{critical_assumption}, Kesten \cite[Sec.~9.24]{kesten86a} conjectured that the length of a geodesic (i.e.~the total number of edges it contains) between the origin and $x$ grows superlinearly in $\|x\|$.
This conjecture was verified by Damron and Tang \cite{damron_tang19}.
To state their result precisely, we let $\lng_{\vc 0,x}$ denote the minimum length of a geodesic between the origin (denoted $\vc 0 = (0,0)$) and $x$.

\begin{theirthm} \textup{\cite[Thm.~1]{damron_tang19}} \label{damron_tang_thm}
Assume \eqref{critical_assumption}. 
There exist constants $C,c>0$ and $\beta>1$ such that
\eeq{ \label{damron_tang_thm_eq}
\P(\lng_{\vc 0,x}\le \|x\|^\beta_1) \leq C\exp(-\|x\|_1^c) \quad \text{for all $x\in\Z^2$}.
}
\end{theirthm}

It should be emphasized that this result is special to the critical case, and at least superficially to two dimensions.
Indeed, if $\P(t_e=0)\neq \tfrac{1}{2}$, then $\NN_{\vc 0,x}$ grows linearly in $\|x\|$ and is even known in some cases to  satisfy a law of large numbers if $x$ is brought to infinity along a fixed direction (see \cite[Sec.~1.5]{bates24} and references therein, including \cite[Thm.~4.9]{auffinger_damron_hanson17} 
for the subcritical case, and \cite[Thm.~4]{zhang95} for supercritical).
What remains unsettled by Theorem \ref{damron_tang_thm} is the exact growth rate of $\lng_{\vc 0,x}$ in critical FPP, as there is no matching upper bound.
In fact, until now, no upper bound whatsoever has been established.


\subsection{Main results for Bernoulli weights} \label{ber_res_sec}
The most well-studied case of critical FPP is that of Bernoulli weights:
\eeq{
\label{Ber}
\P(t_e = 0) = \P(t_e = 1) = \tfrac{1}{2}.
}
Our results are strongest in this case.
The estimates we provide for geodesic length are given in terms of arm events, which are of fundamental interest in the study of Bernoulli percolation and not immediately connected to FPP.
We now recall the relevant definitions.


Define the dual lattice $\wh \Z^2 = \Z^2 + (\f{1}{2},\f{1}{2})$.
For clarity, we will often refer to $\Z^2$ as the primal lattice.
The shift by $(\f{1}{2},\f{1}{2})$ means that for each primal edge $e\in E(\Z^2)$, there is a unique dual edge $e^\star\in E(\wh\Z^2)$ that bisects it.
We say that the endpoints of $e^\star$ are the \textit{dual neighbors} of $e$, and similarly, the endpoints of $e$ are the dual neighbors of $e^\star$. We also define the dual neighbors of a \textit{vertex} $v$, which are the four points on the dual lattice closest to $v$. 
Once the edges of $\Z^2$ are given open or closed status, their dual edges are given the same status.
That is, we declare $e^\star$ to be open if $e$ is open, or closed if $e$ is closed.
We can then speak of open or closed paths; all open paths we consider are on the primal lattice, while all closed paths are on the dual lattice.
Furthermore, it will often be useful to identify self-avoiding paths with the simple curves their edges trace out in the plane.
For instance, this identification makes defining 
disjointness 
very intuitive:
two paths are \textit{disjoint} if their associated curves are disjoint.\footnote{If the two paths are on the same lattice, then disjointness is equivalent to sharing no vertices; if on different lattices, it is equivalent to sharing no dual pair of edges (i.e.~never crossing).} 

Consider the box $B_R = [-R,R]^2\cap\Z^2$.
The boundary of $B_R$ is written $\partial B_R = B_R \setminus B_{R-1}$. 
Let $\pi_3(R)$ denote the probability of the following event depicted in Figure \ref{fig:3arm}:
\begin{figure}[t]

\begin{center}

\tikzset{every picture/.style={line width=0.75pt}} 

\begin{tikzpicture}[x=0.75pt,y=0.75pt,yscale=-1,xscale=1]

\draw  [color={rgb, 255:red, 155; green, 155; blue, 155 }  ,draw opacity=1 ] (154.33,80.15) -- (315.25,80.15) -- (315.25,241.07) -- (154.33,241.07) -- cycle ;
\draw [line width=2.25]    (230.12,160.12) -- (240.12,160.17) ;
\draw    (230.12,160.12) .. controls (201.41,164.3) and (184.33,112.68) .. (154.15,122.31) ;
\draw    (240.12,160.17) .. controls (276.81,138.76) and (285.52,214.96) .. (315.07,230.07) ;
\draw  [dash pattern={on 0.84pt off 2.51pt}]  (235.64,85.97) .. controls (213.09,97.97) and (241.27,138.33) .. (234.91,156) ;

\draw (230.12,162.52) node [anchor=north west][inner sep=0.75pt]    {$e$};
\draw (157.31,222.87) node [anchor=north west][inner sep=0.75pt]    {$\partial B_{R}$};

\end{tikzpicture}
\caption{The polychromatic 3-arm event at edge $e = \{(0,0), (1,0)\}$: two open primal paths (shown solid) and one closed dual path (shown dashed).}\label{fig:3arm}
\end{center}
\end{figure}
there exist two primal paths and one dual path that are disjoint and satisfy the following conditions: 
\begin{itemize}
    \item The two primal paths are open, start at $(0,0)$ and $(1,0)$ respectively, and both end at $\partial B_R$.
    \item The dual path is closed, starts at either $(\f{1}{2},\frac{1}{2})$ or $(\f{1}{2},-\frac{1}{2})$, and eventually reaches a dual neighbor of some vertex belonging to $\partial B_R$.
\end{itemize}
In the percolation literature, this is called a 3-arm event, and $\pi_3(R)$ is the 3-arm probability at distance $R$.
Our bounds are given in terms of this quantity.

Let $\lng_R$ denote the length of the shortest geodesic from the origin to $\partial B_R$.


\begin{theorem}[Bernoulli weights, point-to-box] \label{thm:Ber_ptb}
 
Assume \eqref{Ber}. Then there exists a constant $C$ such that for all $R\geq1$,
\eeq{ \label{ber_expectation_bound}
\E[\lng_R] \le CR^2 \pi_3(R).
}
\end{theorem}

Our next result considers a more general setting.
For finite sets $\mathcal{A},\mathcal{B}\subseteq\Z^2$, let $\lng_{\mathcal A,\mathcal B}$ denote the minimum length of a geodesic from $\mathcal{A}$ to $\mathcal{B}$. 
Denote the $\ell^\infty$ Euclidean distance between the two sets by
\eeq{ \label{dist_def}
\dist(\mathcal A,\mathcal B) = \inf\{\|x-y\|_\infty:\, x\in\mathcal A, y\in\mathcal B\}.
}
When $\mathcal A$ and $\mathcal B$ are single points, the following gives a result for point-to-point geodesics.  


\begin{theorem}[Bernoulli weights, set-to-set] \label{thm:Ber_ptp}
Assume \eqref{Ber}.
Let $\mathcal{A}$ and $\mathcal{B}$ be disjoint finite connected sets of vertices, and let $R = \dist(\mathcal A, \mathcal{B})$. 
There exist constants $C,c > 0$, independent of $\mathcal A$ and $\mathcal B$, such that
\eeq{ \label{cm33}
\P(\lng_{\mathcal A,\mathcal B} \ge \theta R^2\pi_3(R)) \le C(|\mathcal A| + |\mathcal B|)^3\theta^{-c} \quad \text{for all $\theta >  0$}.
}
\end{theorem}

The exact decay rate of $\pi_3(R)$ as $R\to\infty$ has not been established on the square lattice, but for site percolation on the triangular lattice, it is known that $\pi_3(R) = R^{-2/3+o(1)}$ \cite[Thm.~4]{smirnov_werner01} (see also \cite[Thm.~21]{nolin08}). 
It is widely believed that the same asymptotic holds on the square lattice, although we are not aware of any rigorous estimates better than those 
mentioned in Remark~\ref{rmk:3armprob}.

\begin{remark}[For fixed $R$, the tail must be heavy] \label{2m_rmk}
In the case $|\mathcal A| = |\mathcal B| = 1$, Theorem \ref{thm:Ber_ptp} is analogous to \cite[Cor.~2]{damron_hanson_sosoe16} concerning chemical distance in critical percolation. We do not obtain a precise value of $c$ in the proof, but we explain here why $c$ cannot exceed the value $2$ when $\mathcal{A} = \{(0,0)\}$ and $\mathcal{B} = \{(1,0)\}$. Indeed, it was shown in \cite[Prop.~3]{damron_hanson_sosoe16} that
\[
\mathbb E\givenk[\big]{(\lng_{(0,0),(1,0)})^2}{(0,0) \leftrightarrow (1,0)} = \infty,
\]
where $(0,0) \leftrightarrow (1,0)$ denotes the event that the two points are connected by an open path.
In particular, $\E\big[(\lng_{(0,0),(1,0)})^2\big]=\infty$.
But if $c$ in \eqref{cm33} were greater than $2$, then $\lng_{(0,0),(1,0)}$ would have a finite second moment.
\end{remark}

\subsection{Main results for general edge-weights} \label{rsl_general}
To obtain bounds for geodesic length beyond the Bernoulli case, we require either of two possible assumptions on the distribution function $F$ of the edge-weights. 
These are \eqref{eq : intro_limit_bigger_than_one} and \eqref{eq : intro_limit_equal_zero} stated below.
These conditions involve quantities $(p_R)_{R\geq1}$ which are standard in near-critical percolation.
We now recall their definition.

Consider i.i.d.~random variables $(U_e)_{e \in E(\Z^2)}$, each uniformly distributed on $[0,1]$. 
For a specific choice of distribution function $F$, we can realize the edge-weights as $t_e = F^{-1}(U_e)$, where $F^{-1}$ is the usual quantile function:
\eq{
F^{-1}(u) = \inf\{t\in\R:\, F(t)\geq u\}, \quad u\in[0,1].
}
We say a path $\gamma$ is \textit{$p$-open} if $U_e \le p$ for each edge $e \in \gamma$. 
With $p_\cc=\tfrac{1}{2}$, the open paths we have defined previously are the same as $p_\cc$-open paths defined here.
%
 For a positive integer $R\ge1$ and $p \in (p_\cc , 1 ]$, let 
	\eeq{ \label{cross_prob}
	\sigma(R,p) = \P( \text{$\exists$ $p$-open left-right crossing of } [-R,R] \times [-R,R] ), 
	}
	where ``left-right crossing of $[a,b] \times [m,n]$'' means a path in $[a,b] \times [m,n]$ that starts in $\{a\} \times [m,n]$ and ends in $\{b\} \times [m,n]$. For $ \eps > 0 $ and $p > p_\cc ,$ define
	\eeq{ \label{corr_L}
		L(p,\eps) = \min \{ R\geq 1:\, \sigma(R,p) \ge 1 - \eps \} .
	}
	This $L(p,\eps)$ is called the (finite-size scaling) correlation length, measuring the scale up to which the $p$-open model ``looks critical.''
    Indeed, at criticality the Russo--Seymour--Welsh (RSW) theorem \cite[Chap.~6]{kesten82} states that for any aspect ratio $\alpha\in(0,\infty)$, there exists $\eps_\alpha>0$ such that for every $R\ge1$,
    \eeq{ \label{RSW}
    \eps_\alpha < \P( \text{$\exists$ $p_\cc$-open left-right crossing of } \big[-\ceil{\alpha R},\ceil{\alpha R}\big] \times [-R,R] ) < 1-\eps_\alpha.
    }
    We will always assume the $\eps$ in \eqref{corr_L} is less than $\eps_{1}$, so that 
    \eeq{ \label{L_lim}
    \lim_{p\searrow p_\cc}L(p,\eps) = \infty.
    }
    
    It is shown in \cite[disp.~(1.24)]{kesten87} that there exist $C,c,\eps_\star > 0$ such that for all $0 < \eps , \eps' \leq \eps_\star$ and  $p\in(p_\cc,p_\cc+\eps_\star]$, we have $cL(p,\eps)\le L(p,\eps') \le CL(p,\eps)$.
 We therefore write $L(p) = L(p,\eps_\star) $ with this fixed $\eps_\star$, as is customary. For $R \geq 1 $, define
	\begin{equation} \label{def_pn}
		p_R = \inf \{ p > p_\cc :\, L(p) \leq R \}.
	\end{equation}
In other words, $p_R$ gives the largest value of $p$ such that the $p$-open model ``looks critical'' up to scale $R$.
 Therefore, the two maps $R\mapsto p_R$ and $p\mapsto L(p)$ should be thought of as roughly inverse to each other; see \eqref{eq: L_p_n}.
 Note that \eqref{L_lim} implies
 \eeq{ \label{k3r78b}
 p_R > p_\cc \quad \text{for all $R\ge1$.}
 }

\begin{theorem}[General weights, point-to-box]
\label{thm:general_thm}
Assume that the edge-weight distribution function $F$ satisfies \eqref{critical_assumption} and one of the following two conditions:
\eeqs{
\limsup_{n\to\infty} F^{-1}(p_{\b^{n+1}})/F^{-1}(p_{\b^{n}}) < 1  \label{eq : intro_limit_bigger_than_one} \quad \text{for some integer $\b\ge2$,} \\
\text{or} \quad \liminf_{n\to\infty} F^{-1}(p_{\b^{2n}})/F^{-1}(p_{\b^{n}}) > 0 \quad \text{for some integer $\b\ge2$}. \label{eq : intro_limit_equal_zero}     
}
Then for any $\ve > 0$, there exist constants $C,c,s>0$ such that for all $\theta \ge 1$ and $R\ge 1$, 
\eeq{ \label{mxi4}
\P( \lng_R \geq \theta R^{ 2 + \ve} \pi_3(R) ) \leq C \e^{-c (\log (\theta R))^s }  .
}
\end{theorem}

\begin{remark}[Discussion of assumptions]
Since $p_{R} \searrow \f{1}{2}$ as $R \to \infty$ (see \eqref{pn-pcbd}), the conditions \eqref{eq : intro_limit_bigger_than_one} and \eqref{eq : intro_limit_equal_zero} depend only on the behavior of $F$ near its atom at $0$.
Similar conditions have been used in \cite{damron_lam_wang17} for characterizing the asymptotic behavior of $T(\vc 0, \partial B_R)$ as $R\to\infty$.
In \eqref{eq : intro_limit_bigger_than_one} and \eqref{eq : intro_limit_equal_zero}, we demand that $b$ is an integer simply for convenience.
Otherwise one must round powers to $\b$ to nearby integers, and our proofs could be modified provided $\b>1$.
Allowing general $b$ (as opposed to fixing $\b=2$, say) ensures greater generality.
While it would be ideal to eliminate these assumptions entirely, what is striking is that \eqref{eq : intro_limit_bigger_than_one} and \eqref{eq : intro_limit_equal_zero} represent opposite---indeed, mutually exclusive---regimes for the behavior of $F$ near $0$.
Somewhat surprisingly, either of these opposite possibilities is enough to prove the same upper bound \eqref{mxi4}.
\end{remark}

\begin{example}[Allowable edge-weight distributions] \label{gnwig}
Here we provide a broad class of edge-weight distributions satisfying either \eqref{eq : intro_limit_bigger_than_one} or \eqref{eq : intro_limit_equal_zero}.
In each example below, it is assumed that $F(t) = 0$ for $t < 0$ and $F(0) = p_\cc=\tfrac12$. 
We specify how $F(t)$ behaves for small positive $t$.
Roughly speaking, \eqref{eq : intro_limit_bigger_than_one} suggests that as $\eps\searrow0$, the value of $F^{-1}(\tfrac{1}{2}+\eps)$ decreases towards $0$ relatively quickly.
In other words, the graph of $F$ is ``not too flat'' to the right of $0$.
By contrast, \eqref{eq : intro_limit_equal_zero} indicates that the graph of $F$ is quite flat to the right of $0$, so that $F^{-1}(\tfrac{1}{2}+\eps)$ vanishes relatively slowly (or possibly not at all) as $\eps\searrow0$.
\begin{enumerate}[label=\textup{(\alph*)}]

\item \label{gnwig_a}
Suppose there are constants $\alpha,\beta,h>0$ such that $F(t) = p_\cc + \alpha t^\beta$ for all $t\in[0,h]$; see Figure \ref{fig:assump} (left).
Then $F$ satisfies \eqref{eq : intro_limit_bigger_than_one}.

\item \label{gnwig_b}
Suppose the support of the edge-weight distribution has a gap above $0$, i.e.\ there exists $h>0$ such that $F(0)=F(h)=p_\cc$; see Figure \ref{fig:assump} (right).
Then $F$ satisfies \eqref{eq : intro_limit_equal_zero}.
In particular, the Bernoulli distribution \eqref{Ber} satisfies \eqref{eq : intro_limit_equal_zero}.

\item \label{gnwig_c}
Suppose there are constants $\alpha,\beta,h> 0$ such that $F(t) = p_\cc + \alpha\e^{-t^{-\beta}}$  for all $t \in (0,h]$. 
Then $F$ satisfies \eqref{eq : intro_limit_equal_zero}.

\item \label{gnwig_d}
Not all distributions satisfy one of the conditions. 
For example, since $L(p_R)\nearrow\infty$ as $R\to\infty$, there exists a distribution $F$ with $F(0) = p_\cc$ and $F^{-1}(p_R)=\e^{-\sqrt{\log_2 L(p_R)}}$ for all large $R$. 
Such a distribution satisfies neither \eqref{eq : intro_limit_bigger_than_one} nor \eqref{eq : intro_limit_equal_zero}. 
\end{enumerate}
The proof that these examples have the stated properties is given in Section \ref{sec:corr_length}. 
\end{example}

\begin{figure}[t]
\begin{center}

\tikzset{every picture/.style={line width=0.75pt}} 

\begin{tikzpicture}[x=0.75pt,y=0.75pt,yscale=-1,xscale=1]

\draw  (114,194.81) -- (300.1,194.81)(138.38,64.4) -- (138.38,208.5) (293.1,189.81) -- (300.1,194.81) -- (293.1,199.81) (133.38,71.4) -- (138.38,64.4) -- (143.38,71.4)  ;
\draw [line width=2.25]    (67.27,193.9) -- (138.35,194.09) ;
\draw  [fill={rgb, 255:red, 255; green, 255; blue, 255 }  ,fill opacity=1 ] (133.9,194.09) .. controls (133.9,191.63) and (135.89,189.64) .. (138.35,189.64) .. controls (140.81,189.64) and (142.8,191.63) .. (142.8,194.09) .. controls (142.8,196.55) and (140.81,198.54) .. (138.35,198.54) .. controls (135.89,198.54) and (133.9,196.55) .. (133.9,194.09) -- cycle ;
\draw [color={rgb, 255:red, 155; green, 155; blue, 155 }  ,draw opacity=1 ] [dash pattern={on 0.84pt off 2.51pt}]  (138.31,149.26) -- (282.06,149.28) ;
\draw  [fill={rgb, 255:red, 0; green, 0; blue, 0 }  ,fill opacity=1 ] (133.86,149.26) .. controls (133.86,146.8) and (135.85,144.81) .. (138.31,144.81) .. controls (140.76,144.81) and (142.76,146.8) .. (142.76,149.26) .. controls (142.76,151.72) and (140.76,153.71) .. (138.31,153.71) .. controls (135.85,153.71) and (133.86,151.72) .. (133.86,149.26) -- cycle ;
\draw  (403.22,194.81) -- (589.33,194.81)(427.61,64.4) -- (427.61,208.5) (582.33,189.81) -- (589.33,194.81) -- (582.33,199.81) (422.61,71.4) -- (427.61,64.4) -- (432.61,71.4)  ;
\draw [line width=2.25]    (356.5,193.9) -- (427.57,194.09) ;
\draw  [fill={rgb, 255:red, 255; green, 255; blue, 255 }  ,fill opacity=1 ] (423.12,194.09) .. controls (423.12,191.63) and (425.12,189.64) .. (427.57,189.64) .. controls (430.03,189.64) and (432.03,191.63) .. (432.03,194.09) .. controls (432.03,196.55) and (430.03,198.54) .. (427.57,198.54) .. controls (425.12,198.54) and (423.12,196.55) .. (423.12,194.09) -- cycle ;
\draw [color={rgb, 255:red, 155; green, 155; blue, 155 }  ,draw opacity=1 ] [dash pattern={on 0.84pt off 2.51pt}]  (427.57,149.09) -- (571.33,149.11) ;
\draw [line width=2.25]    (427.53,149.26) -- (517.42,148.99) ;
\draw  [fill={rgb, 255:red, 255; green, 255; blue, 255 }  ,fill opacity=1 ] (512.91,148.91) .. controls (512.91,146.45) and (514.91,144.46) .. (517.36,144.46) .. controls (519.82,144.46) and (521.81,146.45) .. (521.81,148.91) .. controls (521.81,151.37) and (519.82,153.36) .. (517.36,153.36) .. controls (514.91,153.36) and (512.91,151.37) .. (512.91,148.91) -- cycle ;
\draw [line width=2.25]    (517.24,99.34) -- (589.57,99.19) ;
\draw [shift={(594.57,99.18)}, rotate = 179.88] [fill={rgb, 255:red, 0; green, 0; blue, 0 }  ][line width=0.08]  [draw opacity=0] (14.29,-6.86) -- (0,0) -- (14.29,6.86) -- cycle    ;
\draw  [fill={rgb, 255:red, 0; green, 0; blue, 0 }  ,fill opacity=1 ] (512.79,99.34) .. controls (512.79,96.89) and (514.79,94.89) .. (517.24,94.89) .. controls (519.7,94.89) and (521.69,96.89) .. (521.69,99.34) .. controls (521.69,101.8) and (519.7,103.79) .. (517.24,103.79) .. controls (514.79,103.79) and (512.79,101.8) .. (512.79,99.34) -- cycle ;
\draw  [fill={rgb, 255:red, 0; green, 0; blue, 0 }  ,fill opacity=1 ] (423.08,149.26) .. controls (423.08,146.8) and (425.07,144.81) .. (427.53,144.81) .. controls (429.99,144.81) and (431.98,146.8) .. (431.98,149.26) .. controls (431.98,151.72) and (429.99,153.71) .. (427.53,153.71) .. controls (425.07,153.71) and (423.08,151.72) .. (423.08,149.26) -- cycle ;
\draw [line width=2.25]    (138.31,149.26) .. controls (241.64,149.17) and (219.33,108.08) .. (268.1,106.22) ;
\draw [shift={(272.83,106.17)}, rotate = 180.69] [fill={rgb, 255:red, 0; green, 0; blue, 0 }  ][line width=0.08]  [draw opacity=0] (14.29,-6.86) -- (0,0) -- (14.29,6.86) -- cycle    ;
\draw [line width=2.25]    (138.31,149.26) .. controls (176.92,123.97) and (200.77,107.1) .. (226.91,104.76) ;
\draw [shift={(231.5,104.5)}, rotate = 178.64] [fill={rgb, 255:red, 0; green, 0; blue, 0 }  ][line width=0.08]  [draw opacity=0] (14.29,-6.86) -- (0,0) -- (14.29,6.86) -- cycle    ;
\draw [line width=2.25]    (138.31,149.26) .. controls (138.49,106.87) and (166.48,104.92) .. (190.88,105.11) ;
\draw [shift={(195.5,105.17)}, rotate = 180.75] [fill={rgb, 255:red, 0; green, 0; blue, 0 }  ][line width=0.08]  [draw opacity=0] (14.29,-6.86) -- (0,0) -- (14.29,6.86) -- cycle    ;

\draw (119.67,134.4) node [anchor=north west][inner sep=0.75pt]    {$\frac{1}{2}$};
\draw (124.5,43.4) node [anchor=north west][inner sep=0.75pt]    {$F( t)$};
\draw (305.67,183.57) node [anchor=north west][inner sep=0.75pt]    {$t$};
\draw (408.89,134.4) node [anchor=north west][inner sep=0.75pt]    {$\frac{1}{2}$};
\draw (413.73,43.4) node [anchor=north west][inner sep=0.75pt]    {$F( t)$};
\draw (594.89,183.57) node [anchor=north west][inner sep=0.75pt]    {$t$};

\end{tikzpicture}

\caption{\small An illustration of conditions \eqref{eq : intro_limit_bigger_than_one} and \eqref{eq : intro_limit_equal_zero}. \textit{Left:} Three distribution functions that behave near $t=0$ as $\tfrac{1}{2} + t^{1/10}$, $\tfrac{1}{2} + t$, and $\tfrac{1}{2} + t^{10}$. 
All three distributions satisfy \eqref{eq : intro_limit_bigger_than_one}. 
\textit{Right:} A distribution function that is constant on $[0,h]$ for some $h>0$. 
Any such distribution satisfies \eqref{eq : intro_limit_equal_zero}.
}\label{fig:assump}

\end{center}

\end{figure}

\subsection{Three-arm heuristic and organization of the paper}
The appearance of $\pi_3(R)$ in \eqref{ber_expectation_bound}, \eqref{cm33}, and \eqref{mxi4} can be intuitively explained by the following heuristic.
From every edge $e$ in a geodesic $\primalp$ between two sets, the geodesic itself provides two paths: one to the starting set and another to the ending set.
These subpaths of $\gamma$ consist mostly of open edges, making it plausible that they resemble open arms.
By duality, there is also a closed dual path $\dualp$ from the starting set to the ending set that consists mostly of closed edges. 
We show that if $\primalp$ is chosen ``as close as possible'' to $\dualp$, then for each edge $e$ in $\gamma$, there must exist a closed path from (a dual neighbor of) $e$ to $\dualp$.
Joining this closed path with $\dualp$ itself, we obtain a path resembling a closed arm.
Hence $e$ resembles an edge admitting a 3-arm event.

Carrying out this argument has two main challenges.
The first is to overcome the obvious issue that all three ``arms'' described above have some defect edges of the wrong type.
This is dealt with in two essentially independent ways. 
In the Bernoulli case, Section \ref{bernoulli_section} employs a patching argument that stitches together arm events on different scales in exchange for only constant factors in front of probabilities. 
The main technical devices are found in the proof of Proposition~\ref{M_est}, which is used 
to prove Theorems~\ref{thm:Ber_ptb} and \ref{thm:Ber_ptp}.
Remark~\ref{rem_no_bernoulli} clarifies exactly where the Bernoulli assumption is used.

For more general edge-weight distributions, Section \ref{sec:general_proof} provides a quite different approach. 
Instead of avoiding the defects by patching as before, we allow the two open arms to actually use the defects.
Each defect incurs a small probabilistic penalty thanks to a result of \cite{nolin08} (this is the reason for the additional factor of $R^\eps$ in \eqref{mxi4}).
In order for these penalties to not accumulate too much, we must control the number of closed edges used by a geodesic within a given annulus (Proposition~\ref{prop:Dknbd}) and between two consecutive open circuits (Proposition~\ref{prop:Xjbd}).
These intermediate results are interesting in their own right, as they shed light on the structure of geodesics in critical FPP.
In the case of \eqref{eq : intro_limit_equal_zero}, the heuristic goes as follows.
Because $F$ is quite flat to the right of $0$, it is rare to find nonzero edges with passage time close to $0$.
Therefore, if there were many nonzero edges on the geodesic, then its passage time would be relatively large, but we know a priori that $T(\vc 0,\partial B_R)$ is small (Proposition~\ref{lem:Tkn0assumption}).

The argument under \eqref{eq : intro_limit_bigger_than_one} is more subtle.
In general, every closed edge on a geodesic generates a pseudo-4-arm event: 2 mostly open arms from the geodesic itself, and 2 genuinely closed arms from the fact that the geodesic must be breaking through a closed circuit whenever it uses a closed edge (Proposition~\ref{prop:geod_cons}\ref{itm:dualcir}).
On the other hand, because $F$ is not too flat to the right of $0$, there is a relatively abundant supply of nonzero edges with passage time close to $0$.
Therefore, the passage time is very small (Proposition~\ref{lem:param}), and the geodesic uses no large edge-weights.
This in turn means that the 2 mostly open arms mentioned before are in fact $p$-open for a value of $p$ that is close to $p_\cc$, but we know a priori that the resulting 4-arm event is rare (Lemma~\ref{lem : kiss}).
Once again, the conclusion is that the geodesic does not contain many closed edges.

The second main challenge is to supply all of these arguments with rigorous topological constructions.
The paths involved in arm events are obtained by particular interactions of the geodesic with open and closed circuits.
Proving these interactions do occur requires a careful treatment of the relevant topological objects.
So as to not distract from the probabilistic arguments, we state the requisite definitions in Section \ref{sec:prelim} and postpone the finer topological details until Sections \ref{sec:separation}, \ref{sec:outer_circuits}, and \ref{sec:geod_cons_proof}. 
The logical order of the sections is \ref{sec:separation} $\to$ \ref{sec:outer_circuits} $\to$ \ref{sec:geod_cons_proof} $\to$ \ref{sec:prelim} $\to$ \ref{bernoulli_section} $\to$ \ref{sec:general_proof}.

Section \ref{sec:separation} recalls a result from \cite{kesten82} regarding percolation on planar graphs, and then collects various consequences for bond percolation on $\Z^2$.
For instance, a key fact is that a bounded open cluster on $\Z^2$ is enclosed by a closed dual circuit (Lemma \ref{lemma:separation}). 
This and other results in Section \ref{sec:separation} are well-known but not conveniently quotable from the literature, hence the inclusion of this section for completeness. Section \ref{sec:outer_circuits} provides some general facts about Jordan curves. Finally, Section \ref{sec:geod_cons_proof} constructs the desired geodesic. 
The outcome is summarized in Proposition~\ref{prop:geod_cons}, which gives a circuit-based description of this distinguished geodesic in any critical FPP model on $\Z^2$.
This and other topological results we develop are quite general and could be useful for future studies.
It should be noted that our construction performs several modifications of the percolation environment (in order to use the separation result following from \cite{kesten82}) that are similar
to techniques used in \cite{kesten_sidoravicius_zhang98} on the triangular lattice. 
In that setting, topological considerations are somewhat simpler because there is no auxiliary dual lattice.
We work on the square lattice to highlight the general applicability of our approach, as our methods work just as well on the triangular lattice. 

\subsection{Related literature and methodology}

The study of the shortest geodesic in critical FPP is similar  to the concept of chemical distance. The chemical distance $\dist_\chem(x,y)$  between $x,y\in\Z^2$ is the minimum number of edges in an open path starting at $x$ and ending at $y$. 
If no such open path exists, then the chemical distance is infinite.

A natural question is how the chemical distance scales with the Euclidean distance. 
For instance, if we condition on the event $\{x\leftrightarrow y\}$, then how does $\E\givenk{\dist_\chem(x,y)}{x\leftrightarrow y}$ depend on $\|x-y\|$? 
It is generally believed that there is some $\beta>1$ such that
\eq{
\E\givenk{\dist_\chem(x,y)}{x\leftrightarrow y} = \|x-y\|^{\beta+o(1)}.
}
%
Establishing the existence of $\beta$, let alone computing its value, remains a very challenging open problem. 
In fact, it was suggested in \cite[Prob.~3.3]{schramm07} that even on the triangular lattice, the value of $\beta$ may not be obtainable through SLE methods; on the other hand, see \cite{pose_schrenk_araujo_herrmann14} for a numerical comparison.
Currently, there is no widely accepted prediction for the exact value of $\beta$. 
Numerical results \cite{herrmann_stanley88, grassberger99, deng_zhang_garoni_sokal_sportiello10, zhou_yang_deng_ziff12} have suggested $\beta\approx 1.13$.

Another version of the chemical distance question is to ask about the minimal length $\mathcal{S}_R$ of an open left-right crossing of the box $B_R$, conditional on the event $\mathsf{Cross}_R$ that such a crossing exists.
Aizenman and Burchard \cite{aizenman_burchard99} give a sense in which $\beta>1$ by proving
\eeq{ \label{cj3x}
\E\givenk{\mathcal{S}_R}{\mathsf{Cross}_R} \ge R^{1+c} \quad \text{for some $c>0$}.
}
It turns out that a more amenable quantity is the length $\mathcal{L}_R$ of the \textit{lowest} such crossing.
This is because the lowest crossing consists exclusively of $3$-arm edges. 
%
Using this fact, Morrow and Zhang \cite{morrow_zhang05} showed (for critical site percolation on the triangular lattice) that 
\eeq{ \label{zne3}
\mathbb{E}\givenk{\mathcal{L}_R}{\mathsf{Cross}_R} = R^{4/3+o(1)}=R^{2+o(1)}\pi_3(R).
}
Since $\mathcal{S}_R \le \mathcal{L}_R$, this result effectively provides an upper bound on the exponent $\beta$, namely $\beta\leq4/3$.
Our arguments in the Bernoulli case \eqref{Ber} can be viewed as a generalization of this approach, where ``lowest'' is replaced by ``closest to a certain closed dual path.''

More recently, it was shown that the shortest crossing is asymptotically much shorter than the lowest crossing: first in a little-O sense \cite{damron_hanson_sosoe17,damron_hanson_sosoe16}, then as a strict inequality of exponents \cite{damron_hanson_sosoe21}:
\eeq{ \label{f283}
\E\givenk{\mathcal{S}_R}{\mathsf{Cross}_R} \le CR^{2-c}\pi_3(R) \quad \text{for some $C,c>0$}.
}
The high-level idea is to identify shortcuts around an edge $e$, conditional on $e$ belonging to the lowest crossing.
This line of work was extended to the point-to-box case in \cite{sosoe_reeves22}, again showing $\beta < 4/3$ for a suitable interpretation of $\beta$.



\subsection{Open problems}
Collected here are some open problems suggested by our work:
\begin{enumerate}
\item Can Theorem \ref{thm:general_thm} be generalized to \textit{any} critical FPP edge-weight distribution? 
\item Does the rate of growth of geodesic length in critical FPP depend on the edge-weight distribution?
\item Find the exact order of growth of geodesic length in some critical FPP model, for instance \eqref{Ber}. 
Does there exist $\beta > 1$ such that $\mathcal N_{R} = R^{\beta + o(1)}$? 
This is unknown even on the triangular lattice, despite knowledge of the arm exponents there.
\item Our results give length bounds for a specific choice of geodesic that is constructed in Proposition \ref{prop:geod_cons}. 
Is $R^2\pi_3(R)$ optimal for this geodesic, as in \eqref{zne3}?
With a different choice of geodesic, can the ideas of \cite{damron_hanson_sosoe17,damron_hanson_sosoe21, sosoe_reeves22, reeves23_arxiv} be adapted to FPP in order to replace the bound $R^{2 + \ve} \pi_3(R)$ with $R^{2 - \delta}\pi_3(R)$ for some small $\delta> 0$? 

\item Related to Remark \ref{2m_rmk}, \cite{damron_hanson_sosoe16} asks whether the expected chemical distance between two points is finite.
In our notation, this is the problem of determining whether, in the setting \eqref{Ber}, we have $\mathbb E\givenk[\big]{\lng_{(0,0),(1,0)}}{(0,0) \leftrightarrow (1,0)}<\infty$.
\end{enumerate}

\section{Topological preliminaries: circuits and associated regions}

\label{sec:prelim}

Critical FPP is intimately connected to the structure of open and closed circuits.
In this section we collect some definitions and conventions (Section~\ref{subsec_term}) and use them to describe a specially chosen geodesic (Section~\ref{sec:Ber_geo_construction}) that will be studied in Sections~\ref{bernoulli_section} and \ref{sec:general_proof}.

\subsection{Terminology and conventions} \label{subsec_term}

\begin{definition} \label{circuit_def}
A \textit{circuit} is a path of length at least $4$ that starts and ends at the same vertex but is otherwise self-avoiding. An \textit{open circuit} is a circuit consisting of open primal edges. A \textit{closed circuit} is a circuit consisting of closed dual edges. 
\end{definition} 
A circuit $\CC$ will often be naturally identified 
with the Jordan curve its edges trace out in the plane; in particular, we can speak of its interior and exterior, and we will denote these two disjoint sets by $\intr(\CC)$ and $\ext(\CC)$.


\begin{definition} \label{def:enclose}
Given a set of vertices $\mathcal{A}$ (either $\mathcal A\subseteq\Z^2$ or $\mathcal{A}\subseteq\wh\Z^2$),
we say a circuit $\cir$ \textit{encloses} $\mathcal{A}$ if $\mathcal A\subseteq\intr(\cir)$. 
We say a circuit $\cir$ \textit{surrounds} another circuit $\cir'$ if $\intr(\cir') \subseteq \intr(\cir)$.
\end{definition} 

We stress that if $\mathcal{A}$ is the vertex set of a circuit $\cir'$, then $\cir$ enclosing $\mathcal{A}$ is a stronger notion than $\cir$ surrounding $\cir'$, as the latter still allows $\cir$ and $\cir'$ to touch.
We will often impose \textit{edge-disjointness}, meaning the two circuits have no common edges. 
A stronger condition is \textit{vertex-disjointness}, which means the two circuits share no vertices.
We will always use ``surround'' when treating the relevant objects as two circuits, and ``enclose'' when thinking of the enclosed object as a set of vertices.

Under the criticality assumption \eqref{critical_assumption}, it is a straightforward application of the RSW theorem \eqref{RSW} together with the FKG inequality, to show that the following event occurs with probability one:

\begin{definition} \label{def:omegainf}
Let $\Omega_\infty$ be the full-probability event on which, for every $R \ge 1$, there exist both an open circuit enclosing $B_R$ and a closed circuit enclosing $B_R$.
\end{definition}




In the context of planar FPP, circuits are useful objects to consider for the simple reason that paths must cross them to go from their interior to their exterior.
However, when we speak of paths going from one circuit to another, there is potential for confusion depending on which lattice---primal or dual---each object belongs to.
Instead of exhausting the reader with clarifications every time, we simply rely on the conventions specified here to maintain precision:

\begin{definition} \label{open_closed_paths_def}
Let $\gamma$ be a primal path. 
\begin{itemize}
\item The path $\primalp$ is \textit{open} if all its edges are open.
    \item If $\mathcal{A}\subseteq\Z^2$, then we say $\primalp$ starts (ends) at $\mathcal{A}$ if its first (last) vertex is an element of $\mathcal{A}$.
    \item If $\EE$ is a collection of primal edges, then we say $\primalp$ starts (ends) at $\EE$ if its first (last) vertex is an endpoint of some element of $\EE$.
    \item If $\wh\EE$ is a collection of dual edges, then we say $\primalp$ starts (ends) at $\wh\EE$ if its first (last) vertex is a dual neighbor of some element of $\wh\EE$.
\end{itemize}
Let $\dualp$ be a dual path. 
\begin{itemize}
\item The path $\dualp$ is \textit{closed} if all its edges are closed.
    \item If $\mathcal A\subseteq\Z^2$, then we say $\dualp$ starts (ends) at $\mathcal{A}$ if its first (last) vertex is a dual neighbor of some element of $\mathcal{A}$.
    \item If $\EE$ is a collection of primal edges, then we say $\dualp$ starts (ends) at $\EE$ if its first (last) vertex is a dual neighbor of some element of $\EE$.
    \item If $\wh\EE$ is a collection of dual edges, then we say $\dualp$ starts (ends) at $\wh\EE$ if its first (last) vertex is an endpoint of some element of $\wh\EE$.
\end{itemize}
\end{definition}

To streamline the exposition, certain topological facts regarding circuits and paths are postponed to Sections \ref{sec:separation} and \ref{sec:outer_circuits}.

\subsection{Constructing the geodesic} \label{sec:Ber_geo_construction}
In the present paper, we are concerned with bounding the length of the shortest geodesic. This is achieved by obtaining an upper bound for the length of a specific geodesic.
Rigorously constructing this geodesic requires careful treatment of the topological details; this is done in Section \ref{sec:geod_cons_proof}. 
Proposition \ref{prop:geod_cons} below states the outcomes of the construction that we need.
Here, we outline how this geodesic is constructed.

Let $\mathcal{A}$ and $\mathcal{B}$ be disjoint finite connected sets of vertices in $\Z^2$. 
We first construct a sequence of edge-disjoint circuits, each of which either encloses $\mathcal{A}$ or $\mathcal{B}$. 
Let $\incir_1$ be the innermost open circuit enclosing $\mathcal{A}$ such that $\mathcal{B}$ is contained in $\ext(\incir_1)$ (if such a circuit exists). 
Next, let $\incir_2$ be the innermost open circuit surrounding and edge-disjoint from $\incir_1$ such that $\mathcal{B}$ is contained in $\ext(\incir_2)$. 
Continue this way, obtaining a finite sequence $\incir_1,\ldots,\incir_L$ of nested open circuits, until we can no longer create any more such circuits. 
Then, let $\outcir_{L + 1}$ be the outermost open circuit that is edge-disjoint from $\incir_L$, encloses $\mathcal{B}$, and keeps $\intr(\incir_L)$ in its exterior. 
Next let $\incir_{L + 2}$ be the outermost open circuit enclosing $\mathcal{B}$ that is surrounded by and edge-disjoint from $\incir_{L + 1}$. Continue this way, obtaining a nested sequence of open circuits $\incir_{L + 1},\incir_{L + 2},\ldots,\incir_P$, each enclosing $\mathcal{B}$. 

Any dual path from $\mathcal{A}$ to $\mathcal{B}$ must cross each of the open circuits $\incir_1,\dots,\incir_P$, thanks to their nested structure.
Therefore, every dual path from $\mathcal{A}$ to $\mathcal{B}$ contains at least $P$ open edges.
In fact, because of the innermost/outermost specifications in the previous paragraph, there exists a dual path $\dualp$ from $\mathcal{A}$ to $\mathcal{B}$ that contains \textit{exactly} $P$ open edges.

Now consider any geodesic $\primalp$ between $\mathcal{A}$ and $\mathcal{B}$. 
We prove that for any closed edge $e$ along the geodesic, the dual edge $e^\star$ belongs to a closed circuit that either encloses $\mathcal{A}$ and keeps $\mathcal{B}$ in its exterior, or encloses $\mathcal{B}$ and keeps $\mathcal{A}$ in its exterior. 
So in exact contrast with $\dualp$, the primal path $\primalp$ is entirely open except for when it crosses a closed circuit.
Moreover, by possibly rerouting $\gamma$ along open primal circuits, and similarly rerouting $\dualp$ along closed dual circuits, we can choose $\primalp$ and $\dualp$ to be disjoint. 

Finally, between any two of its closed edges, the geodesic $\primalp$ takes only open edges, so we may replace that portion of the path with any other path consisting entirely of open edges. 
In this way, we will choose $\gamma$ to be the geodesic that is \textit{closest} (in an appropriate sense) to $\dualp$. 
It will follow that for each open edge $e \in \gamma$, there exists a closed dual path from $e$ to $\dualp$. Then, the edge $e$ satisfies a 3-arm event: one closed arm to $\dualp$, and two open arms by following the geodesic in each direction to the next closed edges. 
Quantitatively controlling the likelihood of this arm event is done in Sections~\ref{bernoulli_section} and \ref{sec:general_proof}.  

The properties of this construction are summarized in the following proposition, whose proof may be found in Section \ref{sec:geod_cons_proof}. 

\begin{proposition} \label{prop:geod_cons}
Let $\mathcal{A}$ and $\mathcal{B}$ be disjoint finite connected subsets of $\Z^2$.  On the full-probability event $\Omega_\infty$ from Definition~\ref{def:omegainf}, there exists a (possibly empty) sequence of edge-disjoint open circuits $\incir_1,\ldots,\incir_L,\outcir_{L + 1},\ldots,\outcir_P$ satisfying the following:
\begin{enumerate}[label=\textup{(\roman*)}]
\item \label{itm:Iinc} $\mathcal{A} \subseteq \intr(\incir_1) \subseteq \intr(\incir_2) \subseteq \cdots \subseteq \intr(\incir_L) \subseteq \intr(\incir_L) \cup \incir_L \subseteq \mathcal{B}^\complement $. 
\item \label{itm:Idisj} $\intr(\incir_{L}) \subseteq \ext(\outcir_{L +1})$.
\item \label{itm:Idec} $\mathcal{A}^\complement \supseteq \intr(\outcir_{L + 1}) \cup \outcir_{L+1} \supseteq  \intr(\outcir_{L + 1}) \supseteq \intr(\outcir_{L + 2}) \supseteq \cdots \supseteq \intr(\outcir_P) \supseteq \mathcal{B}$.
\item \label{itm:2donttouch} For $j \in \{1,\ldots,P - 2\}$, the circuits $\incir_j$ and $\incir_{j + 2}$ are vertex-disjoint. 
\item  \label{itm:alledual} For $j \in \{1,\ldots,P\}$ and every $e \in \incir_j$, there exists a dual path $\dualp_e$ from $e$ to $\mathcal{A}$ that has exactly $j - 1$ open edges, one crossing each of the circuits $\incir_1,\ldots,\incir_{j - 1}$. 
\end{enumerate}
 Furthermore, there exist
a geodesic $\primalp$ from $\mathcal{A}$ to $\mathcal{B}$, and a disjoint dual path $\dualp$ from $\mathcal{A}$ to $\mathcal{B}$, satisfying the following properties (here we note that the paths $\dualp_e$ in Item \ref{itm:alledual} are not necessarily disjoint from $\primalp$): 
\begin{enumerate}[resume,label=\textup{(\roman*)}]
\item \label{itm:zetaopen} $\dualp$ has exactly $P$ open edges, one crossing each of the circuits $\incir_1,\ldots,\incir_L,\outcir_{L +1},\ldots, \outcir_P$.
\item \label{itm:gamma_on_circuit} For each circuit $\incir_j$, let $x_j$ and $y_j$ be the first and last vertices of $\primalp$ on that circuit. Then, the portion of $\primalp$ between $x_j$ and $y_j$ lies entirely on $\incir_j$. If $\incir_j$ and $\incir_{j + 1}$ are not vertex-disjoint, then $y_j = x_{j+1}$.
\item \label{itm:dualconn} For every open edge $e \in \primalp$ with $e \notin \incir_1\cup \cdots \cup \incir_L \cup \outcir_{L + 1}\cup \cdots \cup \outcir_P$, there exists a closed dual path from $e$ to $\dualp$ that is disjoint from $\gamma$.
\item \label{itm:dualcir} The dual of each closed edge along $\primalp$ belongs to a closed circuit $\dincir$ that either contains $\mathcal{A}$ in its interior and $\mathcal{B}$ in its exterior, or vice versa. The circuit $\dincir$ does not contain the dual of any other edges along $\primalp$. 
\item \label{itm:Csequence} With $\{0,1\}$-valued edge-weights, the closed circuits $\dincir$ from Item \ref{itm:dualcir} can be chosen to form a edge-disjoint collection $\dincir_1,\ldots,\dincir_V$. 
For $j\in\{1,\dots,V-2\}$, the circuits $\dincir_j$ and $\dincir_{j+2}$ are vertex-disjoint.
The union of the circuits $\incir_1,\ldots,\incir_P$ and the circuits $\dincir_1,\ldots,\dincir_V$ forms a sequence $\CC_1,\ldots,\CC_K$, which is ordered so that, for some index $W \in \{0,\ldots,K\}$, $\intr(\CC_W) \cap \intr(\CC_{W + 1}) = \varnothing$ (with the convention $\intr(\CC_0) = \intr(\CC_{K +1}) = \varnothing)$, and we have the following inclusions
\[
\mathcal A \subseteq \intr(\CC_1) \subseteq \cdots \subseteq \intr(\CC_W),\quad\text{and} \quad \intr(\CC_{W + 1}) \supseteq \cdots \supseteq\intr(\CC_K) \supseteq \mathcal{B}.
\]
\end{enumerate}
\end{proposition}

Figure \ref{fig:not_ss_cont} shows an example when circuits $\incir_i$ and $\incir_{i+1}$ are not vertex-disjoint. The only way this can happen is if the two circuits touch at a corner. Since each vertex is only seen in four edges on $\Z^2$, $\incir_i$ and  $\incir_{i + 2}$ must always be vertex-disjoint, as stated in Item \ref{itm:2donttouch}.
\begin{figure}
    \centering
    \includegraphics[height = 2in]{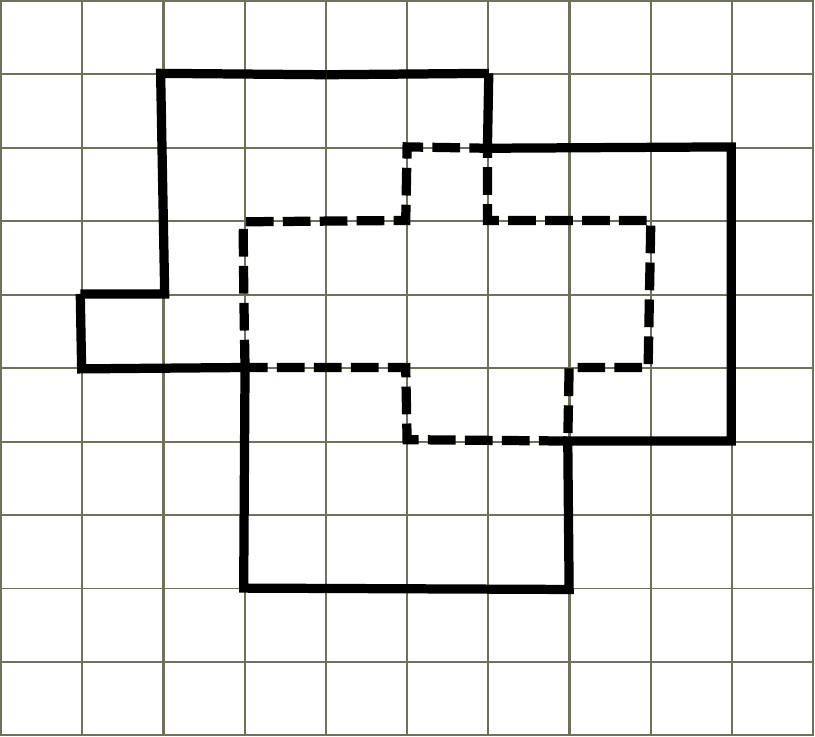}
    \caption{\small Two primal circuits (one in solid, the other dashed) which are edge-disjoint but not vertex-disjoint, and such that one surrounds the other.}
    \label{fig:not_ss_cont}
\end{figure}

\section{Proofs for Bernoulli edge-weights} \label{bernoulli_section}
In this section, we prove Theorems \ref{thm:Ber_ptb} and \ref{thm:Ber_ptp}. 
So we assume the edge-weights have the Bernoulli($\f{1}{2}$) distribution from \eqref{Ber}.
In this case, all closed edges have weight $1$, so the passage time along any path $\gamma$ is simply the number of closed edges in $\gamma$.
Hence the passage time between two sets $\mathcal A,\mathcal B\subseteq\Z^2$ is the minimal number of closed edges that are encountered by a path starting at $\mathcal{A}$ and ending at $\mathcal{B}$. 
By Proposition \ref{prop:geod_cons}\ref{itm:Csequence}, that number of closed edges is equal to the number of edge-disjoint closed circuits $\dincir_1,\ldots,\dincir_V$ that separate $\mathcal A$ and $\mathcal B$.

Throughout our arguments, the symbols $c$ and $C$ will denote positive constants that are chosen sufficiently small or sufficiently large, respectively.
Their values may change from line to line, but they will never depend on the radii of arm events, or on the sets $\AA$ and $\BB$.

\subsection{Preliminaries on arm events} \label{sec:arm_event}

For an edge $e$, let $x_e$ denote the endpoint of $e$ which has the smaller norm, and define 
$B_R(e) = x_e + B_R$.
Recall the definition of  ``start'' and ``end'' from Definition \ref{open_closed_paths_def}.
We now define various types of arm events, for integers $R\ge r \ge 1$.
An \textit{open arm} is a open primal path, and a \textit{closed arm} is a closed dual path.
\begin{itemize}
    \item A \textit{$1$-arm event} from $\partial B_r(e)$ to $\partial B_R(e)$ is the event that there exists an open arm starting at $\partial B_r(e)$ and ending at $\partial B_R(e)$. 
    We denote its probability by $\pi_1(r,R)$.
    \item For $k\ge2$, a \textit{polychromatic $k$-arm event} from $\partial B_r(e)$ to $\partial B_R(e)$ is the event that there exist $k$ disjoint arms---one closed and $k-1$ open---starting at $\partial B_r(e)$ and ending at $\partial B_R(e)$.
    We denote its probability by $\pi_k(r,R)$.
    The probability of any other polychromatic sequence, e.g. (closed, open, closed, open) versus (closed, open, open, open), is within an absolute constant factor of $\pi_k(r,R)$ \cite[Prop.~1.1]{reeves_sosoe22_arxiv}.
    \item For $k\ge2$, a \textit{monochromatic $k$-arm event} from $\partial B_r(e)$ to $\partial B_R(e)$ is the event that there exist either $k$ disjoint open arms starting at $\partial B_r(e)$ and ending at $\partial B_R(e)$, or $k$ disjoint closed arms starting at $\partial B_r(e)$ and ending at $\partial B_R(e)$.
    We denote the probability of the open version by $\pi_k'(r,R)$, and the probability of the closed version by $\pi_k''(r,R)$.
    \item Finally, we say there is a \textit{(polychromatic) $3$-arm event} from edge $e$ to distance $R$ if there exist two disjoint open arms that start at the two endpoints of $e$ and end at $\partial B_R(e)$, and also a closed arm starting at one of the endpoints of $e^\star$ and ending at $\partial B_R(e)$.
    We denote its probability by $\pi_3(R)$.
\end{itemize}
Notice that there is a slight difference between $\pi_3(R)$ and $\pi_3(1,R)$.
We will use both for expositional precision, but the numerical difference is inconsequential.
This is because $\pi_3(R) \le \pi_3(1,R)$ by inclusion of events, while 
\eeq{ \label{two_3_arm}
\pi_3(1,R) \le C\pi_3(R)
}
for some constant $C$, by a simple argument conditioning on finitely edges near $e$.

We prepare for the analysis ahead by quoting three propositions for the arm events. The first proposition is known as quasi-multiplicativity, which is widely used in near-critical percolation. 
The main content is the second inequality in \eqref{quasi_j}, as the first inequality is trivial.




\begin{proposition}[Quasi-multiplicativity, \textup{\cite[Prop.~17]{nolin08}}]\label{qm}
Fix $k\geq 1$ and let $r_0 = r_0(k)$ be the smallest integer such that $\pi_k(r_0,r_0+1)>0$. 
There exists a constant $C = C(k)$ such that
\eeq{ \label{quasi_j}
\pi_k(r,R) \le \pi_k(r,r')\pi_k(r',R) \le C\pi_k(r,R) \quad \text{for all $R \ge r' \ge r \ge r_0$}.
}
The same statement holds for monochromatic arm events:
\eeq{ \label{quasi_mono}
\pi_k'(r,R) \le \pi_k'(r,r')\pi_k'(r',R) \le C\pi_k'(r,R), \\
\pi_k''(r,R) \le \pi_k''(r,r')\pi_k''(r',R) \le C\pi_k''(r,R).
}
\end{proposition}

\begin{remark}[Scaling the terminal distance only costs a constant] \label{rmk:cNquasi}
It follows from the RSW theorem \eqref{RSW} that for any fixed constant $\alpha>1$, there is some constant $c = c(\alpha,k) > 0$ such that $\pi_k(R,\lceil \alpha R\rceil) \ge c$ for all $R\geq r_0(k)$. 
Using this observation and then quasi-multiplicativity, one obtains that for $R \ge r \ge r_0(k)$,
\eeq{ \label{quasi_c}
\pi_k(r,R)
\leq \pi_k(r,R)\frac{\pi_k(R,\lceil\alpha R\rceil)}{c}
\stackref{quasi_j}{\leq} C\pi_k(r,\lceil \alpha R\rceil).
}
\end{remark}


The next result is an upper bound on the 1-arm probability.
It follows from the RSW theorem \eqref{RSW}.
Although \cite[Lem.~8.5]{kesten82} only considers the point-to-box arm event, the argument can be immediately adapted for our box-to-box arm event.
\begin{proposition}\textup{\cite[Lem.~8.5]{kesten82}} \label{prop:1-arm}
There exist constants $C,c>0$ such that
\eeq{ \label{ncb3}
\pi_1(r,R) \leq C(R/r)^{-c} \quad \text{for all $R\geq r \geq 1$.}
}
\end{proposition}

Meanwhile, the polychromatic 5-arm exponent is known to be equal to $2$:
\begin{proposition} \textup{\cite[Thm.~24(3)]{nolin08}} \label{pi5} 
There exist constants $C,c>0$ such that
\eeq{ \label{dm8i}
c(R/r)^{-2} \leq \pi_5(r,R) \leq C(R/r)^{-2} \quad \text{for all $R\geq r\geq 1$}.
}
\end{proposition}
\begin{remark}[Triangular vs.~square lattice]
It should be noted that the result in \cite{nolin08} is stated for the triangular lattice instead of the square lattice.
In fact, on the triangular lattice, all polychromatic arm exponents are known exactly: $\pi_1(1,R) = R^{-5/48+o(1)}$ \cite{lawler_schramm_werner02} while $\pi_k(r_0(k),R) = R^{-(k^2-1)/12+o(1)}$ for $k\geq2$ \cite{smirnov_werner01}.
But unlike the derivations of the other arm exponents on the triangular lattice that use conformal invariance, the proof of Proposition~\ref{pi5} works on the square lattice as well; see the discussion above \cite[Thm.~24]{nolin08}, including \cite[Rmk.~23]{nolin08}.
Alternatively, one can use the argument outlined in \cite[Chap.~1]{werner09} or the more general result \cite[Prop.~6.6]{duminilcopin_manolescu_tassion21}, together with the equivalence of polychromatic arm probabilities \cite[Prop.~1.1]{reeves_sosoe22_arxiv}.
\end{remark}

\begin{remark}[Arm exponents on the square lattice]  \label{rmk:3armprob}
Although the arm exponents in our setting are widely believed to be identical to those on the triangular lattice, here are the best estimates of which we are aware, where $C,c>0$ and $R\ge r\ge 1$:
\begin{subequations}
\eeqs{
c(R/r)^{-1/6} &\le \pi_1(r,R) \le C(R/r)^{-c}, \label{best_1_arm}\\
c(R/r)^{-1/3} &\leq \pi_2(r,R) \leq C(R/r)^{-c}, \label{best_2_arm}\\
c(R/r)^{-(1-c)} &\leq \pi_3(r,R) \leq C(R/r)^{-c}, \label{best_3_arm} \\
c(R/r)^{-4/3} &\leq \pi_4(r,R) \leq C(R/r)^{-(1+c)}. \label{best_4_arm}
}
\end{subequations}
The lower bounds for $\pi_1$ and $\pi_2$ are from \cite[p.~435]{duminilcopin_manolescu_tassion21}, together with quasi-multiplicativity.
The lower bound for $\pi_4$ is also from \cite[p.~435]{duminilcopin_manolescu_tassion21}, once we appeal to equivalence of different polychromatic color sequences \cite[Prop.~1.1]{reeves_sosoe22_arxiv}.
The lower bound for $\pi_3$ can be found by combining \eqref{cj3x} and \eqref{f283} with quasi-multiplicativity, or by appealing to \cite[Lem.~3.1]{damron_hanson_sosoe17} for a direct argument.
The upper bound for $\pi_2$ and $\pi_3$ is immediate from \eqref{ncb3} since $\pi_3(r,R)\le \pi_2(r,R) \le \pi_1(r,R)$.
For a more precise comparison of the 2-arm and 1-arm probabilities, see \cite{gassman_manolescu24_arxiv,ramananradhakrishnan_tassion24_arxiv}.
Finally, the upper bound for $\pi_4$ can be implicitly inferred from \cite{kesten87}, but is thoroughly treated in \cite{vandenberg_nolin21}.
\end{remark}

A key consequence of the lower bound in \eqref{best_3_arm} is the following estimate, in which we take the convention that $\pi_3(0)=1$.


\begin{lemma} \label{lem:sum_arm}
There exists a constant $C$ such that
\eeq{ \label{3r8bbc}
\sum_{r=0}^R \pi_3(r) \leq CR \pi_3(R) \quad \text{for all $R\ge1$}.
}
\end{lemma}
\begin{proof}
We allow the value of $C$ to change with each inequality.
For any $r\in\{1,\dots,R\}$ and some $c\in(0,1)$, we have
\eeq{ \label{kb20x}
\pi_3(r) \le \pi_3(1,r) 
\stackref{quasi_j}{\le} C\frac{\pi_3(1,R)}{\pi_3(r,R)} 
\stackref{two_3_arm}{\le} C\frac{\pi_3(R)}{\pi_3(r,R)}
\stackref{best_3_arm}{\leq} C\frac{\pi_3(R)}{(R/r)^{-(1-c)}}.
}
Furthermore, we recognize the following expression as a lower Riemann sum:
\eeq{ \label{kb20y}
\sum_{r=1}^R (r/R)^{-(1-c)}
\leq \frac{\int_0^R x^{-(1-c)}\, \dd x}{R^{-(1-c)}} = \frac{R}{c}.
}
Combining \eqref{kb20x} and \eqref{kb20y} yields \eqref{3r8bbc} but with the starting index of $r=1$.
We can include $\pi_3(0)$ after increasing $C$.
\end{proof}

Monochromatic arm events are generally believed to be asymptotically less likely than their polychromatic counterparts; this is known on the triangular lattice \cite[Thm.~5]{beffara_nolin11} (see also \cite[Rmk.~5]{ramananradhakrishnan_tassion24_arxiv}), but not on the square lattice.
The following lemma is a much weaker statement which suffices for our purposes.

\begin{lemma}\label{mono_poly}
There exists a constant $C$ and a large integer $k_\star\ge3$ such that
\eeq{ \label{mono_poly_ineq}
\pi'_{k_\star}(r,R) \leq C\pi_3(r,R)\cdot \frac{r}{R} \quad \text{for all $R\ge r\ge 1$.}
}
\end{lemma}

\begin{proof}
Using the van den Berg--Kesten--Reimer (BKR) inequality followed by Proposition~\ref{prop:1-arm}, we have 
\eq{
\pi_{k_\star}'(r,R) \leq \pi_1(r,R)^{k_\star} \stackref{ncb3}{\leq} C_0(r/R)^{ck_\star}.
}
On the other hand, by \eqref{best_3_arm} we have $\pi_3(r,R) \ge \delta(r/R)^{1-\delta}$ for some $\delta>0$. 
Now choose $k_\star$ large enough that ${ck_\star} \ge {2-\delta}$, in order to obtain \eqref{mono_poly_ineq} with $C = C_0/\delta$.
\end{proof}


The next result shows that a closed monochromatic $k$-arm event has comparable probability to an open monochromatic $k$-arm event.

\begin{lemma} \label{lem_closed_like_open}
For any $k\ge1$, there exists a constant $C>0$ such that
\eeq{ \label{closed_like_open}
\pi_k'(r,R)\le \pi_k''(r,R) \le C\pi_k'(r,R) \quad \text{for all $R\ge r\ge r_0(k)$}.
}
\end{lemma}

\begin{proof}
If $R = r$, then $\pi_k'(r,R) = \pi_k''(r,R)=1$, so the claim holds.

Suppose $R = r+1$.  
We have $\pi_k''(r,r+1)=1$ since $\partial B_r$ and $\partial B_{r+1}$ have common dual neighbors.
Meanwhile, it is straightforward to see that
\eeq{ \label{weru7b}
\inf_{r\ge r_0(k)}\pi_k'(r,r+1)>0.
}
So the claim again holds by choosing $C$ sufficiently large.

For the remainder of the proof, we assume $R \ge r+2$ so that $R-1 \ge r+1$.
After accounting for the lattice shift of $(\tfrac12,\tfrac12)$ and using duality, we have
\eq{ 
\pi_k''(r,R)
&= \P\big(\text{$\exists$ $k$ disjoint open arms from $\partial [-r,r+1]^2$ to $\partial[-R+1,R]^2$}\big).
}
The desired lower bound is now evident, since we can pull the boundaries farther apart to obtain
\eq{
&\P\big(\text{$\exists$ $k$ disjoint open arms from $\partial [-r,r+1]^2$ to $\partial[-R+1,R]^2$}\big) \\
&\ge\P\big(\text{$\exists$ $k$ disjoint open arms from $\partial [-r,r]^2$ to $\partial[-R,R]^2$}\big)
=\pi_k'(r,R).
}
The upper bound is obtained by pushing the boundaries closer together:
\eq{
&\P\big(\text{$\exists$ $k$ disjoint open arms from $\partial [-r,r+1]^2$ to $\partial[-R+1,R]^2$}\big) \\
&\stackrefp{quasi_j}{\le} \P\big(\text{$\exists$ $k$ disjoint open arms from $\partial [-r-1,r+1]^2$ to $\partial[-R+1,R-1]^2$}\big) \\
&\stackrefp{quasi_mono}{=}\pi_k'(r+1,R-1) \\
&\stackref{quasi_mono}{\le}\frac{C\pi_k'(r,R)}{\pi_k'(r,r+1)\pi_k'(R-1,R)}
\stackref{weru7b}{\le}\frac{C\pi_k'(r,R)}{c}.
}
To complete the proof, absorb the factor of $1/c$ into $C$.
\end{proof}

Our final preparatory result gives an upper bound for the probability of a sequence of polychromatic arm events, in terms of a single polychromatic 3-arm event.

\begin{proposition}\label{arm_calc}
For any positive integer $\ell$, there exists a constant $C_\ell$ such that the following statement holds.
For any integers $k_1,\ldots, k_{\ell}\ge4$ and $M\ge2$, we have 
\eeq{ \label{arm_calc_ineq}
\sum_{i_1 =0}^{\floor{\log_2 M}-1}\sum_{i_2=1+i_1}^{\floor{\log_2 M}-1} \cdots \sum_{i_\ell = 1+i_{\ell-1}}^{\floor{\log_2 M}-1} \pi_3(2^{i_1}) \pi_{k_1}(2^{1+i_1}, 2^{i_2}) \cdots \pi_{k_\ell}(2^{1+i_\ell}, M) \leq C_\ell \pi_3(M).
}
Furthermore, for any $\eps>0$, there exists a constant $C_{\ell,\eps}$ such that the following statement holds. 
For any integers $k_1, \dots, k_{\ell-1} \geq 4$ and $M\ge2$, we have
\eeq{ \label{arm_calc_ineq2}
\sum_{i_1 =0}^{\floor{\log_2 M}-1}\sum_{i_2=1+i_1}^{\floor{\log_2 M}-1} \cdots &\sum_{i_\ell = 1+i_{\ell-1} }^{\floor{\log_2 M}-1}\bigg[\pi_3(2^{i_1}) \pi_{k_1}(2^{1+i_1}, 2^{i_2}) \cdots \\
&\cdots\pi_{k_{\ell-1}}(2^{1+i_{\ell-1}}, 2^{i_\ell}) \pi_3(2^{1+i_\ell}, M)\Big(\frac{2^{1+i_\ell}}{M} \Big)^\ve\bigg]
\leq C_{\ell,\eps} \pi_3(M).
}
\end{proposition}


\begin{proof}
Throughout the proof, the value of $C$ can change with each inequality but remains independent of $\ell$.
First consider any fixed $0\le i_1 < i_2 < \cdots < i_\ell \le \floor{\log_2 M}-1$.
By the BKR inequality, we have
$\pi_4(r,R) \leq \pi_3(r,R)\pi_1(r,R)$ for any $b\ge a\ge 1$, which justifies the third inequality below:
\eq{
&\pi_3(2^{i_1}) \pi_{k_1}(2^{1+i_1}, 2^{i_2}) \cdots \pi_{k_\ell}(2^{1+i_\ell },M) \\
&\stackrefp{quasi_j}{\le} \pi_3(1,2^{i_1}) \pi_{4}(2^{1+i_1}, 2^{i_2}) \cdots \pi_{4}(2^{1+i_\ell}, M) \\
&\stackref{quasi_c}{\le} C^\ell\pi_3(1,2^{1+i_1}) \pi_{4}(2^{1+i_1}, 2^{1+i_2}) \cdots \pi_{4}(2^{1+i_\ell}, M) \\
&\stackrefp{quasi_j}{\le} C^\ell\pi_3(1,2^{1+i_1})\pi_{3}(2^{1+i_1}, 2^{1+i_2}) \cdots \pi_{3}(2^{1+i_\ell}, M) \cdot
\pi_{1}(2^{1+i_1}, 2^{1+i_2}) \cdots \pi_{1}(2^{1+i_\ell}, M) \\
&\stackref{quasi_j}{\le} C^{3\ell}\pi_3(1,M)\cdot\pi_{1}(2^{1+i_1},M) 
\stackref{two_3_arm,ncb3}{\le} C^{3\ell+2}\pi_3(M)\Big(\frac{2^{1+i_1}}{M}\Big)^c.
}
Applying this inequality to each summand on the left-hand side of \eqref{arm_calc_ineq}, we obtain
\eq{
\text{L.H.S. of \eqref{arm_calc_ineq}}
&\le C^{3\ell+2}\pi_3(M)\sum_{i_1 =0}^{\floor{\log_2 M}-1}\sum_{i_2=1+i_1}^{\floor{\log_2 M}-1} \cdots \sum_{i_\ell = 1+i_{\ell-1}}^{\floor{\log_2 M}-1} \Big(\frac{2^{1+i_1}}{M}\Big)^c \\
& \leq C^{3\ell+2} \pi_3(M)\sum_{i_1 =0}^{\floor{\log_2 M}-1}(\floor{\log_2 M} - i_1-1)^{\ell-1} \Big(\frac{2^{1+i_1}}{M}\Big)^c \quad  \\
&\le C^{3\ell+2} \pi_3(M)\int_0^{\floor{\log_2 M}} (\floor{\log_2 M}-x)^{\ell-1} \Big(\frac{2^{1+x}}{M}\Big)^c\, \dd x \\ 
&\le C^{3\ell+2} \pi_3(M) \int_{0}^{\floor{\log_2 M}} z^{\ell-1}\Big(\frac{2}{2^z}\Big)^c\, \dd z \\
&\le 2^c C^{3\ell+2} \pi_3(M) \int_0^\infty z^{\ell-1}\e^{-(c\log 2)z}\, \dd z \\
&= 2^c C^{3\ell+2}\frac{(\ell-1)!}{(c\log 2)^\ell}\pi_3(M).
}
This proves \eqref{arm_calc_ineq}.
For \eqref{arm_calc_ineq2}, we simply replace the upper bound used before,
\eq{
\pi_{k_\ell}(2^{1+i_\ell},M) \le \pi_3(2^{1+i_\ell},M)\pi_1(2^{1+i_\ell},M),
}
with the quantity $\pi_3(2^{1+i_\ell},M)\cdot\big(\frac{2^{1+i_\ell}}{M} \big)^\ve$, and then replace $c$ with $c \wedge \eps$ in the above argument.
\end{proof}

\subsection{Three-arm events along the specially chosen geodesic}

Notice that the 3-arm event in \eqref{ber_expectation_bound} and \eqref{cm33} requires arms that reach all the way to distance $R$.
But the arms in Proposition \ref{prop:geod_cons} only extend as far as the nearby circuits.
To furnish the quantity $\pi_3(R)$ from these shorter arms, we will strategically ``cover up'' problematic edges and patch together what remains using Proposition~\ref{arm_calc}. 
The goal is to establish the following theorem.

\begin{proposition}\label{M_est}
Assume \eqref{Ber}.
Let $\mathcal{A}$ and $\mathcal{B}$ be disjoint finite connected subsets of $\Z^2$. Let $\gamma$ be the geodesic between $\mathcal{A}$ and $\mathcal{B}$ defined in Proposition \ref{prop:geod_cons}. There exists a constant $C$ (independent of $\mathcal{A}$ and $\mathcal{B}$) such that for every edge $e \in E(\Z^2)$, we have
\eeq{ \label{key3}
\mathbb{P}(e\in \primalp) \leq C\pi_3(M),
}
where $M =  \min\{\dist(e,\mathcal{A}), \dist(e, \mathcal{B})\}$, with $\dist$ defined in \eqref{dist_def}.
\end{proposition}


\begin{proof} 
Throughout the argument, the value of $C$ is allowed to change with each inequality.
By choosing $C\ge\pi_3(1)^{-1}$, we may assume $M\ge2$.  We will write $\llbrack a,b\rrbrack$ for the set of integers $k$ such that $a\le k\le b$.

To prove \eqref{key3}, we split the event $\{e\in \primalp\}$ into various disjoint events listed below.
Let $\dualp$ be the dual path from $\AA$ to $\BB$ as in Proposition \ref{prop:geod_cons}\ref{itm:Csequence}.
Recall from the same proposition the sequence of edge-disjoint circuits $\CC_1,\ldots,\CC_K$, where $\{K=0\}$ means the sequence is empty.
Each one of these circuits is either open primal or closed dual; both $\gamma$ and $\dualp$ intersect or cross all of these circuits in succession.
The only closed edges of $\gamma$ occur when it crosses a closed $\CC_k$, while the only open edges of $\dualp$ occur when it crosses an open $\CC_k$.
We say that an edge $e\in\gamma$ is \textit{between} $\XX$ and $\YY$ if, as we proceed along $\gamma$ from $\AA$ to $\BB$, the midpoint of $e$ occurs strictly after the last intersection with $\XX$ and strictly before the first intersection with $\YY$.
We will show all of the following inequalities:
\begin{subequations} \label{M_est_cases}
\eeqs{
\label{lemmm1}
&\mathbb{P}(e\in \primalp,\ K=0) \leq C\pi_3(M),\\
\label{lemmm4}
&\mathbb{P}(e\in \primalp,\ \text{$e$ is between $\mathcal{A}$ and $\mathcal{C}_{1}$, $\CC_1$ is open}) \le C\pi_3(M),\\
\label{lemmm44}
&\mathbb{P}(e\in \primalp,\ \text{$e$ is between $\mathcal{C}_K$ and $\mathcal{B}$, $\CC_K$ is open})\le C\pi_3(M), \\
\label{lemmm6}
&\mathbb{P}(e\in \primalp\cap\CC_1) \le C\pi_3(M),\\
\label{lemmm55}
&\mathbb{P}(\text{$e \in \primalp\cap\CC_K$}) \le C\pi_3(M), \\
\label{lemmm5}
&\mathbb{P}(\text{$e \in \primalp\cap\CC_k$ for some $k\in\llbrack 2,K-1\rrbrack$, $\CC_{k+1}$ is open}) \le C\pi_3(M), \\
\label{lemmm8}
&\mathbb{P}(\text{$e \in \primalp\cap\CC_k$ for some $k\in\llbrack2,K-1\rrbrack$, $\CC_{k+1}$ is closed}) \le C\pi_3(M), \\
\label{lemmm2}
&\mathbb{P}\Big(\, \pbox{0.8\textwidth}{$e\in\primalp$, $e$ is between $\mathcal{C}_k$ and $\mathcal{C}_{k+1}$ for some $k\in\llbrack1,K-1\rrbrack$, \\ both $\CC_{k}$ and $\CC_{k+1}$ are open}\Big)\leq C\pi_3(M),\\
\label{lemmm3}
&\mathbb{P}\Big(\, \pbox{0.8\textwidth}{$e\in\primalp$, $e$ is between $\mathcal{C}_k$ and $\mathcal{C}_{k+1}$ for some $k\in\llbrack1,K-1\rrbrack$, \\ at least one of $\CC_{k}$ and $\CC_{k+1}$ is closed}\Big)\leq C\pi_3(M),\\
\label{lemmm33}
&\mathbb{P}(e\in \primalp,\ \text{$e$ is between $\mathcal{A}$ and $\mathcal{C}_{1}$, $\CC_1$ is closed}) \le C\pi_3(M),\\
\label{lemmm333}
&\mathbb{P}(e\in \primalp,\ \text{$e$ is between $\mathcal{C}_K$ and $\mathcal{B}$, $\CC_K$ is closed})\le C\pi_3(M), \\
\label{lemmm3333}
&\mathbb{P}(e\in \primalp,\ \text{$e^\star\in\CC_k$ for some $k\in\llbrack1,K\rrbrack$}) \leq C\pi_3(M).
}
\end{subequations}
Note that \eqref{M_est_cases} exhausts the event $\{e\in\gamma\}$, so it suffices to prove these twelve cases.
Before we begin, note that
\eeq{ \label{relax_to_2}
\pi_3(2^{\floor{\log_2 M}}) \le \pi_3(1,2^{\floor{\log_2 M}}) \stackref{quasi_c}{\leq}  C\pi_3(1,M) \stackref{two_3_arm}{\le} C\pi_3(M).
}
Therefore, it suffices to bound the probabilities in \eqref{M_est_cases} by any of the first three quantities instead of $C\pi_3(M)$.

\medskip

\noindent \textbf{Proof of \eqref{lemmm1}.}
If $K=0$, then Proposition \ref{prop:geod_cons}\ref{itm:dualcir} implies that $\gamma$ is entirely open, while Proposition \ref{prop:geod_cons}\ref{itm:zetaopen} implies that $\dualp$ is entirely closed.
In addition, Proposition~\ref{prop:geod_cons}\ref{itm:dualconn} gives a closed dual path $\dualp_e$ that connects $e$ to $\dualp$.
So on the event $\{e\in\gamma\}\cap\{K=0\}$, we obtain a genuine 3-arm event from $e$ to distance $M$: The two open arms are obtained by following $\gamma$ in both directions from $e$ until reaching $\mathcal{A}$ and $\mathcal{B}$, respectively.
The closed arm is obtained by following $\dualp_e$ from $e$ to $\dualp$, and then following $\dualp$ to either $\mathcal{A}$ or $\mathcal{B}$.
This proves \eqref{lemmm1}.

\medskip \noindent \textbf{Proof of \eqref{lemmm4} and \eqref{lemmm44}.} 
We prove only \eqref{lemmm4}, since the argument for \eqref{lemmm44} is completely analogous.
So we assume $e$ is an edge of $\gamma$ between $\mathcal{A}$ and $\mathcal{C}_1$, and that $\CC_1$ is open, i.e.\ $\CC_1=\incir_1$.
We claim that there exists a 3-arm event from $e$ to distance $M$.
One open arm exists because $e$ is connected to $\mathcal{A}$ by $\primalp$, and the portion of $\gamma$ between $\AA$ and $\CC_1$ is open.
The other open arm is obtained by following $\gamma$ in the other direction until reaching $\CC_1$, and then following $\CC_1$ around $\AA$ or $\BB$.
Finally, the closed arm exists because there is a closed dual path from $e$ to $\dualp$ (Proposition \ref{prop:geod_cons}\ref{itm:dualconn}), and the portion of $\dualp$ between $\mathcal{A}$ and $\incir_1$ is closed by Proposition~\ref{prop:geod_cons}\ref{itm:zetaopen}.

\medskip \noindent \textbf{Proof of \eqref{lemmm6}.}
Assume $e\in\gamma\cap\CC_1$, which implies $\CC_1$ is open, i.e.~$\CC_1=\incir_1$.
We claim that there exists a 3-arm event from $e$ to distance $M$.
The two open arms are obtained by following $\CC_1$ in both directions.
The closed arm is from Proposition \ref{prop:geod_cons}\ref{itm:alledual}, which gives a closed dual path $\dualp_e$ from $e$ to $\mathcal{A}$. 
This finishes the proof of \eqref{lemmm6}.


\medskip
In the remaining cases, encounters with certain circuits may prevent some arms from extending all the way to distance $M$.
Therefore, we now introduce notation for identifying where the problematic edges are located.
Recall the notation $B_R(e) = x_e + B_R$ defined in Section~\ref{sec:arm_event}.
Let $E(B_{2^i}(e))$ denote the set of primal edges with both endpoints in $B_{2^i}(e)$.
Let $E^\star(B_{2^i}(e))$ denote the set of dual edges whose associated primal edge belongs to $E(B_{2^i}(e))$.
Let $\overline E(B_{2^i}(e)) = E(B_{2^i}(e)) \cup E^\star(B_{2^i}(e))$.
Finally, define
\eq{
\EE^{(e)}_0 = \overline E(B_{2^1}(e)) \qquad \text{and} \qquad
\EE^{(e)}_i = \overline E(B_{2^{1+i}}(e))\setminus \overline E(B_{2^{i}}(e)) \quad \text{for $i\ge1$}.
}
So $\EE^{(e)}_0,\dots,\EE^{(e)}_{i-1}$ form a partition of all possible edges of a path (primal or dual) from $e$ until it first reaches $\partial B_{2^{i}}(e)$.
Moreover, whenever $1\le i<i'$, the disjoint union $\EE^{(e)}_{i}\cup\cdots\cup\EE^{(e)}_{i'-1}$ contains all the edges needed for a path from $\partial B_{2^i}(e)$ to $\partial B_{2^{i'}}(e)$ (in the sense of Definition~\ref{open_closed_paths_def}).

\medskip \noindent \textbf{Proof of \eqref{lemmm55} and \eqref{lemmm5}.}
We assume $e\in\CC_k$ for some $k$ (which implies $\CC_k$ is open), and either $\{K=k\}$ or $\{\text{$K>k$ and $\CC_{k+1}$ is open}\}$.
Let $\dualp_e$ be the dual path from Proposition \ref{prop:geod_cons}\ref{itm:alledual}, which starts at $e$, ends at $\AA$, and is closed except where it passes through the open circuits among $\CC_{k-1},\CC_{k-2},\dots,\CC_1$.
Let $I_1<\cdots<I_J$ be the random (possibly empty) sequence of $i\in\{0,\dots,\floor{\log_2 M}-1\}$ such that $\EE^{(e)}_i$ contains an open edge belonging to $\dualp_e$.
If the sequence is empty, we set $J=0$.
Our definitions lead to the following two claims.

\begin{claim}[Three-arm event from $e$] \label{claim_co}
The following statements hold:
\begin{enumerate}[label=\textup{(\alph*)}]

\item \label{claim_co_a}
There are $2$ disjoint open arms from $e$ to $\partial B_M(e)$.

\item \label{claim_co_b}
If $J=0$, then there is $1$ closed arm from $e$ to $\partial B_{2^{\floor{\log_2 M}}}(e)$.

\item \label{claim_co_c}
If $J\ge1$, then there is $1$ closed arm from $e$ to $\partial B_{2^{I_1}}(e)$.
\end{enumerate}
\end{claim}

\begin{claim}[Arms across annuli] \label{2claim_co}
Assume $J\ge1$. The following statements hold:
\begin{enumerate}[label=\textup{(\alph*)}] \setcounter{enumi}{3}
\item \label{2claim_co_d}
For each $j\in\llbrack1,J\rrbrack$, there are $\ceil{j/2}\cdot 2+1$ disjoint open arms from $\partial B_{2^{1 + I_j}}(e)$ to $\partial B_{M}(e)$.

\item \label{2claim_co_e}
For each $j\in\llbrack1,J-1\rrbrack$, there is $1$ closed arm from $\partial B_{2^{1 + I_j}}(e)$ to $\partial B_{2^{I_{j+1}}}(e)$.

\item \label{2claim_co_f}
There is $1$ closed arm from $\partial B_{2^{1 + I_{J}}}(e)$ to $\partial B_{2^{\floor{\log_2 M}}}(e)$.

\end{enumerate}
\end{claim}

\begin{proofclaim}
Part~\ref{claim_co_a}:
Since $e\in\CC_k$ and $\CC_k$ is open, we obtain two disjoint open arms to distance $M$ by following $\CC_k$ in both directions.

\medskip \noindent Parts~\ref{claim_co_b} and \ref{claim_co_c}: 
If $J=0$, then the dual path $\dualp_e$ reaches $\partial B_{2^{\floor{\log_2 M}}}(e)$ without using any open edges.
If instead $J\ge1$, then this path reaches $\partial B_{2^{I_1}}(e)$  without using any open edges.

\medskip \noindent Part~\ref{2claim_co_d}: 
Consider any $j\in\llbrack1,J\rrbrack$.
Recall that every open edge of $\dualp_e$ intersects a different open circuit among $\CC_{1},\dots,\CC_{k-1}$ (Proposition~\ref{prop:geod_cons}\ref{itm:alledual}).
There are at least $j$ such edges in $\EE^{(e)}_{I_1}\cup\EE^{(e)}_{I_2}\cup\cdots\cup\EE^{(e)}_{I_j}$, and each one yields two disjoint open arms from $\partial B_{2^{1 + I_j}}(e)$ to $\partial B_{M}(e)$, obtained by following the associated open circuit $\incir_i$.
We have thus identified $j\cdot 2$ open arms; however, if both $\incir_i$ and $\incir_{i+1}$ are used in this process, the open arms they produce may share vertices.
On account of this, we recall that $\incir_i$ and $\incir_{i + 2}$ are always vertex-disjoint (Proposition \ref{prop:geod_cons}\ref{itm:2donttouch}), so the number of \textit{disjoint} open arms is at least $\ceil{j/2}\cdot 2$.

We identify an additional disjoint open arm from $e$ to distance $M$ as follows.
Proceed along $\gamma$ from $e$ toward $\BB$; since $e\in\CC_k$, this portion of $\gamma$ does not intersect any of $\CC_{1},\dots,\CC_{k-1}$ (Proposition~\ref{prop:geod_cons}\ref{itm:gamma_on_circuit}) and thus avoids the open arms from the previous paragraph.
If $K=k$, then this path extends all the way to $\BB$ without encountering any closed edges.
If instead $K>k$, then this path reaches $\CC_{k+1}$ without encountering any closed edges, after which we can follow the open circuit $\CC_{k+1}$ around $\AA$ or $\BB$.  

\medskip \noindent Part~\ref{2claim_co_e}:
Consider any $j\in\llbrack1,J-1\rrbrack$.
Since $\dualp_e$ ends at distance $\ge M > 2^{I_{j+1}}$ away from $e$, it includes a portion between $\partial B_{2^{1+I_j}}(e)$ and $\partial B_{2^{I_{j+1}}}(e)$.
This subpath uses only edges in $\dualp_e\cap(\EE^{(e)}_{1+I_j}\cup\cdots\cup\EE^{(e)}_{I_{j+1}-1})$, so it is closed since every relevant index $i$ satisfies $I_j < i < I_{j+1}$. 

\medskip \noindent Part~\ref{2claim_co_f}:
Since $\dualp_e$ ends at distance $\ge M \ge 2^{1+I_J}$ away from $e$, it includes a portion between $\partial B_{2^{1+I_J}}(e)$ and $\partial B_{2^{\floor{\log_2 M}}}(e)$.
This subpath uses only edges in $\dualp_e\cap(\EE^{(e)}_{1+I_J}\cup\cdots\cup\EE^{(e)}_{\floor{\log_2 M}})$, so it is closed by maximality of $I_J$.
\end{proofclaim}

Claims \ref{claim_co} and \ref{2claim_co} together imply three essential statements listed below, from which the desired estimate will follow.
Let $k_\star$ be the large positive integer from Lemma \ref{mono_poly}.

\begin{enumerate}[label=\textup{(\roman*)}]
    \item \label{essential_1} 
    \textit{There is a 3-arm event from $e$ to $\partial B_{2^{\floor{\log_2 M}}}(e)$ if $J=0$, or to $\partial B_{2^{I_1}}(e)$ if $J\ge1$.}
    This is provided by Claim~\ref{claim_co}.
    
    \item \label{essential_2}
    \textit{Each annulus admits a polychromatic 4-arm event.}
        For $j\in\llbrack 1,J-1\rrbrack$, between $\partial B_{2^{1+I_j}}$ and $\partial B_{2^{I_{j+1}}}$, there are 3 open arms provided by Claim~\ref{2claim_co}\ref{2claim_co_d}, and 1 closed arm provided by Claim~\ref{2claim_co}\ref{2claim_co_e}.   
        Between $\partial B_{2^{1+I_J}}$ and $\partial B_{2^{\floor{\log M}}}(e)$, there are 3 open arms provided by Claim~\ref{2claim_co}\ref{2claim_co_d}, and 1 closed arm provided by Claim~\ref{2claim_co}\ref{2claim_co_f}.
    
    \item \label{essential_3}
    \textit{The final annulus admits a monochromatic $k_\star$-arm event if $J>k_\star$.}
    On the event $\{J>k_\star\}$, Claim~\ref{3claim_co}\ref{2claim_co_d} provides at least $k_\star$ open arms between $\partial B_{2^{1+I_J}}$ and $\partial B_M$.
\end{enumerate}
We stress that in \ref{essential_2}, it does not matter that there are specifically 3 open arms and 1 closed arm; any polychromatic arrangement will suffice thanks to asymptotic equivalence of polychromatic arm events \cite[Prop.~1.1]{reeves_sosoe22_arxiv}.
Similarly, in \ref{essential_3} it does not matter that the arms are open; a closed monochromatic $k_\star$-arm event will suffice thanks to Lemma~\ref{closed_like_open}.
The rest of the argument uses only \ref{essential_1}, \ref{essential_2}, and \ref{essential_3}.
So in future cases, we will just establish these three facts, and not repeat the following.

Let $\Asf$ be the event under consideration, i.e.
\eq{
\Asf = \{e\in\gamma\cap\CC_K\}\cup\{\text{$e \in \primalp\cap\CC_k$ for some $k\in\llbrack 2,K-1\rrbrack$, $\CC_{k+1}$ is open}\}.
}
First suppose $J=0$.
Then statement \ref{essential_1} implies
\eeq{\label{J=0}
\mathbb{P}(\Asf\cap\{J=0\}) \le \pi_3(2^{\floor{ \log M}}).
}
Next suppose $J\in\llbrack1,k_\star\rrbrack$.
Using \ref{essential_1} and then $J$ many applications of \ref{essential_2}, we obtain the second inequality below:
\begin{align}
&\mathbb{P}(\Asf\cap\{J\in\llbrack1,k_\star\rrbrack\})
\le \sum_{\ell=1}^{k_\star}\P(\Asf\cap\{J=\ell\})\nonumber \\
&\stackrefp{arm_calc_ineq}{\le} \sum_{\ell=2}^{k_\star}\sum_{{i}_1=0}^{ \floor{\log_2 M }-1}\sum_{{i}_2=1 + {i}_1}^{ \floor{\log_2 M}-1} \cdots \sum_{{i}_{\ell}=1+ {i}_{{\ell-1}}}^{ \floor{\log_2 M}-1}\Big[\pi_3( 2^{{i}_1})\pi_{4}(2^{1 + {i}_1 }, 2^{{i}_{2}})\cdots \label{L_small} \\
&\hphantom{\stackref{arm_calc_ineq}{\leq} \sum_{{i}_1=0}^{ \floor{\log_2 M }-1}\sum_{{i}_2=1 + {i}_1}^{ \floor{\log_2 M}-1} \cdots \sum_{{i}_{\ell}=1+ {i}_{{\ell-1}}}^{ \floor{\log_2 M}-1}\Big[}\cdots\pi_{4}(2^{1+ {i}_{\ell-1} }, 2^{i_\ell})\pi_{4}(2^{1 + {i}_{{\ell}}}, 2^{\floor{\log_2 M}})\Big]\nonumber\\
&\stackref{arm_calc_ineq}{\le} C\pi_3(1,2^{\floor{\log_2 M}}). \nonumber
\end{align}
Finally, suppose $J>k_\star$.
We proceed as in the case $J\in\llbrack1,k_\star\rrbrack$, except that beyond distance $2^{1+I_{k_\star}}$, we only keep track of open arms.
Namely, we use \ref{essential_1} followed by $k_\star-1$ many applications of \ref{essential_2}, and finally \ref{essential_3}:
\begin{align}
&\mathbb{P}(\Asf\cap\{J>k_\star\}) \nonumber \\
&\stackrefp{mono_poly_ineq}{\leq} \sum_{{i}_1=0}^{ \floor{\log_2 M}-1}\sum_{i_2 = 1 + i_1}^{\floor{\log_2 M}-1} \cdots \sum_{{i}_{k_\star}=1+ {i}_{k_\star-1} }^{ \floor{\log_2 M}-1}\Big[\pi_3(2^{i_1})\pi_4(2^{1 + i_1 } ,2^{i_2})\cdots \nonumber\\
&\hphantom{\stackrefp{mono_poly_ineq}{\leq}\sum_{{i}_1=0}^{ \floor{\log_2 M}-1}\sum_{i_2 = 1 + i_1}^{\floor{\log_2 M}-1} \cdots \sum_{i_{k_\star}=1+ i_{k_\star-1} }^{ \floor{\log_2 M}-1}\Big[}\cdots \pi_{4}(2^{1+ {i}_{k_\star-1} }, 2^{{i}_{k_\star}}) \pi'_{k_\star} (2^{1 + {i}_{k_\star}  }, M) \nonumber \\
&\stackref{mono_poly_ineq}{\leq} \sum_{{i}_1=0}^{ \floor{\log_2 M}-1}\sum_{i_2 = 1 + i_1 }^{\floor{\log_2 M}-1}  \cdots \sum_{{i}_{{J}}=1+ {i}_{k_\star-1} }^{ \floor{\log_2 M}-1}\Big[\pi_3(2^{{i}_1})\pi_4(2^{1 + i_1} ,2^{i_2})\cdots \label{L_large} \\
&\hphantom{\stackref{mono_poly_ineq}{\leq} \sum_{{i}_1=0}^{ \floor{\log_2 M}}\sum_{i_2 = 1 + i_1 }^{\floor{\log_2 M}}  \cdots \sum_{{i}_{{J}}=1+ {i}_{{J}-1} }^{ \floor{\log_2 M}}\Big[}\cdots \pi_{4}(2^{1+ {i}_{k_\star-1} }, 2^{{i}_{k_\star}}) \pi_{3} (2^{1 + {i}_{k_\star}}, M) \Big(\frac{2^{1 + {i}_{k_\star}}} {M}\Big)\Big] \nonumber\\
&\stackrefpp{arm_calc_ineq2}{mono_poly_ineq}{\le}  C\pi_3(1,M). \nonumber
\end{align}
Combining \eqref{J=0}--\eqref{L_large} proves \eqref{lemmm2}.

\medskip \noindent \textbf{Proof of \eqref{lemmm8}.}
Assume $e\in\gamma\cap\CC_k$ for some $k\in\llbrack2,K-1\rrbrack$ (which implies $\CC_k$ is open), and $\CC_{k+1}$ is closed.
This case is very similar to \eqref{lemmm5}, but now the geodesic $\gamma$ encounters a closed edge (specifically on $\CC_{k+1}$) before reaching either $\BB$ or an open circuit.
This means we lose the additional open arm in Claim~\ref{2claim_co}\ref{2claim_co_d}.
We account for this by finding an additional closed arm, as will be seen in the following argument.

Once again, let $\dualp_e$ be the dual path from Proposition \ref{prop:geod_cons}\ref{itm:alledual}, which starts at $e$, ends at $\AA$, and is closed except where it passes through the open circuits among $\CC_{k-1},\CC_{k-2},\dots,\CC_1$.
Let $f$ be the closed edge of $\gamma$ such that $f^\star\in\CC_{k+1}$, and let $I^\star$ be the unique integer such that $f\in\EE^{(e)}_{I^\star}$. 
In particular, the portion of $\gamma$ from $e$ to $\BB$ reaches $\partial B_{2^{I^\star}}(e)$ without using any closed edges.
Let $I_1<\cdots<I_J$ be the random (possibly empty) sequence of $i\in\{0,\dots,\floor{\log_2 M}-1\}$ such that \textit{either} $\EE^{(e)}_i$ contains an open edge belonging to $\dualp_e$, or $i = I^\star$.

Claim~\ref{claim_co} and its proof hold verbatim.
We just need to replace Claim~\ref{2claim_co}.
If $I^\star\ge \floor{\log_2 M}$, then the only modification to Claim~\ref{2claim_co} is that the last open arm in part~\ref{2claim_co_d} extends to $\partial B_{2^{\floor{\log_2 M}}}(e)$ rather than $\partial B_{M}(e)$.
This modified version is still sufficient thanks to \eqref{relax_to_2}.
Otherwise Claim~\ref{2claim_co} is replaced with the following.

\begin{claim}[Arms across annuli] \label{3claim_co}
Assume $I^\star \le \floor{\log_2 M}-1$, and let $J_f$ be the unique element of $\llbrack1,J\rrbrack$ such that $I_{J_f} = I^\star$. 
The following statements hold:
\begin{enumerate}[label=\textup{(\alph*)}] \setcounter{enumi}{3}
\item \label{3claim_co_d}
For each $j\in\llbrack 1,J_f-1\rrbrack$, there are $\ceil{j/2}\cdot 2+1$ disjoint open arms from $\partial B_{2^{1 + I_j}}(e)$ to $\partial B_{2^{I^\star}}(e)$.

\item[\textup{($\mathrm{d}'$)}] \label{3claim_co_dd}
For each $j\in\llbrack J_f,J\rrbrack$, there are $\ceil{(j-1)/2}\cdot 2$ disjoint open arms from $\partial B_{2^{1 + I_j}}(e)$ to $\partial B_{M}(e)$.

\item\label{3claim_co_e}
For each $j\in\llbrack 1,J_f-1\rrbrack$, there is a closed arm from $\partial B_{2^{1 + I_j}}(e)$ to $\partial B_{2^{I_{j+1}}}(e)$.

\item \label{3claim_co_f}
There are two closed arms from $\partial B_{2^{1 + I^\star}}(e)$ to $\partial B_M(e)$.

\end{enumerate}
\end{claim}

\begin{proofclaim}
Part~\ref{3claim_co_d}: 
The argument is the same as for Claim~\ref{2claim_co}\ref{2claim_co_d}, except now the additional open arm provided by $\gamma$ only extends to $\partial B_{2^{I^\star}}(e)$.

\medskip \noindent Part~\hyperref[3claim_co_dd]{\textup{($\mathrm{d}'$)}}: 
The argument is the same as for Claim~\ref{2claim_co}\ref{2claim_co_d}, except for the following detail.
Since $j\ge J_f$, the union $\EE^{(e)}_{I_1}\cup\EE^{(e)}_{I_2}\cup\cdots\cup\EE^{(e)}_{I_j}$ is only guaranteed to contain $j-1$ open edges of $\zeta_e$ (rather than $j$).
Therefore, the number of disjoint open arms is at least $\ceil{(j-1)/2}\cdot 2$.

\medskip \noindent Part~\ref{3claim_co_e}:
The argument is the same as for Claim~\ref{2claim_co}\ref{2claim_co_e}.

\medskip \noindent Part~\ref{3claim_co_f}:
By definition of $I^\star$, the closed edge $f$ belongs to $\partial B_{2^{I^\star}}(e)$.
Since $f^\star$ belongs to the closed circuit $\CC_{k+1}$, we can follow this circuit in both directions to obtain two disjoint closed arms from $\partial B_{2^{1+\star}}(e)$ to $\partial B_M(e)$.
\end{proofclaim}

To complete the proof, it suffices to check the three essential statements from before:
\begin{enumerate}[label=\textup{(\roman*)}]
    \item \textit{There is a 3-arm event from $e$ to $\partial B_{2^{\floor{\log_2 M}}}(e)$ if $J=0$, or to $\partial B_{2^{I_1}}(e)$ if $J\ge1$.}
    This is again provided by Claim~\ref{claim_co}.
    \item \textit{Each annulus admits a polychromatic 4-arm event.}
    \begin{itemize}
    \item  Suppose $I^\star = 1$.
        Between $\partial B_{2^{1+I_1}}$ and $\partial B_{2^{I_2}}$, there are two open arms provided by Claim~\ref{claim_co}\ref{claim_co_a}, and two closed arms provided by Claim~\ref{3claim_co}\ref{3claim_co_f}.
        For $j\in\llbrack 2,J-1\rrbrack$, between $\partial B_{2^{1+I_j}}$ and $\partial B_{2^{I_{j+1}}}$, there are (at least) two open arms provided by Claim~\ref{3claim_co}\hyperref[3claim_co_dd]{\textup{($\mathrm{d}'$)}}, and two closed arms provided by Claim~\ref{3claim_co}\ref{3claim_co_f}.
        Finally, between $\partial B_{2^{1+I_J}}$ and $\partial B_{2^{\floor{\log M}}}(e)$, there are two open arms provided by Claim~\ref{claim_co}\ref{claim_co_a}, and two closed arms provided by Claim~\ref{3claim_co}\ref{3claim_co_f}.
    \item Suppose $I^\star \ge 2$.
        For $j\in\llbrack 1,J_f-1\rrbrack$, between $\partial B_{2^{1+I_j}}$ and $\partial B_{2^{I_{j+1}}}$, there are (at least) three open arms provided by Claim~\ref{3claim_co}\ref{3claim_co_d}, and a closed arm provided by Claim~\ref{3claim_co}\ref{3claim_co_e}.
        For $j\in\llbrack J_f,J-1\rrbrack$, between $\partial B_{2^{1+I_j}}$ and $\partial B_{2^{I_{j+1}}}$, there are (at least) two open arms provided by Claim~\ref{3claim_co}\hyperref[3claim_co_dd]{\textup{($\mathrm{d}'$)}}, and two closed arms provided by Claim~\ref{3claim_co}\ref{3claim_co_f}.    
        Finally, between $\partial B_{2^{1+I_J}}$ and $\partial B_{2^{\floor{\log M}}}(e)$, there are two open arms provided by Claim~\ref{claim_co}\ref{claim_co_a}, and two closed arms provided by Claim~\ref{3claim_co}\ref{3claim_co_f}.
    \end{itemize}
    \item \textit{The final annulus admits a monochromatic $k_\star$-arm event if $J>k_\star$.}
    On the event $\{J>k_\star\}$, Claim~\ref{3claim_co}\hyperref[3claim_co_dd]{\textup{($\mathrm{d}'$)}} provides at least $k_\star$ open arms between $\partial B_{2^{1+I_J}}$ and $\partial B_M$.
\end{enumerate}

\begin{figure}[t]
\begin{center}
\tikzset{every picture/.style={line width=0.75pt}} 

\begin{tikzpicture}[x=0.75pt,y=0.75pt,yscale=-1,xscale=1]

\draw  [draw opacity=0][fill={rgb, 255:red, 155; green, 155; blue, 155 }  ,fill opacity=0.5 ] (112,196.24) .. controls (112,185.19) and (127.67,176.24) .. (147,176.24) .. controls (166.33,176.24) and (182,185.19) .. (182,196.24) .. controls (182,207.29) and (166.33,216.24) .. (147,216.24) .. controls (127.67,216.24) and (112,207.29) .. (112,196.24) -- cycle ;
\draw  [fill={rgb, 255:red, 155; green, 155; blue, 155 }  ,fill opacity=0.5 ] (330.83,93.44) -- (409.23,93.44) -- (409.23,171.04) -- (330.83,171.04) -- cycle(386.23,116.44) -- (353.84,116.44) -- (353.84,148.03) -- (386.23,148.03) -- cycle ;
\draw  [fill={rgb, 255:red, 155; green, 155; blue, 155 }  ,fill opacity=0.44 ] (267.03,29.67) -- (473.03,29.67) -- (473.03,234.8) -- (267.03,234.8) -- cycle(435.83,66.88) -- (304.24,66.88) -- (304.24,197.6) -- (435.83,197.6) -- cycle ;
\draw  [line width=1.5] [line join = round][line cap = round] (357.94,118.04) .. controls (364.99,119.47) and (372.37,119.74) .. (379.08,122.33) .. controls (387.03,125.38) and (391.35,135.52) .. (401.65,132.04) .. controls (412.87,128.25) and (422.57,120.79) .. (433.65,116.61) .. controls (446.59,111.74) and (446.62,119.12) .. (460.51,116.9) .. controls (465.44,116.11) and (469.65,112.9) .. (474.22,110.9) ;
\draw  [line width=1.5] [line join = round][line cap = round] (131.27,175.1) .. controls (119.31,157.15) and (105.17,133.95) .. (133.27,119.9) .. controls (138.06,117.51) and (143.16,115.39) .. (148.47,114.7) .. controls (157.59,113.51) and (166.91,113.49) .. (176.07,114.3) .. controls (196.44,116.1) and (213.88,128.59) .. (231.27,137.5) .. controls (249.96,147.08) and (271.71,158.11) .. (292.07,163.1) .. controls (303.3,165.85) and (327.82,173.36) .. (338.07,163.1) .. controls (364.79,136.38) and (314.2,123.46) .. (296.47,115.9) .. controls (292.41,114.17) and (288.38,112.35) .. (284.47,110.3) .. controls (278.8,107.33) and (273.51,104.1) .. (268.07,100.7) .. controls (259.32,95.23) and (247.35,78.39) .. (256.07,68.7) .. controls (264.8,59) and (288.2,60.59) .. (298.07,67.1) .. controls (313.01,76.95) and (318.86,97.03) .. (334.07,105.9) .. controls (342.25,110.67) and (351.14,112.36) .. (358.07,117.9) .. controls (361.82,120.9) and (366.68,133.1) .. (371.27,133.1) ;
\draw  [draw opacity=0][fill={rgb, 255:red, 155; green, 155; blue, 155 }  ,fill opacity=0.5 ] (547.4,67.27) .. controls (547.4,56.22) and (563.07,47.27) .. (582.4,47.27) .. controls (601.73,47.27) and (617.4,56.22) .. (617.4,67.27) .. controls (617.4,78.31) and (601.73,87.27) .. (582.4,87.27) .. controls (563.07,87.27) and (547.4,78.31) .. (547.4,67.27) -- cycle ;
\draw  [line width=1.5] [line join = round][line cap = round] (470.8,112.73) .. controls (473.47,112.33) and (476.28,109.28) .. (478.87,108.29) .. controls (486.8,105.25) and (495.12,103.35) .. (503.29,101.04) .. controls (523.41,95.34) and (544.19,97.22) .. (562.4,84.31) ;
\draw [color={rgb, 255:red, 155; green, 155; blue, 155 }  ,draw opacity=1 ][line width=3]    (466.94,8.78) .. controls (437.61,36.78) and (442.94,118.11) .. (545.61,208.11) ;
\draw [color={rgb, 255:red, 155; green, 155; blue, 155 }  ,draw opacity=1 ][line width=3]    (258.67,13.17) .. controls (295.33,160.33) and (391.67,79.17) .. (432.67,283.17) ;
\draw  [color={rgb, 255:red, 0; green, 0; blue, 0 }  ,draw opacity=1 ][fill={rgb, 255:red, 155; green, 155; blue, 155 }  ,fill opacity=1 ] (273.36,61.71) .. controls (273.36,60.17) and (274.61,58.91) .. (276.16,58.91) .. controls (277.7,58.91) and (278.96,60.17) .. (278.96,61.71) .. controls (278.96,63.26) and (277.7,64.51) .. (276.16,64.51) .. controls (274.61,64.51) and (273.36,63.26) .. (273.36,61.71) -- cycle ;
\draw  [color={rgb, 255:red, 0; green, 0; blue, 0 }  ,draw opacity=1 ][fill={rgb, 255:red, 155; green, 155; blue, 155 }  ,fill opacity=1 ] (465.94,113.84) .. controls (465.94,112.29) and (467.19,111.04) .. (468.74,111.04) .. controls (470.29,111.04) and (471.54,112.29) .. (471.54,113.84) .. controls (471.54,115.38) and (470.29,116.64) .. (468.74,116.64) .. controls (467.19,116.64) and (465.94,115.38) .. (465.94,113.84) -- cycle ;
\draw [color={rgb, 255:red, 155; green, 155; blue, 155 }  ,draw opacity=1 ][line width=1.5]    (348.25,129.06) .. controls (366.65,125.46) and (442.28,167.44) .. (482.28,137.44) ;
\draw  [fill={rgb, 255:red, 0; green, 0; blue, 0 }  ,fill opacity=1 ] (367.23,132.24) .. controls (367.23,130.69) and (368.49,129.44) .. (370.03,129.44) .. controls (371.58,129.44) and (372.83,130.69) .. (372.83,132.24) .. controls (372.83,133.79) and (371.58,135.04) .. (370.03,135.04) .. controls (368.49,135.04) and (367.23,133.79) .. (367.23,132.24) -- cycle ;
\draw  [color={rgb, 255:red, 0; green, 0; blue, 0 }  ,draw opacity=1 ][fill={rgb, 255:red, 155; green, 155; blue, 155 }  ,fill opacity=1 ] (342.16,152.51) .. controls (342.16,150.97) and (343.41,149.71) .. (344.96,149.71) .. controls (346.5,149.71) and (347.76,150.97) .. (347.76,152.51) .. controls (347.76,154.06) and (346.5,155.31) .. (344.96,155.31) .. controls (343.41,155.31) and (342.16,154.06) .. (342.16,152.51) -- cycle ;

\draw (138,186.74) node [anchor=north west][inner sep=0.75pt]    {$\mathcal{A}$};
\draw (370.83,135.64) node [anchor=north west][inner sep=0.75pt]    {$e$};
\draw (175.35,93.03) node [anchor=north west][inner sep=0.75pt]    {$\zeta $};
\draw (447.83,114.53) node [anchor=north west][inner sep=0.75pt]    {$$};
\draw (577.2,60.57) node [anchor=north west][inner sep=0.75pt]    {$\mathcal{B}$};
\draw (436,278.07) node [anchor=north west][inner sep=0.75pt]    {$\mathcal{C}_{k}$};
\draw (543.33,210.07) node [anchor=north west][inner sep=0.75pt]    {$\mathcal{C}_{k+1}$};

\end{tikzpicture}
\caption{\small An illustration of the case when $e \in \primalp$ is between two open circuits $\mathcal{C}_k$ and $\mathcal{C}_{k+1}$. 
There exists a dual path $\dualp$ between $\mathcal{A}$ and $\mathcal{B}$. 
The open dual edges in $\dualp$ are contained in the collection $\{\EE_{I_j}^{(e)}\}_{j=1}^J$, which correspond to the shaded annuli.
The edge $e$ is connected to $\dualp$ by a closed dual path $\dualp_e$.} \label{fig:OO}
\end{center}
\end{figure}

\medskip


\noindent \textbf{Proof of \eqref{lemmm2}.} 
Assume $e$ is between two open circuits $\CC_k$ and $\CC_{k+1}$.
We mimic the proof strategy of the two previous cases, but our method for obtaining arms needs to be modified.
For an illustration of the argument, see Figure \ref{fig:OO}.
Let $I_1<\cdots<I_J$ be the random (possibly empty) sequence of $i\in\{0,\dots,\floor{\log_2 M}-1\}$ such that $\EE^{(e)}_i$ contains an open edge belonging to the dual path $\dualp$ from $\AA$ to $\BB$.

\begin{claim}[Three-arm event from $e$] \label{claim_2open}
The following statements hold:
\begin{enumerate}[label=\textup{(\alph*)}]

\item \label{claim_2open_a}
There are $2$ disjoint open arms from $e$ to $\partial B_M(e)$.

\item \label{claim_2open_b}
If $J=0$, then there is $1$ closed arm from $e$ to $\partial B_{2^{\floor{\log_2 M}}}(e)$.

\item \label{claim_2open_c}
If $J\ge1$, then there is $1$ closed arm from $e$ to $\partial B_{2^{I_1}}(e)$.
\end{enumerate}
\end{claim}

\begin{claim}[Arms across annuli] \label{2claim_2open}
Assume $J\ge1$. The following statements hold:
\begin{enumerate}[label=\textup{(\alph*)}] \setcounter{enumi}{3}
\item \label{2claim_2open_d}
For each $j\in\llbrack1,J\rrbrack$, there are $\ceil{j/2}\cdot 2$ disjoint open arms from $\partial B_{2^{1 + I_j}}(e)$ to $\partial B_{M}(e)$.

\item \label{2claim_2open_e}
For each $j\in\llbrack1,J-1\rrbrack$,  there are $2$ disjoint closed arms $\partial B_{2^{1 + I_j}}(e)$ to $\partial B_{2^{I_{j+1}}}(e)$.

\item \label{2claim_2open_f}
There are $2$ disjoint closed arms from $\partial B_{2^{1 + I_{J}}}(e)$ to $\partial B_{2^{\floor{\log_2 M}}}(e)$.

\end{enumerate}
\end{claim}

\begin{proofclaim}

Part~\ref{claim_2open_a}:
By Proposition \ref{prop:geod_cons}\ref{itm:gamma_on_circuit}, the assumption that $e$ is between $\CC_k$ and $\CC_{k + 1}$ implies that these two circuits are vertex-disjoint. 
Therefore, the following two arms to distance $M$ are disjoint: from $e$, proceed along $\gamma$ in both directions until reaching $\CC_k$ and $\CC_{k + 1}$, respectively, then follow these circuits around either $\mathcal{A}$ or $\mathcal{B}$. 
The portions along $\gamma$ are open by Proposition \ref{prop:geod_cons}\ref{itm:dualcir}, and the portions along $\CC_k$ and $\CC_{k+1}$ are open by assumption.

\medskip \noindent Parts~\ref{claim_2open_b} and \ref{claim_2open_c}: 
By assumption, $e$ does not belong to any of the circuits $\CC_1,\dots,\CC_K$.
In particular, $e$ is open (Proposition~\ref{prop:geod_cons}\ref{itm:dualcir}), so there exists a closed dual path $\dualp_e$ from $e$ to $\dualp$ (Proposition~\ref{prop:geod_cons}\ref{itm:dualconn}).
Consider the following dual path: follow $\dualp_e$ from $e$ to $\dualp$, then follow $\dualp$ to either $\AA$ or $\BB$ (the choice does not matter).
If $J=0$, then this path reaches $\partial B_{2^{\floor{\log_2 M}}}(e)$ without using any open edges.
If instead $J\ge1$, then this path reaches $\partial B_{2^{I_1}}(e)$  without using any open edges.

\medskip \noindent Part~\ref{2claim_2open_d}: 
Consider any $j\in\llbrack1,J\rrbrack$.
Recall that every open edge of $\dualp$ intersects a different open circuit from the edge-disjoint collection $\incir_1,\dots,\incir_P$ (Proposition~\ref{prop:geod_cons}\ref{itm:zetaopen}).
There are at least $j$ such edges in $\EE^{(e)}_{I_1}\cup\EE^{(e)}_{I_2}\cup\cdots\cup\EE^{(e)}_{I_j}$, and each one yields two disjoint open arms from $\partial B_{2^{1 + I_j}}(e)$ to $\partial B_{M}(e)$, obtained by following the associated open circuit $\incir_i$ in both directions.
Since $\incir_i$ and $\incir_{i + 2}$ are always vertex-disjoint (Proposition \ref{prop:geod_cons}\ref{itm:2donttouch}), in total we have at least $\ceil{j/2}\cdot 2\ge \max\{j,2\}$ disjoint open arms from $\partial B_{2^{1 + I_j}}(e)$ to $\partial B_{M}(e)$.

\medskip \noindent Part~\ref{2claim_2open_e}:
Consider any $j\in\llbrack1,J-1\rrbrack$.
The edge set $\dualp\cap\EE^{(e)}_{I_j}$ is nonempty by definition of $I_j$.
Therefore, $\dualp$ contains a vertex that is (a dual neighbor of) some element of $B_{2^{1+I_j}}(e)$.
Since $\dualp$ starts and ends at distance at least $M > 2^{1+I_j}$ away from $e$, it produces two disjoint paths between $\partial B_{2^{1+I_j}}(e)$ and $\partial B_{2^{I_{j+1}}}(e)$.
These paths use only edges in $\dualp\cap(\EE^{(e)}_{1+I_j}\cup\cdots\cup\EE^{(e)}_{I_{j+1}-1})$, so  they are closed since every relevant index $i$ satisfies $I_j < i < I_{j+1}$. 

\medskip \noindent Part~\ref{2claim_2open_f}:
The edge set $\dualp\cap\EE^{(e)}_{I_J}$ is nonempty by definition of $I_J$.
Therefore, $\dualp$ contains a vertex that is (a dual neighbor of) some element of $B_{2^{1+I_J}}(e)$.
Since $\dualp$ starts and ends at distance $M > 2^{1+I_J}$ away from $e$, it produces two disjoint paths between $\partial B_{2^{1+I_J}}(e)$ and $\partial B_{2^{\floor{\log_2 M}}}(e)$.
These paths use only edges in $\dualp\cap(\EE^{(e)}_{1+I_J}\cup\cdots\cup\EE^{(e)}_{\floor{\log_2 M}})$, so  they are closed by maximality of $I_J$.
\end{proofclaim}

To complete the proof, it suffices to check the three essential statements from before:
\begin{enumerate}[label=\textup{(\roman*)}]
    \item \textit{There is a 3-arm event from $e$ to $\partial B_{2^{\floor{\log_2 M}}}(e)$ if $J=0$, or to $\partial B_{2^{I_1}}(e)$ if $J\ge1$.}
    This is provided by Claim~\ref{claim_2open}.
    \item \textit{Each annulus admits a polychromatic 4-arm event.}
        For $j\in\llbrack 1,J-1\rrbrack$, between $\partial B_{2^{1+I_j}}$ and $\partial B_{2^{I_{j+1}}}$, there are $2$ open arms provided by Claim~\ref{2claim_2open}\ref{2claim_2open_d}, and $2$ closed arms provided by Claim~\ref{2claim_2open}\ref{2claim_2open_e}.   
        Between $\partial B_{2^{1+I_J}}$ and $\partial B_{2^{\floor{\log M}}}(e)$, there are $2$ open arms provided by Claim~\ref{2claim_2open}\ref{2claim_2open_d}, and $2$ closed arms provided by Claim~\ref{2claim_2open}\ref{2claim_2open_f}.
    \item \textit{The final annulus admits a monochromatic $k_\star$-arm event if $J>k_\star$.}
    On the event $\{J>k_\star\}$, Claim~\ref{2claim_2open}\ref{2claim_2open_d} provides at least $k_\star$ open arms between $\partial B_{2^{1+I_J}}$ and $\partial B_M$.
\end{enumerate}

\begin{remark} \label{rem_no_bernoulli}
The proof thus far has not used that the edge-weights are $\{0,1\}$-valued.
Indeed, the only special consequence of this assumption is that the \textit{closed} circuits among $\CC_1,\dots,\CC_K$ are edge-disjoint (Proposition~\ref{prop:geod_cons}\ref{itm:Csequence}), but until now we have only needed the edge-disjointness of the \textit{open} circuits among this sequence, which are $\incir_1,\dots,\incir_P$.
Therefore, $\P(\{e\in\gamma\}\cap\Gsf)\le C\pi_3(M)$ holds assuming only \eqref{critical_assumption} whenever $\Gsf$ is one of the following events: 
\begin{itemize}
    \item $e\in\incir_i$ for some $i$, as in \eqref{lemmm6}, \eqref{lemmm55}, \eqref{lemmm5}, \eqref{lemmm8}.
    \item $e$ lies on an entirely open portion of $\gamma$ that either
    \begin{itemize}
        \item starts at $\AA$ and ends at $\BB$ as \eqref{lemmm1};
        \item starts at $\AA$ and ends at $\incir_1$ as in \eqref{lemmm4};
        \item starts at $\incir_P$ and ends at $\BB$ as in \eqref{lemmm44}; or
        \item starts at $\incir_i$ and ends at $\incir_{i+1}$ as in \eqref{lemmm2}.
    \end{itemize}
\end{itemize}
By contrast, the remainder of the proof requires Proposition~\ref{prop:geod_cons}\ref{itm:Csequence}, which only works for $\{0,1\}$-valued edge weights.
\end{remark}

\medskip \noindent \textbf{Proof of \eqref{lemmm3}, \eqref{lemmm33}, \eqref{lemmm333}, \eqref{lemmm3333}.}
This argument is similar to the case of \eqref{lemmm2}, but now the arm events across annuli are created using the geodesic $\gamma$ instead of the dual path $\dualp$.
Namely, let $I_1<\cdots<I_J$ be the random (possibly empty) sequence of $i\in\{0,\dots,\floor{\log_2 M}-1\}$ such that $\EE^{(e)}_i$ contains a closed edge belonging to $\primalp$.
Then Claim~\ref{2claim_2open} holds once we exchange ``closed'' and ``open,'' and the proof is verbatim except for replacing $\dualp$ with $\primalp$, and applying parts \ref{itm:dualcir} and \ref{itm:Csequence} of Proposition~\ref{prop:geod_cons} instead of parts \ref{itm:zetaopen} and \ref{itm:2donttouch_new}.
We just need to replace Claim~\ref{claim_2open} with the following.

\begin{claim}[Three-arm event from $e$] \label{claim_1closed}
The following statements hold:
\begin{enumerate}[label=\textup{(\alph*)}]

\item \label{claim_1closed_a}
If $J=0$, then there are $2$ disjoint open arms from $e$ to $\partial B_{2^{\floor{\log_2 M}}}(e)$.

\item[\textup{$(\mathrm{a}')$}] \label{claim_1closed_aa}
If $J\ge 1$, then there are $2$ disjoint open arms from $e$ to $\partial B_{2^{I_1}}(e)$.

\item \label{claim_1closed_b}
There is $1$ closed arm from $e$ to $\partial B_M(e)$.

\end{enumerate}
\end{claim}

\begin{proofclaim}

Parts~\ref{claim_1closed_a} and \hyperref[claim_1closed_aa]{\textup{$(\mathrm{a}')$}}:
Follow $\gamma$ from $e$ in both directions.
If $J=0$, then $\gamma$ reaches $\partial B_{2^{\floor{\log_2 M}}}(e)$ without using any closed edges.
If instead $J\ge1$, then $\gamma$ reaches $\partial B_{2^{I_1}}(e)$ without using any closed edges.

\medskip \noindent Part~\ref{claim_2open_b}: 
We consider two possibilities: 
\begin{itemize}
    \item In \eqref{lemmm3333}, $e^\star$ belongs to a closed dual circuit that encloses either $\AA$ or $\BB$.
    Following this circuit, we obtain a closed arm from $e$ to distance $M$ (in fact, we obtain two such arms).
    \item In \eqref{lemmm3}, \eqref{lemmm33}, \eqref{lemmm333}, $e$ must be open (Proposition~\ref{prop:geod_cons}\ref{itm:dualcir}), so there exists a closed dual path $\dualp_e$ from $e$ to $\dualp$ (Proposition~\ref{prop:geod_cons}\ref{itm:dualconn}).
    Furthermore, we assume $e$ is between $\CC_k$ and $\CC_{k+1}$ (replace $\CC_0$ with $\AA$, and replace $\CC_{K+1}$ with $\BB$), and at least one of these circuits is closed.
    If $\dualp_e$ intersects either $\CC_k$ or $\CC_{k+1}$, then a closed arm to distance $M$ is obtained by simply following $\dualp_e$ from $e$ to that circuit (which must be closed since $\dualp_e$ is closed), and then following that circuit around $\AA$ or $\BB$.
    Otherwise, $\dualp_e$ intersects $\zeta$ along the portion of $\zeta$ between $\CC_k$ and $\CC_{k+1}$.
     Since $\zeta$ is entirely closed along this portion (Proposition~\ref{prop:geod_cons}\ref{itm:zetaopen}), a closed arm to distance $M$ is obtained by following $\dualp_e$ from $e$ to $\dualp$, then following $\dualp$ to whichever of $\CC_k$ and $\CC_{k+1}$ is closed, and then following that circuit around $\AA$ or $\BB$. \qedhere
\end{itemize} 
\end{proofclaim}
Now repeat what comes after the proof of Claim~\ref{2claim_2open}, but replace Claim~\ref{claim_2open} with Claim~\ref{claim_1closed}, and exchange ``closed'' and ``open.''
\end{proof}

\subsection{Upper bounds on length of geodesic} \label{sec_ber_proof_final}

We are now ready to prove our first two main results.
As usual, the value of the constant $C$ may change with each inequality.


\begin{proof}[Proof of Theorem \ref{thm:Ber_ptb}]
Let $\primalp$ be the geodesic constructed in Proposition \ref{prop:geod_cons}, where $\mathcal{A} = \{\vc 0\}$ and $\mathcal{B} = \partial B_R$. 
Let $|\primalp|$ denote the number of edges in $\primalp$. 
Note that every edge of $\gamma$ has both endpoints in $B_R$ and is thus separated from $\{\vc 0\}\cup\partial B_R$ by some distance $r\in\{0,\dots,R\}$.
We now have
\eq{
\mathbb E[\lng_R] 
\le \mathbb E|\primalp| 
= \sum_{e\in E(\Z^2)} \mathbb P(e \in \primalp)
&\stackrefp{key3}{=} \sum_{r = 0}^{R} \sum_{\substack{e\in E(\Z^2) \\ \dist(e,\{\vc 0\}) \wedge \dist(e, \partial B_R) = r }} \mathbb P(e \in \primalp) \\
&\stackref{key3}{\le} C \sum_{r = 0}^R  R \pi_3(r) 
\stackref{3r8bbc}{\le} C R^2 \pi_3(R). \qedhere
}
\end{proof}

\begin{proof}[Proof of Theorem \ref{thm:Ber_ptp}]
The inequality \eqref{cm33} is trivial if $\theta<1$, so we assume $\theta\ge1$.
Recall the notation $R = \dist(\mathcal{A},\mathcal{B})$.
By translation invariance, we may assume without loss of generality that
\eeq{ \label{jb210b}
\text{$\mathcal{A}$ contains the origin}.
}
For $i \ge 1$, let $\Hsf_i$ be the event that both of the following statements are true:
\begin{enumerate} [label=\textup{(\roman*)}]
\item there exists an open circuit $\CC$ that encloses $B_{4^{i-1}}$ and whose vertex set is contained in $B_{4^i/2}$; and
\item there exists a closed dual circuit $\DD$ that encloses $B_{4^i/2}$ and whose vertex set is contained in $B_{4^i}$.
\end{enumerate}
Note that every vertex of $\CC$ lies in $\intr(\DD)$, so every vertex of $\DD$ lies in $\ext(\CC)$ by Lemma~\ref{surround_lemma}.
The events $\{\Hsf_i\}_{i\ge1}$ are independent, and by the RSW theorem \eqref{RSW} together with the FKG inequality, there exists a constant $\delta > 0$ such that
$\mathbb P(\Hsf_i) \ge 1 - \e^{-\delta}$ for all $i \ge 1$.
In particular, for all $i'\ge1$ we have
\eeq{ \label{ub7b3c}
\P\Bigl(\bigcap_{i = i'}^{n} \Hsf_i^\complement\Bigr) = \prod_{i = i'}^{n}  (1-\mathbb P(\Hsf_i)) \le \e^{-\delta(n-i'+1)}.
}
Henceforth we take $i'$ large enough that $B_{4^{i'-1}}$ contains both $\mathcal{A}$ and $\mathcal{B}$.
By \eqref{jb210b}, the following choice suffices:
\eeq{ \label{jv6xh}
i' = \ceil{\log_4(|\AA|+|\BB|+R)}+1.
}
We will now argue that for any $i\ge i'$, on the event $\Hsf_i$, all geodesics from $\mathcal{A}$ to $\mathcal{B}$ are contained in $B_{4^i}$. 
That is, we claim
\eeq{ \label{kwefx}
\Hsf_i \subseteq \bigcap_{\gamma\in\Geo(\mathcal{A},\mathcal{B})}\{\text{all vertices of $\gamma$ belong to $B_{4^i}$}\} \quad \text{for all $i\ge i'$},
}
where $\Geo(\mathcal{A},\mathcal{B})$ denotes the set of geodesics between $\mathcal{A}$ and $\mathcal{B}$.
To justify \eqref{kwefx}, let $\CC$ and $\DD$ be any choice of the open and closed circuits guaranteed by the event $\Hsf_i$, and let $\gamma$ be any geodesic between $\mathcal{A}$ and $\mathcal{B}$.
Notice that both $\mathcal{A}$ and $\mathcal{B}$ are contained in $B_{4^{i'}-1}\subseteq B_{4^{i-1}}\subseteq\intr(\CC)$.
Therefore, if the geodesic $\gamma$ has any segment in $\ext(\CC)$, this segment begins and ends at vertices on $\CC$.
The passage time between any two vertices on $\CC$ is $0$, so this segment cannot use any closed edges.
In particular, $\gamma$ never crosses the closed circuit $\DD$, since $\DD$ lies in the exterior of $\CC$.
Therefore, $\gamma$ remains entirely in $\intr(\DD)\subseteq B_{4^i}$.


Now let $\gamma$ be the geodesic between $\mathcal{A}$ and $\mathcal{B}$ from Proposition \ref{prop:geod_cons}.
For any $n \ge i'$, we have
\eeq{ \label{tb2p}
&\mathbb P(\lng_{\mathcal{A}, \mathcal{B}} 
\ge \theta R^2 \pi_3(R)) 
\stackref{kwefx}{\le} \mathbb P\Big(\{\lng_{\mathcal{A}, \mathcal{B}} \ge \theta R^2 \pi_3(R)\}\cup\bigcup_{i=i'}^n \Hsf_i\Big) + \e^{-\delta(n-i'+1)} \\
&\stackref{kwefx}{\le} \mathbb P(|\gamma| \ge \theta R^2 \pi_3(R),\,\text{all vertices of $\gamma$ belong to $B_{4^n}$}) + \e^{-\delta(n-i'+1)} \\
&\stackrefp{kwefx}{\le} \frac{\E[|\gamma| \one(\text{all vertices of $\gamma$ belong to $B_{4^n}$})]}{\theta R^2 \pi_3(R)} + \e^{-\delta(n-i'+1)}.
}
We now wish to control the expectation appearing in the final line of \eqref{tb2p}.
Let $E(B_{4^n})$ denote the set of edges with both endpoints in $B_{4^n}$.
We have
\eq{
\mathbb E[|\gamma| \one( \text{all vertices of $\gamma$ belong to $B_{4^n}$})] 
&\stackrefp{key3}{\le}  \sum_{r = 0}^{4^{n}} \sum_{\substack{e \in E(B_{4^n}) \\ \dist(e,\mathcal{A})\wedge \dist(e,\mathcal{B}) = r}} \mathbb P(e \in \gamma) \\
&\stackref{key3}{\le} \sum_{r = 0}^{4^{n}} \sum_{\substack{e \in E(B_{4^n}) \\ \dist(e,\mathcal{A})\wedge \dist(e,\mathcal{B}) = r}} C \pi_3(r) \\
&\stackrefp{key3}{\le} \sum_{r = 0}^{4^{n}} C(|\mathcal A| + |\mathcal B|)r \pi_3(r) \\
&\stackrefpp{3r8bbc}{key3}{\le} C(|\mathcal A| + |\mathcal B|) (4^{n})^2 \pi_3(4^{n}). 
}
Now \eqref{tb2p} leads to
\[
\mathbb P(\lng_{\mathcal{A}, \mathcal{B}} \ge \theta R^2 \pi_3(R)) 
\le \f{C(|\mathcal A| + |\mathcal B|) 16^{n} \pi_3(4^{n})}{\theta R^2 \pi_3(R)} + \e^{-\delta(n-i'+1)}.
\]
Now set $n = i' + \lfloor \log_{16}\sqrt{\theta} \rfloor$ with $i'$ as in \eqref{jv6xh}, so that $16^n \le 16^{2}(|\AA|+|\BB|+R)^2\sqrt{\theta}$.
In addition, we have $\pi_3(4^n)\le\pi_3(R)$ since $4^{n} \ge 4^{i'} \ge |\AA|+|\BB|+R > R$.
Using these observations in the previous display, we obtain
\eq{
\mathbb P(\lng_{\mathcal{A}, \mathcal{B}} \ge \theta R^2 \pi_3(R))
\leq \frac{C(|\mathcal{A}|+|\mathcal{B}|)(|\mathcal{A}|+|\mathcal{B}|+R)^2}{\sqrt{\theta}R^2} + \e^{-\delta\log_{16}\sqrt{\theta}}.
}
Upon using the inequality $|\AA|+|\BB|+R \le 2R(|\AA|+|\BB|)$, taking $c = \tfrac12\min\{1, \delta/\log 16\}$, and further adjusting the constant $C$, we simplify this final inequality to \eqref{cm33}.
\end{proof}

\section{Proofs for general edge-weights} \label{sec:general_proof}
The goal of this section is to prove Theorem~\ref{thm:general_thm} for any edge-weight distribution function $F$ satisfying \eqref{critical_assumption} and additionally one of following conditions:
\eeqs{
\limsup_{n\to\infty} F^{-1}(p_{\b^{n+1}})/F^{-1}(p_{\b^{n}}) &< 1  
\label{eq : limit_bigger_than_one},\\
\text{or} \quad \liminf_{n\to\infty} F^{-1}(p_{\b^{2n}})/F^{-1}(p_{\b^{n}}) &> 0, \label{eq : limit_equal_zero}
}
where $p_{R}$ is defined in \eqref{def_pn}, and $b$ is an integer (at least $2$) that is fixed throughout this section. 
The conditions \eqref{eq : limit_bigger_than_one} and \eqref{eq : limit_equal_zero} will be used separately in different ways, thereby creating parallel lines of argument at various spots in the exposition.
On the other hand, several intermediate results require only the criticality condition \eqref{critical_assumption}. 
For each result, we will make clear which assumptions are being made.
As before, the constants $C,c>0$ can change with each inequality. 
Their values may depend on $\b$ and the distribution function $F$.

\subsection{Correlation length and verification of examples} \label{sec:corr_length}
The lemma below provides various properties of the correlation length $p\mapsto L(p)$ and its approximate inverse $R\mapsto p_R$, defined in Section \ref{rsl_general}.
Display \eqref{eq: scaling_relation} is commonly referred to as \textit{Kesten's scaling relation}.

\begin{lemma} \label{lem:cor_length}
There exist constants $C,c>0$ such that the following inequalities hold:
\eeq{
\label{eq: L_p_n}
    c R \leq L(p_R) \leq R \quad \text{for all $R\geq1$,}
}
\eeq{
\label{eq: scaling_relation}
    c \leq L(p)^2 \pi_4(1,L(p)) (p-p_\cc) \leq C \quad \text{for all $p>p_\cc$,}
}
\eeq{
\label{eq: comparable_p_n_equation}
    c\Bigl(\f{R}{r}\Bigr)^{-(1-c)}	\le \f{p_R - p_\cc}{p_{r} - p_\cc} \le C\Bigl(\f{R}{r}\Bigr)^{-2/3} \quad \text{for all $R\geq r\geq 1$,}
}
\eeq{
\label{pn-pcbd}
    cR^{-(1-c)} \le p_R - p_\cc \le CR^{-2/3} \quad \text{for all $R\geq 1$.}
}
\end{lemma}

\begin{proof}

For the first claim \eqref{eq: L_p_n}, we start with the upper bound. 
Recall from \eqref{corr_L} that $L(p) = L(p,\eps_\star)$ where $\eps_\star>0$ is chosen sufficiently small, and
\eq{
L(p,\eps_\star) = \min\{ R\geq 1:\, \sigma(R,p) \ge 1 - \eps_\star \}.
}
Since $p\mapsto\sigma(R,p)$ is clearly increasing (and continuous) by definition \eqref{cross_prob}, the map $p\mapsto L(p)$ is weakly decreasing.
It follows from this monotonicity that $L(p)\le R$ for all $p > p_R$, since $p_R = \inf\{p>p_\cc:\, L(p)\le R\}$.
If it were the case that $L(p_R) > R$, then $\sigma(r,p_R) < 1-\eps_\star$ for all $r\le R$, which (by continuity) implies $\sigma(r,p) < 1-\eps_\star$ for all $r\le R$ and $p$ sufficiently close to $p_R$.
But this would mean $L(p)>R$ for all $p$ sufficiently close to $p_R$, which contradicts our earlier observation.
So $L(p_R)\le R$.


Regarding the lower bound in \eqref{eq: L_p_n}, \cite[disp.~(2.10)]{jarai03} states that there exists a constant $D > 0$ such that
\begin{equation} \label{eq:L_Rat}
\lim_{\delta \searrow 0} \f{L(p-\delta)}{L(p)} \le D\quad\text{for all $p > p_\cc$.}
\end{equation}
We now show that this implies the desired lower bound in \eqref{eq: L_p_n}  with $c = \f{1}{2D}$. Assume by way of contradiction that $ \f{R}{2D} > L(p_R)$ for some $R \ge 1$. By definition of $p_R$ as an infimum, every $p\in(p_\cc,p_R)$ has $L(p) > R$.
Hence
\[
\f{L(p_R - \delta)}{L(p_R)} > \f{R}{\frac{R}{2D}} = 2D \quad \text{ for all sufficiently small $\delta > 0$},
\]
a contradiction to \eqref{eq:L_Rat} since $p_R > p_\cc$ by \eqref{k3r78b}.

The second claim \eqref{eq: scaling_relation} appears as \cite[Prop.~34]{nolin08}.
For the third claim \eqref{eq: comparable_p_n_equation}, we begin by making use of the previous parts:
\eeq{ \label{bc63}
\frac{p_R-p_\cc}{p_r-p_\cc}
\stackref{eq: scaling_relation}{\leq}
C\frac{L(p_r)^2\pi_4(1,L(p_r))}{L(p_R)^2\pi_4(1,L(p_R))}
\stackref{eq: L_p_n}{\leq}
C\frac{r^2\pi_4(1,cr)}{R^2\pi_4(1,R)}
\stackref{quasi_c}{\leq}
C\frac{r^2\pi_4(1,r)}{R^2\pi_4(1,R)}.
}
By analogous reasoning, we also have
\eeq{ \label{bc64}
\frac{p_R - p_\cc}{p_r-p_\cc} \geq c\frac{r^2\pi_4(1,r)}{R^2\pi_4(1,R)}.
}
%
Now we use quasi-multiplicativity \eqref{quasi_j} to bound the ratio of $\pi_4$ quantities:
\eq{
\frac{1}{\pi_4(r,R)} 
\le \frac{\pi_4(1,r)}{\pi_4(1,R)}
\le \frac{C}{\pi_4(r,R)}.
}
Insert these upper and lower bounds into \eqref{bc63} and \eqref{bc64} respectively, and then use the known bounds on 4-arm probabilities:
\eq{
&\frac{p_R-p_\cc}{p_r-p_\cc} 
\leq C\frac{r^2}{R^2\pi_4(r,R)}
\stackref{best_4_arm}{\le} C\Big(\frac{R}{r}\Big)^{-2/3}, \\
&\frac{p_R-p_\cc}{p_r-p_\cc} 
\ge c\frac{r^2}{R^2\pi_4(r,R)}
\stackref{best_4_arm}{\ge} c\Big(\frac{R}{r}\Big)^{-(1-c)}.
}
The two previous displays are exactly \eqref{eq: comparable_p_n_equation}.
Finally, \eqref{pn-pcbd} follows from \eqref{eq: comparable_p_n_equation} upon setting $r = 1$. 
%
%
%
 \end{proof}

For completeness, we next check that the distributions in Example \ref{gnwig} have the claimed properties.

\begin{proof}[Proof of Example \ref{gnwig}]
Part \ref{gnwig_a}: Assume $F(t) = p_\cc + \alpha t^\beta$ for $t \in [0,h]$, where $\alpha,\beta,h>0$. 
For $p > p_\cc$ sufficiently close to $p_\cc$, we have $F^{-1}(p) = \bigl((p - p_\cc)/{\alpha}\bigr)^{1/\beta}$. 
So for all $n$ sufficiently large, we have
\[
\frac{F^{-1}(p_{\b^{n+1}})}{F^{-1}(p_{\b^{n}})} 
= \Big(\frac{p_{\b^{n+1}}-p_\cc}{p_{\b^n}-p_\cc}\Big)^{1/\beta}  
\stackref{eq: comparable_p_n_equation}{\le} C\b^{-2/(3\beta)}.
\]
Hence \eqref{eq : limit_bigger_than_one} holds by choosing $\b$ sufficiently large that $C\b^{-2/(3\beta)}<1$.

\medskip \noindent Part \ref{gnwig_b}: Assume $F(0)=F(h) = p_\cc$ for some $h>0$.
Choose $h'>h$ such that $F(h') > p_\cc$.
Then for all $p > p_\cc$ sufficiently close to $p_\cc$, we have $h \leq F^{-1}(p) \leq h'$.
Hence \eqref{eq : limit_equal_zero} holds.

\medskip \noindent Part \ref{gnwig_c}: Assume $F(t) = p_\cc + \alpha\e^{-t^{-\beta}}$ for $t\in(0,h]$, where $\alpha,\beta,h>0$.
Then for $p > p_\cc$ sufficiently close to $p_\cc$, we have
\[
F^{-1}(p) = \Big(\f{1}{\log(\alpha/(p - p_\cc))}\Big)^{1/\beta}.
\]
So for all $n$ sufficiently large, we have
\eq{
\frac{F^{-1}(p_{\b^{2n}})}{F^{-1}(p_{\b^n})}
= \Big(\frac{\log \alpha - \log(p_{\b^n}-p_\cc)}{\log \alpha - \log(p_{\b^{2n}}-p_\cc)}\Big)^{1/\beta}.
}
Sending $n\to\infty$ results in
\eq{
\liminf_{n\to\infty}\frac{F^{-1}(p_{\b^{2n}})}{F^{-1}(p_{\b^n})}
= \Big(\liminf_{n\to\infty}\frac{\log(p_{\b^n}-p_\cc)}{\log(p_{\b^{2n}}-p_\cc)}\Big)^{1/\beta}.
}
So define $\nu_n$ implicitly by 
\eq{
\frac{\log(p_{\b^n}-p_\cc)}{\log(p_{\b^{2n}}-p_\cc)} = \nu_n.
}
In order to establish \eqref{eq : limit_equal_zero}, it suffices to show
\eq{
\liminf_{n\to\infty}\nu_n >0 .
}
To this end, observe that for some $c\in(0,1)$, we have
\eq{
\big(c(\b^{2n})^{-(1-c)})^{\nu_n} 
\stackref{pn-pcbd}{\leq} (p_{\b^{2n}}-p_\cc)^{\nu_n}
= p_{\b^n}-p_\cc
\stackref{pn-pcbd}{\leq} C(\b^n)^{-2/3}.
}
Isolating the first and last expressions, we see that
\eq{
\nu_n(\log_\b  c - 2n(1-c)) &\leq \log_\b (C) - \tfrac{2}{3}n \\
\implies \quad
\liminf_{n\to\infty}\nu_n &\geq \frac{1}{3(1-c)} > 0.
}
Hence \eqref{eq : limit_equal_zero} holds.

\medskip \noindent Part \ref{gnwig_d}: Assume $F$ is such that $F^{-1}(p_R) = \e^{-\sqrt{\log_2 L(p_R)}}$ for all large $R$.
We then have
\eq{
\limsup_{n\to\infty}\frac{F^{-1}(p_{\b^{n+1}})}{F^{-1}(p_{\b^n})}
\stackref{eq: L_p_n}{\ge} \frac{\e^{-\sqrt{(n+1)\log_2 \b}}}{\e^{-\sqrt{n\log_2 \b + \log_2 c}}} \to 1 \quad \text{as $n\to\infty$},
}
as well as
\eq{
\liminf_{n\to\infty}\frac{F^{-1}(p_{\b^{2n}})}{F^{-1}(p_{\b^n})}
\stackref{eq: L_p_n}{\le} \frac{\e^{-\sqrt{2n\log_2 \b + \log_2 c}}}{\e^{-\sqrt{n\log_2 \b}}} \to 0 \quad \text{as $n\to\infty$}.
}
Therefore, neither \eqref{eq : limit_bigger_than_one} nor \eqref{eq : limit_equal_zero} holds.
\end{proof}
\label{sec:corr_length_special}

\subsection{Bounds for passage times of geodesics across annuli} \label{sec_ptgaa}
We will study geodesics from $\vc 0$ to $\partial B_R$ on one scale at a time, examining the subpath between $\partial B_{\b^{k-1}}$ and $\partial B_{\b^k}$.
This task is complicated by the fact that the subpath may not be a geodesic between $\partial B_{\b^{k-1}}$ and $\partial B_{\b^k}$.
Nevertheless, the goal of this subsection is to demonstrate that the passage time restricted to this annulus is ``small'' when measured on a scale depending on $k$ but not on $R$.

We let
$E(B_{\b^k})$ denote the set of edges with both endpoints in $B_{\b^k}$.
To isolate edges in a particular annulus, we define
\eeq{ \label{def_ann}
\EE_1^\ann = E(B_{\b^1}) \qquad \text{and} \qquad \EE_{k}^\ann = E(B_{\b^k}) \setminus E(B_{\b^{k-1}}) \quad \text{for $k\ge2$.}
}
So $\EE_1^\ann,\EE_2^\ann,\dots,\EE_k^\ann$ from a partition of all possible edges used by a path $\gamma$ that does not exit $B_{\b^k}$.
The passage time accumulated in the $k$-th annulus will be denoted by
\eq{
T_k(\gamma) = \sum_{e\in\gamma\cap\EE_k^\ann} t_e.
}
We consider the random variable
\eeq{ \label{def_Tmax}
T_{k,R}  = \sup_{\gamma \in \Geo(\vc 0,\partial B_{R})} T_k(\gamma), \quad k\ge1,
}
where $\Geo(\vc 0,\partial B_{R})$ denotes the set of all geodesics between the origin and $\partial B_{R}$.
Note that $T_{k,R}=0$ if $k > \ceil{\log_\b  R}$ because every edge of $\gamma\in\Geo(\vc 0,\partial B_R)$ has both endpoints in $B_R\subseteq B_{\b^{\ceil{\log_\b R}}}$.

\begin{lemma} \label{lem:big_PT_bound}
Assume \eqref{critical_assumption}.
There exist constants $C,c>0$ such that for all $R\ge1$, $k\ge3$, and $p > p_\cc$ with $L(p)\leq \b^{k-3}$, we have
\eeq{ \label{big_PT_bound_eq}
 \P \Big( T_{k,R}   \geq \lambda F^{-1}(p)\Big( \dfrac{\b^k}{L(p)} \Big)^2 \Big) \leq  C\e^{-c \lambda} + C\exp \Big( -c \dfrac{\b^k}{L(p)}  \Big) \quad \text{for all $\lambda\ge0$.}
}
\end{lemma}


The estimate from Lemma~\ref{lem:big_PT_bound} will be obtained from a similar estimate for passage times across rectangles, quoted below.
Note that by definition, a ``crossing'' stays within the associated rectangle.

\begin{lemma} \textup{\cite[Cor.~1]{damron_tang19}} \label{lem:damron-tang-Tkmax}
Assume \eqref{critical_assumption}.
For any integer $K \ge 2$, there exists a constant $c > 0$ such that for any integer $n\ge1$ and any $p>p_\cc$ with $L(p) \le n$, we have
\eeq{ \label{damron-tang-Tkmax_eq}
\P\Big(\Tsf \ge \lambda F^{-1}(p) \Bigl(\f{n}{L(p)}\Bigr)^2 \Big) \le \e^{-c\lambda} + \exp\Bigl(-c \f{n}{L(p)}\Bigr) \quad \text{for all $\lambda \ge 0$},
}
where $\Tsf$ is the minimal passage time of a left-right crossing $[-Kn,Kn] \times [-n,n]$.
\end{lemma}

\begin{proof}[Proof of Lemma~\ref{lem:big_PT_bound}]
We may assume $k \le \ceil{\log_\b  R}$, since otherwise $T_{k,R} = 0$.
We will bound $T_{k,R}$ from above by the sum of nine first-passage values across different rectangles, each of which can be controlled using Lemma~\ref{lem:damron-tang-Tkmax}.
See Figure \ref{fig:T19} for an illustration.

\begin{figure}[t]
\begin{center}

\tikzset{every picture/.style={line width=0.75pt}} 

\subfloat[$k<\floor{\log_\b  R}$ \label{fig:T19_a}]{
\scalebox{0.85}{
\begin{tikzpicture}[x=0.75pt,y=0.75pt,yscale=-1,xscale=1]

\draw  [draw opacity=0][fill={rgb, 255:red, 184; green, 181; blue, 181 }  ,fill opacity=0.5 ] (302.5,172.5) -- (380,172.5) -- (380,192) -- (302.5,192) -- cycle ;
\draw  [draw opacity=0][fill={rgb, 255:red, 184; green, 181; blue, 181 }  ,fill opacity=0.5 ] (380,172.5) -- (360.5,172.5) -- (360.5,250) -- (380,250) -- cycle ;
\draw  [draw opacity=0][fill={rgb, 255:red, 184; green, 181; blue, 181 }  ,fill opacity=0.5 ] (322,172.5) -- (302.5,172.5) -- (302.5,250.5) -- (322,250.5) -- cycle ;
\draw  [draw opacity=0][fill={rgb, 255:red, 184; green, 181; blue, 181 }  ,fill opacity=0.5 ] (302.5,230.5) -- (380,230.5) -- (380,250) -- (302.5,250) -- cycle ;
\draw  [draw opacity=0][fill={rgb, 255:red, 184; green, 181; blue, 181 }  ,fill opacity=0.5 ] (186.25,56.25) -- (496.5,56.25) -- (496.5,134) -- (186.25,134) -- cycle ;
\draw  [draw opacity=0][fill={rgb, 255:red, 184; green, 181; blue, 181 }  ,fill opacity=0.5 ] (186,288.5) -- (496.25,288.5) -- (496.25,366.25) -- (186,366.25) -- cycle ;
\draw  [draw opacity=0][fill={rgb, 255:red, 184; green, 181; blue, 181 }  ,fill opacity=0.5 ] (263.5,56.25) -- (186,56.25) -- (186,366.25) -- (263.5,366.25) -- cycle ;
\draw  [draw opacity=0][fill={rgb, 255:red, 184; green, 181; blue, 181 }  ,fill opacity=0.5 ] (496.25,56.25) -- (418.75,56.25) -- (418.75,366.25) -- (496.25,366.25) -- cycle ;
\draw    (187,72) .. controls (227,42) and (457.5,103) .. (497.5,73) ;
\draw    (186,333) .. controls (226,303) and (456.5,364) .. (496.5,334) ;
\draw    (302.5,180) .. controls (337,164) and (341.5,194) .. (380.5,178) ;
\draw    (302.5,244) .. controls (337,228) and (341.5,258) .. (380.5,242) ;
\draw    (208.5,57) .. controls (178.5,152) and (260.5,259) .. (233.5,365) ;
\draw    (443.5,56) .. controls (495.5,138) and (444.5,258) .. (479.5,366) ;
\draw  [color={rgb, 255:red, 0; green, 0; blue, 0 }  ,draw opacity=1 ][dash pattern={on 4.5pt off 4.5pt}] (341.5,192) -- (496.5,192) -- (496.5,231) -- (341.5,231) -- cycle ;
\draw    (318,172.5) .. controls (311,224.5) and (321.5,205) .. (316.5,250) ;
\draw    (370.5,172) .. controls (363.5,224) and (379.5,205) .. (374.5,250) ;
\draw  [line width=2.25]  (154.25,24.25) -- (528.25,24.25) -- (528.25,398.25) -- (154.25,398.25) -- cycle ;
\draw    (342,204) .. controls (382,174) and (456.5,242) .. (496.5,212) ;
\draw   [color={rgb, 255:red, 29; green, 92; blue, 216 }  ,draw opacity=1 ] (341.25,211.25) .. controls (338.82,211.52) and (337.46,210.5) .. (337.17,208.17) .. controls (336.94,205.87) and (335.65,204.85) .. (333.3,205.1) .. controls (330.99,205.41) and (329.66,204.4) .. (329.32,202.06) .. controls (328.82,199.69) and (327.41,198.77) .. (325.09,199.31) .. controls (322.95,200.1) and (321.43,199.46) .. (320.52,197.37) .. controls (318.99,195.42) and (317.36,195.29) .. (315.63,196.98) .. controls (314.99,198.95) and (313.62,199.85) .. (311.53,199.66) .. controls (309.2,200.47) and (308.5,202.01) .. (309.44,204.26) .. controls (310.61,206.22) and (310.2,207.83) .. (308.19,209.08) .. controls (306.21,210.39) and (305.88,211.98) .. (307.21,213.86) .. controls (308.5,215.91) and (308.17,217.58) .. (306.2,218.85) .. controls (304.22,220.05) and (303.82,221.7) .. (304.99,223.79) .. controls (306.11,225.84) and (305.59,227.39) .. (303.42,228.45) .. controls (301.27,229.17) and (300.47,230.6) .. (301.04,232.74) .. controls (299.94,235.22) and (298.43,235.52) .. (296.51,233.63) .. controls (295.88,231.39) and (294.42,230.53) .. (292.13,231.05) .. controls (290,231.76) and (288.49,231.08) .. (287.62,229) .. controls (285.83,227.24) and (284.2,227.41) .. (282.73,229.52) .. controls (282.35,231.71) and (281.02,232.72) .. (278.74,232.56) .. controls (276.31,232.88) and (275.45,234.3) .. (276.15,236.82) .. controls (277.81,238.13) and (278.11,239.73) .. (277.04,241.6) .. controls (276.61,243.99) and (277.59,245.28) .. (279.96,245.46) .. controls (282.31,245.37) and (283.59,246.47) .. (283.78,248.74) .. controls (284.27,251.11) and (285.74,252.06) .. (288.17,251.6) .. controls (290.3,250.88) and (291.69,251.62) .. (292.33,253.81) .. controls (293.45,256.16) and (295.11,256.88) .. (297.31,255.99) .. controls (299.45,255.01) and (300.88,255.52) .. (301.61,257.53) .. controls (302.78,259.64) and (304.41,260.11) .. (306.52,258.95) .. controls (308.52,257.7) and (310.17,258.06) .. (311.47,260.01) .. controls (312.9,261.91) and (314.53,262.12) .. (316.38,260.65) .. controls (318.04,259.09) and (319.77,259.15) .. (321.58,260.83) .. controls (323.23,262.42) and (324.86,262.27) .. (326.46,260.39) .. controls (327.69,258.43) and (329.29,258.01) .. (331.26,259.13) .. controls (333.47,259.96) and (334.93,259.17) .. (335.64,256.74) .. controls (335.69,254.45) and (336.8,253.26) .. (338.95,253.19) .. controls (341.34,252.2) and (342.03,250.62) .. (341.03,248.47) .. controls (339.96,246.46) and (340.5,244.94) .. (342.63,243.91) .. controls (344.79,242.8) and (345.36,241.2) .. (344.35,239.11) .. controls (343.36,236.97) and (343.94,235.42) .. (346.09,234.45) .. controls (348.22,233.56) and (348.83,232.05) .. (347.9,229.92) .. controls (347.06,227.68) and (347.75,226.19) .. (349.97,225.44) .. controls (352.2,224.9) and (353.12,223.49) .. (352.72,221.21) .. controls (352.45,218.88) and (353.51,217.58) .. (355.9,217.31) .. controls (358.24,217.2) and (359.35,216) .. (359.24,213.72) .. controls (359.36,211.31) and (360.61,210.19) .. (362.99,210.34) .. controls (365.32,210.65) and (366.66,209.67) .. (367.01,207.42) .. controls (367.74,205.09) and (369.26,204.31) .. (371.57,205.09) .. controls (373.58,206.18) and (375.18,205.78) .. (376.39,203.87) .. controls (378.13,202.16) and (379.77,202.26) .. (381.32,204.16) .. controls (382.51,206.17) and (384.1,206.63) .. (386.1,205.52) .. controls (388.29,204.57) and (389.85,205.2) .. (390.79,207.43) .. controls (391.54,209.63) and (393.03,210.37) .. (395.28,209.64) .. controls (397.5,208.93) and (398.95,209.71) .. (399.62,211.97) .. controls (400.31,214.26) and (401.78,215.08) .. (404.01,214.42) .. controls (406.28,213.77) and (407.76,214.59) .. (408.46,216.88) .. controls (409.09,219.11) and (410.52,219.88) .. (412.77,219.19) .. controls (415.03,218.48) and (416.56,219.26) .. (417.36,221.52) .. controls (418.06,223.73) and (419.54,224.48) .. (421.8,223.78) .. controls (424,223.09) and (425.39,223.87) .. (425.98,226.14) .. controls (426.69,228.53) and (428.1,229.41) .. (430.22,228.8) .. controls (432.58,228.39) and (433.97,229.37) .. (434.39,231.73) .. controls (434.68,234.06) and (435.99,235.11) .. (438.32,234.88) .. controls (440.69,234.76) and (441.87,235.87) .. (441.88,238.21) .. controls (441.81,240.6) and (442.93,241.91) .. (445.26,242.14) .. controls (447.59,242.61) and (448.46,244.04) .. (447.87,246.43) .. controls (446.88,248.46) and (447.35,249.99) .. (449.3,251.03) .. controls (451.03,252.82) and (450.92,254.52) .. (448.97,256.11) .. controls (446.88,256.88) and (446.15,258.32) .. (446.79,260.43) .. controls (446.88,262.9) and (445.73,264.18) .. (443.34,264.27) .. controls (441.1,264.13) and (439.93,265.23) .. (439.83,267.56) .. controls (439.47,269.98) and (438.13,270.96) .. (435.81,270.51) .. controls (433.7,269.72) and (432.17,270.35) .. (431.22,272.38) .. controls (429.59,274.23) and (427.91,274.27) .. (426.18,272.52) .. controls (424.76,270.67) and (423.09,270.47) .. (421.18,271.93) .. controls (419.55,273.51) and (417.89,273.42) .. (416.2,271.67) .. controls (414.78,269.85) and (413.14,269.67) .. (411.28,271.13) .. controls (409.43,272.6) and (407.71,272.41) .. (406.12,270.56) .. controls (404.59,268.75) and (403,268.61) .. (401.35,270.16) .. controls (399.51,271.73) and (397.83,271.67) .. (396.3,269.96) .. controls (394.43,268.31) and (392.7,268.35) .. (391.13,270.08) .. controls (389.68,271.87) and (388.07,272.03) .. (386.3,270.58) .. controls (384.38,269.23) and (382.71,269.59) .. (381.28,271.65) .. controls (380.34,273.7) and (378.76,274.28) .. (376.54,273.39) .. controls (374.33,272.66) and (372.89,273.51) .. (372.24,275.92) .. controls (372.15,278.14) and (370.98,279.2) .. (368.75,279.09) .. controls (366.22,279.62) and (365.26,280.97) .. (365.87,283.15) .. controls (366.63,285.4) and (365.96,286.95) .. (363.86,287.8) .. controls (361.79,288.97) and (361.33,290.57) .. (362.5,292.62) .. controls (363.69,294.62) and (363.24,296.29) .. (361.15,297.64) .. controls (359.14,298.69) and (358.72,300.27) .. (359.89,302.38) .. controls (361.06,304.44) and (360.63,306.04) .. (358.6,307.18) .. controls (356.58,308.25) and (356.13,309.83) .. (357.26,311.94) .. controls (358.39,314.01) and (357.89,315.63) .. (355.76,316.82) .. controls (353.65,317.82) and (353.09,319.33) .. (354.1,321.36) .. controls (354.89,323.71) and (354.16,325.25) .. (351.91,325.98) .. controls (349.68,326.37) and (348.71,327.68) .. (349.02,329.89) .. controls (348.57,332.38) and (347.14,333.28) .. (344.75,332.57) .. controls (343.04,331.2) and (341.43,331.33) .. (339.9,332.95) .. controls (338.1,334.46) and (336.42,334.31) .. (334.85,332.48) .. controls (333.32,330.65) and (331.66,330.47) .. (329.85,331.96) .. controls (327.95,333.43) and (326.31,333.25) .. (324.92,331.42) .. controls (323.48,329.58) and (321.83,329.4) .. (319.97,330.87) .. controls (318.08,332.33) and (316.36,332.14) .. (314.81,330.29) .. controls (313.49,328.47) and (311.85,328.29) .. (309.9,329.75) .. controls (308.2,331.24) and (306.56,331.07) .. (304.97,329.22) .. controls (303.38,327.37) and (301.75,327.2) .. (300.08,328.71) .. controls (298.2,330.2) and (296.5,330.03) .. (294.97,328.21) .. controls (293.48,326.4) and (291.86,326.25) .. (290.11,327.77) .. controls (288.23,329.29) and (286.52,329.16) .. (284.99,327.37) .. controls (283.53,325.6) and (281.94,325.51) .. (280.22,327.09) .. controls (278.43,328.7) and (276.7,328.68) .. (275.05,327.02) .. controls (273.4,325.41) and (271.75,325.5) .. (270.1,327.29) .. controls (268.67,329.12) and (267.02,329.34) .. (265.13,327.93) .. controls (263.25,326.58) and (261.61,326.92) .. (260.2,328.94) .. controls (259.04,330.96) and (257.42,331.41) .. (255.34,330.3) .. controls (253.3,329.23) and (251.72,329.79) .. (250.61,331.99) .. controls (249.75,334.15) and (248.22,334.82) .. (246.02,333.99) .. controls (243.88,333.19) and (242.41,333.95) .. (241.61,336.26) .. controls (241.06,338.5) and (239.66,339.33) .. (237.4,338.76) .. controls (235.05,338.3) and (233.63,339.26) .. (233.14,341.63) .. controls (232.89,343.88) and (231.55,344.9) .. (229.12,344.68) .. controls (226.8,344.43) and (225.58,345.49) .. (225.45,347.86) .. controls (225.61,350.13) and (224.5,351.42) .. (222.12,351.73) .. controls (219.82,352.14) and (218.97,353.46) .. (219.57,355.71) .. controls (220.28,357.98) and (219.53,359.57) .. (217.32,360.47) .. controls (215.21,361.38) and (214.68,362.94) .. (215.73,365.13) .. controls (216.92,367.18) and (216.55,368.75) .. (214.6,369.85) .. controls (212.67,371.26) and (212.41,372.91) .. (213.84,374.8) .. controls (215.33,376.71) and (215.19,378.39) .. (213.43,379.83) .. controls (211.71,381.48) and (211.68,383.12) .. (213.35,384.75) -- (213.39,386.35) -- (214.05,394.32) ;
\draw [shift={(214.5,397)}, rotate = 257.91] [fill={rgb, 255:red, 29; green, 92; blue, 216 }  ,fill opacity=1 ][line width=0.08]  [draw opacity=0] (8.93,-4.29) -- (0,0) -- (8.93,4.29) -- cycle    ;
\draw  [draw opacity=0][fill={rgb, 255:red, 0; green, 0; blue, 0 }  ,fill opacity=1 ] (310,197) .. controls (310,193.96) and (312.46,191.5) .. (315.5,191.5) .. controls (318.54,191.5) and (321,193.96) .. (321,197) .. controls (321,200.04) and (318.54,202.5) .. (315.5,202.5) .. controls (312.46,202.5) and (310,200.04) .. (310,197) -- cycle ;
\draw  [draw opacity=0][fill={rgb, 255:red, 0; green, 0; blue, 0 }  ,fill opacity=1 ] (297,230.5) .. controls (297,227.46) and (299.46,225) .. (302.5,225) .. controls (305.54,225) and (308,227.46) .. (308,230.5) .. controls (308,233.54) and (305.54,236) .. (302.5,236) .. controls (299.46,236) and (297,233.54) .. (297,230.5) -- cycle ;
\draw  [draw opacity=0][fill={rgb, 255:red, 0; green, 0; blue, 0 }  ,fill opacity=1 ] (358,289) .. controls (358,285.96) and (360.46,283.5) .. (363.5,283.5) .. controls (366.54,283.5) and (369,285.96) .. (369,289) .. controls (369,292.04) and (366.54,294.5) .. (363.5,294.5) .. controls (360.46,294.5) and (358,292.04) .. (358,289) -- cycle ;
\draw  [draw opacity=0][fill={rgb, 255:red, 0; green, 0; blue, 0 }  ,fill opacity=1 ] (271,327) .. controls (271,323.96) and (273.46,321.5) .. (276.5,321.5) .. controls (279.54,321.5) and (282,323.96) .. (282,327) .. controls (282,330.04) and (279.54,332.5) .. (276.5,332.5) .. controls (273.46,332.5) and (271,330.04) .. (271,327) -- cycle ;
\draw  [color={rgb, 255:red, 0; green, 0; blue, 0 }  ,draw opacity=1 ] (339.13,211.25) .. controls (339.13,210.08) and (340.08,209.13) .. (341.25,209.13) .. controls (342.42,209.13) and (343.38,210.08) .. (343.38,211.25) .. controls (343.38,212.42) and (342.42,213.38) .. (341.25,213.38) .. controls (340.08,213.38) and (339.13,212.42) .. (339.13,211.25) -- cycle ;
\draw [color={rgb, 255:red, 248; green, 231; blue, 28 }  ,draw opacity=0.5 ][line width=4.5] [line join = round][line cap = round]   (299.5,233) .. controls (291.23,233) and (290.29,227.78) .. (280.5,229) .. controls (274.45,229.76) and (279.1,245.07) .. (283.5,248) .. controls (298.18,257.79) and (306.34,260.31) .. (332.5,259) .. controls (336.59,258.8) and (337.58,252.92) .. (340.5,250) ;
\draw [color={rgb, 255:red, 248; green, 231; blue, 28 }  ,draw opacity=0.5 ][line width=3.75] [line join = round][line cap = round]   (381.5,204) .. controls (389.38,204) and (414.86,222) .. (417.5,222) ;
\draw [color={rgb, 255:red, 248; green, 231; blue, 28 }  ,draw opacity=0.5 ][line width=3.75] [line join = round][line cap = round]   (366.5,284) .. controls (366.5,282.33) and (365.88,280.55) .. (366.5,279) .. controls (366.91,277.97) and (378.5,270) .. (379.5,273) .. controls (383.5,269) and (396.5,269) .. (406.5,270) .. controls (408.36,270.13) and (415.45,273) .. (417.5,273) ;

\draw (336,161.4) node [anchor=north west][inner sep=0.75pt]    {$\mathsf{T}_{1}$};
\draw (374,209.4) node [anchor=north west][inner sep=0.75pt]    {$\mathsf{T}_{2}$};
\draw (349,248.4) node [anchor=north west][inner sep=0.75pt]    {$\mathsf{T}_{3}$};
\draw (318.5,211.9) node [anchor=north west][inner sep=0.75pt]    {$\mathsf{T}_{4}$};
\draw (430,195.4) node [anchor=north west][inner sep=0.75pt]    {$\mathsf{T}_{9}$};
\draw (348,78.4) node [anchor=north west][inner sep=0.75pt]    {$\mathsf{T}_{5}$};
\draw (471,141.4) node [anchor=north west][inner sep=0.75pt]    {$\mathsf{T}_{6}$};
\draw (392,341.4) node [anchor=north west][inner sep=0.75pt]    {$\mathsf{T}_{7}$};
\draw (215,167.4) node [anchor=north west][inner sep=0.75pt]    {$\mathsf{T}_{8}$};
\draw (294,185) node [anchor=north west][inner sep=0.75pt]    {$x_{1}$};
\draw (286,213) node [anchor=north west][inner sep=0.75pt]    {$x_{2}$};
\draw (371,288) node [anchor=north west][inner sep=0.75pt]    {$x_{3}$};
\draw (273,307) node [anchor=north west][inner sep=0.75pt]    {$x_{4}$};
\draw (197,371.4) node [anchor=north west][inner sep=0.75pt]    {$\gamma $};

\end{tikzpicture}
}
}\\[0.03in]

\subfloat[$k=\floor{\log_\b  R}$ \label{fig:T19_b}]{
\scalebox{0.85}{
\begin{tikzpicture}[x=0.75pt,y=0.75pt,yscale=-1,xscale=1]

\draw  [draw opacity=0][fill={rgb, 255:red, 184; green, 181; blue, 181 }  ,fill opacity=0.5 ] (302.5,172.5) -- (380,172.5) -- (380,192) -- (302.5,192) -- cycle ;
\draw  [draw opacity=0][fill={rgb, 255:red, 184; green, 181; blue, 181 }  ,fill opacity=0.5 ] (380,172.5) -- (360.5,172.5) -- (360.5,250) -- (380,250) -- cycle ;
\draw  [draw opacity=0][fill={rgb, 255:red, 184; green, 181; blue, 181 }  ,fill opacity=0.5 ] (322,172.5) -- (302.5,172.5) -- (302.5,250.5) -- (322,250.5) -- cycle ;
\draw  [draw opacity=0][fill={rgb, 255:red, 184; green, 181; blue, 181 }  ,fill opacity=0.5 ] (302.5,230.5) -- (380,230.5) -- (380,250) -- (302.5,250) -- cycle ;
\draw  [draw opacity=0][fill={rgb, 255:red, 184; green, 181; blue, 181 }  ,fill opacity=0.5 ] (186.25,56.25) -- (496.5,56.25) -- (496.5,134) -- (186.25,134) -- cycle ;
\draw  [draw opacity=0][fill={rgb, 255:red, 184; green, 181; blue, 181 }  ,fill opacity=0.5 ] (186,288.5) -- (496.25,288.5) -- (496.25,366.25) -- (186,366.25) -- cycle ;
\draw  [draw opacity=0][fill={rgb, 255:red, 184; green, 181; blue, 181 }  ,fill opacity=0.5 ] (263.5,56.25) -- (186,56.25) -- (186,366.25) -- (263.5,366.25) -- cycle ;
\draw  [draw opacity=0][fill={rgb, 255:red, 184; green, 181; blue, 181 }  ,fill opacity=0.5 ] (496.25,56.25) -- (418.75,56.25) -- (418.75,366.25) -- (496.25,366.25) -- cycle ;
\draw    (302.5,180) .. controls (337,164) and (341.5,194) .. (380.5,178) ;
\draw    (302.5,244) .. controls (337,228) and (341.5,258) .. (380.5,242) ;
\draw  [color={rgb, 255:red, 0; green, 0; blue, 0 }  ,draw opacity=1 ][dash pattern={on 4.5pt off 4.5pt}] (341.5,192) -- (496.5,192) -- (496.5,231) -- (341.5,231) -- cycle ;
\draw    (318,172.5) .. controls (311,224.5) and (321.5,205) .. (316.5,250) ;
\draw    (370.5,172) .. controls (363.5,224) and (379.5,205) .. (374.5,250) ;
\draw  [line width=2.25]  (225.06,95.06) -- (457.44,95.06) -- (457.44,327.44) -- (225.06,327.44) -- cycle ;
\draw    (342,204) .. controls (382,174) and (456.5,242) .. (496.5,212) ;
\draw    [color={rgb, 255:red, 29; green, 92; blue, 216 }  ,draw opacity=1 ] (341.25,211.25) .. controls (338.82,211.52) and (337.46,210.5) .. (337.17,208.17) .. controls (336.94,205.87) and (335.65,204.85) .. (333.3,205.1) .. controls (330.99,205.41) and (329.66,204.4) .. (329.32,202.06) .. controls (328.82,199.69) and (327.41,198.77) .. (325.09,199.31) .. controls (322.95,200.1) and (321.43,199.46) .. (320.52,197.37) .. controls (318.99,195.42) and (317.36,195.29) .. (315.63,196.98) .. controls (314.99,198.95) and (313.62,199.85) .. (311.53,199.66) .. controls (309.2,200.47) and (308.5,202.01) .. (309.44,204.26) .. controls (310.61,206.22) and (310.2,207.83) .. (308.19,209.08) .. controls (306.21,210.39) and (305.88,211.98) .. (307.21,213.86) .. controls (308.5,215.91) and (308.17,217.58) .. (306.2,218.85) .. controls (304.22,220.05) and (303.82,221.7) .. (304.99,223.79) .. controls (306.11,225.84) and (305.59,227.39) .. (303.42,228.45) .. controls (301.27,229.17) and (300.47,230.6) .. (301.04,232.74) .. controls (299.94,235.22) and (298.43,235.52) .. (296.51,233.63) .. controls (295.88,231.39) and (294.42,230.53) .. (292.13,231.05) .. controls (290,231.76) and (288.49,231.08) .. (287.62,229) .. controls (285.83,227.24) and (284.2,227.41) .. (282.73,229.52) .. controls (282.35,231.71) and (281.02,232.72) .. (278.74,232.56) .. controls (276.31,232.88) and (275.45,234.3) .. (276.15,236.82) .. controls (277.81,238.13) and (278.11,239.73) .. (277.04,241.6) .. controls (276.61,243.99) and (277.59,245.28) .. (279.96,245.46) .. controls (282.31,245.37) and (283.59,246.47) .. (283.78,248.74) .. controls (284.27,251.11) and (285.74,252.06) .. (288.17,251.6) .. controls (290.3,250.88) and (291.69,251.62) .. (292.33,253.81) .. controls (293.45,256.16) and (295.11,256.88) .. (297.31,255.99) .. controls (299.45,255.01) and (300.88,255.52) .. (301.61,257.53) .. controls (302.78,259.64) and (304.41,260.11) .. (306.52,258.95) .. controls (308.52,257.7) and (310.17,258.06) .. (311.47,260.01) .. controls (312.9,261.91) and (314.53,262.12) .. (316.38,260.65) .. controls (318.04,259.09) and (319.77,259.15) .. (321.58,260.83) .. controls (323.23,262.42) and (324.86,262.27) .. (326.46,260.39) .. controls (327.69,258.43) and (329.29,258.01) .. (331.26,259.13) .. controls (333.47,259.96) and (334.93,259.17) .. (335.64,256.74) .. controls (335.69,254.45) and (336.8,253.26) .. (338.95,253.19) .. controls (341.34,252.2) and (342.03,250.62) .. (341.03,248.47) .. controls (339.96,246.46) and (340.5,244.94) .. (342.63,243.91) .. controls (344.79,242.8) and (345.36,241.2) .. (344.35,239.11) .. controls (343.36,236.97) and (343.94,235.42) .. (346.09,234.45) .. controls (348.22,233.56) and (348.83,232.05) .. (347.9,229.92) .. controls (347.06,227.68) and (347.75,226.19) .. (349.97,225.44) .. controls (352.2,224.9) and (353.12,223.49) .. (352.72,221.21) .. controls (352.45,218.88) and (353.51,217.58) .. (355.9,217.31) .. controls (358.24,217.2) and (359.35,216) .. (359.24,213.72) .. controls (359.36,211.31) and (360.61,210.19) .. (362.99,210.34) .. controls (365.32,210.65) and (366.66,209.67) .. (367.01,207.42) .. controls (367.74,205.09) and (369.26,204.31) .. (371.57,205.09) .. controls (373.58,206.18) and (375.18,205.78) .. (376.39,203.87) .. controls (378.13,202.16) and (379.77,202.26) .. (381.32,204.16) .. controls (382.51,206.17) and (384.1,206.63) .. (386.1,205.52) .. controls (388.29,204.57) and (389.85,205.2) .. (390.79,207.43) .. controls (391.54,209.63) and (393.03,210.37) .. (395.28,209.64) .. controls (397.5,208.93) and (398.95,209.71) .. (399.62,211.97) .. controls (400.31,214.26) and (401.78,215.08) .. (404.01,214.42) .. controls (406.28,213.77) and (407.76,214.59) .. (408.46,216.88) .. controls (409.09,219.11) and (410.52,219.88) .. (412.77,219.19) .. controls (415.03,218.48) and (416.56,219.26) .. (417.36,221.52) .. controls (418.06,223.73) and (419.54,224.48) .. (421.8,223.78) .. controls (424,223.09) and (425.39,223.87) .. (425.98,226.14) .. controls (426.69,228.53) and (428.1,229.41) .. (430.22,228.8) .. controls (432.58,228.39) and (433.97,229.37) .. (434.39,231.73) .. controls (434.68,234.06) and (435.99,235.11) .. (438.32,234.88) .. controls (440.69,234.76) and (441.87,235.87) .. (441.88,238.21) .. controls (441.81,240.6) and (442.93,241.91) .. (445.26,242.14) .. controls (447.59,242.61) and (448.46,244.04) .. (447.87,246.43) .. controls (446.88,248.46) and (447.35,249.99) .. (449.3,251.03) .. controls (451.03,252.82) and (450.92,254.52) .. (448.97,256.11) .. controls (446.88,256.88) and (446.15,258.32) .. (446.79,260.43) .. controls (446.88,262.9) and (445.73,264.18) .. (443.34,264.27) .. controls (441.1,264.13) and (439.93,265.23) .. (439.83,267.56) .. controls (439.47,269.98) and (438.13,270.96) .. (435.81,270.51) .. controls (433.7,269.72) and (432.17,270.35) .. (431.22,272.38) .. controls (429.59,274.23) and (427.91,274.27) .. (426.18,272.52) .. controls (424.76,270.67) and (423.09,270.47) .. (421.18,271.93) .. controls (419.55,273.51) and (417.89,273.42) .. (416.2,271.67) .. controls (414.78,269.85) and (413.14,269.67) .. (411.28,271.13) .. controls (409.43,272.6) and (407.71,272.41) .. (406.12,270.56) .. controls (404.59,268.75) and (403,268.61) .. (401.35,270.16) .. controls (399.51,271.73) and (397.83,271.67) .. (396.3,269.96) .. controls (394.43,268.31) and (392.7,268.35) .. (391.13,270.08) .. controls (389.68,271.87) and (388.07,272.03) .. (386.3,270.58) .. controls (384.38,269.23) and (382.71,269.59) .. (381.28,271.65) .. controls (380.34,273.7) and (378.76,274.28) .. (376.54,273.39) .. controls (374.33,272.66) and (372.89,273.51) .. (372.24,275.92) .. controls (372.15,278.14) and (370.98,279.2) .. (368.75,279.09) .. controls (366.22,279.62) and (365.26,280.97) .. (365.87,283.15) .. controls (366.63,285.4) and (365.96,286.95) .. (363.86,287.8) .. controls (361.74,289.13) and (361.29,290.72) .. (362.52,292.55) .. controls (363.67,294.68) and (363.23,296.29) .. (361.22,297.39) .. controls (359.17,298.7) and (358.75,300.37) .. (359.97,302.4) .. controls (361.27,304.26) and (360.95,305.85) .. (359.01,307.17) -- (358.66,312.14) -- (359.44,320.02) ;
\draw [shift={(359.5,323)}, rotate = 270] [fill={rgb, 255:red, 29; green, 92; blue, 216 }  ][line width=0.08]  [draw opacity=0] (8.93,-4.29) -- (0,0) -- (8.93,4.29) -- cycle    ;
\draw  [draw opacity=0][fill={rgb, 255:red, 0; green, 0; blue, 0 }  ,fill opacity=1 ] (310,197) .. controls (310,193.96) and (312.46,191.5) .. (315.5,191.5) .. controls (318.54,191.5) and (321,193.96) .. (321,197) .. controls (321,200.04) and (318.54,202.5) .. (315.5,202.5) .. controls (312.46,202.5) and (310,200.04) .. (310,197) -- cycle ;
\draw  [draw opacity=0][fill={rgb, 255:red, 0; green, 0; blue, 0 }  ,fill opacity=1 ] (297,230.5) .. controls (297,227.46) and (299.46,225) .. (302.5,225) .. controls (305.54,225) and (308,227.46) .. (308,230.5) .. controls (308,233.54) and (305.54,236) .. (302.5,236) .. controls (299.46,236) and (297,233.54) .. (297,230.5) -- cycle ;
\draw  [draw opacity=0][fill={rgb, 255:red, 0; green, 0; blue, 0 }  ,fill opacity=1 ] (358,289) .. controls (358,285.96) and (360.46,283.5) .. (363.5,283.5) .. controls (366.54,283.5) and (369,285.96) .. (369,289) .. controls (369,292.04) and (366.54,294.5) .. (363.5,294.5) .. controls (360.46,294.5) and (358,292.04) .. (358,289) -- cycle ;
\draw  [draw opacity=0][fill={rgb, 255:red, 0; green, 0; blue, 0 }  ,fill opacity=1 ] (354,328.5) .. controls (354,325.46) and (356.46,323) .. (359.5,323) .. controls (362.54,323) and (365,325.46) .. (365,328.5) .. controls (365,331.54) and (362.54,334) .. (359.5,334) .. controls (356.46,334) and (354,331.54) .. (354,328.5) -- cycle ;
\draw  [color={rgb, 255:red, 0; green, 0; blue, 0 }  ,draw opacity=1 ] (339.13,211.25) .. controls (339.13,210.08) and (340.08,209.13) .. (341.25,209.13) .. controls (342.42,209.13) and (343.38,210.08) .. (343.38,211.25) .. controls (343.38,212.42) and (342.42,213.38) .. (341.25,213.38) .. controls (340.08,213.38) and (339.13,212.42) .. (339.13,211.25) -- cycle ;
\draw [color={rgb, 255:red, 248; green, 231; blue, 28 }  ,draw opacity=0.5 ][line width=4.5] [line join = round][line cap = round]   (299.5,233) .. controls (291.23,233) and (290.29,227.78) .. (280.5,229) .. controls (274.45,229.76) and (279.1,245.07) .. (283.5,248) .. controls (298.18,257.79) and (306.34,260.31) .. (332.5,259) .. controls (336.59,258.8) and (337.58,252.92) .. (340.5,250) ;
\draw [color={rgb, 255:red, 248; green, 231; blue, 28 }  ,draw opacity=0.5 ][line width=3.75] [line join = round][line cap = round]   (381.5,204) .. controls (389.38,204) and (414.86,222) .. (417.5,222) ;
\draw [color={rgb, 255:red, 248; green, 231; blue, 28 }  ,draw opacity=0.5 ][line width=3.75] [line join = round][line cap = round]   (366.5,284) .. controls (366.5,282.33) and (365.88,280.55) .. (366.5,279) .. controls (366.91,277.97) and (378.5,270) .. (379.5,273) .. controls (383.5,269) and (396.5,269) .. (406.5,270) .. controls (408.36,270.13) and (415.45,273) .. (417.5,273) ;

\draw (336,161.4) node [anchor=north west][inner sep=0.75pt]    {$\mathsf{T}_{1}$};
\draw (374,209.4) node [anchor=north west][inner sep=0.75pt]    {$\mathsf{T}_{2}$};
\draw (349,248.4) node [anchor=north west][inner sep=0.75pt]    {$\mathsf{T}_{3}$};
\draw (318.5,211.9) node [anchor=north west][inner sep=0.75pt]    {$\mathsf{T}_{4}$};
\draw (468,200) node [anchor=north west][inner sep=0.75pt]    {$\mathsf{T}_{9}$};
\draw (293,185) node [anchor=north west][inner sep=0.75pt]    {$x_{1}$};
\draw (285,213) node [anchor=north west][inner sep=0.75pt]    {$x_{2}$};
\draw (372,288) node [anchor=north west][inner sep=0.75pt]    {$x_{3}$};
\draw (352,338) node [anchor=north west][inner sep=0.75pt]    {$x_{4}$};
\draw (343,299) node [anchor=north west][inner sep=0.75pt]    {$\gamma $};

\end{tikzpicture}
}
}
\hfill
\subfloat[$k=\ceil{\log_\b  R}$ \label{fig:T19_c}]{
\scalebox{0.85}{
\begin{tikzpicture}[x=0.75pt,y=0.75pt,yscale=-1,xscale=1]

\draw  [draw opacity=0][fill={rgb, 255:red, 184; green, 181; blue, 181 }  ,fill opacity=0.5 ] (302.5,172.5) -- (380,172.5) -- (380,192) -- (302.5,192) -- cycle ;
\draw  [draw opacity=0][fill={rgb, 255:red, 184; green, 181; blue, 181 }  ,fill opacity=0.5 ] (380,172.5) -- (360.5,172.5) -- (360.5,250) -- (380,250) -- cycle ;
\draw  [draw opacity=0][fill={rgb, 255:red, 184; green, 181; blue, 181 }  ,fill opacity=0.5 ] (322,172.5) -- (302.5,172.5) -- (302.5,250.5) -- (322,250.5) -- cycle ;
\draw  [draw opacity=0][fill={rgb, 255:red, 184; green, 181; blue, 181 }  ,fill opacity=0.5 ] (302.5,230.5) -- (380,230.5) -- (380,250) -- (302.5,250) -- cycle ;
\draw  [draw opacity=0][fill={rgb, 255:red, 184; green, 181; blue, 181 }  ,fill opacity=0.5 ] (186.25,56.25) -- (496.5,56.25) -- (496.5,134) -- (186.25,134) -- cycle ;
\draw  [draw opacity=0][fill={rgb, 255:red, 184; green, 181; blue, 181 }  ,fill opacity=0.5 ] (186,288.5) -- (496.25,288.5) -- (496.25,366.25) -- (186,366.25) -- cycle ;
\draw  [draw opacity=0][fill={rgb, 255:red, 184; green, 181; blue, 181 }  ,fill opacity=0.5 ] (263.5,56.25) -- (186,56.25) -- (186,366.25) -- (263.5,366.25) -- cycle ;
\draw  [draw opacity=0][fill={rgb, 255:red, 184; green, 181; blue, 181 }  ,fill opacity=0.5 ] (496.25,56.25) -- (418.75,56.25) -- (418.75,366.25) -- (496.25,366.25) -- cycle ;
\draw    (302.5,180) .. controls (337,164) and (341.5,194) .. (380.5,178) ;
\draw    (302.5,244) .. controls (337,228) and (341.5,258) .. (380.5,242) ;
\draw  [color={rgb, 255:red, 0; green, 0; blue, 0 }  ,draw opacity=1 ][dash pattern={on 4.5pt off 4.5pt}] (341.5,192) -- (496.5,192) -- (496.5,231) -- (341.5,231) -- cycle ;
\draw    (318,172.5) .. controls (311,224.5) and (321.5,205) .. (316.5,250) ;
\draw    (370.5,172) .. controls (363.5,224) and (379.5,205) .. (374.5,250) ;
\draw  [line width=2.25]  (271.53,143.66) -- (410.97,143.66) -- (410.97,283.09) -- (271.53,283.09) -- cycle ;
\draw    (342,204) .. controls (382,174) and (456.5,242) .. (496.5,212) ;
\draw    [color={rgb, 255:red, 29; green, 92; blue, 216 }  ,draw opacity=1 ] (341.25,211.25) .. controls (338.82,211.52) and (337.46,210.5) .. (337.17,208.17) .. controls (336.94,205.87) and (335.65,204.85) .. (333.3,205.1) .. controls (330.99,205.41) and (329.66,204.4) .. (329.32,202.06) .. controls (328.82,199.69) and (327.41,198.77) .. (325.09,199.31) .. controls (322.95,200.1) and (321.43,199.46) .. (320.52,197.37) .. controls (318.99,195.42) and (317.36,195.29) .. (315.63,196.98) .. controls (314.99,198.95) and (313.62,199.85) .. (311.53,199.66) .. controls (309.2,200.47) and (308.5,202.01) .. (309.44,204.26) .. controls (310.61,206.22) and (310.2,207.83) .. (308.19,209.08) .. controls (306.21,210.39) and (305.88,211.98) .. (307.21,213.86) .. controls (308.5,215.91) and (308.17,217.58) .. (306.2,218.85) .. controls (304.22,220.05) and (303.82,221.7) .. (304.99,223.79) .. controls (306.11,225.84) and (305.59,227.39) .. (303.42,228.45) .. controls (301.27,229.17) and (300.47,230.6) .. (301.04,232.74) .. controls (299.94,235.22) and (298.43,235.52) .. (296.51,233.63) .. controls (295.88,231.39) and (294.42,230.53) .. (292.13,231.05) .. controls (290,231.76) and (288.49,231.08) .. (287.62,229) .. controls (285.83,227.24) and (284.2,227.41) .. (282.73,229.52) .. controls (282.35,231.71) and (281.02,232.72) .. (278.74,232.56) .. controls (276.31,232.88) and (275.45,234.3) .. (276.15,236.82) .. controls (277.81,238.13) and (278.11,239.73) .. (277.04,241.6) .. controls (276.61,243.99) and (277.59,245.28) .. (279.96,245.46) .. controls (282.31,245.37) and (283.59,246.47) .. (283.78,248.74) .. controls (284.27,251.11) and (285.74,252.06) .. (288.17,251.6) .. controls (290.3,250.88) and (291.69,251.62) .. (292.33,253.81) .. controls (293.45,256.16) and (295.11,256.88) .. (297.31,255.99) .. controls (299.45,255.01) and (300.88,255.52) .. (301.61,257.53) .. controls (302.78,259.64) and (304.41,260.11) .. (306.52,258.95) .. controls (308.52,257.7) and (310.17,258.06) .. (311.47,260.01) .. controls (312.9,261.91) and (314.53,262.12) .. (316.38,260.65) .. controls (318.04,259.09) and (319.77,259.15) .. (321.58,260.83) .. controls (323.23,262.42) and (324.86,262.27) .. (326.46,260.39) .. controls (327.69,258.43) and (329.29,258.01) .. (331.26,259.13) .. controls (333.47,259.96) and (334.93,259.17) .. (335.64,256.74) .. controls (335.69,254.45) and (336.8,253.26) .. (338.95,253.19) .. controls (341.34,252.2) and (342.03,250.62) .. (341.03,248.47) .. controls (339.96,246.46) and (340.5,244.94) .. (342.63,243.91) .. controls (344.79,242.8) and (345.36,241.2) .. (344.35,239.11) .. controls (343.36,236.97) and (343.94,235.42) .. (346.09,234.45) .. controls (348.22,233.56) and (348.83,232.05) .. (347.9,229.92) .. controls (347.06,227.68) and (347.75,226.19) .. (349.97,225.44) .. controls (352.2,224.9) and (353.12,223.49) .. (352.72,221.21) .. controls (352.45,218.88) and (353.51,217.58) .. (355.9,217.31) .. controls (358.24,217.2) and (359.35,216) .. (359.24,213.72) .. controls (359.36,211.31) and (360.61,210.19) .. (362.99,210.34) .. controls (365.32,210.65) and (366.66,209.67) .. (367.01,207.42) .. controls (367.74,205.09) and (369.26,204.31) .. (371.57,205.09) .. controls (373.58,206.18) and (375.18,205.78) .. (376.39,203.87) .. controls (378.13,202.16) and (379.75,202.25) .. (381.26,204.14) .. controls (382.59,206.07) and (384.25,206.38) .. (386.23,205.07) .. controls (388.26,203.82) and (389.85,204.22) .. (391,206.29) .. controls (391.95,208.4) and (393.51,209.02) .. (395.7,208.13) -- (397.31,208.86) -- (404.51,212.29) ;
\draw [shift={(407,213.5)}, rotate = 205.93] [fill={rgb, 255:red, 29; green, 92; blue, 216 }  ][line width=0.08]  [draw opacity=0] (8.93,-4.29) -- (0,0) -- (8.93,4.29) -- cycle    ;
\draw  [draw opacity=0][fill={rgb, 255:red, 0; green, 0; blue, 0 }  ,fill opacity=1 ] (310,197) .. controls (310,193.96) and (312.46,191.5) .. (315.5,191.5) .. controls (318.54,191.5) and (321,193.96) .. (321,197) .. controls (321,200.04) and (318.54,202.5) .. (315.5,202.5) .. controls (312.46,202.5) and (310,200.04) .. (310,197) -- cycle ;
\draw  [draw opacity=0][fill={rgb, 255:red, 0; green, 0; blue, 0 }  ,fill opacity=1 ] (297,230.5) .. controls (297,227.46) and (299.46,225) .. (302.5,225) .. controls (305.54,225) and (308,227.46) .. (308,230.5) .. controls (308,233.54) and (305.54,236) .. (302.5,236) .. controls (299.46,236) and (297,233.54) .. (297,230.5) -- cycle ;
\draw  [draw opacity=0][fill={rgb, 255:red, 0; green, 0; blue, 0 }  ,fill opacity=1 ] (407,213.5) .. controls (407,210.46) and (409.46,208) .. (412.5,208) .. controls (415.54,208) and (418,210.46) .. (418,213.5) .. controls (418,216.54) and (415.54,219) .. (412.5,219) .. controls (409.46,219) and (407,216.54) .. (407,213.5) -- cycle ;
\draw  [color={rgb, 255:red, 0; green, 0; blue, 0 }  ,draw opacity=1 ] (339.13,211.25) .. controls (339.13,210.08) and (340.08,209.13) .. (341.25,209.13) .. controls (342.42,209.13) and (343.38,210.08) .. (343.38,211.25) .. controls (343.38,212.42) and (342.42,213.38) .. (341.25,213.38) .. controls (340.08,213.38) and (339.13,212.42) .. (339.13,211.25) -- cycle ;
\draw [color={rgb, 255:red, 248; green, 231; blue, 28 }  ,draw opacity=0.5 ][line width=4.5] [line join = round][line cap = round]   (299.5,233) .. controls (291.23,233) and (290.29,227.78) .. (280.5,229) .. controls (274.45,229.76) and (279.1,245.07) .. (283.5,248) .. controls (298.18,257.79) and (306.34,260.31) .. (332.5,259) .. controls (336.59,258.8) and (337.58,252.92) .. (340.5,250) ;
\draw [color={rgb, 255:red, 248; green, 231; blue, 28 }  ,draw opacity=0.5 ][line width=3.75] [line join = round][line cap = round]   (381.5,204) .. controls (389.38,204) and (413.5,211) .. (407,213.5) ;

\draw (336,161.4) node [anchor=north west][inner sep=0.75pt]    {$\mathsf{T}_{1}$};
\draw (374,209.4) node [anchor=north west][inner sep=0.75pt]    {$\mathsf{T}_{2}$};
\draw (349,248.4) node [anchor=north west][inner sep=0.75pt]    {$\mathsf{T}_{3}$};
\draw (318.5,211.9) node [anchor=north west][inner sep=0.75pt]    {$\mathsf{T}_{4}$};
\draw (459,198.4) node [anchor=north west][inner sep=0.75pt]    {$\mathsf{T}_{9}$};
\draw (293,185) node [anchor=north west][inner sep=0.75pt]    {$x_{1}$};
\draw (285,213) node [anchor=north west][inner sep=0.75pt]    {$x_{2}$};
\draw (417,214) node [anchor=north west][inner sep=0.75pt]    {$x_{3}$};
\draw (286.5,259) node [anchor=north west][inner sep=0.75pt]    {$\gamma $};

\end{tikzpicture}
}
}

\caption{\small Illustration for the proof of Lemma~\ref{lem:big_PT_bound}.
In all three cases, the thick black box identifies $\partial B_R$, and $\gamma\in\Geo(\vc 0,\partial B_R)$ is displayed as a wavy blue path.
The inner shaded annulus is $(B_{\b^{k-1}}\setminus B_{\b^{k-2}})\cup\partial B_{\b^{k-2}}$, the outer shaded annulus is $(B_{\b^{k+1}} \setminus B_{\b^{k}}) \cup \partial B_{\b^{k}}$, and the two annuli are connected by the rectangle $[0,\b^{k+1}]\times[-\b^{k-2},\b^{k-2}]$, whose boundary is shown as dashed lines.
The edges belonging to $\gamma\cap\EE_k^\ann$ are highlighted in yellow.}
\label{fig:T19}

\end{center}
\end{figure}

To begin, we cover the annulus $(B_{\b^{k-1}} \setminus B_{\b^{k - 2}}) \cup \partial B_{\b^{k - 2}}$ by four rectangles corresponding to the top, right, bottom, and left portions of the annulus, in that order:
\begin{enumerate}[label=\textup{\arabic*.}]
\item $[-\b^{k-1},\b^{k-1}] \times [\b^{k - 2},\b^{k-1}]$
\item $[\b^{k - 2},\b^{k-1}] \times [-\b^{k-1},\b^{k-1}]$
\item $[-\b^{k-1},\b^{k-1}] \times [-\b^{k-1},-\b^{k - 2}]$
\item $[-\b^{k-1},-\b^{k - 2}]\times [-\b^{k-1},\b^{k-1}]$.
\end{enumerate}
Let $\Tsf_{1}$ be the minimal passage time among all left-right crossings of the first rectangle. 
Let $\Tsf_{2}$ be the minimal passage time among all top-bottom crossings of the second rectangle. 
Define $\Tsf_{3}$ and $\Tsf_{4}$ analogously to $\Tsf_{1}$ and $\Tsf_{2}$, respectively.
Any four crossings $\gamma_1,\gamma_2,\gamma_3,\gamma_4$ that achieve these passage times $\Tsf_{1},\Tsf_{2},\Tsf_{3},\Tsf_{4}$ must intersect in a cyclical pattern, i.e.\ $\gamma_1$ intersects $\gamma_2$, $\gamma_2$ intersects $\gamma_3$, $\gamma_3$ intersects $\gamma_4$, and $\gamma_4$ intersects $\gamma_1$. 
So by concatenating subpaths at intersection points, we construct a circuit $\CC_1$ that is contained in $B_{\b^{k-1}}$, surrounds the circuit formed by $\partial B_{\b^{k-2}}$ (thanks to Lemma~\ref{surround_lemma}), and has passage time at most $\sum_{i = 1}^4 \Tsf_{i}$. 

In the same way, we cover $(B_{\b^{k+1}} \setminus B_{\b^{k}}) \cup \partial B_{\b^{k}}$ by four rectangles and define $\Tsf_{5},\Tsf_{6},\Tsf_{7},\Tsf_{8}$ analogously to $\Tsf_{1},\Tsf_{2},\Tsf_{3},\Tsf_{4}$.
Once again we identify a circuit, this time called $\CC_2$, that is contained in $B_{\b^{k+1}}$, surrounds the circuit formed by $\partial B_{\b^{k}}$, and has passage time at most $\sum_{i = 5}^8 \Tsf_{i}$.

Finally, let $\Tsf_9$ be the minimal passage time of a left-right crossing of the rectangle $[0,\b^{k + 1}] \times [-\b^{k - 2},\b^{k - 2}]$, and choose a crossing $\gamma'$ achieving this passage time.
Since $\gamma'$ starts in $B_{\b^{k-2}}$ and ends at $\partial B_{\b^{k+1}}$, it must intersect both $\CC_1$ and $\CC_2$.
In the special cases $k=\floor{\log_\b  R}$ or $k=\ceil{\log_\b  R}$, we will instead use the fact that $\gamma'$ intersects both $\CC_1$ and $\partial B_R$.

We now claim 
    \begin{equation} \label{eqn:Tk123_new}
    T_{k,R}  \leq \sum_{i = 1}^9 \mathsf{T}_{i}.
    \end{equation}
To justify this claim, we consider any $\gamma\in\Geo(\vc0,\partial B_{R})$.
We identify four particular vertices on $\gamma$ which---by the geometry of our construction---are reached in the order $x_1,x_2,x_3,x_4$ as $\gamma$ proceeds from the origin to $\partial B_R$:
\begin{itemize}
\item If $k<\floor{\log_\b  R}$, let $x_1\in B_{\b^{k-1}}$ be the first intersection of $\gamma$ with $\CC_1$, let $x_2$ be the first intersection with $\partial B_{\b^{k-1}}$, let $x_3$ be last intersection with $\partial B_{\b^{k}}$, and let $x_4$ be the last intersection with $\CC_2$.
It is possible that $x_1=x_2$ and/or $x_3=x_4$. See Figure~\ref{fig:T19_a}.
\item If $k=\floor{\log_\b  R}$, let $x_1$, $x_2$, and $x_3$ be as above, and let $x_4\in\partial B_{R}\subseteq B_{\b^{k+1}}$ be the final vertex of $\gamma$.  See Figure~\ref{fig:T19_b}.
\item If $k=\ceil{\log_\b  R}$, let $x_1$ and $x_2$ be as above, and let $x_3=x_4\in\partial B_{R}\subseteq B_{\b^{k}}$ be the final vertex of $\gamma$.  See Figure~\ref{fig:T19_c}.
\end{itemize}
Let $\gamma_{x_i,x_j}$ be the subpath of $\gamma$ that starts at $x_i$ and ends at $x_j$.
By extremality of $x_2$ and $x_3$, any edge $e\in\gamma\cap\EE_k^\ann$ must belong to $\gamma_{x_2,x_3}$.
Hence
\eeq{ \label{ju8g4}
T_k(\gamma) \le T(\gamma_{x_2,x_3}).
}
On the other hand, there is an alternative route from $x_1$ to $x_4$ (or from $x_1$ to $\partial B_{R}$, if $k=\floor{\log_\b  R}$ or $k=\ceil{\log_\b  R}$): follow $\CC_1$ until reaching $\gamma'$, then follow $\gamma'$ until reaching $\CC_2$ (or until reaching $\partial B_{R}$), and finally follow $\CC_2$ to $x_4$.
The passage time of this alternative route is at most $\sum_{i=1}^9 \Tsf_i$. 
Since $\gamma$ is a geodesic, we conclude 
\eeq{ \label{ju8g5}
T(\gamma_{x_1,x_4})\le \sum_{i=1}^9 \Tsf_i.
}
Connecting \eqref{ju8g4} and \eqref{ju8g5} with the trivial inequality $T(\gamma_{x_2,x_3})\le T(\gamma_{x_1,x_4})$, we obtain \eqref{eqn:Tk123_new}.

From \eqref{eqn:Tk123_new} a simple union bound gives
     \eq{
     \P \Big( T_{k,R}   \geq \lambda F^{-1}(p)\Big( \dfrac{\b^k}{L(p)} \Big)^2  \Big)\le \sum_{i = 1}^9 \Big( \Tsf_{i}  \geq \f{\lambda}{9} F^{-1}(p)\Big(\dfrac{\b^k}{L(p)} \Big)^2\Big). 
     }
     Finally, we apply \eqref{damron-tang-Tkmax_eq} to each term on the right-hand side. Note that after translation and rotation, the nine rectangles match the setup of Lemma \ref{lem:damron-tang-Tkmax}: with $(n,K) = (\frac{\b^{k-1}-\b^{k-2}}{2},\frac{2b}{b-1})$ for $\Tsf_1,\dots,\Tsf_4$, $(n,K) = (\frac{\b^{k+1}-\b^k}{2},\frac{2b}{b-1})$ for $\Tsf_5,\dots,\Tsf_8$, and $(n,K) = (\b^{k-2},\frac{\b^3}{2})$ for $\Tsf_9$.
     Also note that in all nine cases, we have $n\ge\b^{k-3}\ge L(p)$.
\end{proof}

Since the second term on the right-hand side of \eqref{big_PT_bound_eq} does not depend on $\lambda$, it makes sense to lower $\lambda$ all the way to $\b^k/L(p)$.
This is essentially what is done in the following result, for various $p$ satisfying the constraint $L(p)\le \b^{k-3}$ from Lemma \ref{lem:big_PT_bound}.
We will use (here and in numerous future locations) the following simple inequalities:
\eeq{ \label{floor_bound}
x/2 \le \floor{x} \le \ceil{x} \le 2x \quad \text{for all $x\ge1$}.
}

\begin{lemma} \label{cor:Tkmaxbd} 
Assume \eqref{critical_assumption}.
There exist constants $C,c > 0 $ such that
\eeq{ \label{cor:Tkmaxbd_eq}
\mathbb{P}\Big(T_{k,R} \geq  \theta^3 F^{-1}( p_{\floor{\b^{k-3}/\theta}}) \Big) \leq C  \e^{-c\theta}
\quad \text{for all $R\ge1$, $k\ge3$, $\theta\in[1,\b^{k-3}]$.}
}
\end{lemma}
	
\begin{proof}
By the upper bound in \eqref{eq: L_p_n}, we have $L(p_{\floor{\b^{k-3}/\theta}}) \le \floor{\b^{k-3}/\theta} \le \b^{k-3}$.
Therefore, we may use Lemma \ref{lem:big_PT_bound} with $p = p_{\floor{\b^{k-3}/\theta}}$ to obtain the following inequality for all $\lambda\ge0$:
\eq{
 \P \bigg( T_{k,R}   \geq \lambda F^{-1}(p_{\floor{\b^{k-3}/\theta}})\Big( \dfrac{\b^k}{L(p_{\floor{\b^{k-3}/\theta}})} \Big)^2 \bigg) 
 &\leq  C\e^{-c\lambda} + C\exp \Big( -c \dfrac{\b^k}{L(p_{\floor{\b^{k-3}/\theta}})}  \Big) \\
  &\leq  C\e^{-c\lambda} + C\exp \Big( -c \dfrac{\b^k}{\b^{k-3}/\theta}\Big)  \Big) \\
  &= C\e^{-c\lambda} + C\e^{-c\b^3\theta}.
}
By the lower bound in \eqref{eq: L_p_n}, there is a constant $\delta\in(0,\b^{-3}/2)$ (not depending on $R$, $k$, or $\theta$) such that
\eq{
L(p_{\floor{\b^{k-3}/\theta}}) \ge \delta\floor{\b^{k-3}/\theta} \stackref{floor_bound}{\ge} \delta\cdot \b^{k-3}/(2\theta) \ge \delta^2\cdot \b^k/\theta,
}
so the previous display implies
\eq{
 \P \Big( T_{k,R}   \geq \lambda F^{-1}(p_{\floor{\b^{k-3}/\theta}})\Big(\frac{\theta}{\delta^2}\Big)^2 \Big) \leq  C\e^{-c\lambda} + C\e^{-c\b^3\theta}.
}
Finally set $\lambda = \delta^4\theta$ to obtain \eqref{cor:Tkmaxbd_eq} after suitably adjusting $c$ and $C$. 
	\end{proof}

The drawback of \eqref{cor:Tkmaxbd_eq} is that $\theta^3F^{-1}( p_{\floor{\b^{k-3}/\theta}})$ can diverge as $\theta$ and $k$ jointly tend to $\infty$.
But under the additional hypothesis of \eqref{eq : limit_bigger_than_one}, we can keep prevent this divergence provided $\theta$ is not too close to $\b^{k}$:

\begin{proposition} \label{lem:param}
Assume \eqref{critical_assumption} and \eqref{eq : limit_bigger_than_one}.
There exist constants $C,c>0$ and $q\ge 6$ such that
\eeq{ \label{param_eq}
\mathbb{P}\Big(T_{k,R} \geq F^{-1}(p_{\floor{\b^{k}/\theta^q}}) \Big) \leq C \e^{- c\theta}
\quad \text{for all $R\ge1$, $k\ge3$, $\theta\in[\tfrac{3}{2},c\b^{k/q}]$.}
}
	\end{proposition}

	\begin{proof}
    Comparing to Lemma~\ref{cor:Tkmaxbd}, we just need to find sufficiently large $q\geq 6$ and sufficiently small $c_0>0$ such that
\eeq{ \label{e7gxg}
F^{-1}(p_{\floor{\b^k/\theta^q}  }) \geq   \theta^{3}F^{-1}(p_{\floor{\b^{k-3}/\theta} } )
\quad \text{for all $k\ge3$, $\theta\in[\tfrac{3}{2},c_0\b^{k/q}]$}.
}
	By \eqref{eq : limit_bigger_than_one}, there exist $\eps>0$ and a nonnegative integer $n_0$ such that $F^{-1}(p_{\b^{n}})/F^{-1}(p_{\b^{n+1}}) \geq 1+\eps$ for all $n\ge n_0$. 
    Iterating this equality yields
	\begin{equation} \label{eq:Lell}
     F^{-1}(p_{\b^{n}})/F^{-1}(p_{\b^{n+\ell}}) \geq (1+\eps)^{\ell} \quad \text{whenever $n\ge n_0$, $\ell\ge0$}.
     \end{equation}
     Now choose $q\ge 6$ large enough that
     \eeq{ \label{4fxvb}
     (q-1)\log_\b \theta \ge 5 \quad \text{and} \quad
     \theta^{(q-1)\log_\b (1+\eps)-3} \ge (1+\eps)^6 \qquad 
     \text{for all $\theta\ge\tfrac{3}{2}$},
     }
     and set $c_0 = \b^{-n_0/q} \wedge \b^{-3/q}$.
     We will prove \eqref{e7gxg} with these choices, so assume henceforth that $k\ge3$ and $\theta\in[\tfrac{3}{2},c_0\b^{k/q}]$.

The inequality $c_0\le \b^{-n_0/q}$ gives 
$\theta \le \b^{-n_0/q}\cdot b^{k/q}$.
Solving this inequality for $n_0$ yields
\eeq{ \label{wj76c}
n_0 \le k - q\log_\b  \theta \le k - \floor{\log_\b  \theta} - \floor{(q-1)\log_\b  \theta}.
}
Using the second inequality of \eqref{wj76c}, we note that
\eq{
\floor{\b^k/\theta^q} \le \b^k/\theta^q \le \b^{k-\floor{\log_\b \theta}-\floor{(q-1)\log_\b  \theta}}. 
}
On the other hand, the inequality $c_0\le \b^{-3/q}$ gives $\theta \le b^{-3/q}\cdot\b^{k/q}\le \b^{k-3}$,
which enables the first inequality below:
\eq{
\floor{\b^{k-3}/\theta} \stackref{floor_bound}{\ge} 
\b^{k-3}/(2\theta) \ge
\b^{k-4}/\theta = \b^{k-4-\log_\b \theta} \ge \b^{k-5-\floor{\log_\b \theta}}. 
}
Since $r\mapsto p_r$ is weakly decreasing and $F^{-1}$ is weakly increasing, the two previous displays together imply
\eeq{ \label{32gbe}
\frac{F^{-1}(p_{\floor{\b^k/\theta^q}  })}{F^{-1}(p_{\floor{\b^{k-3}/\theta} } )}
\ge \frac{F^{-1}(\b^{k-\floor{\log_\b \theta}-\floor{(q-1)\log_\b  \theta}})}{F^{-1}(\b^{k-5-\floor{\log_\b \theta}})}.
}
The inequality \eqref{wj76c} enables us to apply \eqref{eq:Lell} with $n = k-\floor{\log_\b \theta}-\floor{(q-1)\log_\b \theta}$ and $\ell = \floor{(q-1)\log_\b \theta}-5$, where $\ell\ge0$ because of the first inequality in \eqref{4fxvb}.
This application of \eqref{eq:Lell} produces a lower bound on the right-hand side of \eqref{32gbe}, resulting in
\begin{align*}
\frac{F^{-1}(p_{\floor{\b^k/\theta^q}  })}{F^{-1}(p_{\floor{\b^{k-3}/\theta} } )}
\ge (1+\eps)^{\floor{(q-1)\log_\b \theta}-5}
&\ge (1+\eps)^{(q-1)\log_\b (\theta)-6} \\
&= \theta^{(q-1)\log_\b(1+\eps)}(1+\eps)^{-6}
\stackref{4fxvb}{\ge} \theta^3. \qedhere
\end{align*}
	\end{proof}

 We next prove a result similar to Proposition~\ref{lem:param}, but now assuming 
 \eeq{ \label{weak_assump_2}
    \liminf_{n\to\infty} F^{-1}(p_{\b^{n+1}})/F^{-1}(p_{\b^n}) > 0.
 }
 Note that \eqref{eq : limit_equal_zero} implies \eqref{weak_assump_2}, since $p_{\b^{n+1}} \ge p_{\b^{2n}}$ for any $n\ge1$.

 \begin{proposition} \label{lem:Tkn0assumption} 
 Assume \eqref{critical_assumption} and \eqref{weak_assump_2}. 
 There exist constants $C,c>0$ and $q\ge3$ such that
\eeq{ \label{Tkn0assumption_eq}
\mathbb{P}\Big(T_{k,R} \geq \lambda F^{-1}(p_{\b^k} ) \Big) \leq C \e^{-c\lambda^{1/q}} \quad \text{for all $R\ge1$, $k\ge3$, $\lambda\in[0,c\b^{qk}]$}.
}
 	\end{proposition}

\begin{proof}
By the assumption \eqref{weak_assump_2}, there exist $\eps\in(0,1]$ and a nonnegative integer $n_0$ such that 
 $F^{-1}(p_{\b^{n+1}})/F^{-1}(p_{\b^n}) \geq  \eps$ for all $n\ge n_0$.
Iterating this inequality yields
\eeq{ \label{elf12}
F^{-1}(p_{\b^{n+\ell}})/F^{-1}(p_{\b^n}) \ge \eps^\ell \quad \text{whenever $n\ge n_0$, $\ell\ge0$}.
}
For any $\theta\in[1,\b^{k-3}]$, we have the following inequalities:
\eq{
\floor{\b^{k-3}/\theta} 
\stackref{floor_bound}{\ge} \b^{k-3}/(2\theta)
\ge \b^{k-4}/\theta = \b^{k-4-\log_\b \theta} \ge \b^{k-5-\floor{\log_\b \theta}}.
}
Since $r\mapsto p_r$ is weakly decreasing and $F^{-1}$ is weakly increasing, it follows that
\eeq{ \label{mv29bx}
F^{-1}(p_{\floor{\b^{k-3}/\theta}}) \le F^{-1}(p_{\b^{k-5-\floor{\log_\b \theta}}}) \quad \text{for all $\theta\in[1,\b^{k-3}]$}.
}
Under the more restrictive condition $\theta \in[1,\b^{k-5-n_0}]$, we have
\eq{
n_0 \le  k-5-\log_\b \theta \le k-5-\floor{\log_\b \theta},
}
so we can apply \eqref{elf12} with $n = k-5-\floor{\log_\b \theta}$ and $\ell = 5+\floor{\log_\b \theta}$.
This results in
\eeq{ \label{mv29by}
F^{-1}(p_{\b^k})/F^{-1}(p_{\b^{k-5-\floor{\log_\b \theta}}}) 
\ge \eps^{5+\floor{\log_\b \theta}} 
&\ge \eps^5\eps^{\log_\b \theta} \\
&= \eps^5\theta^{\log_\b \eps}\quad \text{for all $\theta\in[1,c_1\b^k]$},
}
where $c_1 = \b^{-5-n_0}$.
Combining \eqref{mv29bx} and \eqref{mv29by} yields
\eeq{ \label{mv29bz}
F^{-1}(p_{\b^k}) \ge \eps^{5}\theta^{\log_\b \eps}F^{-1}(p_{\floor{\b^{k-3}/\theta}}) \quad \text{for all $\theta\in[1,c_1\b^k]$}.
}
Finally, set $q = 3-\log_\b \eps$.
Given any $\lambda\in[\eps^{-5},\eps^{-5}(c_1\b^{k})^q]$, we choose $\theta = (\lambda\eps^5)^{1/q} \in [1,c_1\b^k]$ so that
\eq{
\P\Big(T_{k,R} \ge \lambda F^{-1}(p_{\b^k})\Big) 
&\stackref{mv29bz}{\le} \P\Big(T_{k,R} \ge \lambda\eps^{5}\theta^{\log_\b \eps}F^{-1}(p_{\floor{\b^{k-3}/\theta}})\Big) \\
&\stackrefp{mv29bz}{=} \P\Big(T_{k,R} \ge \theta^3F^{-1}(p_{\floor{\b^{k-3}/\theta}})\Big) \\
&\stackrefpp{cor:Tkmaxbd_eq}{mv29bz}{\le} C\e^{-c\theta} 
= C\e^{-c\eps^{5/q}\lambda^{1/q}}.
}
We have thus proved \eqref{Tkn0assumption_eq} after suitably adjusting $c$, for $\lambda\in[\eps^{-5},c\b^{qk}]$.
We may also include $\lambda\in[0,\eps^{-5}]$ after adjusting $C$.
\end{proof}

\subsection{Bounds for number of closed edges on geodesic within annuli} \label{sec_ceba}

The goal of this subsection is to upper bound the number of closed edges used by a geodesic on its subpath between $\partial B_{\b^{k-1}}$ and $\partial B_{\b^k}$.
The eventual conclusion will be that this number is tight (with respect to $k$) with a stretched exponential tail (Proposition~\ref{prop:Dknbd}).
Recall that we can realize the edge-weights as $t_e = F^{-1}(U_e)$, where $(U_e)_{e\in E(\Z^2)}$ are i.i.d.\ uniform random variables on $[0,1]$.
We say a path $\gamma$ is \textit{$p$-open} if $U_e\le p$ for each edge $e\in\gamma$, or \textit{$p$-closed} if $U_e>p$ for all $e\in\gamma$.

The first result below (Lemma~\ref{lem : kiss}) is a general statement about near-critical percolation and does not depend on the distribution function $F$.
For $p>p_\cc$ and an edge $e\in\EE_k^\ann$, let $\Asf_{k,R}(p,e)$ be the event that all of the following occur:
\begin{enumerate}[label=\textup{(\roman*)}]

\item \label{Aevent_1} $U_e\in(p_\cc,p]$;

\item \label{Aevent_2} there are two $p$-open paths starting at the two vertices of $e$: one ending at $\partial B_{\b^{k-2}}(e)$ and the other at $\partial B_{r}(e)$, where
\eq{
r = \min\{\b^{k-2},\dist(e,\partial B_{R})\}.
}

\item \label{Aevent_3} there are two $p_\cc$-closed dual paths starting at the two vertices of $e^\star$: both ending at $\partial B_{\b^{k-2}}(e)$;

\item \label{Aevent_4} all four of these paths are disjoint and remain in $B_R$.

\end{enumerate}
For an illustration of this event, see the yellow box in Figure~\ref{fig_kiss_and_beyond}.
Let $N_{k,R}(p)$ denote the number of $e\in\EE_k^\ann$ for which $\Asf_{k,R}(p,e)$ occurs:
\eeq{ \label{wefu7b}
N_{k,R}(p) = \sum_{e\in\EE_k^\ann} \one_{\Asf_{k,R}(p,e)},
}
where $\EE_k^\ann$ is defined in \eqref{def_ann}.
The following estimate is an adaptation of \cite[Lem.~2]{damron_tang19}, which considers an analogous quantity where the annulus is replaced by a rectangle.

\begin{lemma} \label{lem : kiss}
There exist constants $C,c>0$ such that for all $R\ge1$, $k\in\{2,\dots,\ceil{\log_\b  R}\}$, and $p>p_\cc$ with $L(p)\le \b^{k-2}$, we have
\eeq{ \label{ibk3c0}
 \P \Big( N_{k,R}(p) \geq \lambda \Big( \dfrac{\b^k}{L(p)} \Big)^2 \Big) \leq C\e^{-c \lambda} \quad \text{for all $\lambda\ge0$}.
}
\end{lemma}

Since the proof is essentially the same as \cite[Lem.~2]{damron_tang19}, which in turn borrows parts of \cite{kiss14,kiss_manolescu_sidoravicius15}, we omit the details.
 Instead we explain here why $(\b^k/L(p))^2$ is the correct scale in \eqref{ibk3c0}, and then point out the differences from \cite{damron_tang19}.
 First observe that the event in \ref{Aevent_1} has probability $p-p_\cc$. 
 Next observe that the event in \ref{Aevent_2}--\ref{Aevent_4} is independent of $U_e$ and effectively implies a polychromatic 4-arm event to distance at least $r$.
 If $r = \b^{k-2}$, then all four arms extend to distance at least $L(p)$, and the fact that the arms in \ref{Aevent_2} are $p$-open rather than $p_\cc$-open is inconsequential at scale $L(p)$ (see \cite[Lem.~6.3]{damron_sapozhnikov_vagvolgyi09}), so the probability of this event is at most $C\pi_4(1,L(p))$.
 If $k\in\{\floor{\log_\b  R},\ceil{\log_\b  R}\}$, then it is possible that $r<L(p)$, in which case the second arm in \ref{Aevent_2} is too short.  But in this scenario the remaining three arms continue beyond this distance while confined to a half plane created by the nearest side of $\partial B_R$.
 Since a 3-arm event in a half plane is far less likely than a 4-arm event in the full plane (see \cite[Thm.~24(2)]{nolin08}), the upper bound of $C\pi_4(1,L(p))$ still applies after we invoke quasi-multiplicativity.
 Finally, there are at most $C(\b^k)^2$ edges in $\EE_k^\ann$.
 All of these observations together suggest that the first moment of $N_{k,R}(p)$ is at most
 \eq{
 C(\b^k)^2 \cdot \pi_4(1,L(p)) \cdot (p-p_\cc) \stackref{eq: scaling_relation}{\asymp} \Big(\frac{\b^k}{L(p)}\Big)^2.
 }

 To obtain \eqref{ibk3c0}, one needs also to control the higher moments of $N_{k,R}(p)$.
 This is done in \cite[p.~112--114]{damron_tang19} by using the probability estimate described above (see \cite[disp.~(15)]{damron_tang19}), together with a quasi-mulitiplicativity argument that controls the probability that a particular collection of edges is a subset of $\{e\in\EE_k^\ann:\, \Asf_{k,R}(p,e)\text{ occurs}\}$ (see \cite[disp.~(14)]{damron_tang19}).
 Compared to our setting, there are three differences:
 \begin{itemize}
 \item In \cite{damron_tang19}, the event $A_n(p,e)$ (their notation on p.~109) uses $p_\cc$-closed arms that extend to a particular domain boundary rather than to a certain distance from $e$.
 This means that one of their closed arms is shortened from encountering the boundary.
 By contrast, both of our closed arms always extend to a certain distance, namely $\b^{k-2}$.
 Since this distance is what matters in the analysis (i.e.~greater distance creates greater rarity), our setting suffers no disadvantage in this regard.

\item On the other hand, unlike in \cite{damron_tang19}, the presence of a boundary (namely $\partial B_R$) can shorten one of our $p$-open arms.
 This means that our half-space 3-arm event (mentioned in the previous paragraph) has two $p_\cc$-closed paths and one $p$-open path, rather than two $p$-open paths and one $p_\cc$-closed path as in \cite{damron_tang19}.
 Our event is less likely since $p_\cc<p$ (although the two events have comparable probabilities at scale $L(p)$), so this difference causes no issue.
 
 \item Instead of considering all edges in a rectangle as in \cite{damron_tang19}, our sum \eqref{wefu7b} is restricted to edges in an annulus.  Since this restriction can only lower the value of $N_{k,R}$ (when compared to a rectangle that contains the annulus, which has a comparable number of edges), we can safely ignore this difference.
 \end{itemize}

\begin{figure}[ht]
\centering

\tikzset{every picture/.style={line width=0.75pt}} 

\begin{tikzpicture}[x=0.75pt,y=0.75pt,yscale=-1,xscale=1]

\draw  [line width=2.25]  (154.25,24.25) -- (528.25,24.25) -- (528.25,398.25) -- (154.25,398.25) -- cycle ;
\draw  [color={rgb, 255:red, 155; green, 155; blue, 155 }  ,draw opacity=1 ] (496.25,56.25) -- (186.25,56.25) -- (186.25,366.25) -- (496.25,366.25) -- cycle ;
\draw  [color={rgb, 255:red, 155; green, 155; blue, 155 }  ,draw opacity=1 ] (419,134) -- (263,134) -- (263,288) -- (419,288) -- cycle ;
\draw  [color={rgb, 255:red, 155; green, 155; blue, 155 }  ,draw opacity=1 ] (380,172.5) -- (302.5,172.5) -- (302.5,250) -- (380,250) -- cycle ;
\draw  [color={rgb, 255:red, 155; green, 155; blue, 155 }  ,draw opacity=1 ] (360.5,192.5) -- (322.25,192.5) -- (322.25,229.5) -- (360.5,229.5) -- cycle ;
\draw  [draw opacity=0][fill={rgb, 255:red, 184; green, 181; blue, 181 }  ,fill opacity=0.5 ] (418.5,134.5) -- (418.5,288.5) -- (262.5,288.5) -- (262.5,134.5) -- (418.5,134.5) -- cycle (380,172.5) -- (302,172.5) -- (302,250.5) -- (380,250.5) -- (380,172.5) -- cycle ;
\draw  [color={rgb, 255:red, 155; green, 155; blue, 155 }  ,draw opacity=1 ][fill={rgb, 255:red, 248; green, 231; blue, 28 }  ,fill opacity=0.5 ] (293.63,199.5) -- (255.38,199.5) -- (255.38,236.5) -- (293.63,236.5) -- cycle ;
\draw [color={rgb, 255:red, 208; green, 2; blue, 27 }  ,draw opacity=1 ][line width=2.25]    (271.88,218) -- (274.5,218) -- (276.88,218) ;
\draw [color={rgb, 255:red, 208; green, 2; blue, 27 }  ,draw opacity=1, dashed]   (274.5,218) .. controls (269.25,175.5) and (277.25,160) .. (283.75,149) .. controls (290.25,138) and (307.16,135.25) .. (327.38,142.97) .. controls (347.6,150.69) and (375.69,113.77) .. (390.88,120.5) .. controls (406.06,127.23) and (401.25,147) .. (372.25,158.5) .. controls (343.25,170) and (333.17,178.96) .. (348.75,197) .. controls (364.33,215.04) and (482.25,122.5) .. (503.75,163) .. controls (525.25,203.5) and (525.25,254.5) .. (457.75,300.5) .. controls (390.25,346.5) and (306.22,344.76) .. (281.25,340) .. controls (256.28,335.24) and (265.4,295.97) .. (268.58,280.73) .. controls (271.75,265.5) and (272.75,260) .. (273.25,246.5) .. controls (273.75,233) and (275.38,223.75) .. (274.5,218) -- cycle ;
\draw [color={rgb, 255:red, 74; green, 144; blue, 226 }  ,draw opacity=1 ]   (187.25,88.5) .. controls (188.44,86.23) and (190.08,85.73) .. (192.19,87.01) .. controls (193.91,88.48) and (195.58,88.31) .. (197.19,86.51) .. controls (198.78,84.76) and (200.49,84.7) .. (202.32,86.33) .. controls (203.83,87.98) and (205.57,87.99) .. (207.56,86.35) .. controls (209.28,84.73) and (210.75,84.78) .. (211.98,86.49) .. controls (213.65,88.24) and (215.43,88.35) .. (217.33,86.8) .. controls (219.32,85.27) and (220.95,85.4) .. (222.22,87.19) .. controls (223.96,89.03) and (225.64,89.2) .. (227.26,87.69) .. controls (228.92,86.2) and (230.63,86.41) .. (232.4,88.31) .. controls (233.65,90.16) and (235.13,90.37) .. (236.85,88.94) .. controls (239.1,87.61) and (240.84,87.9) .. (242.05,89.79) .. controls (243.22,91.7) and (244.94,92.02) .. (247.21,90.77) .. controls (249.02,89.44) and (250.47,89.75) .. (251.57,91.7) .. controls (253.08,93.77) and (254.74,94.17) .. (256.55,92.91) .. controls (258.84,91.8) and (260.44,92.25) .. (261.35,94.25) .. controls (262.6,96.38) and (264.34,96.95) .. (266.57,95.95) .. controls (268.4,94.84) and (269.83,95.38) .. (270.86,97.57) .. controls (271.71,99.75) and (273.21,100.43) .. (275.36,99.6) .. controls (277.89,99.05) and (279.39,99.89) .. (279.85,102.1) .. controls (280.32,104.43) and (281.7,105.45) .. (283.99,105.15) .. controls (286.34,105.1) and (287.46,106.33) .. (287.36,108.84) .. controls (286.59,110.79) and (287.28,112.28) .. (289.43,113.29) .. controls (291.31,114.83) and (291.37,116.45) .. (289.61,118.14) .. controls (287.68,119.3) and (287.26,120.87) .. (288.35,122.84) .. controls (289.09,125.14) and (288.28,126.61) .. (285.93,127.24) .. controls (283.59,127.57) and (282.56,128.83) .. (282.83,131.02) .. controls (282.62,133.56) and (281.34,134.74) .. (279,134.56) .. controls (276.7,134.25) and (275.36,135.26) .. (274.99,137.58) .. controls (274.45,139.95) and (273.13,140.8) .. (271.03,140.15) .. controls (268.66,139.6) and (267.08,140.5) .. (266.3,142.83) .. controls (265.75,145) and (264.27,145.75) .. (261.84,145.1) .. controls (259.78,144.23) and (258.24,144.95) .. (257.23,147.24) .. controls (256.51,149.38) and (255.13,149.98) .. (253.08,149.04) .. controls (250.71,148.21) and (249.15,148.84) .. (248.39,150.94) .. controls (247.26,153.17) and (245.72,153.76) .. (243.75,152.73) .. controls (241.48,151.8) and (239.81,152.42) .. (238.73,154.59) .. controls (238,156.63) and (236.56,157.15) .. (234.41,156.15) .. controls (232.01,155.24) and (230.36,155.83) .. (229.46,157.92) .. controls (228.63,159.99) and (227.15,160.54) .. (225,159.55) .. controls (222.71,158.64) and (221.11,159.27) .. (220.22,161.43) .. controls (219.47,163.59) and (218.03,164.3) .. (215.89,163.57) .. controls (213.36,163.52) and (212.2,164.69) .. (212.39,167.08) .. controls (212.94,169.32) and (212.07,170.84) .. (209.79,171.65) .. controls (207.64,172.48) and (207.06,174.07) .. (208.05,176.4) .. controls (209.3,178.35) and (208.98,179.88) .. (207.11,181) .. controls (205.3,182.67) and (205.23,184.43) .. (206.9,186.28) .. controls (208.71,187.59) and (208.93,189.14) .. (207.57,190.93) .. controls (206.54,193.31) and (207.14,194.99) .. (209.35,195.96) .. controls (211.46,196.38) and (212.29,197.76) .. (211.84,200.09) .. controls (211.57,202.4) and (212.75,203.71) .. (215.37,204.02) .. controls (217.46,203.63) and (218.71,204.63) .. (219.1,207.01) .. controls (219.51,209.26) and (221.04,210.18) .. (223.67,209.75) .. controls (225.52,208.8) and (226.9,209.44) .. (227.83,211.67) .. controls (229.02,213.91) and (230.6,214.48) .. (232.55,213.39) .. controls (234.48,212.23) and (236.26,212.72) .. (237.88,214.87) .. controls (239.19,216.88) and (240.58,217.18) .. (242.07,215.78) .. controls (244.12,214.45) and (245.94,214.75) .. (247.53,216.68) .. controls (248.61,218.51) and (250.25,218.7) .. (252.46,217.27) .. controls (253.98,215.74) and (255.74,215.88) .. (257.74,217.7) .. controls (259.13,219.45) and (260.63,219.53) .. (262.22,217.93) .. controls (263.82,216.31) and (265.4,216.35) .. (266.95,218.05) .. controls (268.59,219.72) and (270.25,219.72) .. (271.92,218.04) .. controls (273.6,216.35) and (275.31,216.32) .. (277.06,217.95) .. controls (278.71,219.6) and (280.69,219.59) .. (283.01,217.92) .. controls (284.51,216.27) and (285.95,216.3) .. (287.32,218) .. controls (289.23,219.73) and (290.85,219.8) .. (292.18,218.2) .. controls (294.07,216.66) and (295.77,216.8) .. (297.26,218.61) .. controls (298.53,220.44) and (300.18,220.67) .. (302.2,219.3) .. controls (304.17,218) and (305.81,218.4) .. (307.12,220.49) .. controls (307.81,222.54) and (309.24,223.24) .. (311.39,222.58) .. controls (313.85,222.83) and (314.6,224.18) .. (313.63,226.63) .. controls (311.72,227.79) and (311.15,229.44) .. (311.92,231.57) .. controls (312.43,233.77) and (311.42,235.1) .. (308.9,235.56) .. controls (306.52,235.7) and (305.41,236.89) .. (305.57,239.13) .. controls (305.62,241.43) and (304.48,242.55) .. (302.15,242.48) .. controls (299.58,242.63) and (298.27,243.89) .. (298.22,246.24) .. controls (298.21,248.57) and (297.04,249.72) .. (294.72,249.67) .. controls (292.4,249.66) and (291.32,250.79) .. (291.49,253.07) .. controls (291.64,255.46) and (290.57,256.8) .. (288.28,257.07) .. controls (285.87,257.77) and (285.07,259.25) .. (285.87,261.5) .. controls (287.19,263.21) and (287.13,264.85) .. (285.68,266.42) .. controls (285.1,268.77) and (285.95,270.13) .. (288.23,270.52) .. controls (290.6,271.09) and (291.5,272.46) .. (290.92,274.65) .. controls (290.46,277.07) and (291.35,278.52) .. (293.6,279) .. controls (295.89,279.58) and (296.75,280.99) .. (296.16,283.24) .. controls (295.64,285.57) and (296.53,286.97) .. (298.82,287.44) .. controls (301.12,287.83) and (302.12,289.2) .. (301.83,291.54) .. controls (301.73,293.93) and (302.9,295.11) .. (305.33,295.08) .. controls (307.54,294.61) and (308.94,295.46) .. (309.54,297.63) .. controls (310.75,299.79) and (312.4,300.21) .. (314.47,298.9) .. controls (316.26,297.38) and (317.91,297.53) .. (319.43,299.35) .. controls (320.87,301.2) and (322.52,301.41) .. (324.39,299.98) .. controls (326.2,298.55) and (327.81,298.78) .. (329.23,300.65) .. controls (330.73,302.5) and (332.38,302.66) .. (334.19,301.15) .. controls (335.98,299.56) and (337.67,299.61) .. (339.26,301.3) .. controls (340.97,302.93) and (342.61,302.87) .. (344.19,301.12) .. controls (345.82,299.32) and (347.49,299.17) .. (349.22,300.67) .. controls (351.13,302.12) and (352.78,301.91) .. (354.19,300.04) .. controls (355.46,298.08) and (357.13,297.61) .. (359.21,298.64) .. controls (361.42,299.46) and (362.9,298.71) .. (363.66,296.4) .. controls (364.05,294.21) and (365.49,293.27) .. (367.98,293.59) .. controls (370.34,293.92) and (371.56,293) .. (371.65,290.82) .. controls (371.96,288.41) and (373.35,287.24) .. (375.82,287.31) .. controls (378.07,287.53) and (379.19,286.5) .. (379.19,284.23) .. controls (379.42,281.68) and (380.69,280.44) .. (382.98,280.5) .. controls (385.3,280.5) and (386.41,279.34) .. (386.32,277.01) .. controls (386.19,274.69) and (387.28,273.48) .. (389.61,273.39) .. controls (391.95,273.24) and (393.02,272) .. (392.82,269.66) .. controls (392.83,267.03) and (394,265.61) .. (396.31,265.4) .. controls (398.62,265.15) and (399.61,263.88) .. (399.26,261.59) .. controls (398.87,259.34) and (399.8,258.07) .. (402.05,257.8) .. controls (404.52,257.18) and (405.5,255.78) .. (404.97,253.59) .. controls (404.57,251.16) and (405.45,249.8) .. (407.62,249.5) .. controls (409.94,248.88) and (410.8,247.43) .. (410.19,245.16) .. controls (409.6,242.74) and (410.37,241.25) .. (412.51,240.7) .. controls (414.72,239.85) and (415.38,238.28) .. (414.49,235.99) .. controls (413.33,234.09) and (413.7,232.51) .. (415.61,231.25) .. controls (417.11,229.18) and (416.68,227.66) .. (414.33,226.67) .. controls (412.04,226.85) and (410.69,225.87) .. (410.26,223.72) .. controls (409.22,221.48) and (407.61,220.88) .. (405.42,221.93) .. controls (403.49,223.14) and (401.91,222.76) .. (400.67,220.78) .. controls (399.25,218.82) and (397.64,218.54) .. (395.84,219.93) .. controls (393.81,221.32) and (392.13,221.08) .. (390.78,219.23) .. controls (389.39,217.38) and (387.67,217.18) .. (385.64,218.63) .. controls (383.9,220.12) and (382.35,219.96) .. (380.98,218.14) .. controls (379.4,216.29) and (377.68,216.1) .. (375.81,217.57) .. controls (373.99,219.03) and (372.33,218.82) .. (370.83,216.93) .. controls (369.58,215.03) and (367.95,214.72) .. (365.95,215.99) .. controls (363.84,217.06) and (362.42,216.33) .. (361.69,213.81) -- (361.5,213) ;

\draw (258.5,77) node [anchor=north west][inner sep=0.75pt]    {$\gamma _{R}$};
\draw (275,207) node [anchor=north west][inner sep=0.75pt]    {$e$};
\draw (321,213) node [anchor=north west][inner sep=0.75pt]    {$\partial B_{b^{k-2}}$};
\draw (340,250.4) node [anchor=north west][inner sep=0.75pt]    {$\partial B_{b^{k-1}}$};
\draw (456,367.4) node [anchor=north west][inner sep=0.75pt]    {$\partial B_{b^{k+1}}$};
\draw (391.5,289.9) node [anchor=north west][inner sep=0.75pt]    {$\partial B_{b^{k}}$};
\draw (158.5,380) node [anchor=north west][inner sep=0.75pt]    {$\partial B_{R}$};

\end{tikzpicture}

\caption{ \small Illustration for Lemmas~\ref{lem : kiss} and \ref{cor:DknMknbd}.
Both lemmas are concerned with edges $e$ in the shaded annulus, i.e.\ $e\in\EE_k^\ann$.
Shown here is the case $k<\floor{\log_\b R}$.
The event $\Asf_{k,R}(p,e)$ depends on the edges in $B_{\b^{k-2}}(e)$, shown as the yellow box centered at $e$.
In the proof of Lemma~\ref{cor:DknMknbd}, this event occurs because the geodesic $\gamma_R$ (wavy blue curve) provides two $p$-open arms, and the closed dual circuit (dashed red curve) provides two $p_\cc$-closed arms.}
\label{fig_kiss_and_beyond}
\end{figure}

Henceforth we let $\gamma_R\in\Geo(\vc 0,\partial B_R)$ be the geodesic constructed in Proposition \ref{prop:geod_cons}, with $\mathcal A = \{\vc 0\}$ and $\mathcal B = \partial B_R$.
From Lemma~\ref{lem : kiss} we can limit the number of edges $e\in\gamma_R\cap\EE_k^\ann$ that are closed, provided that any such edge has $U_e\le p$ for some $p$ satisfying $L(p)\le \b^{k-2}$.
This is captured by Lemma~\ref{cor:DknMknbd} below, which uses the following definitions.
The number of closed edges in the $k$-th annulus will be denoted by
\begin{equation} \label{Dkn}
Y_{k,R}   = |\{ e \in \gamma_R \cap \EE_k^\ann:\,  t_e > 0 \}|. 
\end{equation}
The maximum value of a closed edge in either the $k$-th annulus or one of the two neighboring annuli is encoded by the random variable
 \begin{equation}
 \label{Mkn}
 M_{k,R} =  \max \{ U_e :\,  e \in \gamma_R \cap (\EE_{k-1}^\ann \cup \EE_{k}^\ann \cup \EE_{k+1}^\ann) \}.
 \end{equation}

\begin{lemma} \label{cor:DknMknbd}
Assume \eqref{critical_assumption}.
There exist constants $C,c > 0$ (not depending on the edge-weight distribution) such that
\eeq{ \label{DknMknbd_eq}
\mathbb P(Y_{k,R} \ge \lambda,\ M_{k,R} \le p_{\floor{\b^{k-2}/\theta}}) \le C\e^{-c \lambda/\theta^2}
\quad \text{for all $R\ge1$, $k\ge2$, $\lambda\ge0$, $\theta\in[1,\b^{k-2}]$}.
}
\end{lemma}

\begin{proof}
Since $Y_{k,R}=0$ for all $k>\ceil{\log_\b R}$, we may assume $k \le\ceil{\log_b R}$.
We will show
\begin{equation} \label{DknMknNkpconn}
\{Y_{k,R} \ge \lambda,\ M_{k,R} \le p_{\floor{\b^{k-2}/\theta}}\} \subseteq  
    \mathbb \{N_{k}(p_{\floor{\b^{k-2}/\theta}}) \ge \lambda\}.
\end{equation}
Before establishing \eqref{DknMknNkpconn}, we explain why it is sufficient to prove the proposition. 

By the upper bound in \eqref{eq: L_p_n}, we have $L(p_{\floor{\b^{k-2}/\theta}}) \le \floor{\b^{k-2}/\theta} \le \b^{k-2}$, so we may use Lemma~\ref{lem : kiss} with $p=p_{\floor{\b^{k-2}/\theta}}$.
Meanwhile, the lower bound in \eqref{eq: L_p_n} tells us there exists a constant  $\delta\in(0,\b^{-2}/2)$ (not depending on $R$, $k$, or $\theta$) such that
\eeq{ \label{kb71x}
L(p_{\floor{\b^{k-2}/\theta}}) 
\ge \delta\floor{\b^{k-2}/\theta} 
\stackref{floor_bound}{\ge} \delta\b^{k-2}/(2\theta) 
\ge \delta^2\b^k/\theta.
}
We now have
\eq{
\mathbb P\Big(N_{k}(p_{\floor{\b^{k-2}/\theta}}) \ge \lambda\Big) 
&\stackrefp{kb71x}{=} \mathbb P\Big(N_{k}(p_{\floor{\b^{k-2}/\theta}}) \ge \lambda\cdot\Big(\f{L(p_{\floor{\b^{k-2}/\theta}})}{\b^k}\Big)^2 \cdot \Big(\f{\b^k}{L(p_{\floor{\b^{k-2}/\theta}} )}\Big)^2 \Big) \\
&\stackref{kb71x}{\le} \mathbb P\Big(N_{k}(p_{\floor{\b^{k-2}/\theta}}) \ge \lambda\cdot\f{\delta^4}{\theta^2} \cdot \Big(\f{\b^k}{L(p_{\floor{\b^{k-2}/\theta}} )}\Big)^2 \Big) \\
&\stackrefpp{ibk3c0}{kb71x}{\le} C\e^{-c\lambda\delta^4/\theta^2}.
}
The proposition follows from this inequality together with \eqref{DknMknNkpconn}, after absorbing $\delta^4$ into $c$. 
The remainder of the proof is to establish \eqref{DknMknNkpconn}.

Suppose $Y_{k,R} \ge \lambda$ and $M_{k,R} \le p_{\floor{\b^{k-2}/\theta}}$. 
For each edge $e \in \gamma_R \cap \EE_k^\ann$ with $t_e > 0$ (i.e.~$U_e>p_\cc$), we verify that the event $\Asf_{k,R}(p,e)$ occurs with $p = p_{\floor{\b^{k-2}/\theta}}$; from this it follows that $N_k(p_{\floor{\b^{k-2}/\theta}}) \ge Y_{k,R} \ge \lambda$.
That is, we verify conditions \ref{Aevent_1}--\ref{Aevent_4} above \eqref{wefu7b}.
See Figure~\ref{fig_kiss_and_beyond} for an illustration.
Since $M_{k,R} \le p$, every $e'\in\gamma_R\cap(\EE_{k-1}^\ann\cup\EE_k^\ann\cup\EE_{k+1}^\ann)$ satisfies $U_{e'} \le p$.
This implies \ref{Aevent_1} and \ref{Aevent_2}, where in the latter the $p$-open arms arise as follows:
\begin{itemize}
    \item Consider the subpath of $\gamma_R$ extending backward from $e$ (i.e.~toward the origin) until first reaching either $\partial B_{\b^{k-2}}$ or $\partial B_{\b^{k+1}}$.
    All edges on this path belong to $\EE_{k-1}^\ann\cup\EE_k^\ann\cup\EE_{k+1}^\ann$ (so the path is $p$-open), and the $\ell^\infty$-distance traveled from $e$ is at least
    \eeq{ \label{gekuc}
    \min\big\{\dist(e,\partial_{B_{\b^{k-2}}}), \dist(e,\partial B_{\b^{k+1}})\big\}
    &\ge (\b^{k-1}-\b^{k-2})\wedge(\b^{k+1}-\b^k) \\
    &= (\b^{k-1}-\b^{k-2}) \ge \b^{k-2}.
    }
    \item Consider the subpath of $\gamma_R$ extending forward from $e$ until first reaching either $\partial B_{2^{k-2}}$, $\partial B_{2^{k+1}}$, or $\partial B_R$.
    All edges on this path belong to $\EE_{k-1}^\ann\cup\EE_k^\ann\cup\EE_{k+1}^\ann$ (so the path is $p$-open), and the $\ell^\infty$-distance traveled from $e$ is at least
    \eq{
    \min\big\{\dist(e,\partial_{B_{\b^{k-2}}}), \dist(e,\partial B_{\b^{k+1}}), \dist(e,\partial B_R)\big\}
    &\stackref{gekuc}{\ge} \min\{\b^{k-2}, \dist(e,\partial B_R)\}.
    }
\end{itemize}
Regarding \ref{Aevent_3}, since $t_e > 0$, Proposition \ref{prop:geod_cons}\ref{itm:dualcir} implies that $e^\star$ belongs to a closed dual circuit that encloses the origin, and this circuit does not intersect $\gamma_R$ at any location other than $e^\star$. 
The two arcs of this circuit emanating from $e^\star$ constitute disjoint $p_\cc$-closed dual arms traveling $\ell^\infty$-distance at least $\dist(e,\vc0) \ge \b^{k-1}>\b^{k-2}$.
Moreover, these closed arms are disjoint from the $p$-open arms discussed above, since the latter are subpaths of $\gamma_R$.
Finally, the requirement in \ref{Aevent_4} that all these arms remain in $B_R$ is obvious for the $p$-open arms (since $\gamma_R$ remains in $B_R$), and follows from Proposition \ref{prop:geod_cons}\ref{itm:dualcir} for the $p_\cc$-closed dual arms.
\end{proof}

Ultimately we are interested in controlling just $Y_{k,R}$ without requiring $M_{k,R}$ to be small as in \eqref{DknMknNkpconn}.
The following corollary accomplishes this when either \eqref{eq : limit_bigger_than_one} or \eqref{eq : limit_equal_zero} holds.

 	  \begin{proposition} \label{prop:Dknbd}
 	  Assume \eqref{critical_assumption} and either \eqref{eq : limit_bigger_than_one} or \eqref{eq : limit_equal_zero}. 
      There exist constants $C,c , s  >0$ such that
      \eeq{ \label{Dknbd_eq}
      \mathbb{P}(Y_{k,R} \geq \lambda ) \leq C \e^{-c \lambda^{s}} \quad \text{for all $R\ge1$, $k\ge1$, $\lambda\ge0$.}  
    }
 	  \end{proposition}

    

\begin{proof}
Note that $Y_{k,R}$ is trivially upper bounded by the cardinality of $\EE_k^\ann$, which is at most $4(2\b^k+1)^2 \le 4(3\b^k)^2 = 36\cdot \b^{2k}$.
Therefore, we may assume without loss of generality that
\eeq{ \label{g7m23b}
\lambda \le 36 \cdot \b^{2k}, \quad \text{or equivalently} \quad
\tfrac{1}{6}\sqrt{\lambda} \le \b^k.
}
We may also assume $\lambda\ge1$ since $Y_{k,R}$ is integer-valued.
In both cases below, we will prove the proposition assuming $k\ge k_0$, for some sufficiently large $k_0$ determined below, not depending on $R$ or $\lambda$.
The inequality \eqref{Dknbd_eq} is then trivially extended to $k<k_0$ by choosing $C \ge \e^{c(36\cdot \b^{2k_0})^{2s}}$.

\medskip \noindent \textbf{Case 1: Assuming \eqref{eq : limit_bigger_than_one}.}
	Let $c_\star>0$ and $q\ge 6$ be the constants $c$ and $q$ from Proposition~\ref{lem:param}.
    By possibly adjusting $c_\star$, we may assume $c_\star\le\b^{-2/q}$.
    Fix any $\eta  \in (0,1/q)$, and then take $k_0$ large enough that
    \eeq{ \label{k0_condition}
k_0 \ge q \qquad \text{and} \qquad b < \ceil{\b^{\eta k}} \le c_\star\b^{k/q} \le \b^{(k-2)/q} \quad \text{for all $k\ge k_0$}.
    }
    These conditions will allow us to apply Proposition~\ref{lem:param} and Lemma~\ref{cor:DknMknbd} in the appropriate parameter ranges, as we assume henceforth that $k\ge k_0$.
    By decomposing the event of interest according to the value of $M_{k,R}$, we have
 \begin{equation} \label{Dknbd}
	\begin{aligned}
	        \mathbb{P} (Y_{k,R} \geq \lambda)  
            &=  
	      \mathbb{P} \Big(Y_{k,R} \geq \lambda ,\  M_{k,R} \leq p_{\floor{\b^{k}/\b^q}} \Big)  \\
	     &\phantom{=} + \sum_{\ell=b}^{\ceil{\b^{\eta k}}-1} \mathbb{P}\Big(Y_{k,R} \geq \lambda ,\ p_{\floor{\b^k/\ell^q}} < M_{k,R} \leq p_{\floor{\b^k/(\ell+1)^q}}  \Big)   \\
	     &\phantom{=} + \mathbb{P}\Big(Y_{k,R} \geq \lambda ,\   M_{k,R} > p_{\floor{\b^k/\ceil{ \b^{\eta k}}^q}}  \Big).
	\end{aligned}
 \end{equation}
	We treat the three terms on the right-hand side separately, with the goal of an upper bound of the form $C\e^{-c\lambda^s}$ for each one.
    The values of $C$ and $c$ may change with each inequality.
    
    For the first term, we simply apply Lemma~\ref{cor:DknMknbd} with $\theta = \b^{q-2}$ (which is allowed since the first inequality in \eqref{k0_condition} guarantees $q-2\le k-2$ and thus $\b^{q-2}\in[1,\b^{k-2}]$):
    \eq{
    \P(Y_{k,R} \ge \lambda,\ M_{k,R} \le p_{\floor{\b^{k}/\b^{q}}})
    = \P(Y_{k,R} \ge \lambda,\ M_{k,R} \le p_{\floor{\b^{k-2}/\b^{q-2}}}) 
    \stackref{DknMknbd_eq}{\le} C\e^{-c \lambda/\b^{2(q-2)}}.
    }
    The right-hand side has the desired form $C\e^{-c\lambda^s}$ once we absorb $1/\b^{2(q-2)}$ into $c$.

    For the last term on the right-hand side of \eqref{Dknbd}, we simply ignore the event $\{Y_{k,R}\ge\lambda\}$ and make the following observation: if $M_{k,R} > p$, then there exists $e\in\gamma_R\cap(\EE_{k-1}^\ann \cup \EE_{k}^\ann \cup \EE_{k+1}^\ann)$ with $U_e > p$ and thus $t_e = F^{-1}(U_e)\ge F^{-1}(p)$.
    Therefore, a union bound justifies the first inequality below.
    The second inequality uses the fact that $r\mapsto p_r$ is weakly decreasing, the third inequality uses Proposition~\ref{lem:param}, and the final inequality is just adjusting constants:
    \eeq{ \label{kb7inn}
    \P\Big(M_{k,R}>p_{\floor{\b^k/\theta^q}}\Big)
    &\stackrefp{param_eq}{\le}
    \P\Big(T_{k-1,R} \ge F^{-1}(p_{\floor{\b^k/\theta^q})}\Big) \\
    &\hphantom{\stackrefp{param_eq}{\le}}+ \P\Big(T_{k,R} \ge F^{-1}(p_{\floor{\b^k/\theta^q}})\Big) \\
    &\hphantom{\stackrefp{param_eq}{\le}}+ \P\Big(T_{k+1,R} \ge F^{-1}(p_{\floor{\b^k/\theta^q}})\Big) \\
    &\stackrefp{param_eq}{\le} 
    \P\Big(T_{k-1,R} \ge F^{-1}(p_{\floor{\b^{k-1}/(\theta\b^{-1/q})^q})}\Big) \\
    &\hphantom{\stackrefp{param_eq}{\le}}+ \P\Big(T_{k,R} \ge F^{-1}(p_{\floor{\b^k/\theta^q}})\Big) \\
    &\hphantom{\stackrefp{param_eq}{\le}}+ \P\Big(T_{k+1,R} \ge F^{-1}(p_{\floor{\b^{k+1}/\theta^q}})\Big) \\
    &\stackref{param_eq}{\le} C\big(\e^{-c\theta b^{-1/q}} + \e^{-c_\star\theta} + \e^{-c_\star\theta}\big) \quad
    \text{for all $\theta\in[\tfrac{3}{2}\b^{1/q},c_\star\b^{k/q}]$}\\
    &\stackrefp{param_eq}{\le} C\e^{-c\theta}.
    }
Note that $\frac{3}{2}b^{1/q}< b$ since $b\ge 2$ and $q\ge6$, so the interval $[\tfrac{3}{2}\b^{1/q},c_\star\b^{k/q}]$ is nonempty by \eqref{k0_condition} and contains $[b,b^{\eta k}]$. 
In the special case $\theta = \b^{\eta k}$, we obtain
\eq{
\P\Big(M_{k,R}>p_{\floor{\b^k/\ceil{ \b^{\eta k}}^q}}\Big)
\le \P\Big(M_{k,R}>p_{\floor{\b^k/(\b^{\eta k})^q}}\Big)
\stackref{kb7inn}{\le} C\e^{-c\b^{\eta k}}
\stackref{g7m23b}{\le} C\e^{-c(\sqrt{\lambda}/6)^{\eta}}.
}
This again has the desired form $C\e^{-c\lambda^s}$ once we absorb $1/6^\eta$ into $c$.

Our final task is to control the sum on the right-hand side of \eqref{Dknbd}. 
By H\"older's inequality, 
\eeq{\label{Dkbd0}
&\mathbb{P}\Big(Y_{k,R} \geq \lambda ,\ p_{\floor{\b^k/\ell^q}} < M_{k,R}  \leq p_{\floor{\b^k/(\ell+1)^q}}  \Big)  \\
&\leq  \mathbb{P}\Big(Y_{k,R} \geq \lambda ,\ M_{k,R} \leq p_{\floor{\b^k/(\ell+1)^q}}  \Big) ^{1/2}  \mathbb{P}\Big(M_{k,R}>p_{\floor{\b^k/\ell^q}}  \Big)^{1/2}.
}
We wish to upper bound the right-hand side for an arbitrary $\ell\in\{\b,\dots,\ceil{\b^{\eta k}}-1\}$.
For the first factor, we trivially adjust the power of $\b$ and then apply Lemma~\ref{cor:DknMknbd} with $\theta = (\ell+1)^q$ (which is allowed since $(\ell+1)^q \in [(\b+1)^q, \ceil{\b^{\eta k}}^q]\subseteq[1,\b^{k-2}]$ by \eqref{k0_condition}):
\begin{subequations} \label{Dkbd_both}
\eeq{ \label{Dkbd1}
\mathbb{P}\Big(Y_{k,R} \geq \lambda ,\ M_{k,R} \leq p_{\floor{\b^k/(\ell+1)^q}}  \Big)
&\stackrefp{DknMknbd_eq}{\leq} \mathbb{P}\Big(Y_{k,R} \geq \lambda ,\ M_{k,R} \leq p_{\floor{\b^{k-2}/(\ell+1)^q}}  \Big) \\
&\stackref{DknMknbd_eq}{\leq} C\e^{-c \lambda/(\ell+1)^{2q}}.
}
For the second factor on the right-hand side of \eqref{Dkbd0}, we apply \eqref{kb7inn} with $\theta = \ell$ (which is allowed since $\ell\in[\b,\ceil{\b^{\eta k}}-1]\subseteq[\tfrac{3}{2}\b^{1/q},c_\star\b^{k/q}]$ as discussed below \eqref{kb7inn}):
    \begin{equation} \label{Dkbd2}
    \mathbb{P}\Big(  M_{k,R} >   p_{\floor{\b^k/\ell^q}}    \Big) \le \mathbb P\Big( T_{k,R} \ge F^{-1}(p_{\floor{\b^k/\ell^q}} ) \Big) \leq C \e^{-c \ell }.
    \end{equation}
\end{subequations}
Now insert \eqref{Dkbd_both} into \eqref{Dkbd0}, absorb $1/2$ into $c$ to account for the square roots in \eqref{Dkbd0}, and sum over $\ell$:
\eeq{ \label{jb2ll}
\sum_{\ell=\b}^{\ceil{\b^{\eta k}}-1}\mathbb{P}\Big(Y_{k,R} \geq \lambda ,\ p_{\floor{\b^k/\ell^q}} < M_{k,R}  \leq p_{\floor{\b^k/(\ell+1)^q}}  \Big)
\le C\sum_{\ell=2}^{\infty}\e^{-c\lambda/(\ell+1)^{2q}}\e^{-c\ell}.
}
To obtain the desired form $C\e^{-c\lambda^s}$, we simply observe that for any $A,a>0$,
\eeq{ \label{dkb7}
\sum_{\ell=1}^\infty \e^{-c\lambda^{a}/\ell^{A}}\e^{-c\ell}
&= \sum_{\ell=1}^{\floor{\lambda^{a/(A+1)}}}\e^{-c\lambda^{a}/\ell^A}\e^{-c\ell}
+ \sum_{\ell=\floor{\lambda^{a/(A+1)}}+1}^\infty\e^{-c\lambda^{a}/\ell^A}\e^{-c\ell} \\
&\le \sum_{\ell=1}^{\floor{\lambda^{a/(A+1)}}}\e^{-c\lambda^{a/(A+1)}}\e^{-c\ell}
+ \sum_{\ell=\floor{\lambda^{a/(A+1)}}+1}^\infty1\cdot \e^{-c\ell} \\
&\le C\e^{-c\lambda^{a/(A+1)}} + C\e^{-c\lambda^{a/(A+1)}},
}
where the middle inequality uses $\ell\le\lambda^{a/(A+1)}$ for the first sum.

\medskip \noindent \textbf{Case 2: Assuming \eqref{eq : limit_equal_zero}.}
Let $c_\star>0$ and $q\ge3$ be the constants $c$ and $q$ from Proposition~\ref{lem:Tkn0assumption}, and let $C_\star$ be the constant from the upper bound in \eqref{pn-pcbd}.
Choose $k_0$ large enough so that
\eeq{ \label{jd4vx}
36\b^{2k} \le \min\{c_\star\b^{qk},\tfrac{1}{C_\star\e}b^{8k/3}\} \quad \text{for all $k\ge k_0$}.
}
Write $Y_{k,R} = Y_{k,R}^{(1)}+Y_{k,R}^{(2)}$, where
\eq{
 Y_{k,R}^{(1)} &=  |\{ e \in \gamma_R \cap \EE_k^\ann  :\, U_e \in (p_\cc,p_{\b^{4k}})\}|, \\
  Y_{k,R}^{(2)} &= |\{ e \in \gamma_R \cap \EE_k^\ann  :\, U_e \ge p_{\b^{4k}}\}|.
}
A simple union bound gives
\eeq{ \label{eq:Dkn12bd}
 \P(Y_{k,R} \geq \lambda ) \leq  \P(Y_{k,R}^{(1)} \geq \lambda/2 ) + \P(Y_{k,R}^{(2)} \geq \lambda/2 ). 
}
We treat the two terms on the right-hand side separately, with the goal of an upper bound of the form $C\e^{-c\lambda^s}$ for each one.
The values of $C$ and $c$ may change with each inequality.

For the first term, we trivially bound $Y_{k,R}^{(1)}$ by the total number of edges $e \in \EE_k^\ann$ with $U_e \in (p_\cc,p_{\b^{4k}})$, which has the Binomial($N,\rho$) distribution with 
\eq{
N = |\EE_k^\ann| \le 36\b^{2k} 
\quad \text{and} \quad
\rho = p_{\b^{4k}} - p_\cc  \stackref{pn-pcbd}{\le} C_\star \b^{-8k/3} \stackref{jd4vx}{\le} (\e N)^{-1}.
}
This comparison implies
\eq{
\P(Y_{k,R}^{(1)} \geq \lambda) 
&\le \sum_{\ell = \ceil{\lambda} }^{N} { N \choose \ell } \rho^\ell (1-\rho)^{N-\ell}  \\
&\leq  \sum_{\ell= \ceil{\lambda} }^{N} \frac{N^\ell}{\ell!} \rho^\ell
\le \frac{(N\rho)^{\ceil{\lambda}}}{\ceil{\lambda}!}\sum_{\ell=0}^{N-\ceil{\lambda}} \frac{(N\rho)^\ell}{\ell!}
\le \frac{\e^{-\lambda}}{\ceil{\lambda}!}\cdot \e^{\e^{-1}}.
}
This has the desired form $C\e^{-c\lambda^s}$ after we drop the factorial.
 
 For the second term on the right-hand side of \eqref{eq:Dkn12bd}, we begin by noting 
\[
\liminf_{n \to \infty} \f{F^{-1}(p_{\b^{4n}})}{F^{-1}(p_{\b^{n}})} 
= \liminf_{n \to \infty} \f{F^{-1}(p_{\b^{4n}})}{F^{-1}(p_{\b^{2n}})}\cdot \f{F^{-1}(p_{\b^{2n}})}{F^{-1}(p_{\b^{n}})}
\ge \Big(\liminf_{n \to \infty} \f{F^{-1}(p_{\b^{2n}})}{F^{-1}(p_{\b^{n}})}\Big)^2
\stackref{eq : limit_equal_zero}{>} 0.
\]
Therefore, there exists $\delta\in(0,1]$ such that
\begin{equation} \label{Cbetadef}
F^{-1}(p_{\b^{4k}}) \ge \delta F^{-1}(p_{\b^{k}}) \quad \text{for all $k\ge0$.}
\end{equation}
Next observe that every edge $e$ with $U_e\ge p_{\b^{4k}}$ has $t_e = F^{-1}(U_e) \ge F^{-1}(p_{\b^{4k}})$, and thus $T_{k,R} \ge Y_{k,R}^{(2)}F^{-1}(p_{\b^{4k}})$.
This observation implies
\eq{
 \P(Y_{k,R}^{(2)} \geq \lambda) 
 \le \P(T_{k,R} \geq \lambda F^{-1}(p_{\b^{4k} } ) )
 \stackref{Cbetadef}{\leq}  \P(T_{k,R} \geq \delta \lambda F^{-1}(p_{\b^{ k} } )   ).
 }
Now apply Proposition~\ref{lem:Tkn0assumption} to the final probability.
This application is allowed since
\eq{
\delta \lambda
\le \lambda 
\stackref{g7m23b}{\le} 36\b^{2k} 
\stackref{jd4vx}{\le} c_\star\b^{qk},
}
and it yields
\eq{
 \P(Y_{k,R}^{(2)} \geq \lambda)
 \stackrefpp{Tkn0assumption_eq}{Cbetadef}{\leq} C \e^{-c_\star(\delta\lambda)^{1/q}}.
}
This has the desired form $C\e^{-c\lambda^s}$ with $c = c_\star\delta^{1/q}$.
\end{proof}

\subsection{Spacings within and between circuits} \label{sec:spacings}
Section~\ref{sec_ceba} achieved control on the number of closed edges found on the geodesic $\gamma_R$ between $\partial B_{\b^{k-1}}$ and $\partial B_{\b^k}$.
Ultimately we wish to do the same when $(\partial B_{\b^k})_{k\ge1}$ are replaced with the circuits $(\incir_j)_{j\ge1}$ from Proposition~\ref{prop:geod_cons}, since these circuits enable the long arm events needed to yield Theorem~\ref{thm:general_thm}.
This task will be completed in Section~\ref{sec_cebc} but is complicated by the fact that the circuits are random.
The present subsection gives the necessary control over this randomness.
The results here are general statements about critical bond percolation on $\Z^2$.
Our estimates are similar to those in the proof of \cite[Thm.~3.3]{chayes_chayes_durrett86}; nonetheless we provide a complete argument since the precise definition of the circuits we use is quite delicate.

For any primal circuit $\CC$, define its \textit{inner reach} and \textit{outer reach} as follows:
\begin{subequations}
\eeqs{
\label{def_inner_reach}
\mathfrak{i}(\mathcal{C}) &= \min \{ k \ge 1 :\, \EE_k^\ann \cap \mathcal{C} \not= \varnothing \}, \\
\label{def_outer_reach}
\mathfrak{o}(\mathcal{C}) &= \max \{ k \ge 1:\, \EE_k^\ann \cap \mathcal{C} \not= \varnothing \}.
}
\end{subequations}
We will need the following deterministic facts:


\begin{lemma} \label{lemma_reacon}
For any edge $e\in E(\Z^2)$, let $k_e$ be the unique integer such that $e\in \EE_{k_e}^\ann$.
Let $\CC$ be a primal circuit.
The following statements hold:
\begin{enumerate}[label=\textup{(\alph*)}]

\item \label{lemma_reacon_a}
If the midpoint of $e$ is not in $\ext(\CC)$, then $k_e\le\mathfrak{o}(\CC)$.

\item \label{lemma_reacon_b}
If the midpoint of $e$ is not in $\intr(\CC)$, and $\vc0\in\intr(\CC)$, then $k_e\ge\mathfrak{i}(\CC)$.

\end{enumerate}
Let $\CC_1$ and $\CC_2$ be primal circuits such that $\CC_2$ surrounds $\CC_1$.
The following statements hold:
\begin{enumerate}[label=\textup{(\alph*)}] \setcounter{enumi}{2}

\item \label{lemma_reacon_c}
$\mathfrak{o}(\CC_1)\le\mathfrak{o}(\CC_2)$.

\item \label{lemma_reacon_d}
If $\vc 0\in\intr(\CC_1)$, then $\mathfrak{i}(\CC_1)\le\mathfrak{i}(\CC_2)$.

\end{enumerate}
\end{lemma}

\begin{proof}
We call $k_e$ the ``annuli index'' of $e$.
Note that for any $x=(a,b)\in\Z^2$ with $a,b\geq0$, the annuli index of the west and south edges away from $x$ is at most that of the east and north edges:
\eeq{ \label{aw53f}
k_{\{(a-1,b),(a,b)\}} = k_{\{(a,b-1),(a,b)\}}
\le \min\big\{k_{\{(a,b),(a+1,b)\}},k_{\{(a,b),(a,b+1)\}}\big\}.
}
For parts \ref{lemma_reacon_a} and \ref{lemma_reacon_b}, we assume without loss of generality that both coordinates of both endpoints of $e$ are nonnegative (otherwise rotate everything by 90, 180, or 270 degrees).
Denote the midpoint of $e$ by $y$.

\medskip Part \ref{lemma_reacon_a}:
If $y$ lies on $\CC$, then $k_e\le\mathfrak{o}(\CC)$ by definition of outer reach in \eqref{def_outer_reach}.
So assume $y\in\intr(\CC)$.
Take any infinite up-right path $\gamma=(e_1,e_2,\dots)$ whose first edge is $e_1=e$.
Since $y\in\intr(\CC)$, the path $\gamma$ must intersect $\CC$ at some vertex $z$ after passing through $y$.
Let $i$ be the unique index such that $z\in e_i\cap e_{i+1}$.
We assume $z$ is the \textit{first} intersection of $\gamma$ with $\CC$ after passing through $y$, so that $e_i$ does not belong to $\CC$.

If we write $z = (a,b)$, then $e_i$ is either $\{(a-1,b),(a,b)\}$ or $\{(a,b-1),(a,b)\}$, since $\gamma$ is an up-right path.
On the other hand, there are two edges $e',e''\in \CC$ that contain $z$; since neither one is $e_i$, at least one of these two is equal to either $\{(a,b),(a+1,b)\}$ or $\{(a,b),(a,b+1)\}$.
This observation justifies the second inequality below:
\eq{
\mathfrak{o}(\CC)
\stackref{def_outer_reach}{\ge}
\max\{k_{e'},k_{e''}\}
\ge \min\big\{k_{\{(a,b),(a+1,b)\}},k_{\{(a,b),(a,b+1)\}}\big\}
\stackref{aw53f}{\ge} k_{e_i}.
}
Since $k_{e_i}\ge k_{e_1}=k_e$ (again because of \eqref{aw53f} and the fact that $\gamma$ is an up-right path confined to the northeast quadrant), we are done.

\medskip Part \ref{lemma_reacon_b}:
This argument is ``dual'' to that of part \ref{lemma_reacon_a}.
If $y$ lies on $\CC$, then $k_e\ge\mathfrak{i}(\CC)$ by definition of inner reach in \eqref{def_inner_reach}.
So assume $y\in\ext(\CC)$.
Take an up-right path $\gamma=(e_1,\dots,e_n)$ that starts at the origin and whose last edge is $e_n = e$.
Since $\vc 0\in\intr(\CC)$ while $y\in\ext(\CC)$, the path $\gamma$ must intersect $\CC$ at some vertex $x$ before reaching $y$.
Let $i\in\{2,\dots,n\}$ be the unique index such that $x = e_{i-1}\cap e_{i}$.
We assume $x$ is the \textit{last} intersection of $\gamma$ with $\CC$ before reaching $y$, so that $e_i$ does not belong to $\CC$.

If we write $x = (a,b)$, then $e_i$ is either $\{(a,b),(a+1,b)\}$ or $\{(a,b),(a,b+1)\}$ since $\gamma$ is an up-right path.
On the other hand, there are two edges $e',e''\in \CC$ that contain $x$; since neither one is $e_i$, at least one of these two is equal to either $\{(a-1,b),(a,b)\}$ or $\{(a,b-1),(a,b)\}$.
This observation justifies the second inequality below:
\eq{
\mathfrak{i}(\CC)
\stackref{def_inner_reach}{\le}
\min\{k_{e'},k_{e''}\}
\le \max\big\{k_{\{(a-1,b),(a,b)\}},k_{\{(a,b-1),(a,b)\}}\big\}
\stackref{aw53f}{\le} k_{e_i}.
}
Since $k_{e_i}\le k_{e_n} = k_e$ (again because of \eqref{aw53f} and the fact that $\gamma$ is an up-right path confined to the northeast quadrant),  we are done.

\medskip Part \ref{lemma_reacon_c}:
Fix $e\in\CC_1$ such that $k_e = \mathfrak{o}(\CC_1)$.
By Lemma~\ref{surround_lemma}\ref{surround_lemma_c}, the midpoint of $e$ does not belong to $\ext(\CC_2)$.
Now apply part~\ref{lemma_reacon_a}.

\medskip Part \ref{lemma_reacon_d}:
Fix $e\in\CC_2$ such that $k_e = \mathfrak{i}(\CC_2)$.
By Lemma~\ref{surround_lemma}\ref{surround_lemma_d}, the midpoint of $e$ does not belong to $\intr(\CC_1)$.
Now apply part~\ref{lemma_reacon_b}. 
\end{proof}

Let $\incir_{1},\incir_2,\ldots$ be the sequence of successively innermost edge-disjoint open circuits enclosing the origin. 
On the full-probability event $\Omega_\infty$ from Definition~\ref{def:omegainf}, this sequence is infinite.
The rigorous construction of this sequence proceeds just as in Section~\ref{sec:circons}, except that the circuits only need to enclose $\AA = \{\vc 0\}$, i.e.~one takes $\BB = \varnothing$ in \eqref{Si}.
We also take the convention that $\incir_0 = \{\vc 0\}$ is the ``empty circuit at $\vc 0$'' with $\intr(\incir_0) = \{\vc0\}$, $\ext(\incir_0) = \R^2\setminus\{\vc0\}$, and $\mathfrak{i}(\incir_0)=\mathfrak{o}(\incir_0)=0$.

Let $\EE_j'$ denote the set of edges whose midpoint does not belong to $\ext(\incir_j)$.
To isolate edges between consecutive circuits, we define
\eeq{ \label{def_edges_btw_cir}
\EE_1^\crc = \EE_1' \qquad \text{and} \qquad
\EE_j^\crc = \EE_j' \setminus \EE_{j-1}' \quad \text{for $j\ge2$.}
}
So $\EE_1^\crc,\EE_2^\crc,\dots,\EE_j^\crc$ from a partition of all possible edges used by a path $\gamma$ that starts at $\vc 0$ and does not enter the exterior of $\incir_j$.
These sets are related to the sequence $(\EE_k^\ann)_{k\ge1}$ as follows.

\begin{proposition} \label{prop_crc_to_ann}
For any $j\ge1$, we have $\EE_j^\crc \subseteq \bigcup_{k=\mathfrak{i}(\incir_{j-1})\vee1}^{\mathfrak{o}(\incir_j)}\EE_k^\ann$.
\end{proposition}

\begin{proof}
Take any $e\in\EE_j^\crc=\EE_j'\setminus\EE'_{j-1}$. 
In the notation of Lemma~\ref{lemma_reacon}, we wish to show $\mathfrak{i}(\incir_{j-1}) \le k_e \le \mathfrak{o}(\incir_{j})$.
Since $e\in\EE_j'$, the midpoint of $e$ does not belong to $\ext(\incir_j)$, so
Lemma~\ref{lemma_reacon}\ref{lemma_reacon_a} implies $k_e \le \mathfrak{o}(\incir_{j})$.
Since $e\notin\EE_{j-1}'$, the midpoint of $e$ belongs to $\ext(\incir_{j-1})$, so 
Lemma~\ref{lemma_reacon}\ref{lemma_reacon_b} implies $k_e\ge \mathfrak{i}(\incir_{j-1})$.
If $j=1$, then this last conclusion can be replaced by the trivial inequality $k_e \ge 1$.
\end{proof}

Proposition~\ref{prop_crc_to_ann} is most useful when $\mathfrak{o}(\incir_j) - \mathfrak{i}(\incir_{j-1})$ is small.
This motivates the next result.

\begin{proposition} \label{lem : circuitBound}
 There exist constants \(C,c > 0\) such that
\eeq{ \label{reb8c}
  \P(\mathfrak{o}(\incir_{j}) - \mathfrak{i}(\incir_{j-1}) \geq \ell) \leq C j \e^{-c\ell} \quad \text{for all $j\ge1$, $\ell\ge0$}.
}
\end{proposition}

The proof requires that we control two quantities: the radial distance traveled by a given circuit, and the radial distance between successive circuits.
These two quantities appear in the two inequalities of the following lemma:

\begin{lemma} \label{lem_reaches}
There exist constants \(C,c > 0\) such that
\begin{subequations} \label{gx_new}
\eeq{ \label{gx6v5_new}
  \P(\mathfrak{o}(\incir_{j}) - \mathfrak{i}(\incir_{j}) \geq \ell) \leq C j \e^{-c\ell}
  \quad \text{for all $j\ge1$, $\ell\ge0$,}
}
and
\eeq{ \label{gx6v6_new}
\P(\mathfrak{i}(\incir_{j}) - \mathfrak{o}(\incir_{j-1}) \geq \ell) \leq C j \e^{-c\ell} \quad \text{for all $j
\ge1$, $\ell\ge0$.}
}
\end{subequations}
\end{lemma}

Before proving Lemma~\ref{lem_reaches}, we use it to establish Proposition~\ref{lem : circuitBound}.

\begin{proof}[Proof of Proposition~\ref{lem : circuitBound}]
The result is trivial if $\ell=0$, so assume $\ell\ge1$.
By a union bound, $\P(\mathfrak{o}(\incir_{j}) - \mathfrak{i}(\incir_{j-1}) \geq \ell)$ is at most
\eq{
        \P\big(\mathfrak{o}(\incir_{j}) - \mathfrak{i}(\incir_{j}) \geq \ceil{\ell/3} \big) +   \P\big(\mathfrak{i}(\incir_{j}) - \mathfrak{o}(\incir_{j-1}) \geq \ceil{\ell/3} \big) +  \P\big(\mathfrak{o}(\incir_{j-1}) - \mathfrak{i}(\incir_{j-1}) \geq \ceil{\ell/3} \big),
}
where the third probability is $0$ if $j=1$.
Now apply \eqref{gx6v5_new} to the first and last terms, and \eqref{gx6v6_new} to the middle term.
\end{proof}

The strategy for Lemma~\ref{lem_reaches} is to leverage the fact that each of the events in \eqref{gx_new} induces a 1-arm event across an annulus.
There are two complications.
First, the location of the annulus is random because the locations of the circuits $(\incir_j)_{j\ge1}$ is random.
Second, the arm relevant for \eqref{gx6v6_new} is closed dual rather than open primal.
These two complications are dealt with in the following two lemmas.

\begin{lemma} \label{lem:stochast_bd}
There exists $\rho\in(0,1)$ such that the following statement holds.
For every $j \ge 1$, the random variable $\mathfrak o(\incir_j)$ (and therefore also $\mathfrak i(\incir_j)$)  is stochastically dominated by the sum of $j$ independent Geometric($\rho$) variables (supported on $\{1,2,\ldots\}$). 
In particular, for any $m\ge1$, there exists a constant $C = C(\rho,m)$ such that for every $j\ge1$, 
\eeq{ \label{jd7bx}
\sum_{k=1}^\infty\P(\mathfrak i(\incir_j)=k)^{1/m} \le Cj
\quad \text{and} \quad
\sum_{k=1}^\infty\P(\mathfrak o(\incir_j)=k)^{1/m} \le Cj.
}
\end{lemma}

\begin{proof}
 Consider the following event:
 \[
 \Omega_k = \{\exists \text{ an open circuit $\CC$ enclosing the origin such that $\CC\subseteq\EE_k^\ann$}\}.
 \]
 Since $(\EE_k^\ann)_{k\ge1}$ are disjoint, the events $(\Omega_k)_{k\ge1}$ are independent.
 In addition, by the RSW theorem \eqref{RSW} together with the FKG inequality, there exists $\rho\in(0,1)$ such that $\P(\Omega_k)\ge \rho$ for all $k\ge1$.
 By the second Borel--Cantelli lemma, it follows that with probability one, the event $\Omega_k$ occurs for infinitely many $k$.
 Let $\mathfrak{K}_1<\mathfrak{K}_2<\cdots$ be the (random) list of these values of $k$, and set $\mathfrak{K}_0=0$.
 Since the events $(\Omega_k)_{k\ge1}$ are independent, the increments $(\mathfrak{K}_{j}-\mathfrak{K}_{j-1})_{j\ge1}$ are independent and satisfy $\P(\mathfrak{K}_{j}-\mathfrak{K}_{j-1} \ge n) \le (1-\rho)^{n-1}$ for all $n\ge1$.
 Therefore,
 \eeq{ \label{jx7bd}
 \text{$\mathfrak{K}_j$ is stochastically dominated by the sum of $j$ independent Geometric($\rho$) variables.}
 }
 
 Now consider a sequence of open circuits $\CC_1,\CC_2,\dots$, where $\CC_j$ is any open circuit fitting the description in $\Omega_{\mathfrak{K}_j}$ (chosen in some measurable way from the finitely many possible circuits, given the value of $\mathfrak{K}_j$).
By construction, $\CC_{j+1}$ surrounds $\CC_j$ and is edge-disjoint from $\CC_j$ (in fact, they are also vertex-disjoint). 
Since $(\incir_j)_{j\ge1}$ was defined as the sequence of successively \textit{innermost} edge-disjoint open circuits enclosing the origin, it follows by induction that each $\incir_j$ is surrounded by $\CC_j$. 
So Lemma~\ref{lemma_reacon}\ref{lemma_reacon_c} implies $\mathfrak o(\incir_j)\le \mathfrak o(\CC_j) = \mathfrak{K}_j$.
The first statement of the lemma now follows from \eqref{jx7bd}.

For \eqref{jd7bx}, let $\mu_m$ be the $m$-th moment of a Geometric($\rho$) random variable.
We claim
\eeq{ \label{eub83}
\E[\mathfrak i(\incir_j)^m] \le \mathbb \E[\mathfrak o(\incir_j)^m] \le j^m\mu_m.
}
Indeed, the first inequality is trivial, and the second follows from the stochastic domination and the following deterministic inequality for nonnegative numbers $x_1,\ldots,x_j$:
\[
(x_1+\cdots+x_j)^m = j^m\Big(\frac{x_1+\cdots+x_j}{j}\Big)^m \le j^{m}\Big(\frac{x_1^m+\cdots+x_j^m}{j}\Big).
\]
The desired statement \eqref{jd7bx} is trivial for $m=1$.
For $m>1$, we use H\"older's inequality followed by the moment bound:
\eq{
\sum_{k=1}^\infty\P(\mathfrak i(\incir_j)=k)^{1/m}
&\le \bigg(\sum_{k=1}^\infty k^m\cdot\P(\mathfrak i(\incir_j)=k)\bigg)^{1/m}\
\bigg(\sum_{k=1}^\infty k^{-m/(m-1)}\bigg)^{(m-1)/m} \\
&= \E[\mathfrak{i}(\incir_j)^m]^{1/m}\cdot \bigg(\sum_{k=1}^\infty k^{-m/(m-1)}\bigg)^{(m-1)/m}
\stackref{eub83}{\le} Cj.
}
The second inequality in \eqref{eub83} is proved in exactly the same way.
\end{proof}

We are now ready for the proof of  Lemma~\ref{lem_reaches} that had been postponed.
Let $\{\mathcal A \leftrightarrow \mathcal B\}$ denote the event that there exists an open (primal) path starting at $\mathcal{A}$ and ending at $\mathcal{B}$.
Let $\{\mathcal A \overset{\mathsf{cd}} {\leftrightarrow} \mathcal B\}$ denote the event that there exists a closed (dual) path starting at $\mathcal A$ and ending at $\mathcal B$, in the sense of Definition~\ref{open_closed_paths_def}.

\begin{proof}[Proof of Lemma~\ref{lem_reaches}]
If $C$ is large enough, then the result is trivial for $\ell=0$ and $\ell=1$, so assume $\ell\ge2$.
First we prove \eqref{gx6v5_new}.
If $\mathfrak i(\incir_{j}) \le k$ and $\mathfrak o(\incir_{j}) \ge k+\ell$, then there exists an open path from $\partial B_{\b^{k}}$ to $\partial B_{\b^{k + \ell-1}}$ formed by following a portion of $\incir_{j}$.
That is,
\eeq{ \label{2hfv6x}
\{\mathfrak i(\incir_{j}) \le k,\ \mathfrak o(\incir_{j}) \ge k+\ell\} \subseteq \{\partial B_{\b^{k}} \leftrightarrow \partial B_{\b^{k+\ell-1}}  \}
\quad \text{for any $j,k,\ell\ge1$.}
}
Moreover, Proposition \ref{prop:1-arm} gives an upper bound on 1-arm probabilities:
\eeq{ \label{ewb6x}
\P(\partial B_{\b^{k}} \leftrightarrow \partial B_{\b^{k+\ell-1}})
= \pi_1(\b^{k},\b^{k+\ell-1}) \stackref{ncb3}{\leq} C(\b^{\ell-1})^{-c}.
}
Now decompose the event $\{\mathfrak{o}(\incir_{j}) - \mathfrak{i}(\incir_{j}) \geq \ell\}$ based on the value of $\mathfrak{i}(\incir_{j})$:
\eq{
\P(\mathfrak{o}(\incir_{j}) - \mathfrak{i}(\incir_{j}) \geq \ell) 
&\stackrefp{jd7bx,ewb6x}{=}  \sum_{k=1}^{\infty } \P(\mathfrak{i}(\incir_{j}) = k,\ \mathfrak{o}(\incir_{j}) - \mathfrak{i}(\incir_{j}) \geq \ell ) \\
&\stackrefpp{2hfv6x}{jd7bx,ewb6x}{\leq} \sum_{k=1}^{\infty } \P(\mathfrak{i}(\incir_{j})=  k,\   \partial B_{\b^{k}} \leftrightarrow \partial B_{\b^{k+\ell-1}}) \\
&\stackrel{\parbox{\widthof{\footnotesize\eqref{jd7bx},\eqref{ewb6x}}}{\centering\footnotesize (H\"older)}}{\le} \sum_{k=1}^{\infty } \P(\mathfrak{i}(\incir_{j}) =  k)^{1/2}\P(\partial B_{\b^{k}} \leftrightarrow \partial B_{\b^{k+\ell-1}})^{1/2}  \\
&\stackref{jd7bx,ewb6x}{\le} Cj(\b^{\ell-1})^{-c}.
}
This proves \eqref{gx6v5_new} after we adjust $C$ and $c$ to obtain the desired form $Cj\e^{-c\ell}$. 

Next we prove \eqref{gx6v6_new}.
By Proposition \ref{prop:geod_cons}\ref{itm:zetaopen}, there exists a closed dual path $\dualp'$ starting at $\incir_{j-1}$ and ending at $\incir_{j}$.
Therefore, if $\mathfrak o(\incir_{j-1}) \le k$ and $\mathfrak i(\incir_{j}) \ge k+\ell$, then there exists a closed dual path from $\partial B_{\b^{k}}$ to $\partial B_{\b^{k + \ell-1}}$ formed by a portion of $\dualp'$.
That is,
\eeq{ \label{3hfv6x}
\{\mathfrak o(\incir_{j-1}) \le k,\ \mathfrak i(\incir_{j}) \ge k+\ell\} \subseteq \{\partial B_{\b^{k}} \overset{\mathsf{cd}}{\leftrightarrow} \partial B_{\b^{k+\ell-1}}  \}
\quad \text{for any $j,k,\ell\ge1$.}
}
Moreover, by putting together Lemma~\ref{lem_closed_like_open} and  Proposition \ref{prop:1-arm}, we have
\eeq{  \label{q9bp3}
\P(\partial B_{\b^{k}} \overset{\mathsf{cd}}{\leftrightarrow} \partial B_{\b^{k+\ell-1}})
= \pi_1''(\b^k,\b^{k+\ell-1})
\stackref{closed_like_open,ncb3}{\le} C(\b^{\ell-1})^{-c}.
}
Now decompose the event $\{\mathfrak{i}(\incir_{j}) - \mathfrak{o}(\incir_{j-1}) \geq \ell\}$ based on the value of $\mathfrak{i}(\incir_{j})$:
\eq{
\P(\mathfrak{i}(\incir_{j}) - \mathfrak{o}(\incir_{j-1}) \geq \ell) 
&\stackrefp{jd7bx,q9bp3}{=}  \sum_{k=1}^{\infty } \P(\mathfrak{i}(\incir_{j}) = k,\ \mathfrak{i}(\incir_{j}) - \mathfrak{o}(\incir_{j-1}) \geq \ell ) \\
&\stackrefpp{3hfv6x}{jd7bx,ewb6x}{\leq} \sum_{k=1}^{\infty } \P(\mathfrak{i}(\incir_{j})=  k,\   \partial B_{\b^{k}} \overset{\mathsf{cd}}{\leftrightarrow} \partial B_{\b^{k+\ell-1}}) \\
&\stackrel{\parbox{\widthof{\footnotesize\eqref{jd7bx},\eqref{q9bp3}}}{\centering\footnotesize (H\"older)}}{\le} \sum_{k=1}^{\infty } \P(\mathfrak{i}(\incir_{j}) =  k)^{1/2}\P(\partial B_{\b^{k}} \overset{\mathsf{cd}}{\leftrightarrow} \partial B_{\b^{k+\ell-1}})^{1/2}  \\
&\stackref{jd7bx,q9bp3}{\le} Cj(\b^{\ell-1})^{-c}.
}
This proves \eqref{gx6v6_new} after we adjust $C$ and $c$ to obtain the desired form $Cj\e^{-c\ell}$. 
\end{proof}

We conclude this subsection with one more proposition that is needed in Section~\ref{sec_cebc}.
Namely, we wish to know that there are not too many circuits among $(\incir_j)_{j\ge1}$ that remain in the interior of $B_R$.
Let $L_R$ denote the number of such circuits:
\begin{equation} \label{YR}
L_R = \sup\{j \ge 0:\, \text{$e\subseteq B_{R-1}$ for every $e\in\incir_j$}\}.
\end{equation}
In other words, $L_R$ is the value of $L$ in Proposition \ref{prop:geod_cons} applied with $\mathcal A = \{\vc 0\}$ and $\mathcal B = \partial B_R$.
Note that $L_R$ takes the value $0$ if there are no open circuits that both enclose the origin and remain in $B_{R-1}$.

\begin{proposition} \label{prop_LR_tail}
There exists a constant $c>0$ such that
\eeq{ \label{LR_tail}
\P(L_R \ge \theta \log_\b R) \le \ceil{\log_\b R}\e^{-c\theta} \quad \text{for all $\theta\ge0$, $R\ge2$.}
}
\end{proposition}

\begin{proof}
By definition \eqref{YR}, all edges of $\incir_{L_R}$ have both endpoints in $B_{R-1}\subseteq B_{\b^{\ceil{\log_\b R}}-1}$ and thus belong to $\EE_{\ceil{\log_\b R}}$. 
Therefore, we always have $\mathfrak o(\incir_{L_R}) \le \ceil{\log_\b R}$.
Consequently, we have $L_R = \sum_{k=1}^{\ceil{\log_\b R}}V_k$, where $V_k$ is the number of circuits in the sequence $(\incir_j)_{j\ge1}$ with $\mathfrak o(\incir_j) = k$. 
For $k\ge3$, let $\Asf_k$ denote the event that both of the following occur:
\begin{itemize}
    \item there exists a closed circuit $\DD$ that encloses the origin and whose vertices lie in the annulus $[-\b^{k-1},\b^{k-1}]^2\setminus[-\b^{k-2},\b^{k-2}]^2$;
    \item $\partial B_{\b^{k-3}} \overset{\mathsf{cd}}{\leftrightarrow} \partial B_{\b^{k+1}}$.
\end{itemize}
By the RSW theorem \eqref{RSW} together with the FKG inequality, there exists $\eps>0$ such that $\P(\Asf_k) \ge \ve$ for every $k$.
We claim that on the event $\Asf_k$, there exists no open circuit $\CC$ enclosing the origin with $\mathfrak o(\CC) = k$.

Indeed, suppose toward a contradiction that such a circuit $\CC$ exists.
By Lemma \ref{lemma:openclosed}, either $\DD$ surrounds $\CC$, or $\CC$ surrounds $\DD$.
In fact, it must be the latter.
To see this, note that the assumption $\mathfrak{o}(\CC) = k$ implies $\CC$ contains a vertex $v$ with $\|v\|_\infty \ge \b^{k-1}$ and thus $v\in\ext(\DD)$.
It now follows from Lemma~\ref{surround_lemma}\ref{surround_lemma_c} that $\CC$ is not surrounded by $\DD$.
Now we know $\CC$ surrounds $\DD$.
Let $\dualp$ a closed dual path from $\partial B_{\b^{k-3}}$ to $\partial B_{\b^{k+1}}$.
Since the first vertex of $\dualp$ belongs to 
$\intr(\DD)\subseteq\intr(\CC)$ while the last vertex of $\dualp$ belongs to $\ext(\CC)$,
there must be some intersection of $\dualp$ with $\CC$.
But this is impossible since $\dualp$ is closed dual while $\CC$ is open primal.

We have thus shown
\[
\mathbb P(\exists \text{ an open circuit $\CC$ with $\mathfrak o(\CC) = k$}) \le \P(\Asf_k^\complement) \le 1 -\ve.
\]
By the BKR inequality (applied second below), it follows that
\eeq{ \label{Vitail}
\mathbb P(V_k \ge n)  \le \mathbb P(\text{$\exists\ n$ edge-disjoint open circuits $\CC$ with $\mathfrak o(\CC) = k$}) \le (1-\ve)^n.
}
Using our earlier inequality for $L_R$ and then a union bound, we obtain
\eq{
\P(L_R\ge\theta \log_\b R) 
\le \P\bigg(\sum_{k=1}^{\ceil{\log_\b R}}V_k \ge \theta \log_\b R\bigg)
&\stackrefp{Vitail}{\le} \sum_{k=1}^{\ceil{\log_\b R}}\P\Big(V_k \ge \Big\lceil\frac{\theta \log_\b R}{\ceil{\log_\b R}}\Big\rceil\Big) \\
&\stackref{Vitail}{\le} \ceil{\log_\b R}(1-\eps)^{\theta{\log_b R}/\ceil{\log_\b R}} \\
&\stackrefp{Vitail}{\le} \ceil{\log_b R}e^{-c\theta},
}
where $c > 0$ is an appropriately chosen constant.
\end{proof}

\subsection{Bounds for number of closed edges on geodesic between circuits} \label{sec_cebc}
Recall that $\gamma_R\in\Geo(\vc 0,\partial B_R)$ is the geodesic from Proposition \ref{prop:geod_cons} with $\mathcal A = \{\vc 0\}$ and $\mathcal B = \partial B_{R}$.
This subsection will prove a tail bound for the number the closed edges in $\gamma_R\cap\EE_j^\crc$.
This is analogous to Proposition~\ref{prop:Dknbd}, but now using circuits instead of annuli.
The random variable analogous to $Y_{k,R}$ from \eqref{Dkn} is
\eeq{\label{eq: xi}
X_{j,R} = |\{e\in\gamma_R\cap\EE_j^\crc:\, t_e>0\}|,
}
where $\EE_j^\crc$ is defined in \eqref{def_edges_btw_cir}.
In other words, $X_{j,R}$ is the number of closed edges in $\gamma_R$ between its last intersection with $\incir_{j-1}$ and its last intersection with $\incir_{j}$ (this equivalent description is using Proposition~\ref{prop:geod_cons}\ref{itm:gamma_on_circuit} to guarantee that all edges between the first and last intersections with $\incir_j$ are in fact on $\incir_j$).
The previous sentence is true even if $j=1$, since we take the convention $\incir_0 = \{\vc 0\}$.

\begin{proposition} \label{prop:Xjbd}
 	Assume \eqref{critical_assumption} and either \eqref{eq : limit_bigger_than_one} or \eqref{eq : limit_equal_zero}. 
    There exist constants \(C,c ,s > 0\) such that
\eeq{ \label{eq_Xjbd}
  \P(X_{j,R} \geq \lambda)\leq C  j^{4/3} \e^{-c \lambda^s} \quad \text{for all $j \ge 1$, $R\ge1$, $\lambda \ge 0$}.
}
\end{proposition}

\begin{proof}
As usual, the values of $C$ and $c$ may change with each inequality.
By Proposition~\ref{prop_crc_to_ann}, we have $X_{j,R} \le \sum_{k=\mathfrak{i}(\incir_{j-1})\vee1}^{\mathfrak{o}(\incir_{j})}Y_{k,R}$.
This justifies the first inequality below:
\eq{
      &\P(X_{j,R}\geq \lambda) 
\le \P \bigg( \sum_{k=\mathfrak{i}(\incir_{j-1})\vee1}^{\mathfrak{o}(\incir_{j})} Y_{k,R}  \geq \lambda \bigg)  \\
&\stackrefp{Dknbd_eq,reb8c}{=} \sum_{\ell=0}^{\infty} \sum_{n=1}^{\infty}  \P \bigg( \sum_{i=k-\ell}^{k} Y_{k,R}   \geq \lambda ,\ \mathfrak{o}(\incir_j) = n,\ \mathfrak{o}(\incir_{j})  - \mathfrak{i}(\incir_{j-1})  = \ell  \bigg) \\
&\stackrel{\parbox{\widthof{\footnotesize\eqref{Dknbd_eq},\eqref{reb8c}}}{\centering\footnotesize (H\"older)}}{\le}  \sum_{\ell=0}^{\infty }\sum_{n=1}^{\infty}  \P \bigg( \sum_{k=n-\ell}^{k} Y_{k,R}  \geq \lambda \bigg)^{1/3} \P(\mathfrak{o}(\incir_j) = n)^{1/3} \P( \mathfrak{o}(\incir_{j})  - \mathfrak{i}(\incir_{j-1}) = \ell )^{1/3} \\
&\stackref{Dknbd_eq,reb8c}{\le} C\sum_{\ell=0}^{\infty }\sum_{n=1}^{\infty} (\ell+1)^{1/3}\e^{-c(\lambda/(\ell+1))^s}\P \Big( \mathfrak{o}(\incir_j) = n \Big)^{1/3} \cdot j^{1/3}\e^{-c\ell} \\
&\stackrefpp{jd7bx}{Dknbd_eq,reb8c}{\le} Cj^{4/3}\sum_{\ell=1}^\infty \ell^{1/3}\e^{-c\lambda^s/\ell^s}\e^{-c\ell}.
}
By increasing $C$ and decreasing $c$, we may absorb the factor of $\ell^{1/3}$ into $\e^{-c\ell}$.
We are then left with
\[
\P(X_{j,R}\ge\lambda)
\le Cj^{4/3}\sum_{\ell=1}^\infty\e^{-c\lambda^s/\ell^s}\e^{-c\ell}
\stackref{dkb7}{\le} Cj^{4/3}\e^{-c\lambda^{s/(s+1)}}. \qedhere
\]
\end{proof}

\subsection{Upper bound on length of geodesic} \label{sec_gen_final_proof}
In this subsection we prove our last main result, Theorem~\ref{thm:general_thm}.
We will need the following notation: let $\pi_3^{(\lambda)}(M)$ denote the probability of the following polychromatic 3-arm event to distance $M$ with at most $\lambda$ ``defects'':
\begin{itemize}
    \item There exist two primal paths, starting at $(0,0)$ and $(1,0)$ respectively, both ending at $\partial B_M$, and each containing at most $\lambda$ closed edges.
    \item There exists a dual path, starting at either $(\f{1}{2},\frac{1}{2})$ or $(\f{1}{2},-\frac{1}{2})$, ending at a dual neighbor of some vertex belonging to $\partial B_M$, and containing at most $\lambda$ open edges.
    \item All three of these paths are disjoint.
\end{itemize}
A key estimate is \cite[Prop.~18]{nolin08}, which gives
\begin{equation}\label{defect_bd}
\pi_3^{(\lambda)}(M) \leq \big(C(1+\log M)\big)^\lambda \pi_3(M) \quad \text{for all $\lambda\ge0$, $M\ge1$}.
\end{equation}
While \cite[Prop.~18]{nolin08} is stated as $\pi_3^{(\lambda)}(M) \le C_\lambda (1+\log M)^\lambda \pi_3(M)$ for some constant $C_\lambda$ that depends on $\lambda$, \eqref{defect_bd} follows from their proof since it is shown that $C_\lambda = C_{\lambda-1} C'$ for some constant $C'$.

\begin{proof}[Proof of Theorem \ref{thm:general_thm}] 
Let $\primalp_R$ and $\dualp_R$ be the geodesic and the disjoint dual path constructed in Proposition~\ref{prop:geod_cons} with $\AA = \{\vc 0\}$ and $\BB = \partial B_R$.
The statement to be shown is that
\begin{equation} \label{eq:1.4reform}
\P\big(|\gamma_R| \ge \theta R^{2+\ve} \pi_3(R)\big) \le C\e^{-c(\log(\theta R))^s}
\quad \text{for all $\theta\ge1$, $R\ge1$.}
\end{equation}
In the case of $R=1$, we have $|\gamma_R|=1$, so the result trivially holds provided $C\ge\pi_3(1)^{-1}$.
So henceforth we assume $R\ge2$.

As in previous subsections, let $(\incir_j)_{j\ge1}$ denote the collection of successively innermost edge-disjoint open circuits enclosing $\incir_0 = \{\vc 0\}$.
Recall from \eqref{YR} that $L_R$ is the number of these circuits that remain in the interior of $\partial B_R$.
Recall from \eqref{def_edges_btw_cir} that $\EE_j^\crc$ contains all edges that could \textit{possibly} be used by $\gamma_R$ between its last intersection with $\incir_{j-1}$ and its last intersection with $\incir_j$.
Recall from \eqref{eq: xi} that $X_{j,R}$ is the number of those edges that are \textit{actually} used by $\gamma_R$ and are closed.

\begin{claim} \label{claim_last_circ}
$X_{j,R}=0$ for all $j\ge L_R+2$.
\end{claim}

\begin{proofclaim}
If $\gamma_R$ does not intersect $\incir_{L_R+1}$, then $\gamma_R$ does not intersect $\incir_{L_R+2},\incir_{L_R+3},\dots$, which makes the claim trivial.
So assume $\gamma_R$ intersects $\incir_{L_R+1}$, and
let $x$ be the first vertex of $\gamma_R$ that lies on $\incir_{L_R+1}$.
Since any intersection of $\gamma_R$ with $\incir_{L_R+2},\incir_{L_R+3},\dots$ occurs after $x$, it suffices to show $T(x,\partial B_R) = 0$.
To this end, consider the Jordan curve $\jor = \{z\in\R^2:\, \|z\|_\infty = R\}$.
By maximality of $L_R$, the circuit $\incir_{L_R+1}$ contains a vertex in $\Z^2\setminus B_{R-1} \subseteq \jor\cup\ext(\jor)$.
On the other hand, $x$ belongs to $B_R\subseteq\jor\cup\intr(\jor)$.
Therefore, $\incir_{L_R+1}$ intersects both $\jor\cup\ext(\jor)$ and $\jor\cup\intr(\jor)$.
We conclude that $\incir_{L_R+1}$ must intersect $\jor$ at some vertex $y\in\partial B_R$.
Since $x$ and $y$ both lie on the open circuit $\incir_{L_R+1}$, we have $T(x,y)=0$.
\end{proofclaim}

Now define
\eeq{ \label{frakXR_def}
\mathfrak{X}_R
=\max_{j\ge1} X_{j,R}
\stackrel{\text{\footnotesize (Claim~\ref{claim_last_circ}})}{=} \max_{1\leq j\le L_R+1} X_{j,R}.
}
We prove \eqref{eq:1.4reform} via the following union bound:
\eeq{ \label{kb7s3}
\P\big(|\gamma_R| \ge \theta R^{2+\ve} \pi_3(R)\big) 
&\le \P\big( |\primalp_R | \geq \theta  R^{2 + \ve} \pi_3(R),\  \mathfrak{X}_R <   \lambda \big) \\
&\phantom{\le} +  \P(\mathfrak{X}_R  \ge \lambda,\ L_R <  \ell) 
+  \P(L_R \ge \ell).
}
We seek an upper bound of the form $C\e^{-c(\log(\theta R))^s}$ for each of the three terms on the right-hand side, which will be achieved with the choices
\eeq{ \label{choices}
\lambda = \ceil{\log(\theta R)}^{1/2}, \qquad
\ell = \ceil{\log(\theta R)\cdot\log_\b R}.
}
As usual, we allow the constants $C,c,s>0$ to change with each inequality, but they will never depend on $\theta$ or $R$.

The third term on the right-hand side of \eqref{kb7s3} is controlled using Proposition~\ref{prop_LR_tail} (in the second ineqaulity below):
\eq{
\P(L_R\ge\ell) 
\stackref{choices}{\le} \P(L_R \ge \log(\theta R)\cdot\log_\b R)
\stackref{LR_tail}{\le} \ceil{\log_\b R}\e^{-c\log(\theta R)}.
}
By decreasing $c$, we may absorb the factor of $\ceil{\log_\b R}$ into the exponential and obtain an upper bound of the form $C\e^{-c\log(\theta R)}$.

The second term on the right-hand side of \eqref{kb7s3} is controlled using Proposition \ref{prop:Xjbd} (in the third inequality below):
\eq{
&\P\big(\mathfrak{X}_R  \ge \lambda,\ L_R <  \ell \big)
   \stackref{frakXR_def}{\le} \P\Big(\max_{1 \leq j \leq \ell } X_{j,R} \ge \lambda \Big)
   \le \sum_{j=1}^{\ell} \P( X_{j,R} \ge \lambda) \\
   &\stackrefpp{eq_Xjbd}{choices}{\leq} C  \e^{- c\lambda^s } \sum_{j=1}^{\ell} j^{4/3} 
   \le C  \e^{- c\lambda^s }\ell^{7/3} \stackref{choices}{\le} C\e^{-c(\log(\theta R))^{s/2}}\ceil{\log(\theta R)\cdot \log_\b R}^{7/3}.
}
By increasing $C$ and decreasing $c$, we may absorb the factor of $\ceil{\log(\theta R)\cdot \log_\b R}^{7/3}$ into the exponential and obtain an upper bound of the form $C\e^{-c(\log(\theta R))^{s/2}}$.

Finally, the first term on the right-hand side of \eqref{kb7s3} will be controlled with the help of the following claim.
Let $E(B_R)$ denote the set of edges with both endpoints in $B_{R}$.

\begin{claim} \label{claim_defect}
Given an edge $e\in E(B_R)$, let \(M = \min \{ \dist(e,\{\vc 0\}) , \dist(e , \partial B_{R} ) \} \). Then
\eeq{\label{gen_3arm}
\mathbb{P}(e\in \primalp_R,\ \mathfrak{X}_R \leq \lambda) \leq C \pi_3^{(\lambda)}(M),
}
where $C$ does not depend on $e$ or $R$.
\end{claim}

\begin{proofclaim}
Suppose $e\in\gamma_R$ and $\mathfrak{X}_R\le\lambda$.
Let $J$ be the unique integer such that $e\in\EE_J^\crc$.
That is, $e$ lies on the portion of $\gamma_R$ between its last intersection with $\incir_{J-1}$ and its last intersection with $\incir_J$.
Either the midpoint of $e$ belongs to $\ext(\incir_{J-1})\cap\intr(\incir_J)$, or $e\in\incir_J$.
We split the relevant event into three cases and prove \eqref{gen_3arm_1}--\eqref{gen_3arm_3} below, which together prove \eqref{gen_3arm}.

\medskip\noindent\textbf{Case 1: $e\in\incir_{J}$.} 
Then the first bullet point in Remark~\ref{rem_no_bernoulli} gives 
\eeq{ \label{gen_3arm_1}
\P(e\in\gamma_R,\ e\in\incir_J) \le C\pi_3(M).
}

\medskip\noindent\textbf{Case 2: $e\notin\incir_{J}$ and $X_{j,R}=0$.} 
Then the second bullet point in Remark~\ref{rem_no_bernoulli} gives 
\eeq{ \label{gen_3arm_2}
\P(e\in\gamma_R,\ e\notin\incir_J,\ X_{j,R}=0) \le C\pi_3(M).
}

\medskip\noindent\textbf{Case 3: $e\notin\incir_{J}$ and $X_{j,R} \in \{1,\dots,\lambda\}$.} 
Proposition~\ref{prop:geod_cons}\ref{itm:gamma_on_circuit} implies that $\incir_{J-1}$ and $\incir_J$ must be vertex-disjoint in this case.
Therefore, the following two primal paths to distance $M$ are disjoint: follow $\gamma_R$ from $e$ in both directions until reaching $\incir_J$ (or possibly $\partial B_R$ if $J>L_R$) and $\incir_{J-1}$ respectively, then follow $\incir_J$ and $\incir_{J-1}$ around the origin.
These paths use at most $X_{j,R}\le\lambda$ closed edges.
Next we construct a closed arm from $e$ to distance $M$ that is disjoint from the two paths just discussed, which will show
\eeq{ \label{gen_3arm_3}
\P(e\in\gamma_R,\ e\notin\incir_{J},\ X_{j,R} \in \{1,\dots,\lambda\}) \le C\pi_3^{(\lambda)}(M).
}
In fact, this closed arm from $e$ to $\partial B_M(e)$ will not intersect $\gamma_R$, $\incir_J$, or $\incir_{J-1}$.

Since $X_{j,R}>0$, there exists a closed edge $f\in\gamma_R\cap\EE_J^\crc$.
If $e$ is closed, we take $f=e$; otherwise choose $f$ arbitrarily.
By Proposition \ref{prop:geod_cons}\ref{itm:dualcir}, the dual edge \( f^\star \) lies on a closed circuit $\dincir$ with the following properties: $\vc 0\in\intr(\dincir)$, $\partial B_R\subseteq\ext(\dincir)$, and $\UU$ does not intersect the geodesic \( \primalp_R \) anywhere other than the midpoint of $f$.
Moreover, the closed dual circuit $\dincir$ obviously cannot intersect the open primal circuits $\incir_{J-1}$ and $\incir_J$.
Therefore, if $e$ is closed (so $f=e$), then $\dincir$ immediately provides the desired closed arm (in fact, $\dincir$ provides two such arms).

If instead $e$ is open, then consider the closed dual path $\dualp_e$ from $e$ to $\dualp_R$, guaranteed by Proposition \ref{prop:geod_cons}\ref{itm:dualconn}.
Let $x$ be the vertex (on the dual lattice) at which $\dualp_e$ meets $\dualp_R$.
Since $\dualp_e$ is closed dual, it cannot intersect the open primal circuits $\incir_{J-1}$ and $\incir_J$.
Hence $\dualp_e$ remains entirely in $\ext(\incir_{J-1})\cap\intr(\incir_J)$. 
The same logic applies to $\dincir$ since the midpoint of $f$ lies in $\ext(\incir_{J-1})\cap\intr(\incir_J)$.
Since $\dualp_R$ starts at $\vc 0\in\intr(\UU)$ and ends at $\partial B_R\subseteq\ext(\UU)$, $\dualp_R$ must intersect $\dincir$ at some vertex $y$.
We now know that $x$ and $y$ both belong to $\ext(\incir_{J-1})\cap\intr(\incir_J)$.
Since $\dualp_R$ intersects $\incir_{J-1}$ and $\incir_J$ only once each (Proposition~\ref{prop:geod_cons}\ref{itm:zetaopen}), and any such intersection is actually a crossing from interior to exterior, the portion of $\dualp_R$ between $x$ and $y$ remains entirely in $\ext(\incir_{J-1})\cap\intr(\incir_J)$.
Therefore, the following closed path does not intersect $\incir_{J-1}$ or $\incir_J$:
\begin{itemize}
    \item Follow $\dualp_e$ from $e$ to $x$. This path is disjoint from $\gamma_R$ by Proposition~\ref{prop:geod_cons}\ref{itm:dualconn}.
    \item Next follow $\dualp_R$ from $x$ to $y$.  This path is disjoint from $\gamma_R$ by Proposition~\ref{prop:geod_cons}.
    If we have reached $\partial B_M(e)$, then we can stop.
    \item If we have not reached $\partial B_M(e)$, then consider the two arcs of $\dincir$ emanating from $y$, both of which intersect $\partial B_M(e)$ since $\dincir$ encloses the origin.
    Follow the arc that avoids $f^\star$ so that it is disjoint from $\gamma_R$.
\end{itemize}
It is possible that the path just constructed is not self-avoiding ($\zeta_e$ and $\zeta_R$ may use edges on $\dincir$), but removing any loops yields the desired closed arm from $e$ to distance $M$.
\end{proofclaim}

We can now carry out the first-moment method:
\eeq{ \label{kb82x}
&\P\big( |\primalp_R | \geq  \theta R^{2 + \ve} \pi_3(R),\  \mathfrak{X}_R \le\lambda \big)  \\
&\stackrefp{defect_bd}{=} \P \bigg ( \sum_{e \in E(B_R)}  \one(  e \in \primalp_R,\mathfrak{X}_R \le\lambda  ) \geq \theta R^{2 + \ve} \pi_3(R) \bigg) \\
&\stackrefp{defect_bd}{\leq} \frac{1}{\theta R^{2 + \ve} \pi_3(R)} \sum_{e \in E(B_{R})}   \P (  e \in \primalp_R,\ \mathfrak{X}_R\leq \lambda) \\
 &\stackrefp{defect_bd}{=} \frac{1}{\theta R^{2 + \ve} \pi_3(R)}\sum_{r=0}^{R} \sum_{\substack{e \in E(B_{R}) \\ \dist(e,\{\vc 0\}) \wedge \dist(e, \partial B_{R}) = r }}  \P (  e \in \primalp_R, \mathfrak{X}_R \le \lambda) \\
 &\stackrefpp{gen_3arm}{defect_bd}{\le} \frac{C}{\theta R^{2 + \ve} \pi_3(R)}\sum_{r = 0}^{R} R \pi_3^{(\lambda)}(r) \\
& \stackref{defect_bd}{\leq} \frac{C}{\theta R^{2 + \ve} \pi_3(R)} \big(C(1+\log R)\big)^{\lambda} \sum_{r = 0}^{R} R  \pi_3(r)
 \\
 &\stackrefpp{3r8bbc}{defect_bd}{\leq} \dfrac{C\big(C(1+\log R)\big)^{\lambda}}{\theta R^{\ve}} 
 \stackrefpp{choices}{defect_bd}{=} \frac{C\big(C(1+\log R)\big)^{\ceil{\log(\theta R)}^{1/2}}}{\theta R^\eps}.
}
Since $\theta\ge1$, we have
\eq{
\big(C(1+\log R)\big)^{\ceil{\log(\theta R)}^{1/2}}
\le \big(C\log(\theta R)\big)^{\ceil{\log(\theta R)}^{1/2}}
\le C\e^{\log(\theta R)^{2/3}}.
}
On the other hand,
\eq{
\theta R^\eps = \e^{\log(\theta R^\eps)}
\ge c\e^{2(\log(\theta R))^{2/3}}.
}
Using the two previous displays in the final line of \eqref{kb82x}, we obtain an upper bound of the form $C\e^{-(\log(\theta R))^{2/3}}$.
\end{proof}

\section{Kesten's separation results and consequences}\label{sec:separation}
In the remaining three sections of the paper, we develop the topological theory needed to rigorously prove Proposition \ref{prop:geod_cons}. 

Kesten \cite{kesten82} developed a general theory of percolation on lattices embedded in $\R^d$ and satisfying appropriate conditions. Here we review how general separation theorems proved in \cite{kesten82} apply to the square lattice. 

A \textit{mosaic} (Definition 2 in Section 2.2 of \cite{kesten82}) is a graph $\MM$ embedded in $\R^2$ satisfying the following three conditions. Here, we consider each edge as a set of points in $\R^2$.
\begin{enumerate} [label=\rm(\roman{*}), ref=\rm(\roman{*})]  \itemsep=3pt 
    \item $\MM$ has no loops.
    \item All edges of $\MM$ are bounded, and every compact set of $\R^2$ intersects only finitely many edges of $\MM$. 
    \item  Any two distinct edges of $\MM$ are either disjoint or their intersection consists of a single vertex of the graph.
\end{enumerate}
An important fact, as noted by Kesten, is that any mosaic is a planar graph. 
A mosaic $\MM$ splits $\R^2$ into connected components called faces. For a face $F$ of $\MM$, to \textit{close-pack} $F$ means to add an edge between any pair of vertices on the boundary of $F$ that are not yet connected (\cite{kesten82}, Definition 3 in Section 2.2).
Given a subset $\FF$ of faces, let $\GG$ and $\GG^\star$ be the graphs obtained from $\MM$ by close-packing all faces in $\FF$ and not in $\FF$, respectively.
Then $(\GG,\GG^\star)$ is called a \textit{matching pair}; we consider a specific example in Lemma \ref{lem:cov_graph} which is illustrated in Figures \ref{fig:mosaic} and \ref{fig:matching_pair}.

\begin{remark}[Triangular vs.\ square lattice] \label{rmk:triang}
As an example, let $\MM$ be the triangular lattice. Every face in $\MM$ is a triangle, so no matter the choice of $\FF$, the induced matching pair is $(\MM,\MM)$. As we will see, this self-similarity does not hold for the primal lattice. This is one reason why working with the primal lattice is technically more challenging than the triangular lattice. 
\end{remark}

Recall Definition \ref{def:enclose} that a circuit $\CC$ \textit{encloses} a set $\AA$ if $\AA \subseteq \intr(\CC)$, and that a circuit $\CC$ \textit{surrounds} another circuit $\CC'$ if $\intr(\CC') \subseteq \intr(\CC)$. When we say a circuit encloses a vertex $v$, we mean that $\CC$ encloses the set $\{v\}$. 

The following theorem relates site percolation on the graph $\GG$ to site percolation on $\GG^\star$. In site percolation, the open cluster of an open vertex $v$ is the set of all vertices connected to $v$ by a path  whose every vertex is open. 
\begin{theorem} \label{Kesten-cor2.2}
{\textup{\cite[Cor.~2.2]{kesten82}}}
Let $(\GG,\GG^\star)$ be the matching pair constructed from a mosaic $\MM$ and a set of faces $\FF$. Consider site percolation on the graph $\GG$. Also, assume the following: 

\begin{enumerate} [label=\rm(\roman{*}), ref=\rm(\roman{*})]  \itemsep=3pt 
    \item $\GG$ has bounded degree.
    \item All edges of $\GG$ have finite diameter, and every compact set of $\R^d$ intersects only finitely many edges of $\GG$. 
    \item $\GG$ is connected. 
\end{enumerate}
Under these assumptions, if $W(v)$, the open cluster of the vertex $v$, is nonempty and bounded, then there exists a closed circuit $J$ on $\GG^\star$ enclosing $v$.
\end{theorem}

In the setting of the present paper, instead of site percolation, we study edge percolation on $\Z^2$. However, we may still apply Theorem \ref{Kesten-cor2.2} in the appropriate setting by noting that edge percolation on  a graph $\GG$ is equivalent to site percolation on its \textit{covering graph} $\wt \GG$ (The vertex set of $\wt\GG$ is the set of midpoints of the edges of $\GG$, and two vertices in $\wt\GG$ are connected by an edge if the corresponding edges in $\GG$ share an endpoint).
This observation was first made by Fisher and Essam \cite{fisher61,fisher_essam61}, and is detailed in \cite[Sec.~2.5]{kesten82}. 

In the setting of the primal lattice $\Z^2$, let $\MM$ be the mosaic defined as follows and depicted in Figure \ref{fig:mosaic}. We place a vertex of $\MM$ at the midpoint of each edge of the primal lattice; that is, points of the form $(m,n + \f{1}{2})$ and $(m + \f{1}{2},n)$ for integers $m,n$. Each vertex $v$ of $\MM$ has edges connecting to four other vertices, namely $v\pm (\f{1}{2},\f{1}{2})$ and $v\pm(\f{1}{2},-\f{1}{2})$.
\begin{figure}[t]
    \centering
    \includegraphics[height =2in]{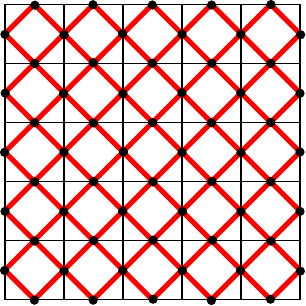}
    \caption{\small The vertices of the mosaic $\MM$ are denoted with black dots, and the edges are denoted with red/thick edges. The mosaic is overlaid onto the primal lattice (black/thin edges). }
    \label{fig:mosaic}
\end{figure}

Let $\FF$ be the set of faces of $\MM$ that contain a point of $\Z^2$, and let $(\wt \GG,\wt \GG^\star)$ be the associated matching pair, shown in Figure \ref{fig:matching_pair}.
Then, we have the following lemma.
\begin{figure}[t]
    \centering
    \includegraphics[height = 2in]{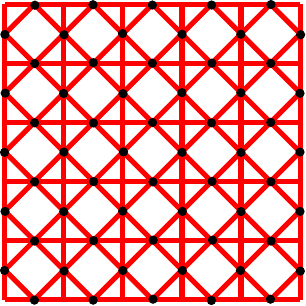} \hspace{0.5in}
    \includegraphics[height = 2in]{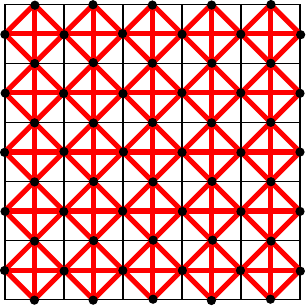}
    \caption{\small On the left is the graph $\wt \GG$, created by adding an edge between any two vertices of $\MM$ in faces that contain a point of $\Z^2$. No additional vertices are added. For faces that contain points of $(\Z^2)^\star$, no edges are added. On the right is the graph $\wt \GG^\star$, which is created by adding edges between vertices in $\MM$ in faces that contain a point of $(\Z^2)^\star$. Again, no additional vertices are added.  In the right figure, we also see the edges from the primal lattice in black/thin for reference, but these edges are not part of $\GG^\star$. }
    \label{fig:matching_pair}
\end{figure}

\begin{lemma} \label{lem:cov_graph}
$\wt \GG$ is the covering graph for the primal lattice, and $\wt \GG^\star$ is the covering graph for the dual lattice.
\end{lemma}
\begin{proof}
We prove that  $\wt \GG$ is the covering graph for the primal lattice, and the other statement follows analogously. It is helpful to refer to Figure \ref{fig:matching_pair}. Both graphs have the same vertex set by definition. In $\MM$, each vertex $v$ is connected by an edge to four other vertices. The corresponding edge $e$ in the primal lattice shares an endpoint with these edges; they are the edges connected to $e$ that are perpendicular to $e$. The two additional edges that share an endpoint with $e$ in the primal lattice correspond to the vertices in $\MM$ whose connections to $v$ are added when constructing $\wt \GG$. 
\end{proof}


\begin{definition} \label{def:dualBD}
For a set $W$ of primal edges, we define the dual boundary of $W$, denoted $\dualBD(W)$, as the collection of dual edges $e^\star$ such that the associated primal edge $e$ is not in $W$, but shares at least one endpoint with an edge of $W$. Similarly, for a set $W^\star$ of dual edges, the primal boundary of the open cluster is defined as the set of primal edges whose dual edge shares at least one endpoint with an edge of $\primBD(e^\star)$.
\end{definition}


We say that a connected set of open edges $W$ is a complete open cluster if no other open edges in $\Z^2$ are connected to an endpoint of an edge in $W$ by an open path.


\begin{lemma} \label{lem:bd_edges_closed}
 If $W$ is a complete open (closed) cluster of primal edges, then $\dualBD(W)$ consists entirely of closed (open) dual edges. Similarly, if $W^\star$ is a complete closed (open) cluster of dual edges, then $\primBD(W^\star)$ consists entirely of open (closed) primal edges. 
\end{lemma}

\begin{proof}
Let $W$ be an open cluster of primal edges, and assume to the contrary that $\dualBD(W)$ has an open dual edge $e^\star$. Then, the associated primal edge $e$ is open, and by definition $e$ is connected to an edge of $W$. Hence, $e \in W$, a contradiction to the definition of $\dualBD(W)$.
\end{proof}

Let $\GG$ be a graph and $\wt \GG$ its associated covering graph. Let $\wt r = (\wt v_1,\wt e_2,\ldots,\wt e_k,\wt v_k)$ be a path on $\wt \GG$. For each $i$, let $e_i$ be the associated edge on $\GG$ corresponding to the vertex $\wt v_i$ on $\wt \GG$. As noted in Comment (iii) in Section 2.5 of \cite{kesten82}, for suitable choices $v_0$ and $v_k$ of endpoints of $e_1$ and $e_k$, $r = (v_0,e_1,\ldots,e_k,v_k)$ is a path on $\GG$. However, $r$ is not necessarily a self-avoiding path, even if $\wt r$ is. Nevertheless, if $\wt r$ is a circuit, then the associated path $r$ on $\GG$ contains a circuit. Finally, with $\phi:\GG \to \wt \GG$ denoting the map sending edges of $\GG$ to vertices of $\wt \GG$, and $W(e)$ the open cluster of the edge $e$ in $\GG$, \cite[Prop.~3.1]{kesten82} gives $\phi(W(e)) = W(\phi(e))$.  We therefore obtain the following lemma as a corollary of Theorem \ref{Kesten-cor2.2} and Lemma \ref{lem:cov_graph}. See Figure \ref{fig:contain_closed} for clarity.

\begin{figure}[t]
    \centering
    \includegraphics[height =2in]{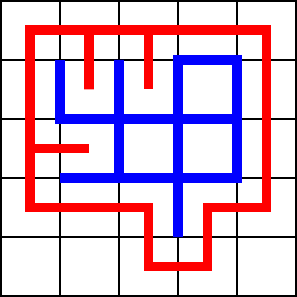} \hspace{0.5in}
\includegraphics[height = 2in]{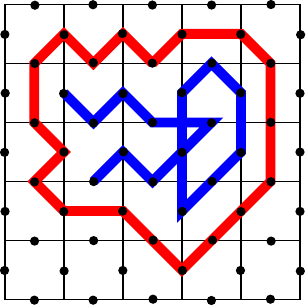}
    \caption{\small On the left is a picture of an open cluster of open edges in $\GG$ (blue/dark). On the right is the associated cluster of open vertices in $\wt \GG$, connected by a blue/dark path. By Theorem \ref{Kesten-cor2.2}, there exists a closed circuit in $\wt \GG^\star$ enclosing this open cluster. This circuit of closed vertices is shown on the right (red/light).  The set of edges on the left associated with the closed vertices on the right is not a circuit, but contains a circuit enclosing the original open cluster of edges. }
    \label{fig:contain_closed}
\end{figure}

\begin{lemma} \label{lemma:separation}
Let $W\subseteq\Z^2$ be a collection of vertices that are each connected to each other by an open path and such that for $v \notin W$, $v$ is not connected to any vertex in $W$ by an open path. Then, if the set $W$ is bounded, there exists a closed circuit $\DD$ on the dual lattice that encloses $W$. If $W$ consists of more than a single point, then we can consider $W$ as a complete open cluster of edges, and the circuit can be chosen to consist only of edges of $\dualBD(W)$. By duality and symmetry, the same holds with the roles of the primal and dual lattices reversed and/or with the roles of ``open'' and ``closed'' reversed.
\end{lemma}
\begin{proof}
If $W$ consists of a single vertex $v$, then there are no open edges incident to $v$. Then, the dual edges to each of the four closed edges incident to $v$ forms a dual closed circuit enclosing $v$. 

Now, assume that $W$ has more than one point so that we can consider it as a complete open cluster of edges. We perform a modification of the environment as follows. Set all edges whose dual does not lie on $\dualBD(W)$ to open. We claim that in the modified environment, $W$ is still a complete open cluster of edges. By definition, all primal edges $e$ sharing an endpoint with an edge in $W$ are either already in $W$, or their dual lies in $\dualBD(W)$. If the dual edge $e^\star$ lies in $\dualBD(W)$, then by Lemma \ref{lem:bd_edges_closed}, $e$ is closed. Thus, a vertex outside $W$ cannot connect to a vertex in $W$ by an open path without taking a closed edge in $\dualBD(W)$. 

Since $W$ is still a complete open cluster, the remarks above the lemma imply that, in the modified environment, there exists a closed dual circuit enclosing $W$. But these edges are also closed in the original environment. 
\end{proof}

\begin{lemma} \label{lem:separate_sets}
    Let $\DD$ be a closed circuit. Assume that there exist two disjoint connected sets $\mathcal{A}$ and $\mathcal{B}$ of dual vertices lying in $\intr(\DD)$ such that all edges connecting vertices of $\mathcal{A}$ are closed and the same with edges connecting vertices in $\mathcal{B}$. Let $W_{\mathcal A}$ be the cluster of dual vertices connected by a closed dual path to $\mathcal{A}$, and let $W_{\mathcal B}$ be the cluster of dual vertices connected by a closed dual path to $\mathcal{B}$.  Then, if $W_{\mathcal A} \neq W_{\mathcal B}$, there exists an open circuit $\CC$ contained in $\intr(\DD)$ that contains $W_{\mathcal A}$ in its interior and $W_{\mathcal B}$ in its exterior, or vice versa.  

By duality, this lemma also holds by interchanging the roles of open/closed and simultaneously interchanging the roles of primal/dual.
\end{lemma}
Figure \ref{fig:two_clusters} gives an illustration of Lemma \ref{lem:separate_sets}.
\begin{figure}
    \centering
    \includegraphics[height = 2in]{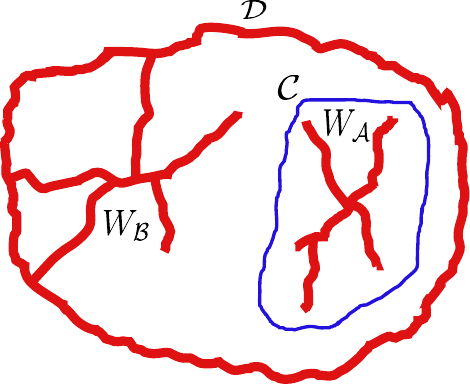}
    \caption{\small Two disjoint clusters of closed edges, $W_\AA$ and $W_\BB$, which lie inside a circuit $\DD$ and are separated by an open circuit $\CC$ (blue/thin)}
    \label{fig:two_clusters}
\end{figure}

\begin{proof}
We modify the environment so that all edges whose midpoint is in $\ext(\DD)$ are set to closed. 
Let $W_{\DD}$ be the infinite cluster of points connected to $\DD$ by a closed dual path in this modified environment, and adjust $W_{\mathcal A}$ and $W_{\mathcal B}$ to include these dual edges in the case that one of the clusters connects to $\DD$. We note that both clusters cannot connect to $\DD$ because then, there is a closed dual path from $\mathcal{A}$ to $\DD$, then from $\DD$ to $\mathcal{B}$, a contradiction. 
We consider two cases:

\medskip \noindent \textbf{Case 1:} $W_{\mathcal A} = W_{\DD}$ or $W_{\mathcal B} = W_\DD$.
Without loss of generality, say $W_{\mathcal B} = W_\DD$. Then, $W_{\mathcal A}$ must be bounded because it cannot cross $\DD$. Lemma \ref{lemma:separation} implies that there exists an open circuit $\CC$ containing $W_{\mathcal A}$ in its interior. Vertices in $W_{\mathcal B}$ cannot lie on the circuit $\CC$ because $\CC$ is open (and therefore primal). $W_{\mathcal B}$ contains points in the exterior of $\CC$ because $W_{\mathcal B} = W_\DD$ is unbounded by assumption. Furthermore, $W_{\mathcal B}$ cannot contain points in the interior of $\CC$; otherwise, the closed dual path from the points of $W_{\mathcal B}$ in $\ext(\CC)$ to the points of $W_{\mathcal B}$ in $\intr(\CC)$  crosses $\CC$, a contradiction because $\CC$ is open. Hence, $\CC$ contains $W_{\AA}$ in its interior and $W_{\BB}$ in its exterior.

\medskip \noindent \textbf{Case 2:} $W_{\mathcal B} \neq W_\DD$ and $W_{\mathcal A} \neq W_\DD$. We modify the environment further to obtain the previous case. We first argue that there exists an infinite closed dual path from either $W_{\mathcal A}$ or $W_{\mathcal B}$ that avoids the other. To see this, take any dual path from $W_{\mathcal A}$. There is a last vertex in $W_{\mathcal A} \cup W_{\mathcal B}$: the portion of the path starting from that vertex is the desired path. Set all the edges on this path to closed. Thus, we have reduced to the previous case. Since edges are only being changed to closed in this modified environment, the open circuit we obtain is still open in the original environment. 
\end{proof}

For the benefit of Section \ref{sec:outer_circuits}, we note that Theorem \ref{Kesten-cor2.2} and Lemmas \ref{lemma:separation} and \ref{lem:separate_sets} are deterministic statements that hold for any given configuration on the graphs. 

\section{Topological details concerning circuits and paths} \label{sec:outer_circuits}
Here we provide rigorous justification of certain topological constructions used throughout the paper.
The first four lemmas will be applications of the Jordan--Schönflies theorem, which says that for every Jordan curve $\jor\subseteq\R^2$, there is a homeomorphism $f\colon\R^2\to\R^2$ such that $f(\jor)$ is the unit circle.
In particular, because homeomorphisms preserve path-connectedness and compactness, the sets $\intr(\jor)$ and $\ext(\jor)$ are separately path-connected, and $\intr(\jor)$ is bounded while $\ext(\jor)$ is unbounded.
Furthermore, for every $z\in\jor$, the set $\{z\}\cup\intr(\jor)\cup\ext(\jor)$ is path-connected.

Recall our terminology that $\jor_1$ is \textit{surrounded} by $\jor_2$ if $\intr(\jor_1)\subseteq\intr(\jor_2)$.

\begin{lemma} \label{surround_lemma}
Let $\jor_1,\jor_2$ be Jordan curves.
The following are equivalent:
\begin{enumerate}[label=\textup{(\alph*)}]

\item \label{surround_lemma_a}
$\jor_1$ is surrounded by $\jor_2$.


\item \label{surround_lemma_c}
$\jor_1 \subseteq \jor_2 \cup \intr(\jor_2)$.

\item \label{surround_lemma_d}
$\jor_2 \subseteq \jor_1 \cup \ext(\jor_1)$ and $\intr(\jor_1)\cap\intr(\jor_2)\neq\varnothing$.

\end{enumerate} 
\end{lemma}
\begin{proof}
\ref{surround_lemma_a}$\implies$\ref{surround_lemma_c}: 
For any Jordan curve $\jor$, the set $\jor\cup\intr(\jor)$ is the closure of $\intr(\jor)$.
Therefore, the assumption $\intr(\jor_1)\subseteq\intr(\jor_2)$ implies $\jor_1\cup\intr(\jor_1)\subseteq\jor_2\cup\intr(\jor_2)$.
Dropping $\intr(\jor_1)$ on the left-hand side yields \ref{surround_lemma_c}.

\medskip \noindent \ref{surround_lemma_a}$\implies$\ref{surround_lemma_d}: 
From the assumption $\intr(\jor_1)\subseteq\intr(\jor_2)$ it trivially follows that $\intr(\jor_1)\cap\intr(\jor_2)$ is nonempty, and also that $\intr(\jor_1)\subseteq\intr(\jor_2)\cup\ext(\jor_2)$.
From the latter fact, taking complements yields $\jor_1\cup\ext(\jor_1)\supseteq\jor_2$.


\medskip \noindent \ref{surround_lemma_c}$\implies$\ref{surround_lemma_a} and \ref{surround_lemma_d}$\implies$\ref{surround_lemma_a}, by contrapositive: 
Assume \ref{surround_lemma_a} is false, meaning there exists some point
$x\in\intr(\jor_1) \cap \intr(\jor_2)^\complement$.
Since $x\notin\intr(\jor_2)$, there exists a continuous path starting at $x$, extending to infinity, and remaining entirely in $\ext(\jor_2)\cup\{x\}$. Since $x \in \intr(\jor_1)$, this path must intersect $\jor_1$ at some point $y_1\ne x$, so $y_1\in\jor_1\cap\ext(\jor_2)$.
Since $\ext(\jor_2)=(\jor_2\cup\intr(\jor_2))^\complement$, the existence of $y_1$ shows that \ref{surround_lemma_c} is false.

To show that \ref{surround_lemma_d} is also false, we suppose the second statement in \ref{surround_lemma_d} is true, and then prove that the first statement is false.
That is, we assume there exists $z\in\intr(\jor_1)\cap\intr(\jor_2)$, and then prove $\jor_2\cap(\jor_1\cup\ext(\jor_1))^\complement$ is nonemtpy.
Since both $x$ and $z$ belong to $\intr(\jor_1)$, there exists a continuous path from $x$ to $z$ that remains in $\intr(\jor_1)$.
Since $x\notin\intr(\jor_2)$ while $z\in\intr(\jor_2)$, this path must intersect $\jor_2$ at some point $y_2$.
Hence $y_2\in\jor_2\cap\intr(\jor_1) = \jor_2\cap(\jor_1\cup\ext(\jor_1))^\complement$,
as desired.
\end{proof}


\begin{lemma} \label{crossing_lemma}
Let $\jor_1$ and $\jor_2$ be Jordan curves, neither of which surrounds the other.
If $\intr(\jor_1)\cap\intr(\jor_2)$ is nonempty, then the following statements hold:
\begin{enumerate}[label=\textup{(\alph*)}]

\item \label{crossing_lemma_a} Both $\intr(\jor_1)\cap\jor_2$ and $\ext(\jor_1)\cap\jor_2$ are nonempty.

\item \label{crossing_lemma_b} $\jor_1\cap\jor_2$ contains at least two distinct points.

\end{enumerate}

\end{lemma}

\begin{proof}
By assumption, we can find $x_1 \in \intr(\jor_1)\setminus\intr(\jor_2)$, $x_2\in\intr(\jor_2)\setminus\intr(\jor_1)$, and $y\in\intr(\jor_1)\cap\intr(\jor_2)$.
Since $x_1$ and $y$ both belong to $\intr(\jor_1)$, there is a continuous path starting at $y$, ending at $x_1$, and remaining entirely in $\intr(\jor_1)$.
But since $y$ belongs to $\intr(\jor_2)$ while $x_1$ does not, this path must contain a point lying on $\jor_2$.
On the other hand, because $x_2\notin\intr(\jor_1)$, there is a continuous path starting at $x_2$, extending to infinity, and remaining entirely in $\ext(\jor_1)\cup\{x_2\}$.
Given that $x_2\in\intr(\jor_2)$, this path must intersect $\jor_2$, hence $\jor_2\cap\ext(\jor_1)$ is nonempty.
We have now seen that $\jor_2$ intersects both $\intr(\jor_1)$ and $\ext(\jor_1)$, thereby proving part \ref{crossing_lemma_a}.
Now take some $z_\mathrm{int}\in\jor_2\cap\intr(\jor_1)$ and $z_\mathrm{ext}\in\jor_2\cap\ext(\jor_1)$.
Any continuous path between $z_\mathrm{int}$ and $z_\mathrm{ext}$ must intersect $\jor_1$, and by definition $\jor_2$ offers two such paths which are disjoint (except at $z_\mathrm{int}$ and $z_\mathrm{ext}$, of course).
This observation implies part \ref{crossing_lemma_b}.
\end{proof}

\begin{lemma} \label{lem:dual_inext}
Let $\jor$ be a Jordan curve. Consider two adjacent boxes on the square lattice, and let $e$ be the edge shared by their boundaries. Assume that all of $e$ lies on $\jor$, and further assume that $\jor$ has empty intersection with the interior of each box. Then, exactly one of the dual neighbors of $e$ lies in $\intr(\jor)$, and the other dual neighbor lies in $\ext(\jor)$. 
\end{lemma}

\begin{proof}
Let $z_0$ and $z_1$ be the dual neighbors of $e$.
Denote the line segment between these two points by $L = \{(1-t)z_0^\star + tz_1^\star:\,t\in[0,1]\}$, and consider the open set $U = \{x\in\R^2:\, \dist(x,L)<\tfrac14\}$.
 By the assumptions on $\jor$, the set $U\setminus \jor$ has two connected components:  $U_0 = \{x\in U:\, \dist(x,z_0)<\dist(x,z_1)\}$ and $U_1 = \{x\in U:\, \dist(x,z_1)<\dist(x,z_0)\}$.
 Since $\jor$ is the boundary of its interior and exterior, any open set intersecting $\jor$ (such as $U$) has nonempty intersection with both $\intr(\jor)$ and $\ext(\jor)$. Hence one of $U_0$ and $U_1$ must be contained in $\intr(\jor)$, and the other is contained in $\ext(\jor)$. 
 Since $z_0\in U_0$ and $z_1\in U_1$, we are done.
\end{proof}

The following is similar to Lemma \ref{lem:dual_inext} but applies in a more general setting. 
The notation $\ball_\eps(z)$ denotes the open ball centered at $z\in\R^2$ with radius $\eps>0$.

\begin{lemma} \label{z_ball_2_lemma}
For any Jordan curve $\jor$, any $z\in\jor$, and any $\eps>0$, there exists an open set $U\subseteq \ball_\eps(z)$ containing $z$ such that
$U\setminus\jor$ has exactly two connected components.
\end{lemma}

\begin{proof}
Let $f\colon\R^2\to\R^2$ be a homeomorphism such that $f(\jor)$ is the unit circle.
Then $f(\ball_\eps(z))$ is an open set containing $f(z)$.
Therefore, we can find $\delta\in(0,2)$ small enough that $\ball_\delta(f(z))\subseteq f(\ball_\eps(z))$.
Since $\delta<2$, the set $\ball_\delta(f(z))\setminus f(\jor)$ has exactly two connected components. 
So the claim is satisfied by taking $U = f^{-1}\big(\ball_\delta(f(z))\big)$.
\end{proof}

Lemma \ref{splitting_lemma} concerns the ``splitting'' of a Jordan curve with a path through its interior;
Figure \ref{fig:split_Jordan} gives a depiction.
\begin{figure}
    \centering
    \includegraphics[height = 1.5in]{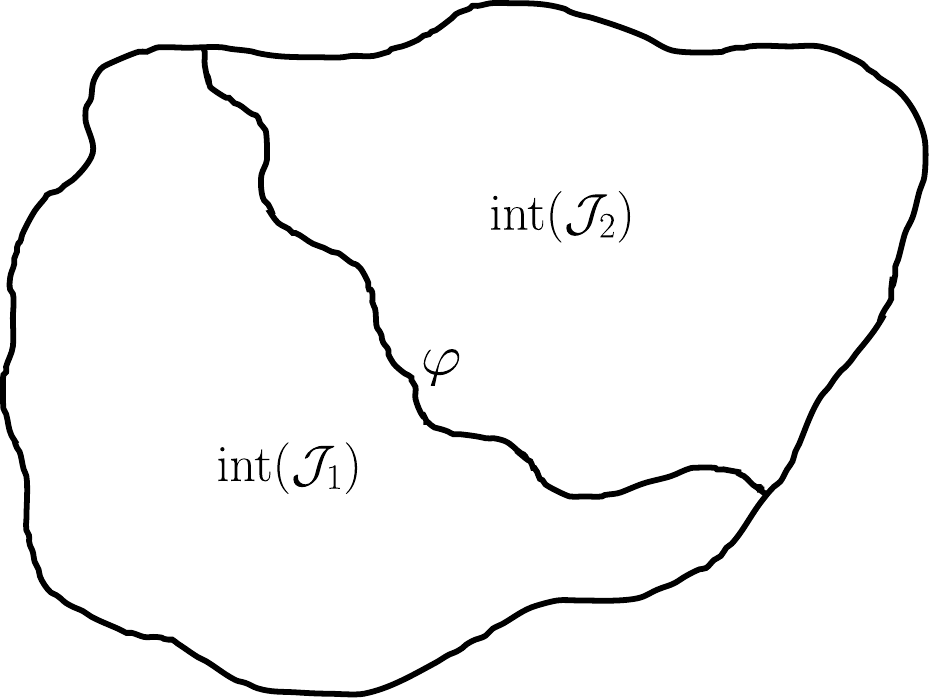}
    \caption{\small Illustration of Lemma \ref{splitting_lemma}: the path $\varphi$ splits $\intr(\jor)$ into three disjoint sets.}
    \label{fig:split_Jordan}
\end{figure}

\begin{lemma} \label{splitting_lemma}
Let $\jor$ be a Jordan curve, and let $\vphi\colon[0,1]\to\R^2$ be a continuous path such that $\vphi(0)$ and $\vphi(1)$ are distinct points on $\jor$, and $\vphi(t)\in\intr(\jor)$ for all $t\in(0,1)$.
Denote by $\jor_1$ the Jordan curve starting at $\vphi(0)$, proceeding in some direction along $\jor$ until reaching $\vphi(1)$, and then following $\vphi$ back to $\vphi(0)$.
Let $\jor_2$ be the same but proceeding along $\jor$ in the other direction.
We can then decompose the interior of $\jor$ as the disjoint union
\eeq{ \label{interior_decomp}
\intr(\jor) = \intr(\jor_1) \uplus \intr(\jor_2) \uplus \{\vphi(t):\, t\in(0,1)\}.
}
\end{lemma}

\begin{proof}
First we argue that the right-hand side of \eqref{interior_decomp} is contained in the left-hand side, by demonstrating the reverse containment of the complements.
Consider any $x\in\jor\cup\ext(\jor)$; clearly $x$ cannot be equal to $\vphi(t)$ (which belongs to $\intr(\jor)$) for any $t\in(0,1)$.
Furthermore, there exists a continuous path $\phi$ starting at $x$, extending to infinity, and remaining entirely in $\{x\}\cup\ext(\JJ)$.
In particular, this path never intersects $(\jor_1\cup\jor_2)\setminus\{x\}$, from which it follows that $x$ does not belong to $\intr(\jor_1)$ nor to $\intr(\jor_2)$.

We next argue that $\intr(\jor_1)$ and $\intr(\jor_2)$ must be disjoint. 
%
Since $\jor_1$ and $\intr(\jor_1)$ are disjoint, we have
\eq{
(\jor_1\cup\jor)\cap\intr(\jor_1) = \jor\cap\intr(\jor_1)\subseteq \jor\cap\intr(\jor) = \varnothing.
}
Meanwhile, because $\jor$ and $\ext(\jor)$ are disjoint, we have
\eq{
(\jor_1\cup\jor)\cap\ext(\jor) &= (\jor_1 \setminus \jor)\cap\ext(\jor) \\
&= \{\vphi(t):\, t\in(0,1)\}\cap\ext(\jor)
\subseteq \intr(\jor)\cap\ext(\jor) = \varnothing.
}
Considering that $\jor_2\subseteq\jor_1\cup\jor$, it follows from these two observations that
\eeq{ \label{J2_empties}
\jor_2\cap\intr(\jor_1) = \varnothing = \jor_2\cap\ext(\jor).
}
Now take any $x\in\intr(\jor_1)$ and $z\in\jor\setminus\jor_2$.
Since $z$ necessarily belongs to $\jor_1$, there exists a continuous path $\phi$ from $x$ to $z$ remaining entirely in $\intr(\jor_1)\cap\{z\}$.
By the first equality in \eqref{J2_empties}, this path never intersects $\jor_2$.
Next we extend $\phi$ from $z$ to infinity using a continuous path remaining entirely in $\ext(\jor)\cap\{z\}$; by the second equality in \eqref{J2_empties}, this extension also avoids $\jor_2$.
We have thus constructed a continuous path disjoint from $\jor_2$ that starts at $x$ and extends to infinity, implying that $x\in\ext(\jor_2)$.
In particular, $x$ does not belong to $\intr(\jor_2)$.

All that remains to show is that the left-hand side of \eqref{interior_decomp} is contained in the right-hand side.
So suppose $x\in\intr(\jor)$ but belongs to neither $\intr(\jor_2)$ nor $\{\vphi(t):\, t\in(0,1)\}$.
It follows that $x\notin\jor_2$ since $\jor_2\cap\intr(\jor)$ is equal to $\{\vphi(t):\, t\in(0,1)\}$.
This leaves only the possibility that $x\in\ext(\jor_2)$, meaning there is a continuous path $\phi$ starting at $x$ and extending to infinity that never intersects $\jor_2$.
But because $x\in\intr(\jor)$, this path must intersect $\jor$, necessarily at some point belonging to $\jor\setminus\jor_2\subseteq\jor_1$.
Let $z$ be the first intersection of $\phi$ with $\jor$ upon leaving $x$, and let $\phi_{x\to z}$ be portion of $\phi$ from $x$ to $z$.
In particular, $z$ is the unique intersection point of $\phi_{x\to z}$ with $\jor$.
Furthermore, $\phi_{x\to z}$ can only intersect $\jor_1$ at points belonging to $\jor_1\setminus\jor_2\subseteq\jor$, meaning $z$ is also the unique intersection point of $\phi_{x\to z}$ with $\jor_1$.


\begin{claim}
There exists $\eps>0$ such that
\eeq{ \label{z_ball_1}
\jor\cap \ball_\eps(z) = \jor_1\cap \ball_\eps(z).
}
\end{claim}

\begin{proofclaim}
Being a Jordan curve, $\jor_2$ is closed, so $\R^2\setminus\jor_2$ is open.
Since $z\in\R^2\setminus\jor_2$, we can thus choose $\eps>0$ such that $\ball_{\eps}(z)\cap\jor_2$ is empty.
This justifies the first and last equalities below, while the middle equality uses the fact that $\jor\setminus\jor_2=\jor_1\setminus\jor_2$:
\[
\jor\cap \ball_\eps(z) 
= (\jor\setminus\jor_2)\cap \ball_\eps (z)
= (\jor_1\setminus\jor_2)\cap \ball_\eps (z)
= \jor_1\cap \ball_\eps(z). \qedhere
\]
\end{proofclaim}


%
Now let $\eps>0$ satisfy \eqref{z_ball_1}, and then take $U\subseteq\ball_\eps(z)$ as in Lemma \ref{z_ball_2_lemma}.
Note that
\eeq{ \label{U_still_good}
\jor\cap U
= \jor \cap \ball_\eps(z)\cap U
\stackref{z_ball_1}{=} \jor_1 \cap \ball_\eps(z)\cap U
= \jor_1 \cap U.
}
Because $U$ has nonempty intersection with both $\intr(\jor)$ and $\ext(\jor)$, we conclude that the two connected components of $U\setminus\jor$ are precisely $U\cap\intr(\jor)$ and $U\cap\ext(\jor)$.
But $U\setminus\jor_1 = U\setminus\jor$ by \eqref{U_still_good}, so by the exact same logic, these two connected components are also expressible as $U\cap\intr(\jor_1)$ and $U\cap\ext(\jor_1)$.
In particular, $U\cap\intr(\jor_1)$ is equal to either $U\cap\intr(\jor)$ or $U\cap\ext(\jor)$.
Considering that $\intr(\jor_1)\subseteq\intr(\jor)$, it must be the former:
\eeq{ \label{interiors_same}
U\cap\intr(\jor_1) = U\cap\intr(\jor).
}
To complete the proof, we return to our consideration of $\phi_{x\to z}$.
This path starts in $\intr(\jor)$ and only encounters $\jor$ at the terminal point $z$, so it must otherwise remain in $\intr(\jor)$.
Since $z\in U$ and $U$ is open, we now know that $\phi_{x\to z}$ intersects $U\cap\intr(\jor)$.
It now follows from \eqref{interiors_same} that $\phi_{x\to z}$ intersects $\intr(\jor_1)$, meaning that $x$ is connected to some element of $\intr(\jor_1)$ by a continuous path avoiding $\jor_1$.
Hence $x\in\intr(\jor_1)$, as needed.
\end{proof}

Lemma \ref{expanding_lemma} concerns the ``expansion'' of a Jordan curve with a path on its exterior;
Figure \ref{fig:Jordan_extend} gives a depiction.
\begin{figure}
    \centering
    \includegraphics[height = 2in]{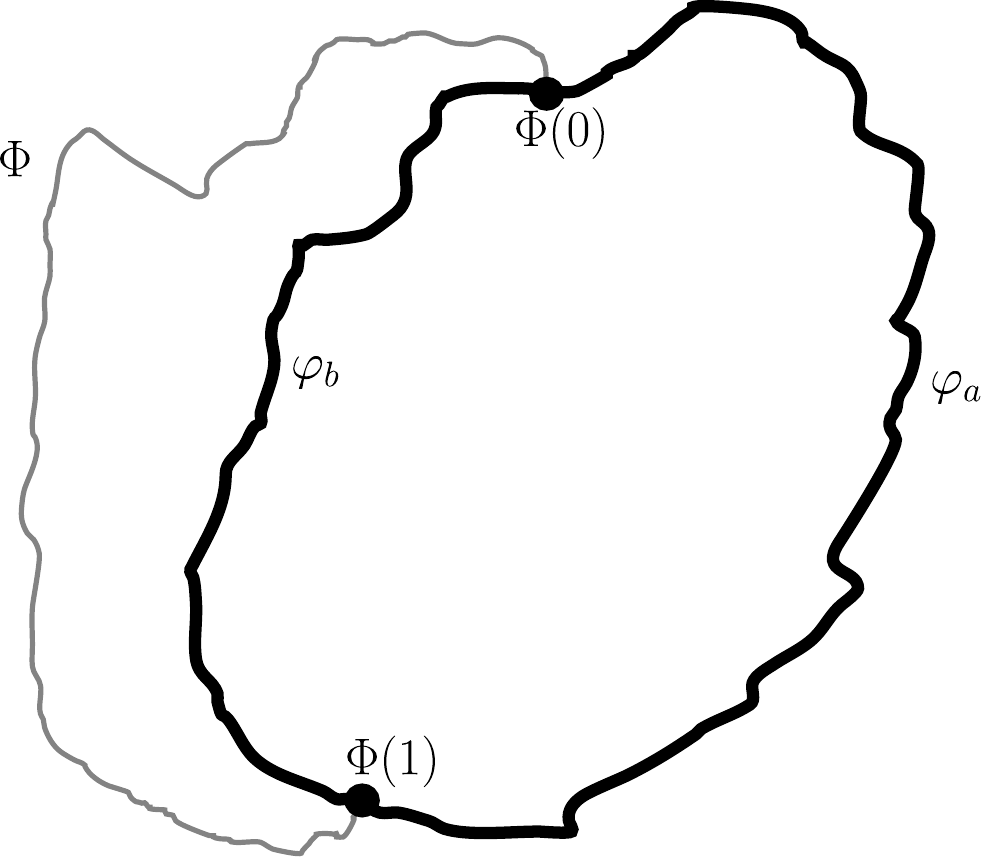}
    \caption{\small One possible conclusion of Lemma \ref{expanding_lemma}: the Jordan curve $\jor_1$ is shown in thick black, and the path $\Phi$ is depicted in thin gray. In this case, \eqref{expanding_lemma_eqs1} holds.}
    \label{fig:Jordan_extend}
\end{figure}

\begin{lemma} \label{expanding_lemma}
Let $\jor_1$ be a Jordan curve, and let $\Phi\colon[0,1]\to\R^2$ be a continuous path such that $\Phi(0)$ and $\Phi(1)$ are distinct points on $\jor_1$, and $\Phi(t)\in\ext(\jor_1)$ for all $t\in(0,1)$.
Let $\vphi_a$ denote the portion of $\jor_1$ proceeding clockwise from $\Phi(0)$ to $\Phi(1)$, and let $\vphi_b$ be the same but proceeding counterclockwise.
We parameterize these functions $\vphi_a,\vphi_b\colon[0,1]\to\R^2$ such that $\vphi_a(0) = \vphi_b(0) = \Phi(0)$ and $\vphi_a(1)=\vphi_b(1)=\Phi(1)$.
Denote by $\jor_a$ the Jordan curve starting at $\Phi(0)$, following $\vphi_a$ until reaching $\Phi(1)$, and then following $\Phi$ back to $\Phi(0)$.
Let $\jor_b$ be the same but using $\vphi_b$.
Then either
\begin{subequations} \label{expanding_lemma_eqs}
\eeq{
\intr(\jor_a)=\intr(\jor_1)\uplus\intr(\jor_b)\uplus \{\vphi_b(t):\, t\in(0,1)\} \label{expanding_lemma_eqs1}
}
or
\eeq{
\intr(\jor_b)=\intr(\jor_1)\uplus\intr(\jor_a)\uplus\{\vphi_a(t):\, t\in(0,1)\}. \label{expanding_lemma_eqs2}
}
\end{subequations}
\end{lemma}

\begin{proof}
We divide our argument into four steps. \medskip

\noindent\textbf{Step 1}: Show that $\intr(\jor_a)\cap\ext(\jor_1)$ is nonempty. 

For any $y\in\jor_a\setminus\jor_1$, we have $y\in\ext(\jor_1)$ by assumption.
Since $\ext(\jor_1)$ is open, there is $\eps>0$ small enough that $\ball_\eps(y)\subseteq\ext(\jor_1)$.
Because $y\in\jor_a$, $\ball_\eps(y)$ has nonempty intersection with $\intr(\jor_a)$.
This completes the first step. \medskip

\noindent\textbf{Step 2}: Show that $\intr(\jor_a)\cap\intr(\jor_b)$ is nonempty.

By the first step, there exists $x\in\intr(\jor_a)\cap\ext(\jor_1)$.
In particular, there exists a path $\phi$ starting at $x$, extending to infinity, and remaining entirely in $\ext(\jor_1)$.
But because $x\in\intr(\jor_a)$, this path must intersect $\jor_a$.
Denote by $z\in\jor_a$ the last intersection point, and let $\phi_{z\to\infty}$ be the portion of $\phi$ starting at $z$ and extending to infinity; with this choice, $\phi_{z\to\infty}$ never again intersects $\jor_a$.
Now note that
\eeq{ \label{same_outside}
\jor_a\cap\ext(\jor_1)=\jor_b\cap\ext(\jor_1)=\{\Phi(t):\, t\in(0,1)\}.
}
By construction we have $z\in\ext(\jor_1)$.
Choose $\eps>0$ sufficiently small that $\ball_\eps(z)\subseteq\ext(\jor_1)$, so as to guarantee
\eeq{ \label{z_ball_ext}
\jor_a\cap\ball_\eps(z)
=\jor_a\cap\ext(\jor_1)\cap\ball_\eps(z)
=\jor_b\cap\ext(\jor_1)\cap\ball_\eps(z)
= \jor_b\cap\ball_\eps(z).
}
By Lemma \ref{z_ball_2_lemma}, there exists a open set $U\subseteq\ball_\eps(z)$ containing $z$ such that $U\setminus\jor_1$ has exactly two components.
Moreover, \eqref{z_ball_ext} can be specialized this open set:
\eeq{ \label{U_still_good_ext}
\jor_a\cap U
= \jor_a \cap \ball_\eps(z)\cap U
= \jor_b \cap \ball_\eps(z)\cap U
= \jor_b \cap U.
}
Because $U$ has nonempty intersection with both $\intr(\jor_a)$ and $\ext(\jor_a)$, we conclude that the two connected components of $U\setminus\jor_a$ are precisely $U\cap\intr(\jor_a)$ and $U\cap\ext(\jor_a)$.
But $U\setminus\jor_a = U\setminus\jor_b$ by \eqref{U_still_good_ext}, and so by the exact same logic, these two connected components are also expressible as $U\cap\intr(\jor_b)$ and $U\cap\ext(\jor_b)$.
In particular, $U\cap\ext(\jor_b)$ is equal to either $U\cap\intr(\jor_a)$ or $U\cap\ext(\jor_a)$.
We claim it must be the latter.
Indeed, since $\phi_{z\to\infty}$ extends to infinity yet never intersects $\jor_a\cup\jor_1\supseteq\jor_b$ except at its starting point $z$, it must otherwise remain in $\ext(\jor_a)\cap\ext(\jor_b)$.
Hence $U\cap\ext(\jor_a)$ and $U\cap\ext(\jor_b)$ have nonempty intersection, forcing the two sets to be equal.
This leaves us to conclude that $U\cap\intr(\jor_a)=U\cap\intr(\jor_b)$; in particular, $\intr(\jor_a)\cap\intr(\jor_b)$ is nonempty. \medskip

\noindent \textbf{Step 3}: Argue that one of $\jor_a$ and $\jor_b$ must surround the other.

Consider the two portions of $\jor_b$ between $\Phi(0)$ and $\Phi(1)$.
The portion formed by the path $\Phi$ is shared entirely with $\jor_a$.
The other portion of $\jor_b$ intersects $\jor_a$ only at the endpoints $\Phi(0)$ and $\Phi(1)$,  so this portion otherwise remains either entirely in $\intr(\jor_a)$ or entirely in $\ext(\jor_a)$.
Putting these two observations together, we determine that either $\jor_b\subseteq\jor_a\cup\intr(\jor_a)$ or $\jor_b\subseteq\jor_a\cup\ext(\jor_a)$.
In either case, by (the contrapositive of) Lemma \ref{crossing_lemma}\ref{crossing_lemma_a}, one of $\jor_a$ and $\jor_b$ surrounds the other. 
For concreteness, let us assume $\jor_a$ surrounds $\jor_b$, since the reverse scenario would use exactly the same argument. \medskip

\noindent \textbf{Step 4}: Appeal to Lemma \ref{splitting_lemma}.

Recall that $\vphi_b$ is the portion of $\jor_b$ which is not shared with $\jor_a$ except at the endpoints $\vphi_b(0)=\Phi(0)$ and $\vphi_b(1)=\Phi(1)$. 
Consider any $y=\vphi_b(t)$ for $t\in(0,1)$.
Since $y\notin\jor_a$, we must have either $y\in\intr(\jor_a)$ or $y\in\ext(\jor_a)$.
We claim the former is true.
Indeed, because $y\in\jor_b$, every open neighborhood of $y$ must intersect both $\intr(\jor_b)$ and $\ext(\jor_b)$.
But we have assumed $\intr(\jor_b)\subseteq\intr(\jor_a)$, so every open neighborhood of $y$ must intersect $\intr(\jor_a)$, forcing $y\in\intr(\jor_a)$.
We have thus shown that
\eq{
\{\vphi_b(t):\, t\in(0,1)\} \subseteq \intr(\jor_a).
}
Therefore, we are in the setting of Lemma \ref{splitting_lemma} with $\jor = \jor_a$, $\jor_2 = \jor_b$, and $\vphi=\vphi_b$.
Indeed, in this notation, $\jor_1$ is the union of $\vphi_b$ with some portion of $\jor_a$ (namely the clockwise arc of $\jor_1$ between $\Phi(0)$ and $\Phi(1)$), while $\jor_2=\jor_b$ is the union of $\vphi_b$ with the complementary portion of $\jor_a$ (namely $\Phi$ itself).
See Figure \ref{fig:Jordan_extend} for reference.
The desired conclusion now follows from \eqref{interior_decomp}.
\end{proof}

Recall Definition \ref{circuit_def} for \textit{circuits}, which are Jordan curves formed entirely by edges of $\Z^2$ or entirely by edges of $\wh\Z^2$.

\begin{lemma} \label{area_lemma}
If $\jor_1$ and $\jor_2$ are Jordan curves such that $\intr(\jor_1)\subsetneq\intr(\jor_2)$, then the area of $\intr(\jor_1)$ is strictly less than that of $\intr(\jor_2)$.
Furthermore, if $\cir$ is a circuit, then the area of $\intr(\cir)$ is a positive integer.
\end{lemma}

\begin{proof}
For the first claim of the lemma, it suffices to find some open set $U\subseteq\ext(\jor_1)\cap\intr(\jor_2)$.
By assumption there exists $x\in\intr(\jor_2)\setminus\intr(\jor_1)$.
If $x\in\ext(\jor_1)$, then take $U$ to be any open set which contains $x$ and remains in $\ext(\jor_1)\cap\intr(\jor_2)$.
Otherwise we must have $x\in\jor_1$, and then we take some open set $V\subseteq\intr(\jor_2)$ containing $x$.
This set $V$ must contain some $y\in \ext(\jor_1)$, so we take $U$ to be any open set containing $y$ which remains in $V\cap\ext(\jor_1)\subseteq\ext(\jor_1)\cap\intr(\jor_2)$.

For the second claim, let us consider the case when $\CC$ is a primal circuit; the case of a dual circuit is entirely analogous.
For every $x\in\Z^2$, the set $U_x = x + (0,1)^2$ can have no intersection with $\CC$.
(This is because $\CC$ consists entirely of edges between nearest-neighbor vertices.)
And clearly $U_x$ is connected, so it must lie either entirely in $\intr(\CC)$ or entirely in $\ext(\CC)$.
Since each $U_x$ has area $1$ and the area of $\R^2\setminus\biguplus_{x\in\Z^2}U_x$ is equal to $0$, we conclude that the area of $\intr(\CC)$ is precisely the number of $x$ such that $U_x\subseteq\intr(\CC)$.
This number is positive because the interior of any Jordan curve is open.
\end{proof}

\begin{lemma} \label{lemma:openclosed}
Let $\opc$ be an open primal circuit and $\cdc$ a closed dual circuit.
If $\intr(\opc)\cap\intr(\cdc)$ is nonempty, then one must surround the other.
\end{lemma}

\begin{proof}
Suppose the conclusion were false, in which case Lemma~\ref{crossing_lemma}\ref{crossing_lemma_b} would guarantee that the two circuits intersect.
But since $\opc$ is primal and $\cdc$ is dual, this intersection must be the midpoint of some edge $e$. 
Hence $e$ is both open and closed, a contradiction.
\end{proof}

The following lemma justifies the notion of an ``outermost'' circuit.

\begin{lemma} \label{join_circuits_lemma}
Consider any circuits $\cir_1,\dots,\cir_k$ such that $\bigcap_{i=1}^k\intr(\cir_i)$ is nonempty.
There exists a unique circuit $\cir$ such that
\eeq{ \label{expand_circuits}
\bigcup_{i=1}^k\intr(\cir_i)\subseteq\intr(\cir) \quad \text{and} \quad \cir\subseteq\bigcup_{i=1}^k\cir_i.
}
\end{lemma}

\begin{proof}
First we prove uniqueness.
Consider two circuits $\cir,\wt\cir$ that both satisfy \eqref{expand_circuits} with the given collection $(\cir_i)_{i=1}^k$.
Suppose towards a contradiction that $\cir\neq\wt\cir$.
Then we claim there is some edge $e$ belonging to one of the circuits whose midpoint is in the exterior of the other circuit.
To see this claim, first consider the case that one circuit surrounds the other, say $\intr(\wt\cir)\subseteq\intr(\cir)$.
Then every edge of $\cir$ either belongs to $\wt\cir$ or has its midpoint in $\ext(\wt\cir)$.
Since we assumed $\cir\neq\wt\cir$, there must be at least one edge with the latter property.
On the other hand, if neither circuit surrounds the other, then we can simply appeal to Lemma~\ref{crossing_lemma}\ref{crossing_lemma_a}.

From the claim, we find some $z\in\ext(\wt\cir)$ which is the midpoint of an edge $e$ belonging to $\cir$.
The edge $e$ must belong to $\cir_i$ for some $i$, so every open neighborhood of $e$ has nonempty intersection with both $\intr(\cir_i)$ and $\ext(\cir_i)$.
In particular,
\eq{
\ball_{1/2}(z)\cap\intr(\cir_i)\neq\varnothing.
}
Next notice that because the midpoint $z$ is in the exterior of $\wt\cir$, the open set $\ball_{1/2}(z)$ has no intersection with $\wt\cir$; it follows that
\eq{
\ball_{1/2}(z) \subseteq \ext(\wt\cir).
}
But when read together, the two previous displays imply that $\intr(\cir_i)$ has nonempty intersection with $\ext(\wt\cir)$, which contradicts the hypothesis that $\intr(\cir_i)\subseteq\intr(\wt\cir)$.
We are left to conclude that $\wt\cir$ must be equal to $\cir$.

Having established uniqueness, we look to prove existence. 
We begin with just two circuits.
That is, we wish to find some circuit $\cir$ such that
\eeq{ \label{existence_for_2_join}
\intr(\cir_1)\cup\intr(\cir_2)\subseteq\intr(\cir) \quad \text{and} \quad \cir \subseteq \cir_1\cup\cir_2.
}
If $\intr(\cir_1)\subseteq\intr(\cir_2)$ or $\intr(\cir_2)\subseteq\intr(\cir_1)$, then we simply take $\cir$ to be $\cir_2$ or $\cir_1$, respectively.
So let us assume neither containment is true, and then Lemma~\ref{crossing_lemma}\ref{crossing_lemma_a} implies the existence of some 
$z_\mathrm{ext}\in\ext(\cir_1)\cap\cir_2$.
%
From $z_\mathrm{ext}$ we follow the circuit $\cir_2$ in both directions; by Lemma~\ref{crossing_lemma}\ref{crossing_lemma_b}, in each direction we will encounter a distinct first intersection with $\cir_1$.
Let $\Phi$ denote the subpath of $\cir_2$ between these intersection points and containing $z_\mathrm{ext}$, so that $\Phi$ is entirely in $\ext(\cir_1)$ except at its endpoints.
We are thus in the setting of Lemma \ref{expanding_lemma}, with the possibility to complete $\Phi$ to a full Jordan curve by following $\cir_1$ either clockwise or counterclockwise.
The circuit resulting from one of these directions will surround the circuit resulting from the other direction;
let $\cir$ be the surrounding circuit and $\cir_b$ the surrounded circuit, so that
\eqref{expanding_lemma_eqs} gives
\eeq{ \label{after_one_join_step}
\intr(\cir) = \intr(\cir_1)\uplus\intr(\cir_b)\uplus\Phi.
}
Notice that the edges in $\cir$ belong to either $\cir_1$ or to $\Phi$, which was a subpath of $\cir_2$.
So if $\cir$ surrounds $\cir_2$, then $\cir$ satisfies \eqref{existence_for_2_join}.
Otherwise, we proceed inductively with $\cir$ replacing $\cir_1$.
To complete the proof, we just need to argue that the procedure just performed can only be repeated finitely many times.
Indeed, in light of \eqref{after_one_join_step}, Lemma \ref{area_lemma} tells us that the area of $\intr(\cir)$ is at least 1 unit greater than that of $\intr(\cir_1)$.
At the same time, the area of a circuit using only edges in $\cir_1\cup\cir_2$ has a finite upper bound, so the argument can only be repeated finitely many times.
Once it terminates, the resulting circuit $\cir$ must satisfy \eqref{existence_for_2_join}.

Our final step is to prove existence for general $n$ from the $n=2$ case just handled.
Given any two circuits $\cir_1$ and $\cir_2$ such that $\intr(\cir_1)\cap\intr(\cir_2)$ is nonempty, let $\cir_1\join\cir_2$ denote the unique circuit $\cir$ satisfying \eqref{existence_for_2_join}.
Then given a sequence $(\cir_i)_{i=1}^k$ such that $\bigcap_{i=1}^k\intr(\cir_i)\neq\varnothing$, we inductively define
\eq{
\cir_1' &= \rlap{$\cir_1$,}\phantom{\cir_{i-1}'\join\cir_i}\quad \text{and} \\
\cir_i' &= \cir_{i-1}'\join\cir_i \quad\text{for $i\in\{2,\dots,n\}$}.
}
By simple induction, the final circuit $\cir_n'$ satisfies \eqref{expand_circuits}.
\end{proof}

The following lemma justifies the notion of an ``innermost'' circuit.

\begin{lemma} \label{meet_circuits_lemma}
Let $S\subseteq\R^2$ be a nonempty connected subset of $\R^2$, and consider any collection of circuits $(\cir_i)_{i\in I}$ such that $S\subseteq\intr(\cir_i)$ for every $i$.
There exists a unique circuit $\cir$
such that
\eeq{ \label{collapse_circuits}
S\subseteq\intr(\cir)\subseteq\bigcap_{i\in I}\intr(\cir_i) \quad \text{and} \quad \cir\subseteq\bigcup_{i\in I}\cir_i.
}
\end{lemma}

\begin{proof}
First we prove uniqueness.
Consider two circuits $\cir,\wt\cir$ that both satisfy \eqref{collapse_circuits} with the given collection $(\cir_i)_{i\in I}$.
Suppose toward a contradiction that $\cir\neq\wt\cir$.
Then we claim there is some edge $e$ belonging to one of the circuits whose midpoint is in the interior of the other circuit.
To see this claim, first consider the case that one circuit surrounds the other, say $\intr(\cir)\subseteq\intr(\wt\cir)$.
Then every edge of $\cir$ either belongs to $\wt\cir$ or has its midpoint in $\intr(\wt\cir)$.
Since we assumed $\cir\neq\wt\cir$, there must be at least one edge with the latter property.
On the other hand, if neither circuit surrounds the other, then we can simply appeal to Lemma~\ref{crossing_lemma}\ref{crossing_lemma_a}.

From the claim, we find some $z\in\intr(\wt\cir)$ which is the midpoint of an edge $e$ belonging to $\cir$.
The edge $e$ must belong to $\cir_i$ for some $i$, so every open neighborhood of $e$ has nonempty intersection with both $\intr(\cir_i)$ and $\ext(\cir_i)$.
In particular,
\eq{
\ball_{1/2}(z)\cap\ext(\cir_i)\neq\varnothing.
}
Next notice that because the midpoint $z$ is in the interior of $\wt\cir$, the open set $\ball_{1/2}(z)$ has no intersection with $\wt\cir$; it follows that
\eq{
\ball_{1/2}(z) \subseteq \intr(\wt\cir).
}
But when read together, the two previous displays imply that $\intr(\wt\cir)$ has nonempty intersection with $\ext(\cir_i)$, which contradicts the hypothesis that $\intr(\wt\cir)\subseteq\intr(\cir_i)$.
We are left to conclude that $\wt\cir$ must be equal to $\cir$.

Having established uniqueness, we look to prove existence. 
We begin with just two circuits.
That is, we wish to find some circuit $\cir$ such that
\eeq{ \label{existence_for_2}
S\subseteq\intr(\cir)\subseteq\intr(\cir_1)\cap\intr(\cir_2) \quad \text{and} \quad \cir \subseteq \cir_1\cup\cir_2.
}
If $\intr(\cir_1)\subseteq\intr(\cir_2)$ or $\intr(\cir_2)\subseteq\intr(\cir_1)$, then we simply take $\cir$ to be $\cir_1$ or $\cir_2$, respectively.
Let us assume neither containment is true, and then Lemma~\ref{crossing_lemma}\ref{crossing_lemma_a} implies the existence of some 
$z_\mathrm{int}\in\intr(\cir_1)\cap\cir_2$.
%
From $z_\mathrm{int}$ we follow the circuit $\cir_2$ in both directions; by Lemma~\ref{crossing_lemma}\ref{crossing_lemma_b}, in each direction we will encounter a distinct first intersection with $\cir_1$.
Let $\vphi$ denote the subpath of $\cir_2$ between these intersection points and containing $z_\mathrm{int}$, so that $\vphi$ is entirely in $\intr(\cir_1)$ except at its endpoints.
We are thus in the setting of Lemma \ref{splitting_lemma}, with the possibility to complete $\vphi$ to a full Jordan curve by following $\cir_1$ either clockwise or counterclockwise.
We choose the direction that contains $S$ in its interior and call the resulting circuit $\cir$.
Because $S$ is connected, the decomposition \eqref{interior_decomp} shows that such a choice exists and is unique; we thus have $S\subseteq\intr(\cir)\subsetneq\intr(\cir_1)$.
Notice that the edges in $\cir$ belong to either $\cir_1$ or to $\vphi$, which was a subpath of $\cir_2$.
So if $\cir$ is surrounded by $\cir_2$, then we are done.
Otherwise, we proceed inductively with $\cir$ replacing $\cir_1$.
To complete the proof, we just need to argue that the procedure just performed can only be repeated finitely many times.
Indeed, 
Lemma \ref{area_lemma} tells us that the area of $\intr(\cir)$ is at least 1 unit less than that of $\intr(\cir_1)$.
Since $\intr(\cir)$ has finite area, it is now evident that the argument can only be repeated finitely many times.
Once it terminates, we are left with a circuit $\cir$ satisfying \eqref{existence_for_2}.

Our final step is to prove existence for an arbitrary index set $I$.
There are only countably many distinct circuits in the lattice, so it suffices to assume the index set $I$ is the set of positive integers.
Given any two circuits $\cir_1$ and $\cir_2$ such that $S\subseteq\intr(\cir_1)\cap\intr(\cir_2)$, let $\cir_1\meet\cir_2$ denote the unique circuit $\cir$ satisfying \eqref{existence_for_2}.
Then given a sequence $(\cir_i)_{i=1}^\infty$ such that $S\subseteq\intr(\cir_i)$ for every $i$, we inductively define
\eq{
\cir_1' &= \rlap{$\cir_1$,}\phantom{\cir_{i-1}'\meet\cir_i}\quad \text{and} \\
\cir_i' &= \cir_{i-1}'\meet\cir_i \quad \text{for $i\geq2$}.
}
By definition, we have
\eeq{ \label{collapse_circuits_induction}
S\subseteq\intr(\cir_i')
&\subseteq\intr(\cir_{i-1}')\cap\intr(\cir_i) \\
&\subseteq\intr(\cir_{i-2}')\cap\intr(\cir_{i-1})\cap\intr(\cir_i) \\
&\hspace{1.3ex}\vdots \\
&\subseteq\intr(\cir_1)\cap\cdots\cap\intr(\cir_i).
}
Since there are only finitely many circuits surrounded by $\cir_1$, we must eventually have $\cir_i'$ equal to some $\cir$ for all large $i$.
For this circuit $\cir$, \eqref{collapse_circuits} follows from the fact that \eqref{collapse_circuits_induction} holds for every $i$.
\end{proof}

\section{Proof of Proposition \ref{prop:geod_cons} (construction of the geodesic)} \label{sec:geod_cons_proof}

 For the benefit of the reader, we recall the relevant conventions as well as the statements to be proved. 
 We say a circuit $\cir$ 
 \textit{encloses} a set of vertices $\mathcal{A}$ (i.e.~either $\mathcal A\subseteq\Z^2$ or $\mathcal A\subseteq\wh\Z^2$) if $\mathcal A\subseteq\intr(\cir)$. We say a circuit $\cir$ \textit{surrounds} another circuit $\cir'$ if $\intr(\cir') \subseteq \intr(\cir)$. 
 Every edge $e\in E(\Z^2)$ has two \textit{dual neighbors} (the endpoints of the dual edge $e^\star$), while every vertex $v\in \Z^2$ has four dual neighbors. In both cases, dual neighbors are vertices on the dual lattice. 
A primal path $\primalp$ is \textit{open} if all its edges are open, and we use the following language:
\begin{itemize}
    \item If $\mathcal{A}\subseteq\Z^2$, then we say $\primalp$ starts (ends) at $\mathcal{A}$ if its first (last) vertex is an element of $\mathcal{A}$.
    \item If $\EE$ is a collection of primal edges, then we say $\primalp$ starts (ends) at $\EE$ if its first (last) vertex is an endpoint of some element of $\EE$.
    \item If $\wh\EE$ is a collection of dual edges, then we say $\primalp$ starts (ends) at $\wh\EE$ if its first (last) vertex is a dual neighbor of some element of $\wh\EE$.
\end{itemize}
A dual path  $\dualp$ is \textit{closed} if all its edges are closed, and we use the following language:
\begin{itemize}
    \item If $\mathcal A\subseteq\Z^2$, then we say $\dualp$ starts (ends) at $\mathcal{A}$ if its first (last) vertex is a dual neighbor of some element of $\mathcal{A}$.
    \item If $\EE$ is a collection of primal edges, then we say $\dualp$ starts (ends) at $\EE$ if its first (last) vertex is a dual neighbor of some element of $\EE$.
    \item If $\wh\EE$ is a collection of dual edges, then we say $\dualp$ starts (ends) at $\wh\EE$ if its first (last) vertex is an endpoint of some element of $\wh\EE$.
\end{itemize}

Proposition \ref{prop:geod_cons} states the following.
Let $\mathcal{A}$ and $\mathcal{B}$ be disjoint finite connected subsets of $\Z^2$.  On the full-probability event $\Omega_\infty$ from Definition~\ref{def:omegainf}, there exists a (possibly empty) sequence of edge-disjoint open circuits $\incir_1,\ldots,\incir_L,\outcir_{L + 1},\ldots,\outcir_P$ satisfying the following:
\begin{enumerate}[label=\textup{(\roman*)}]
\item \label{itm:Iinc_new} $\mathcal{A} \subseteq \intr(\incir_1) \subseteq \intr(\incir_2) \subseteq \cdots \subseteq \intr(\incir_L) \subseteq \intr(\incir_L) \cup \incir_L \subseteq \mathcal  B^\complement$. 
\item \label{itm:Idisj_new} $\intr(\incir_{L}) \subseteq \ext(\outcir_{L +1})$.
\item \label{itm:Idec_new}  $\mathcal{A}^\complement \supseteq \intr(\outcir_{L + 1}) \cup \outcir_{L+1} \supseteq \intr(\outcir_{L + 1}) \supseteq \intr(\outcir_{L + 2}) \supseteq \cdots \supseteq \intr(\outcir_P) \supseteq \mathcal{B}$.
\item \label{itm:2donttouch_new}  For $j \in \{1,\ldots,P - 2\}$, the circuits $\incir_j$ and $\incir_{j + 2}$ are vertex-disjoint. 
\item \label{itm:alledual_new}   For $j \in \{1,\ldots,P\}$ and every $e \in \incir_j$, there exists a dual path $\dualp_e$ from $e$ to $\mathcal{A}$ that has exactly $j - 1$ open edges, one crossing each of the circuits $\incir_1,\ldots,\incir_{j - 1}$. 
\end{enumerate}
 Furthermore, there exist
a geodesic $\primalp$ from $\mathcal{A}$ to $\mathcal{B}$, and a disjoint dual path $\dualp$ from $\mathcal{A}$ to $\mathcal{B}$, satisfying the following properties (here we note that the paths $\dualp_e$ in Item~\ref{itm:alledual_new} are not necessarily disjoint from $\primalp$): 
\begin{enumerate}[resume,label=\textup{(\roman*)}]
\item \label{itm:zetaopen_new}  $\dualp$ has exactly $P$ open edges, one crossing each of the circuits $\incir_1,\ldots,\incir_L,\outcir_{L +1},\ldots, \outcir_P$.
\item \label{itm:gamma_on_circuit_new}  For each circuit $\incir_j$, let $x_j$ and $y_j$ be the first and last vertices of $\primalp$ on that circuit. Then, the portion of $\primalp$ between $x_j$ and $y_j$ lies entirely on $\incir_j$. If $\incir_j$ and $\incir_{j + 1}$ are not vertex-disjoint, then $y_j = x_{j+1}$.
\item \label{itm:dualconn_new}  For every open edge $e \in \primalp$ with $e \notin \incir_1\cup \cdots \cup \incir_L \cup \outcir_{L + 1}\cup \cdots \cup \outcir_P$, there exists a closed dual path from $e$ to $\dualp$ that is disjoint from $\gamma$.
\item \label{itm:dualcir_new} The dual of each closed edge along $\primalp$ belongs to a closed circuit $\dincir$ that either contains $\mathcal{A}$ in its interior and $\mathcal{B}$ in its exterior, or vice versa. The circuit $\dincir$ does not contain the dual of any other edges along $\primalp$. 
\item \label{itm:Csequence_new}  With $\{0,1\}$-valued edge-weights, the closed circuits $\dincir$ from Item~\ref{itm:dualcir_new} can be chosen to form a edge-disjoint collection $\dincir_1,\ldots,\dincir_V$. 
For $j\in\{1,\dots,V-2\}$, the circuits $\dincir_j$ and $\dincir_{j+2}$ are vertex-disjoint.
The union of the circuits $\incir_1,\ldots,\incir_P$ and the circuits $\dincir_1,\ldots,\dincir_V$ forms a sequence $\CC_1,\ldots,\CC_K$, which is ordered so that, for some index $W \in \{0,\ldots,K\}$, $\intr(\CC_W) \cap \intr(\CC_{W + 1}) = \varnothing$ (with the convention $\intr(\CC_0) = \intr(\CC_{K +1}) = \varnothing)$, and we have the following inclusions
\[
\mathcal A \subseteq \intr(\CC_1) \subseteq \cdots \subseteq \intr(\CC_W),\quad\text{and} \quad \intr(\CC_{W + 1}) \supseteq \cdots \supseteq\intr(\CC_K) \supseteq \mathcal B.
\]
\end{enumerate}

We recall that, on the event $\Omega_\infty$, there exists a closed circuit $\wh \DD$ enclosing both $\mathcal{A}$ and $\mathcal{B}$. Choose such a circuit $\wh \DD$ in a deterministic fashion. One way to do this is to let $\wh \DD$ be the unique innermost such circuit from Lemma \ref{meet_circuits_lemma}, but it is not necessary to make this particular choice.

\subsection{Construction of the circuits (proof of Items \ref{itm:Iinc_new}--\ref{itm:2donttouch_new})} \label{sec:circons}
Assuming that the set
\begin{subequations} \label{Si}
\eeq{
S_0 = \{\text{open circuits } \cir:\, \AA \subseteq \intr(\CC)  \text{ and } \mathcal B \subseteq \ext(\cir)\}
}
is nonempty,
let $\incir_1$ be the unique innermost circuit from Lemma \ref{meet_circuits_lemma} consisting of edges of circuits in $S_0$ and satisfying
\[
\mathcal  A \subseteq \intr(\incir_1) \subseteq \bigcap_{\cir \in S_0} \intr(\cir).
\]
Then, also, 
\[
\mathcal B \subseteq \bigcup_{\cir \in S_0} (\ext(\cir) \cup \cir) \subseteq \ext(\incir_1) \cup \incir_1,
\]
but points of $\mathcal{B}$ cannot lie on $\incir_1$ because they do not lie on any of the $\cir \in S_0$ by assumption. Hence, $\mathcal{B} \subseteq \ext(\incir_1)$.

Inductively, for $j \ge 1$, assume that $\incir_j$ has been constructed, and assume that the set 
\eeq{ 
S_j = \{\text{open circuits }\cir \text{ edge-disjoint from and surrounding }\incir_j:\, \mathcal B \subseteq \ext(\cir)\}
}
\end{subequations}
is nonempty. Then, let $\incir_{j + 1}$ be the unique innermost circuit chosen from the set $S_j$ from Lemma \ref{meet_circuits_lemma}. Again, $\mathcal{A} \subseteq \intr(\incir_{j +1})$ and $\mathcal{B} \subseteq \ext(\incir_{j + 1})$. 

Let $L$ be the smallest index $j$ such that $S_j = \varnothing$. We see from induction that 
\[
\mathcal A \subseteq \intr(\incir_1) \subseteq \intr(\incir_2) \subseteq \cdots \subseteq \intr(\incir_L) \subseteq \intr(\incir_L) \cup \incir_L \subseteq \mathcal  B^\complement,
\]
so Item \ref{itm:Iinc_new} holds. We now  argue that $L$ must be finite. Take $\dualp$ to be any finite dual path from $\mathcal{A}$ to $\mathcal{B}$. The path $\dualp$ must cross any circuit that encloses $\mathcal{A}$ and keeps $\mathcal{B}$ in its exterior. The path $\dualp$ can only cross $\incir_j$ at the midpoint of an edge, and since the circuits $\incir_1,\dots,\incir_L$ are edge-disjoint, $L$ can be no greater than the finite number of edges in $\dualp$.   

Next, define $S_L'$ to be
\[
S_{L}' = \{\text{open circuits }\cir \text{ edge-disjoint from } \incir_L:\, \intr(\incir_L) \subseteq \ext(\cir) \text{ and }\mathcal B \subseteq \intr(\cir)\}.
\]
In the case $L = 0$, we replace $\intr(\incir_L)$ with $\mathcal{A}$ in the definition of $S_L'$. In the proof below, we often refer to $\intr(\incir_L)$, but the reader should replace this with $\mathcal{A}$ for the case $L = 0$. 

Assuming that $S_L'$ is nonempty, we show that $S_L'$ is finite. Lemma \ref{lemma:openclosed} implies that any open circuit then either surrounds or is surrounded by  $\wh \DD$. If an open circuit surrounds $\wh \DD$, it encloses both $\mathcal{A}$ and $\mathcal{B}$ and is therefore not in $S_L'$. Thus, $S_L'$ is a subset of open circuits surrounded by $\wh \DD$ and is therefore finite. Let $\incir_{L + 1}$ be the unique circuit from Lemma \ref{join_circuits_lemma} satisfying $\bigcup_{\CC \in S_L'} \intr(\CC) \subseteq \intr(\incir_{L+1})$ and $\incir_{L+1} \subseteq \bigcup_{\CC \in S_L'} \CC$    . It follows immediately that $\mathcal{B} \subseteq \intr(\outcir_{L + 1})$ and that $\outcir_{L + 1}$ is edge-disjoint from $\incir_L$. We now prove Item \ref{itm:Idisj_new}, which asserts $\intr(\incir_{L}) \subseteq \ext(\outcir_{L +1})$.
Refer to Figure \ref{fig:AB_next} for an explanation of why this is nontrivial and a visual representation of the contradiction derived in what follows.

\begin{figure}
    \centering
    \includegraphics[height = 2in]{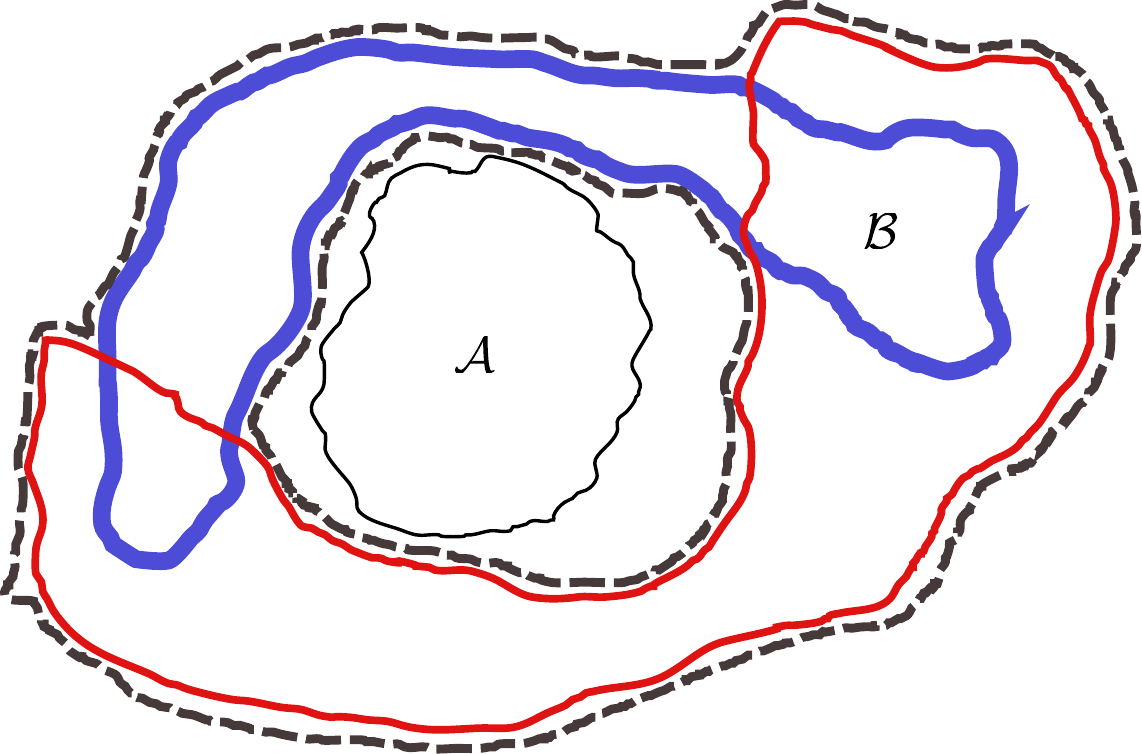}
    \caption{\small In this figure, $\incir_L$ is the circuit enclosing $\mathcal A$ denoted in black/thin. Consider the red/medium thickness and blue/thickest circuits. These both enclose $\BB$ and both keep $\incir_L$ in their exteriors. However, the unique maximal circuit that joins the two circuits as in Lemma \ref{join_circuits_lemma} (the outermore dashed gray circuit) does not keep $\incir_L$ in the exterior. We show that these pairs of circuits cannot exist by showing the existence of the innermore dashed gray circuit built from the two circuits. This contradicts the maximality of $\incir_L$. }
    \label{fig:AB_next}
\end{figure}

Pick arbitrary $\cir_1,\cir_2 \in S_L'$, and let $\cir$ be the unique circuit in Lemma \ref{join_circuits_lemma} satisfying $\intr(\CC_1) \cup \intr(\CC_2) \subseteq \intr(\CC)$ and $\CC \subseteq \CC_1 \cup \CC_2$. By induction, it suffices to show that $\intr(\incir_{L}) \subseteq \ext(\cir)$. Assume, to the contrary, that this is false. The set $\intr(\incir_L)$ cannot contain points on $\cir$ because $\cir$ consists of edges of $\cir_1,\cir_2$, and $\intr(\incir_L)$ is contained in the exterior of both of these. Then, since $\intr(\incir_L)$ is path connected, it cannot contain points in both $\intr(\cir)$ and $\ext(\cir)$ (If $L =0$, we replace $\intr(\incir_L)$ with the path connected set consisting of vertices of $\mathcal{A}$ and all edges connecting them). Hence, our assumption implies that $\intr(\incir_L) \subseteq \intr(\cir)$. Thus, the following set is nonempty:
\[
\wt S = \{\text{open circuits }\wt \cir \subseteq \cir_1 \cup \cir_2 \text{ surrounding }\incir_L \}.
\]
Let $\wt \cir_1$ be the unique circuit from Lemma \ref{meet_circuits_lemma} satisfying $\intr(\wt \cir_1) \subseteq \bigcap_{\wt \cir \in \wt S} \intr(\wt \cir)$ and $\wt \cir_1 \subseteq \bigcup_{\wt \cir \in \wt S} \wt \cir$. By similar reasoning as before, $\mathcal{B}$ must be contained in exactly one of $\intr(\wt \cir_1)$ or $\ext(\wt \cir_1)$. If $\mathcal B\subseteq \ext(\wt \cir_1)$, then $\wt \cir_1 \in S_L$, a contradiction because $S_L$ is empty. Therefore, $\mathcal{B} \subseteq \intr(\wt \cir_1)$. Consider two cases:

\medskip \noindent \textbf{Case 1:} $\cir_1$ and $\wt \cir_1$ have one or fewer vertices in common. Then, since $\wt \cir_1$ consists only of edges of $\cir_1$ and $\cir_2$, we have $\wt \cir_1 = \cir_2$. This is a contradiction because $\wt \cir_1$ surrounds $\incir_{L}$, but $\intr(\incir_{L}) \subseteq \ext(\cir_2)$. 

\medskip \noindent \textbf{Case 2:} $\cir_1$ and $\wt \cir_1$ have at least two vertices in common. We show here that there exists a point along $\cir_1$ in $\intr(\wt \cir_1)$. Once this is shown, we follow $\cir_1$ in each direction along this point until hitting the two vertices of $\wt \cir_1$. Lemma \ref{splitting_lemma} allows us to construct a smaller circuit built from edges of $\cir_1,\cir_2$ that surrounds $\incir_L$, a contradiction to the minimality of $\wt \cir_1$.

We know $\cir_1 \neq \wt \cir_1$ by the same reasoning as in Case 1. Furthermore, the circuit $\cir_1$ cannot surround $\wt \cir_1$ because then, it would surround $\incir_L$. We now have two subcases:

\medskip \noindent \textbf{Subcase 2.1:} $\wt \cir_1$ surrounds $\cir_1$. Then, all points of $\cir_1$ are on or in the interior of $\wt \cir_1$. Not all points on $\cir_1$ can be on $\wt \cir_1$ because no proper subset of a Jordan curve is a Jordan curve and $\wt \cir_1 \neq \cir_1$. 

\medskip \noindent \textbf{Subcase 2.2:} Neither of the circuits $\wt \cir_1,\cir_1$ surrounds the other. Both circuits contain $\mathcal{B}$ in their interior, so Lemma \ref{crossing_lemma} provides the desired point along $\cir_1$ in the interior of $\wt \cir_1$.

\medskip \noindent
Lastly, we construct the circuits $\outcir_{L + 2},\ldots,\outcir_P$. Given $\outcir_j$ for $j \ge L+1$, consider the set 
\[
S_j' = \{\text{open circuits }\cir \text{ enclosing }\mathcal B \text{ that are also surrounded by and edge-disjoint from }\outcir_j \}.
\]

If $S_j'$ is nonempty, let $\outcir_{j + 1}$ be the unique outermost circuit constructed in Lemma \ref{join_circuits_lemma} from the circuits in $S_j'$. Then, $\outcir_{j + 1}$ encloses $\mathcal{B}$ and is edge-disjoint from $\outcir_j$. The points along the edges of $\outcir_{j + 1}$ all lie on or in the interior of $\outcir_j$, so $\outcir_j$ must surround $\outcir_{j + 1}$ by Lemma \ref{surround_lemma}.

The index $P$ is the smallest index $j$ such that $S_j' = \varnothing$. Then, 
\[
\mathcal A^\complement \supseteq \incir_{L + 1} \cup \intr(\outcir_{L + 1})  \supseteq \intr(\outcir_{L + 1}) \supseteq \intr(\outcir_{L + 2}) \supseteq \cdots \supseteq \intr(\outcir_P) \supseteq \mathcal B,
\]
thus proving Item \ref{itm:Idec_new}. 


To establish Item \ref{itm:2donttouch_new}, suppose toward a contradiction that $\incir_j$ and $\incir_{j+2}$ share a vertex $x$.
Since $\incir_j$ and $\incir_{j+2}$ are edge-disjoint, these two circuits exhaust all edges containing $x$, which implies that $x$ is not a vertex on $\incir_{j+1}$.  We now consider four cases:
\begin{itemize}
    \item If $j\le L-2$, then $\incir_{j+1}$ surrounds $\incir_j$, and $\incir_{j+2}$ surrounds $\incir_{j+1}$.
    The first statement implies $x\in\intr(\incir_{j+1})$ by Lemma~\ref{surround_lemma}\ref{surround_lemma_c}, while the second statement implies $x\in\ext(\incir_{j+1})$ by Lemma~\ref{surround_lemma}\ref{surround_lemma_d}.
    These two conclusions are contradictory.
    \item If $j\ge L+1$, then reverse the roles of $\incir_j$ and $\incir_{j+2}$ in the previous bullet.
    \item If $j=L-1$, then $\incir_{j+1}$ surrounds $\incir_j$, and $\intr(\incir_{j+1})\subseteq\ext(\incir_{j+2})$.
    The first statement implies $x\in\intr(\incir_{j+1})$ by Lemma~\ref{surround_lemma}\ref{surround_lemma_c}, so the second statement forces $x\in\ext(\incir_{j+2})$, which contradicts the assumption that $x$ is a vertex of $\incir_{j+2}$.
    \item If $j=L$, then $\intr(\incir_{j})\subseteq\ext(\incir_{j+1})$, and $\incir_{j+1}$ surrounds $\incir_{j+2}$.
    Taking closures in the first statement, we obtain $\intr(\incir_{j})\cup\incir_{j}\subseteq\ext(\incir_{j+1})\cup\incir_{j+1}$.
    Now taking complements yields $\ext(\incir_j)\supseteq\intr(\incir_{j+1})$, so we can reverse the roles of $\incir_j$ and $\incir_{j+2}$ in the previous bullet.
\end{itemize}


\subsection{Construction of the dual paths $\dualp_e$ and $\dualp$ (proof of Items \ref{itm:alledual_new} and \ref{itm:zetaopen_new})} \label{sec:dualcon}
We argue here that there exists a dual path from $\mathcal{A}$ to $\mathcal{B}$ that has exactly $P$ open edges, one crossing each of the circuits $\incir_1,\ldots,\incir_L,\outcir_{L + 1},\ldots,\outcir_P$. This is done by starting at $\mathcal{B}$ and moving backwards to $\mathcal{A}$. There are five stages of the construction:
\begin{enumerate} [label=\textup{(\alph*)}]
\item \label{cp1} Show the existence of a closed dual path from $\mathcal{B}$ to $\outcir_P$.
\item \label{cp2} For $j \in \{L + 1,\ldots, P - 1\}$, given a fixed dual neighbor, $x^\star$, of $\outcir_{j + 1}$, lying in $\ext(\outcir_{j + 1})$, show the existence of a closed dual path from $x^\star$ to $\outcir_{j}$.
\item \label{cp3} Given a fixed dual neighbor $x^\star$ of  $\outcir_{L + 1}$ lying in $\ext(\outcir_{L + 1})$ show the existence of a closed dual path from $x^\star$ to $\incir_L$ (or to $\mathcal{A}$ if $K = 0$). If $P = L$ so that there is no circuit $\outcir_{L+1}$, then
the path is from $\mathcal{B}$ (without specifying the vertex) to $\incir_L$ (or to $\mathcal{A}$ if $L = 0$).
\item \label{cp4} For $j \in \{1,\ldots,L - 1\}$, given a fixed dual neighbor $x^\star$ of  $\incir_{j + 1}$ lying in $\intr(\incir_{j + 1})$, show the existence of a closed dual path from $x^\star$ to $\incir_{j}$.
\item \label{cp5} Given a fixed dual neighbor $x^\star$ of $\incir_1$ lying in $\intr(\incir_1)$, show the existence of a closed dual path from $x^\star$ to $\mathcal{A}$. 
\end{enumerate}

These steps give us a map for constructing the path $\dualp$ in Item \ref{itm:zetaopen_new}: Starting from $\mathcal{B}$, we travel to a dual neighbor of $\outcir_P$. Then, we can cross to the exterior of $\outcir_P$ with a single open edge. From the endpoint of that edge in $\ext(\outcir_P)$, we take a path to a dual neighbor of $\outcir_{P - 1}$, then cross $\outcir_{P - 1}$ with a single edge. Continue this procedure until we reach $\mathcal{A}$.  The construction of the paths $\dualp_e$ from Item \ref{itm:alledual_new} is the same, except that we now start at the edge $e$ and go backwards to $\mathcal{A}$. The choice of the geodesic $\primalp$ will come later. In Section \ref{sec:disj}, we show that $\primalp$ and $\dualp$ can be chosen to be disjoint.

\medskip \noindent \textbf{Part \ref{cp1}}: If $P = L$ (meaning the set $\{\outcir_{L + 1},\ldots,\outcir_P\}$ is empty), then this case is deferred to part \ref{cp3}. Otherwise, we start by showing there is a closed path from $\mathcal{B}$ to $\outcir_P$. Here, we interpret $\mathcal{B}$ as a collection of vertices, and $\outcir_P$ as a collection of edges, so this means that there exists a closed dual path from some dual neighbor of a \textit{vertex} in $\mathcal{B}$ to some dual neighbor of an \textit{edge} in $\outcir_P$.  

Since $\mathcal{B} \subseteq \intr(\outcir_P)$, then every dual neighbor of a point $x \in B$ lies in $\intr(\outcir_P)$. To see this, note that dual vertices cannot lie on $\outcir_P$ so they must either lie in $\intr(\outcir_P)$ or $\ext(\outcir_P)$. If $x \in \intr(\outcir_P)$ and a dual neighbor $x^\star$ lies in $\ext(\outcir_P)$, then the diagonal path connecting the two vertices must cross $\outcir_P$. As $\outcir_P$ lies on the primal lattice, this is not possible. Now, we claim that the set of dual neighbors of vertices in $\mathcal{B}$ is a connected set (connected by dual edges). The set of northwest dual neighbors is connected, since this is just a shifted version of the set $\mathcal{B}$. Likewise, the sets of southwest, northeast, and southeast dual neighbors are all connected. But each of these sets is connected to each other, so the entire set of dual neighbors is connected. 

We perform a modification of the environment as follows. Figure \ref{fig:blocked_W} gives an illustration. 
Set to closed all edges whose midpoint lies on $\outcir_P$ or in $\ext(\outcir_P)$.
Also set to closed all dual edges (along with their primal counterparts) that have both endpoints dual to $\mathcal{B}$. This latter set of edges is connected (since $\BB$ is connected), and we let $W$ be the cluster of closed edges that contains these. Then, there is no  closed dual path from $\mathcal{B}$ to $\outcir_P$ in the original environment if and only if $W$ is bounded in the modified environment.  If $W$ is bounded in  this environment, then Lemma \ref{lemma:separation} guarantees the existence of an open circuit $\CC$ that contains $W$ in its interior. The circuit $\CC$ is surrounded by and edge-disjoint from $\outcir_P$ because all edges on or in the exterior of $\outcir_P$ are closed. The open circuit $\CC$ must also be open in the original environment. 

\begin{figure}
    \centering
    \includegraphics[height = 2in]{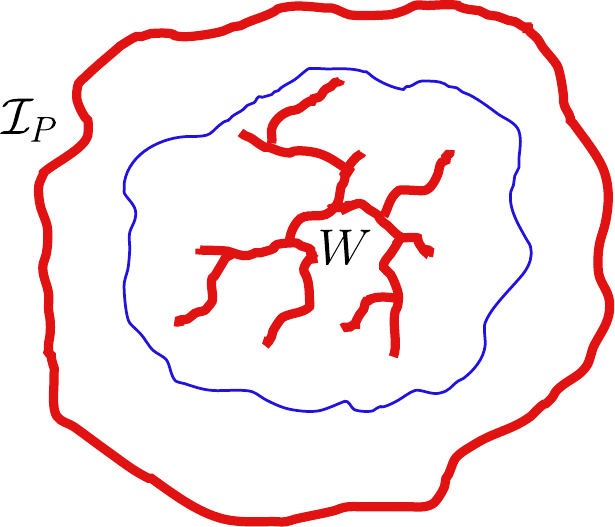}
    \caption{\small Closed edges are denoted by red/thick, and open edges are denoted by blue/thin. The edges on $\incir_P$ have been set to closed, as well as all edges outside $\incir_P$ and all  edges connecting vertices of $\BB$. In this example, the closed component $W$ of $\BB$ is bounded, so $B$ is not connected to $\incir_P$ by a closed path, and there exists an open circuit enclosing $W$.}
    \label{fig:blocked_W}
\end{figure}

Since $\outcir_P$ is the last circuit in the sequence, we derive a contradiction once we show that $\CC$ also contains $\mathcal{B}$ in its interior. Any vertex $x \in B$ cannot lie in $\ext(\CC)$ because otherwise, the diagonal path from $x$ to any of its dual neighbors (which lie in $\intr(\CC)$) must cross $\CC$. But $\CC$ consists of primal edges and so this is not possible. The only remaining possibility for $\mathcal B \not \subseteq \intr(\CC)$  comes when $x$ lies on the circuit $\CC$. Then, exactly two edges of $\CC$ share the vertex $x$. Take one of these edges $e$. By Lemma \ref{lem:dual_inext}, one of the two dual neighbors of $e$ lies in $\ext(\CC)$. But this dual neighbor of the edge $e$ is a dual neighbor of the vertex $x$, which is contained in $\intr(\CC)$ by assumption. 

\medskip \noindent \textbf{Part \ref{cp2}}: 
Similarly as in part \ref{cp1}, modify the environment so that all edges with midpoint on $\outcir_j$ or in $\ext(\outcir_j)$ are set to closed.
Further, since $x^\star$ is a dual neighbor of $\outcir_{j + 1}$, there exists a dual edge $e^\star$ crossing $\outcir_{j + 1}$. Set $e^\star$ and its corresponding primal edge to closed. If there is no closed path from $x^\star$ to $\outcir_j$ in the original environment, then in the modified environment, the closed cluster of dual edges connected to $x^\star$ is bounded. Lemma \ref{lemma:separation} impies the existence of an open circuit $\CC$ that contains $x^\star$ in its interior. $\CC$ is surrounded by and edge-disjoint from $\outcir_j$ because all edges on or in the exterior of $\outcir_j$ were set to closed. The other endpoint of $e^\star$ also lies in $\intr(\CC)$ because $e^\star$ was set to closed and therefore cannot cross $\CC$. But since $x^\star \in \ext(\outcir_{j + 1})$, this other endpoint lies in $\intr(\outcir_{j+1})$ by Lemma \ref{lem:dual_inext}. By Lemma \ref{join_circuits_lemma}, there exists a circuit $\CC'$ which (i) consists entirely of edges of $\outcir_{j + 1}$ and $\CC$ (and therefore is edge-disjoint from and surrounded by $\outcir_j$), and (ii) surrounds $\outcir_{j + 1}$. This contradicts the definition of $\outcir_{j + 1}$ as the outermost next circuit in the sequence after $\outcir_j$. 

\medskip \noindent \textbf{Part \ref{cp3}:} Recall the definition of the circuit $\wh \DD$ at the beginning of this section and the observation that all the circuits $\incir_1,\ldots,\incir_L,\outcir_{L + 1},\ldots,\outcir_P$ are contained in $\intr(\wh \DD)$. 

We consider a modification of the environment as follows. If $L = P$ (meaning there is no circuit $\outcir_{L + 1}$), then consider the connected set of dual vertices that are dual neighbors of at least one vertex in $\mathcal{B}$. Set all the dual edges connecting this set to closed (along with their primal counterparts). Call the cluster of closed dual edges connected to these edges $W_{\mathcal B}$. If $P > L$, then we are given a fixed dual neighbor $x^\star$ of $\outcir_{L + 1}$ lying in $\ext(\outcir_{L + 1})$. The vertex $x^\star$ is the endpoint of an edge crossing $\outcir_{L + 1}$: set that edge (and its primal counterpart) to closed. 
Also set to closed all edges whose midpoint is in $\intr(\outcir_{L+1})$.
Then, $W_{\mathcal B}$ is the cluster of dual closed edges that connects to these edges inside $\outcir_{L + 1}$. 

We construct another closed cluster, denoted $W_{\mathcal A}$, as follows. If $L = 0$ so that there is no circuit $\incir_L$, then we set all dual edges connecting dual neighbors of $\mathcal A$ to closed and let  $W_{\mathcal A}$ be the closed dual cluster of these edges.  
If $L > 0$, then we set all edges whose midpoint lies on $\incir_L$ or in $\intr(\incir_L)$ to closed, and we let $W_{\mathcal A}$ be the closed cluster of the associated dual edges in the modified environment.
We also set all edges whose midpoint is in $\ext(\wh \DD)$ to closed, and let $W_{\wh \DD}$ be the infinite cluster of closed edges connected to $\wh \DD$ in the modified environment. 

Now, assume, by way of contradiction, that in the orignial environment, there is no closed dual path from $x^\star$ (or from $\mathcal{B}$ if $L = P$) to $\incir_{L}$  (or to $\mathcal{A}$ if $L = 0$). Then, in the modified environment, $W_{\mathcal A} \neq W_{\mathcal B}$. Lemma \ref{lem:separate_sets} implies the existence of an open circuit $\CC$, lying in the interior of $\DD$ that contains $W_{\mathcal A}$ in its interior and $W_{\mathcal B}$ in its exterior, or vice versa (it may be helpful to refer back to Figure \ref{fig:two_clusters}). The circuit $\CC$ must also be open in the original environment.  

In the case that $\CC$ contains $W_\BB$ in its interior and $W_\AA$ in its exterior, $\CC$ is edge-disjoint from $\incir_L$ because all the edges on $\incir_L$ were set to closed in the modified environment. Additionally, $\CC$ contains the point $x^\star$ in its interior, which lies in $\ext(\outcir_{L + 1})$ by assumption. $\CC$ and $\outcir_{L + 1}$ also share points in their interior because they both enclose $\mathcal{B}$. Lemma \ref{join_circuits_lemma} implies the existence of another open circuit $\CC'$ that is edge-disjoint from $\incir_L$ and such that $\intr(\outcir_{L +1})\cup \intr(\CC) \subseteq \outcir(\CC')$. This is a contradiction to our construction of the outermost circuit $\outcir_{L + 1}$ in Section \ref{sec:circons}.
In the case that $\CC$ contains $W_{\mathcal A}$ in its interior and $W_{\mathcal B}$ in its exterior, since all edges on or inside $\incir_L$ were set to closed, $\CC$ is edge-disjoint from and surrounds $\incir_L$. This contradicts the definition of $\incir_L$ as the last edge-disjoint innermost circuit enclosing $\mathcal{A}$ but containing $\mathcal{B}$ in its exterior.

\medskip \noindent \textbf{Part \ref{cp4}:} The chosen dual vertex $x^\star$ of $\incir_{j +1}$ is the endpoint of a dual edge that crosses $\incir_{j + 1}$. Set that edge (along with its primal counterpart) to closed.
Additionally, set all edges whose midpoint is in $\ext(\incir_{j + 1})$ to closed.
Set all edges whose midpoint lies on $\incir_j$ or in $\intr(\incir_j)$ to closed.
Then, there is no closed dual path between $x^\star$ and $\incir_{j}$ in the original environment if and only if, in the modified environment, the cluster of dual closed edges containing those edges passing through $\incir_j$ is bounded. In the latter case, Lemma \ref{lemma:separation} implies that there is an open circuit $\CC$ surrounding $\incir_j$ in the modified environment (which is also open in the original environment). This open circuit is edge-disjoint from $\incir_j$.
The point $x^\star$ lies in $\ext(\CC)$ since, in the dual environment, there is an infinite, closed dual path starting from $x^\star$--this path cannot cross $\CC$. 

Similarly as in part \ref{cp2}, Lemma \ref{meet_circuits_lemma} contradicts the construction of $\incir_{j +1}$ as the next innermost circuit.


\medskip \noindent \textbf{Part \ref{cp5}:} This is almost the same as part \ref{cp4}. The only difference is that, instead of setting edges on and inside $\incir_j$ to closed, we set all edges which connect vertices in $\mathcal{A}$ to closed. The rest of the proof goes the same.


\subsection{Existence of closed circuits hitting the closed edges (proof of Items \ref{itm:gamma_on_circuit_new}  and \ref{itm:dualcir_new})} \label{sec:closed_cir}
Take $\primalp$ to be any geodesic from $\mathcal{A}$ to $\mathcal{B}$.
For each $j$, let $x_j$ and $y_j$ be the first and last vertices of $\primalp$ on $\incir_j$. Then, we modify the portion of $\primalp$ between $x_j$ and $y_j$ to follow the open circuit $\incir_j$ (in either direction) if it does not already. 
We can similarly force $y_j = x_{j+1}$ if $\incir_j$ and $\incir_{j +1}$ are not vertex-disjoint. This concludes Item \ref{itm:gamma_on_circuit_new}.  

To prove Item \ref{itm:dualcir_new}, enumerate the closed edges along the path $\gamma$, in order starting from $\mathcal{A}$ as $e_1,\ldots,e_V$, and choose one of the edges $e_i$.  We continue to work on the event $\Omega_\infty$ and this time make use of a deterministically chosen open circuit $\wh \CC$ that encloses both $\mathcal{A}$ and $\mathcal{B}$. 
Modify the environment in the following manner. 
Set all edges whose midpoint is in $\ext(\wh \CC)$ to open. 
Set all edges along $\primalp$ to open, except for the edge $e_i$.
Set all edges connecting two vertices in $\mathcal{A}$ to open, as well as all edges connecting two vertices in $\mathcal{B}$. 
This creates three open clusters of edges: the cluster $W_{\mathcal A}$ consisting of all open edges with an open path to $\mathcal{A}$, the cluster $W_{\mathcal B}$ of open edges with a path to $\mathcal{B}$, and the unbounded cluster $W_{\wh \CC}$ of the circuit $\wh \CC$. 

As an intermediate step, we now argue that $W_{\mathcal A} \neq W_{\mathcal B}$. Suppose, by way of contradiction, that $W_{\mathcal A} = W_{\mathcal B}$. Then, there exists an open path $\wt \gamma$ from $\mathcal{A}$ to $\mathcal{B}$ in the modified environment. 
If $\wt \gamma$ exits the interior of $\wh\CC$, then it must also re-enter (to reach $\mathcal{B}$); in this case, one could reroute $\wt \gamma$ along the open circuit $\wh\CC$ so that it never enters the exterior of $\wh\CC$.
Then, in the original environment, the only (potentially) closed edges along $\wt \gamma$ are the edges $e_j$ for $j \neq i$, since these are the only edges that were switched from closed to open and have empty intersection with the exterior of $\wh\CC$. However, $\wt \gamma$ does not take the closed edge $e_i$, so the passage time of $\wt \gamma$ (in the original environment) is strictly less than the passage time along $\primalp$. This contradicts the fact that $\primalp$ is a geodesic.

Now that we have established $W_{\mathcal A} \neq W_{\mathcal B}$, the dual counterpart of Lemma \ref{lem:separate_sets} implies the existence of a closed circuit $\dincir$ lying in $\intr(\wh \CC)$ that contains $W_{\mathcal A}$ (and therefore also $\mathcal{A}$) in its interior and $W_{\mathcal B}$ (therefore also $\mathcal{B}$) in its exterior, or vice versa. The geodesic from $\mathcal{A}$ to $\mathcal{B}$ thus passes through vertices in $\ext(\UU)$ and vertices in $\intr(\UU)$. The path between these must then cross $\UU$. Since all edges on the geodesic $\primalp$ except for $e_i$ were set to open, the edge $e_i$ is the only edge on $\primalp$ whose dual belongs to $\dincir$.

\subsection{Separate argument in the Bernoulli case (proof of Item \ref{itm:Csequence_new})}
In the case of Bernoulli weights, the dual version of the construction of the circuits $\incir_1,\ldots,\incir_L,\outcir_{L + 1},\ldots,\outcir_P$ allows us to construct successively innermost edge-disjoint closed circuits $\dincir_1,\ldots,\dincir_{U}$ enclosing $\mathcal{A}$, followed by successively outermost edge-disjoint closed circuits $\doutcir_{U + 1},\ldots,\doutcir_V$ enclosing $\mathcal{B}$. 
Since the circuits are edge-disjoint, any geodesic $\primalp$ must pass through each circuit at least once. By a symmetric argument to the construction in Section \ref{sec:dualcon}, we may construct a primal path $\primalp$ that crosses each of the circuits $\dincir_1,\ldots,\dincir_{U},\doutcir_{U + 1},\ldots,\doutcir_V$ exactly once. This path is therefore a geodesic. 

 By Lemma \ref{lemma:openclosed}, for each pair of open and closed circuit in this sequence, they either have disjoint interior, or one surrounds the other. This gives a natural ordering of the sequence $\CC_1,\ldots,\CC_K$: we start with the circuits that enclose $\mathcal{A}$, starting from the innermost, then move to the circuits that enclose $\mathcal{B}$, starting from the outermost, and this implies that, for some index $W \in \{0,\ldots,K\}$, $\intr(\incir_{W}) \cap \intr(\incir_{W + 1}) = \varnothing$. Furthermore, 
 \[
\mathcal A \subseteq \intr(\CC_1) \subseteq \cdots \subseteq \intr(\CC_W),\quad\text{and} \quad \intr(\CC_{W + 1}) \supseteq \cdots \intr(\CC_K) \supseteq B.
\] 
Finally, the fact that $\dincir_j$ and $\dincir_{j+2}$ are vertex-disjoint is proved exactly as was Item \ref{itm:2donttouch_new} at the end of Section~\ref{sec:circons}.

\subsection{Modification of $\primalp$ and $\dualp$ to be disjoint} \label{sec:disj} 
Recall that $\primalp$ is a primal path while $\dualp$ is a dual path. If $\primalp$ and $\dualp$ meet at the midpoint of an open edge, then that edge must lie along one of the circuits $\incir_1,\ldots,\incir_L,\outcir_{L + 1},\ldots,\outcir_{P}$, and so $\primalp$ contains an edge $e$ that lies on this circuit. 
Denote the circuit at which the intersection occurs by $\CC$. There exists an open path between the endpoints of $e$ that goes the opposite direction around $\CC$, thus avoiding $e$. Since $\dualp$ only crosses $\CC$ once, by rerouting $\primalp$ in this manner, we have avoided intersection with $\dualp$ at $e$ without incurring additional passage time. 

If $\primalp$ and $\dualp$ intersect at the midpoint of a closed edge, then we showed that this edge of $\primalp$ is the primal counterpart of an edge on a closed circuit. We similarly reroute $\dualp$ to avoid this edge.

We make the observation here that, because the circuits $\incir_1,\ldots,\incir_P$ (as well as the dual circuits containing the closed edges along $\gamma$) are not necessarily vertex-disjoint, this procedure could have created paths $\primalp$ and/or $\dualp$ that are not self-avoiding. If this is the case, we can simply create self-avoiding paths by deleting some of the edges along the paths.

\subsection{Modifying the geodesic $\primalp$ between any two of its closed edges} \label{JCcons}
We perform a final modification of the geodesic $\gamma$ to place it as close as possible to the dual path $\dualp$.
With the circuits $\incir_1,\ldots,\incir_L,\outcir_{L + 1},\ldots,\outcir_P$ and the geodesic $\primalp$ constructed thus far, for $i \in \{0,\ldots,P\}$, let $e_{i,1},\ldots,e_{i,J_i}$ be the (possibly empty) list of closed edges of the geodesic between the $i$-th circuit ($\AA$ if $i=0$) and the $(i+1)$-th circuit ($\BB$ if $i=P$).
These edges are listed in the order they appear on $\gamma$ as it travels from $\AA$ to $\BB$. 
For each $e_{i,j}$, let $x_{i,j}^1$ be its first vertex along the path, and let $x_{i,j}^2$ be the second vertex. 
We make modifications to 9 types of open subpaths:
\begin{enumerate}[label=\textup{\arabic*.},ref=\textup{\arabic*}]
\item \label{sce_1} 
Path from $\mathcal{A}$ to $\mathcal{B}$ (if the list $\incir_1,\ldots,\incir_L,\outcir_{L + 1},\ldots,\outcir_P$ is empty, and there are no closed edges on $\primalp$).
\item \label{sce_2} 
Path from $\mathcal{A}$ to $\incir_1$ (if there are no closed edges on $\primalp$ before it reaches $\incir_1$).
\item \label{sce_3}
Path from $\mathcal{A}$ to $x_{0,1}^1$.
\item \label{sce_4}
Path from $\incir_P$ to $\mathcal{B}$ (if there are no closed edges on $\primalp$ after it leaves $\incir_P$).
\item \label{sce_5}
Path from $x_{P,J_P}^2$ to $\mathcal{B}$.
\item \label{sce_6} 
Path from $\incir_i$ to $\incir_{i+1}$ (if there are no closed edges of $\primalp$ between the two circuits).
\item \label{sce_7} 
Path from $\incir_i$ to $x_{i,1}^1$. 
\item \label{sce_8} 
Path from $x_{i,J_i}^2$ to $\incir_{i + 1}$. 
\item \label{sce_9} 
Path from $x_{i,j}^2$ to $x_{i,j + 1}^1$. 
\end{enumerate}
By Item~\ref{itm:gamma_on_circuit_new}, every open edge $e\in\gamma$ with $e \notin \incir_1 \cup \cdots \cup \incir_L\cup \incir_{L + 1} \cup \cdots \cup \incir_P$ belows to one of these subpaths.
We may replace any one of these subpaths with another open path of the same starting and ending description, and the modified path will still be a geodesic.
If one of these subpaths is empty (e.g.\ $x_{0,1}^1\in\AA$ in scenario~\ref{sce_2}, $\incir_{i}$ and $\incir_{i+1}$ share a vertex in scenario~\ref{sce_6}, or $x_{i,j}^2 = x_{i,j + 1}^1$ in scenario~\ref{sce_9}), then no modification is made to that subpath.

We reduce the number of cases by careful use of the following definitions.
When the path starts at $\mathcal{A}$ or $\incir_i$ for some $i$, we say the first end of the path is \textit{free}. 
This means that our modification is allowed to change the starting point (to a different element of $\mathcal A$, or to a different vertex of $\incir_i$) provided the new subpath is still open.
The portion of $\gamma$ on the open circuit $\incir_i$ can be rerouted to reach this new starting point, by using an arc of $\incir_i$ that avoids the single intersection point of $\incir_i$ with $\dualp$.
Similarly, when the path ends at $\mathcal{B}$ or $\incir_i$, we say the second end is \textit{free}.
Otherwise we say the end is \textit{fixed}.
This notation will greatly simplify the construction.

For a given pair of ends (either free or fixed), we let $\Pi$ be the set of open paths between them.
When the first end is free, we assume that every $\pi\in\Pi$ intersects $\AA$ or the associated open circuit $\incir_i$ only at its initial vertex.
Similarly, when the second end is free, we assume that every $\pi\in\Pi$ intersects $\BB$ or the associated open circuit $\incir_{i+1}$ only at its terminal vertex. 
These assumptions guarantee that $\pi$ and $\dualp$ are disjoint, since the only open edges on $\dualp$ are dual to some edge on an open circuit.

Ultimately our goal is to choose $\pi \in \Pi$ such that every edge $e\in\pi$ has a closed dual connection to $\dualp$. 
The existence of this dual connection will be proved in Section~\ref{subsec_closed_arms_exist}. 
The preliminary task here is to associate with each $\pi \in \Pi$ a Jordan curve $\jor_\pi$ that lives on nearest-neighbor grid associated to $(\tfrac12\Z)^2$.
By a simple modification of Lemma \ref{area_lemma}, the area enclosed by each such Jordan curve is an integer multiple of $\f{1}{4}$.
So once $\jor_\pi$ is defined for every $\pi\in\Pi$, we will simply choose $\pi$ such that this area is minimal, breaking ties according to some deterministic total ordering of all finite paths. 

We construct the Jordan curve $\jor_\pi$ in three separate and exhaustive cases.
Figure \ref{fig:fixed_free} illustrates the case when one end is free and one end is fixed.
To unify the presentation, we declare that any edge $e$ such that $e\subseteq\AA$ or $e\subseteq\BB$ is automatically open.
The important features of our construction can then be stated as follows\footnote{In \eqref{jor_pi_open}  we mean not only that both vertices of $e$ lie on the Jordan curve $\jor_\pi$, but also that the line segment between the two vertices is a subset of the curve. The word \textit{entirely} is meant to distinguish from the case when only half this line segment belongs to the curve.}:
\eeqs{ \label{jor_pi_open}
\pbox{0.6\textwidth}{\centering every primal edge $e$ that lies entirely on $\jor_\pi$ is open;} \\
\label{jor_pi_property}
\pbox{0.58\textwidth}{\centering every dual vertex $x\in\wh\Z^2\cap\jor_\pi$ is connected to $\dualp$ by a closed dual path that does not intersect the geodesic.}
}
The first feature \eqref{jor_pi_open} will be obvious from construction.
The second feature \eqref{jor_pi_property} will also be obvious since the portion of $\jor_\pi$ on the dual lattice will be entirely closed and include a segment of $\zeta$, except possibly in Subcase 3.2 where we will say more.
The following observation will be implicitly used several times: any circuit that contains $\mathcal{A}$ in its interior and $\mathcal{B}$ in its exterior (or vice versa) must intersect the path $\dualp$.
\begin{figure}
    \centering
    \includegraphics[height = 2in]{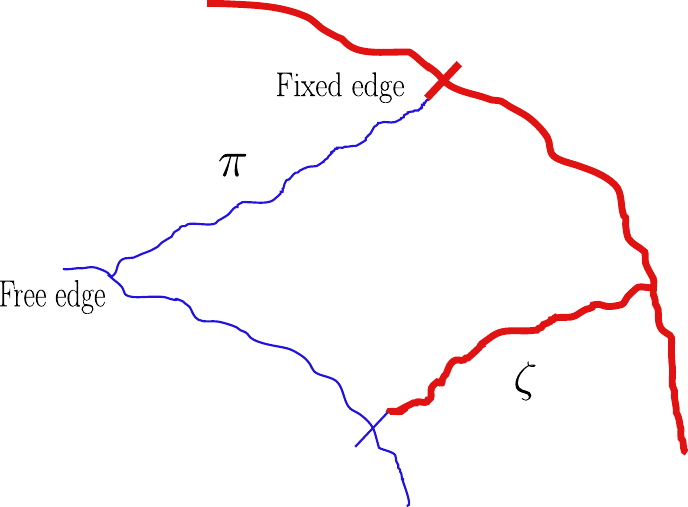}
    \caption{\small The case when one end is free and one end is fixed. Open edges are shown in blue/thin and closed edges are shown in red/thick. The fixed end crosses through a closed dual circuit with a closed primal edge. The path $\zeta$ crosses one of the $\incir_i$ with an open dual edge. We use edges and half-edges to create a Jordan curve seen in the figure.  }
    \label{fig:fixed_free}
\end{figure}

\medskip \noindent \textbf{Case 1: Both ends are free (scenarios \ref{sce_1}, \ref{sce_2}, \ref{sce_4}, \ref{sce_6})}. 
Traverse the open path $\pi$ between its ends. The second end belongs to either an open circuit $\incir_{j+1}$ or to $\BB$.
In the former case, follow $\incir_{j+1}$ (in either direction) until it crosses $\dualp$.
In the latter case, take a self-avoiding path in $\mathcal{B}$ (this path is automatically open by convention) until we reach a primal neighbor of the last vertex of $\dualp$; then connect to this vertex using half of a primal edge and half of a dual edge. 
Now comes the closed dual portion: follow $\dualp$ backwards until reaching either $\mathcal{A}$ or the open circuit containing the first end of $\pi$.
Finally, by mimicking the two cases above in reverse, create a connection to the first end of $\pi$.
A Jordan curve has now been formed, since $\dualp$ intersects each free end only once, and---for scenario \ref{sce_6}---the two open circuits are assumed to be vertex-disjoint (otherwise this subpath of the geodesic is empty).

\medskip \noindent \textbf{Case 2: One end is free, one end is fixed (scenarios \ref{sce_3}, \ref{sce_5}, \ref{sce_7}, \ref{sce_8})}. Without loss of generality, assume that the free end comes second as we travel from $\mathcal{A}$ toward $\mathcal{B}$. 
The fixed end is thus the second vertex $x^2$ of some closed edge $e = \{x^1,x^2\}$.
Let $\dincir$ be the closed dual circuit associated with $e$ (constructed in Section \ref{sec:closed_cir}), and recall that $\UU$ intersects the geodesic only once: at the midpoint of $e$.
To form the Jordan curve, traverse $\pi$ from $x^2$ until meeting the free end. 
Connect this path to $\dualp$ just as in Case 1.
Next create the ``closed dual portion'': follow $\dualp$ backwards until meeting $\dincir$, and then follow $\dincir$ (in either direction) until reaching an endpoint of $e^\star$.
Finally, connect to $x^2$ using half of the dual edge $e^\star$ and then half of the primal edge $e$.
A Jordan curve has now been formed, since $\dualp$ intersects the free end only once.

\medskip \noindent \textbf{Case 3: Both ends are fixed (scenario \ref{sce_9})}. 
The two ends are $x_{i,j}^2$ and $x_{i,j+1}^1$.
Denote the associated edges by $e = \{x_{i,j}^1,x_{i,j}^2\}$ and $e'=\{x_{i,j+1}^1,x_{i,j+1}^2\}$, and let $\dincir$ and $\dincir'$ be the associated closed dual circuits, recalling once more that $\dincir$ and $\dincir'$ intersect the geodesic only at the midpoints of $e$ and $e'$, respectively.

\medskip \noindent \textbf{Subcase 3.1: $\dincir$ and $\dincir'$ are vertex-disjoint.} 
Traverse $\pi$ from $x_{i,j}^2$ to $x_{i,j+1}^1$.
Connect to vertex of $\dincir'$ by using half of the primal edge $e'$ and then (either) half of the dual edge $(e')^\star$.
Now create the ``closed dual portion'': follow $\dincir'$ (in either direction) until meeting $\dualp$, follow $\dualp$ until meeting $\dincir$, and follow $\dincir$ (in either direction) until reaching an endpoint of $e^\star$.
Finally, connect to $x_{i,j}^2$ using half of the dual edge $e^\star$ and then half of the primal edge $e$.
A Jordan curve has now been formed, since we are assuming $\dincir$ and $\dincir'$ are vertex-disjoint.

\medskip \noindent \textbf{Subcase 3.2: $\dincir$ and $\dincir'$ share at least one vertex.} 
We begin as in the previous subcase:
traverse $\pi$ from $x_{i,j}^2$ to $x_{i,j+1}^1$, and connect to a vertex of $\dincir'$ by using half of $e'$ and (either) half of $(e')^\star$.
But the ``closed dual portion'' is formed differently: follow $\dincir'$ (in either direction) until reaching $\dincir$, then follow $\dincir$ (in either direction) until reaching an endpoint of $e^\star$ (it is possible that $\dincir$ hits $\dincir'$ again, but it will not hit the portion of $\dincir'$ used as part of the Jordan curve since we transition to $\dincir$ at the \textit{first} place the two circuits meet).
Finally, connect to $x_{i,j}^2$ using half of the dual edge $e^\star$ and then half of the primal edge $e$.
A Jordan curve has now been formed as noted parenthetically, and \eqref{jor_pi_property} holds because the only dual vertices on $\jor_\pi$ belong to the closed circuits $\dincir$ and $\dincir'$, both of which can be followed to $\dualp$ in whichever direction avoids the single intersection with the geodesic.

\medskip \noindent We emphasize that even for a \textit{given} path $\pi$, there were several choices made to define $\jor_\pi$ (traversing the circuits in either direction, etc.).
We assume that these choices are made so that the resulting curve encloses minimal area (breaking ties deterministically).
Afterward, we modify the geodesic by choosing $\pi\in\Pi$ such that $\jor_\pi$ encloses minimal area.


\subsection{Existence of closed arms (proof of Item 
\ref{itm:dualconn_new})} \label{subsec_closed_arms_exist}

\begin{figure}
    \centering
    \includegraphics[width=0.4\linewidth]{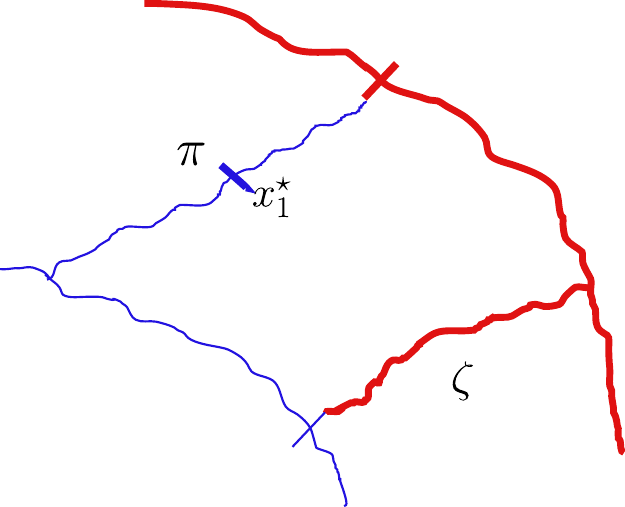} \qquad 
    \includegraphics[width=0.4\linewidth]{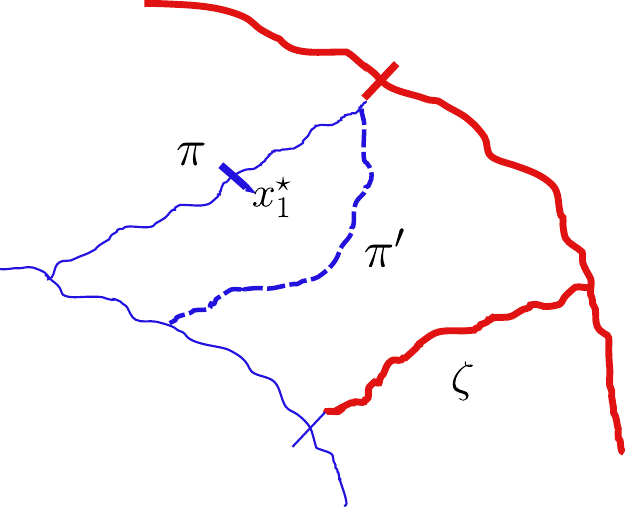}
    \caption{\small On the left, we reproduce Figure \ref{fig:fixed_free}, where $\pi$ is the portion of the path between two successive circuits (open or closed) and $\jor_\pi$ is the associated Jordan curve built from $\pi$, $\zeta$, and portions of the circuits. The point $x_1^\star$ is the endpoint of the open dual edge $e^\star$ that belongs to the interior of $\jor_\pi$. Recall from Section~\ref{JCcons} that $\pi$ was chosen such that $\intr(\jor_\pi)$ has minimal area. We argue that if $x_1^\star$ were not connected to $\zeta$ by a closed dual path, then there would exist an open circuit enclosing $x_1^\star$ that blocks it from $\zeta$. A portion of this open circuit could be used to construct a new path $\pi'$ (depicted on the right) such that $\jor_{\pi'}$ encloses smaller area than $\jor_\pi$, a contradiction to the definition of $\pi$. }
    \label{fig:fixed_free_2}
\end{figure}
See Figure \ref{fig:fixed_free_2} for a visual sketch of the argument.  Let $e \in \primalp$ be an open edge with $e \notin \incir_1 \cup \cdots \cup \incir_L\cup \incir_{L + 1} \cup \cdots \cup \incir_P$. 
Let $\pi\ni e$ be the associated portion of $\primalp$ between successive closed edges or open circuits, as in Section \ref{JCcons}, and let $\jor_\pi$ be the associated Jordan curve enclosing minimal area. 
In each case of the construction, $\pi$ is a segment of $\jor_\pi$. 
In particular, $e$ lies entirely on $\jor_\pi$. 
Let $x_1^\star,x_2^\star$ be the endpoints of the associated dual edge $e^\star$. Since $\jor_\pi$ lives on the union of the primal and dual lattices, either at least one of these points lies on $\jor_\pi$ (in which case the desired path is immediate from \eqref{jor_pi_property}), or Lemma \ref{lem:dual_inext} implies that one of these points belongs to $\intr(\jor_\pi)$ and the other belongs to $\ext(\jor_\pi)$. 
Without loss of generality, assume $x_1^\star\in\intr(\jor_\pi)$.

We perform a modification of the environment as follows. Along the Jordan curve, there are potentially up to $4$ half-edges. 
For all of these, set the associated full edge to closed, if not already closed.  
In addition, set all edges whose midpoint is in $\ext(\jor_\pi)$ to closed.

Suppose the closed dual cluster of $x_1^\star$ is unbounded in the modified environment.
That is, there is a closed dual path $\dualp_1$ starting at $x_1^\star$ and extending to infinity.
Let $y$ be the first intersection point of $\dualp_1$ with $\jor_\pi$, so either $y\in\wh\Z^2$ or $y$ is the midpoint of some primal edge that lies entirely on $\jor_\pi$.
Because of \eqref{jor_pi_open}, the latter is not true, so $y\in\wh\Z^2$.
Since $y$ is the \textit{first} intersection point of $\dualp_1$ with $\jor_\pi$, the portion of $\dualp_1$ from $x_1^\star$ to $y$ does not intersect $\pi$ and is not affected by the modification, so it is closed in the original environment.
In addition, by \eqref{jor_pi_property}, $y$ is connected to $\dualp$ by a closed dual path in the original environment that does not intersect the geodesic.
Therefore, $x_1^\star$ is connected to $\dualp$ by a closed dual path in the original environment that does not intersect the geodesic, as desired.


The only remaining task is to argue that in the modified environment, the closed dual cluster of $x_1^\star$ is unbounded.
Suppose not.
Then Lemma~\ref{lem:separate_sets} gives an open circuit $\CC$ containing $x_1^\star$ in its interior.
This open circuit must lie entirely on or in the interior of $\jor_\pi$ because all edges outside were set to closed. 
Hence $\intr(\CC) \subseteq \intr(\jor_\pi)$ by Lemma~\ref{surround_lemma}.
In particular, both $\jor_\pi$ and $\CC$ contain $x_1^\star$ in their interior and $x_2^\star$ in their exterior, so $e \in \jor_\pi \cap \CC$. 
Let $W$ be the connected component of $\jor_\pi \cap \CC$ containing $e$.
Then, modifying $\pi$ by replacing $W$ with $\CC \setminus W$ results in a new path between the two ends (possibly after changing the starting and/or ending point in the case of free ends).
Let $\pi'$ be the modified path, and let $\jor_{\pi'}$ be the associated Jordan curve.
By examining each of the cases in Section~\ref{JCcons}, it is straightforward to check that $\jor_{\pi'}\subseteq\jor_\pi\cup\CC$.
Since $\CC$ is surrounded by $\jor_\pi$, it now follows from two applications of Lemma~\ref{surround_lemma} that $\jor_{\pi'}$ is surrounded by $\jor_\pi$.
Clearly $\pi\ne\pi'$ (since $W\neq \CC\setminus W$), so
Lemma \ref{area_lemma} tells us that the area of $\intr(\jor_{\pi'})$ is smaller than the area of $\intr(\jor_\pi)$, which contradicts minimality since $\pi'$ is open in the original environment as well. \hfill \qedsymbol

%

\section{Acknowledgments}
We thank Michael Damron for suggesting this problem, many helpful discussions, and invaluable input.
We also thank Jack Hanson for beneficial conversations.
Part of this work was conducted at the International Centre for Theoretical Sciences (ICTS), Bengaluru, India during the program ``First-passage percolation and related models” in July 2022 (code: ICTS/fpp-2022/7).
We thank ICTS for the privilege of attending this meeting and for its hospitality.
We thank the referees for numerous suggestions and corrections; their thorough reports are tremendously appreciated. 

\section{Funding}
E.B. was partially supported by NSF grants DMS-1902734 and DMS-2412473.
D.H. was partially supported by NSF grant DMS-2054559.
X.S. was partially supported by the Wylie Research Fund at the University of Utah.
E.S. was partially supported by the Fernholz Foundation. This work was partly performed while E.S. was a PhD student at the University of Wisconsin--Madison, where he was partially supported by Timo Sepp\"al\"ainen under NSF grants DMS-1854619 and DMS-2152362.

\bibliographystyle{myacm}
\bibliography{erikbib}

 \end{document}